\documentclass[10pt]{article}
\usepackage{amsmath,amssymb,amsthm,enumitem,pst-all,pstricks,pstricks-add,pst-bspline}
\usepackage{textgreek}
\newtheorem{thm}{Theorem}[section]
\newtheorem{thm*}{Theorem}
\newtheorem{proposition}[thm]{Proposition}
\newtheorem{lemma}[thm]{Lemma}

\newtheorem{corollary}[thm]{Corollary}
\newtheorem{hypothesis}[thm]{Hypothesis}
\newtheorem{definition}[thm]{Definition}

\numberwithin{equation}{section}
\renewcommand{\-}{\longrightarrow}
\renewcommand{\=}{\Longrightarrow}
\newcommand{\lap}{\Delta}
\newcommand{\K}{\bar{K}}
\newcommand{\eps}{\varepsilon}

\newcommand{\gdw}{\Longleftrightarrow}

\newcommand{\N}{\mathbb{N}}
\newcommand{\R}{\mathbb{R}}
\newcommand{\e}{\;\exists \;}
\newcommand{\fa}{\;\forall\;}
\renewcommand{\{}{\lbrace}
\renewcommand{\}}{\rbrace}

\newcommand{\var}{\varphi}

\def\Item{\item\abovedisplayskip=0pt\abovedisplayshortskip=0pt~\vspace*{-\baselineskip}}

\date{}

\begin{document}
\title{A scalar curvature flow in low dimensions}
\author{Martin Mayer}

\maketitle
\vfill
\begin{abstract}
\noindent
Let $(M^{n},g_{0})$ be a $n=3,4,5$ dimensional, closed Riemannian manifold
of positive  Yamabe invariant.
For a smooth function $K>0$ on $M$ we consider 
a scalar curvature flow, that tends to prescribe $K$ as the scalar curvature of a metric $g$ conformal to $g_{0}$.
We show global existence and in case $M$ is not conformally equivalent to the standard sphere smooth flow convergence and solubility of the prescribed scalar curvature problem under suitable conditions on $K$.  
\end{abstract}
\vfill
\vfill
\begin{center}
Inauguraldissertation zur Erlangung des Doktorgrades \\
der Naturwissenschaftlichen Fachbereiche der\\
Justus-Liebig-Universit\"at Gie\ss en
\end{center}

\begin{center}
\small{Supervisor}\\
Prof. Dr. Ould Ahmedou
\end{center}
\newpage
\tableofcontents

\newpage
\section{Introduction}
\label{sec:Introduction}
\subsection{Overview and related works}
\label{subsec:OverviewAndRelatedWorks}
We study the problem of prescribing the scalar curvature of a closed Riemannian manifold  within its conformal class, called the prescribed scalar curvature problem. Many work  has been devoted to this topic in the last decades and we refer to \cite{Aubin_book}, \cite{LeeAndParker} and the references therein for an overview.
More precisely we consider the problem of conformally prescribing a smooth function $K>0$ as the scalar curvature in case the underlying manifold already admits a conformal metric of positive scalar curvature.

\smallskip
The problem has variational structure and solutions of the prescribed scalar curvature problem then correspond to critical points of a non negative energy functional $J$, which does not satisfy a compactness criterion known as the Palais-Smale condition. So direct variational methods can not be applied. 
Indeed considering a minimizing or more general a Palais-Smale sequence the possible obstacle of finding a minimizer or a critical point of the associated energy functional 
is, what we call a critical point at infinity - a blow up phenomenon, whose profile however is well understood \cite{StruweConcentrationCompactness}.

\begin{figure}[h]
\begin{center}
\psscalebox{1}
{\begin{pspicture}(-4,-0.5)(8,5.5)\psset{xunit=25pt,yunit=25pt,runit=20pt}
\def\coordinates
  {
  \psaxes[linewidth=0.02,linecolor=gray,labels=none,ticks=none]{->}(0,0)(-3,-1)(9.1,6)
  }

\def\Curve
  {
  \psBspline[linewidth=0.02,fillstyle=none,fillcolor=lightgray,hatchwidth=0.01,showpoints=false]
  (-3,-1)(2,6.4)(3,3.6)(6,2.2)(7,2.1)(8,2)(9,1.96)
  }

\def\HorizontalForCpAtInfinity
 {
 \psline[linewidth=0.01,linecolor=gray,linestyle=dashed,showpoints=false](0,1.95)(9,1.95)
 }
\def\HorizontalForCp
 {
 \psline[linewidth=0.01,linecolor=gray,linestyle=dashed,showpoints=false](0,5)(2,5)
 }
\def\TopChange
 {
 \psline[linewidth=0.01,linecolor=gray,linestyle=dashed,showpoints=false](-3.2,-1)(-3.2,6)
 \psline[linewidth=0.02,linecolor=black,showpoints=false](-3.4,1.95)(-3,1.95)
 \psline[linewidth=0.02,linecolor=black,showpoints=false](-3.4,5)(-3,5)
 }
 
\def\Box 
 {
 \psline[linewidth=0.01,linecolor=gray,linestyle=dashed,showpoints=false]
  (9,1.75)(6.9,1.75)(6.9,2.2)(9,2.2)
 }

\Box
\Curve
\coordinates
\HorizontalForCpAtInfinity
\HorizontalForCp
\TopChange

\put(10.2,3)
{\psscalebox{0.6 0.4}{\psplot[algebraic,linewidth=0.01]{-1.7}{1.4}{4/(1+16*x^2)}}
}
\put(9.7,3)
{\psscalebox{0.5 0.6}{\psplot[algebraic,linewidth=0.01]{-1.7}{1.4}{4/(1+16*x^2)}}
}
\put(10.2,2.9)
{\psscalebox{0.4 0.1.6}{\psplot[algebraic,linewidth=0.01]{-3}{2.1}{1/2*sin(14*x)*cos(10*x)}}
}

\put(-1,6.2){\small $c_{p_{r}}$}
\put(-0.85,2.4){\small $c_{p_{\infty}}$}
\put(-4.7,6.7){\small $0$}
\put(-4.7,4.2){\small $2$}
\put(-4.7,1){\small $1$}
\put(9,1.7){\tiny{near infinity}}
\end{pspicture}
}
\end{center}
\caption{Blow up at infinity and topological contribution}\label{Figure1}
\end{figure}
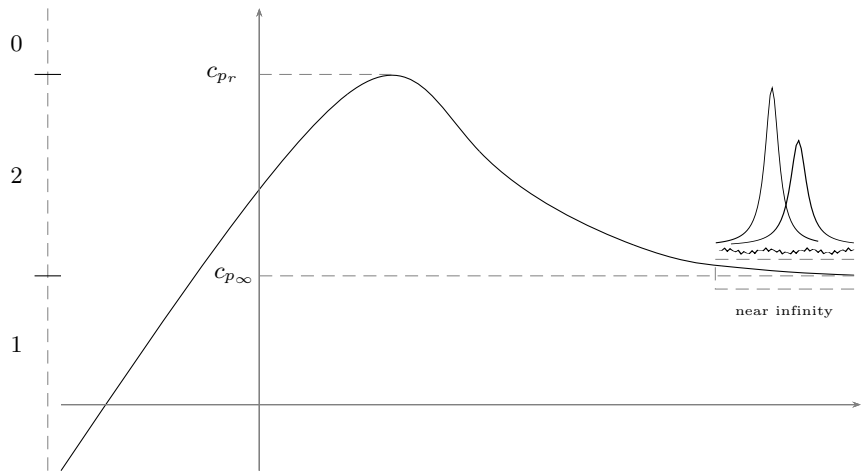

The problem of prescribing a constant scalar curvature is known as the Yamabe problem.
In this case the critical energy levels, at which a blow up may occur, are quantized.
Thus to prove existence of a minimizer, it is sufficient to find a test function, whose 
energy is below the least critical energy level \cite{Aubin_n>=6}, \cite{Schoen_deformation}.
Even, if this is not possible, one can show existence of critical points by analysing the critical points at infinity and their topological contribution to the underlying space as indicated in the above figure, cf. \cite{Bahri_without}, \cite{Bahri_Brezis}, \cite{BahriCoronCriticalExponent} and \cite{BenAyed_Chtioui_Hammami} for some genuine algebraic topological argument.

In addition to these two approaches one may recover solutions by perturbation arguments \cite{Ambrosetti_Malchiodi}, \cite{Chang_Gursky_Yang}.

\smallskip
Besides pure existence results it is a natural idea to find critical points as the limit of the gradient flow or more general of a pseudo gradient flow related to the energy functional. 
In this context one has to show long time existence and flow convergence with the crucial task being to ensure, that a flow line does not escape from the variational space towards a critical point at infinity. 
In the Yamabe case the question of flow convergence reduces to proving, that along a flow line, which becomes highly concentrated, the associated will eventually be below the critical energy levels, at which blow up may occur,  and thus can not blow up at all 
\cite{BrendleArbitraryEnergies}, \cite{Chow}, \cite{StruweLargeEnergies}, \cite{Ye}.

\smallskip
When prescribing the scalar curvature however the critical energy levels are not necessarily quantized. Nonetheless to show existence of a minimizer one may construct a test function with energy strictly below the least critical energy like for the Yamabe problem \cite{AubinHebey}, \cite{Escobar_Schoen} and one may use as well topological arguments to show existence of solutions as critical points \cite{AubinBahri}, \cite{BenAyed_n=4}, \cite{BenAyed&OuldAhmedou}, \cite{YanYanLi_Part1}, \cite{YanYanLi_Part2}.

The strategy of finding solutions by starting a flow is more complicated. The first task is to show long time existence. Secondly one has to prove, that the flow or at least one flow line does not converge to a critical point at infinity instead of a critical point - the ingredient of quantized energy levels being not available. To overcome this deficit one may impose assumptions on the function to be prescribed and therefore on the energy functional to be considered, which ensure a quantization of the critical energy levels \cite{ChenXu}.

\smallskip
One may object, that, when using deformations in the context of topological arguments, some pseudo gradient flow is always used, so there is nothing new. But the freedom of possibly choosing another more suitable pseudo gradient flow, in case some lines of a given flow do blow up, as sketched in figure \ref{Figure2}, is lost, once we limit ourselves to considering one fixed pseudo gradient flow. And a priori there is no equivalence in using different flows.

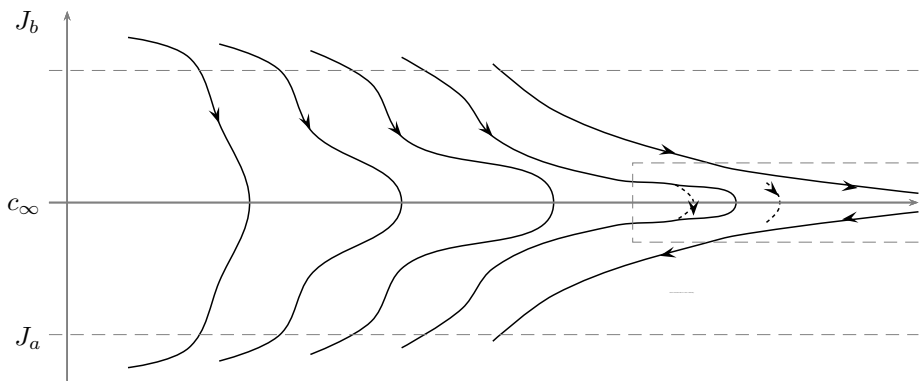
\begin{figure}[h]
\begin{pspicture}(-1,-2.5)(9,2.5)\psset{xunit=23pt,yunit=25pt}
\def\coordinates
  {
  \psaxes[linewidth=0.02,linecolor=gray,labels=none,ticks=none]{->}(0,0)(-0.3,-2.7)(14,2.9)
  }

\def\CurveA
  {
  \pscurve[linewidth=0.02,showpoints=false]
  (1,2.5)(2,2.2)(2.5,1.2)
  (3,0)
  (2.5,-1.2)(2,-2.2)(1,-2.5)
  }

\def\CurveB
  {
  \pscurve[linewidth=0.02,showpoints=false]
  (2.5,2.4)(3.5,2)(4,1.1)
  (5.5,0)
  (4,-1.1)(3.5,-2)(2.5,-2.4)
  }

\def\CurveC
  {
  \pscurve[linewidth=0.02,,showpoints=false]
  (4,2.3)(5,1.8)(5.5,1)
  (8,0)
  (5.5,-1)(5,-1.8)(4,-2.3)
  }

\def\CurveD
  {
  \pscurve[linewidth=0.02,showpoints=false]
  (5.5,2.2)(6.5,1.6)(7,1)(9,0.35)(10,0.27)
  (11,0)
  (10,-0.27)(9,-0.35)(7,-1)(6.5,-1.6)(5.5,-2.2)
  }
  
\def\CurveDdefomed
 {
 \pscurve[linewidth=0.02,linestyle=dashed,dash=0.5mm,showpoints=false]
  (10,0.27)
  (10.3,0)
  (10,-0.27)
 }

\def\CurveE
  {
  \pscurve[linewidth=0.02,linestyle=dashed,dash=0.5mm,showpoints=false]
  (11.5,0.3)(11.7,0.1)
  (11.7,-0.1)(11.5,-0.3)
  }

\def\CurveFPlus
  {
  \pscurve[linewidth=0.02,showpoints=false]
  (7,2.1)(8,1.4)(10.75,0.55)(11,0.5)
  (14,0.14)
  }

\def\CurveFMinus
  {
  \pscurve[linewidth=0.02,showpoints=false]
  (14,-0.14)
  (11,-0.5)(10.75,-0.55)(8,-1.4)(7,-2.1)
  }

\def\HorizontalForJb
 {
 \psline[linewidth=0.01,linecolor=gray,linestyle=dashed,showpoints=false](-0.3,2)(14,2)
 }

\def\HorizontalForJa
 {
 \psline[linewidth=0.01,linecolor=gray,linestyle=dashed,showpoints=false](-0.3,-2)(14,-2)
 }

\def\ArrowA
 {
 \psline[linewidth=0.01,linestyle=none,linecolor=black,showpoints=false,arrows=->,arrowsize=0.15]
 (1,5)(2.5,1.2)
 }

\def\ArrowB
 {
 \psline[linewidth=0.01,linestyle=none,linecolor=black,showpoints=false,arrows=->,arrowsize=0.15]
 (2,5)(4,1.1)
 }

\def\ArrowC
 {
 \psline[linewidth=0.01,linestyle=none,linecolor=black,showpoints=false,arrows=->,arrowsize=0.15]
 (2,5)(5.5,1)
 }

\def\ArrowD
 {
 \psline[linewidth=0.01,linestyle=none,linecolor=black,showpoints=false,arrows=->,arrowsize=0.15]
 (3,5)(7,1)
 }
 
\def\ArrowE
 {
 \psline[linewidth=0.01,linestyle=none,linecolor=black,showpoints=false,arrows=->,arrowsize=0.15]
 (11.2,0.55)(11.7,0.1)
 }
 
\def\ArrowDdeformed
 {
 \psline[linewidth=0.01,linestyle=none,linecolor=black,showpoints=false,arrows=->,arrowsize=0.15]
 (10.3,5)(10.3,-0.2)
 }
 
\def\ArrowFPlus
 {
 \psline[linewidth=0.01,linestyle=none,linecolor=black,showpoints=false,arrows=->,arrowsize=0.15]
 (8,1.2)(10,0.75)
 \psline[linewidth=0.01,linestyle=none,linecolor=black,showpoints=false,arrows=->,arrowsize=0.15]
 (11,0.5)(13,0.22)
 }
 
\def\ArrowFMinus
 {
 \psline[linewidth=0.01,linestyle=none,linecolor=black,showpoints=false,arrows=-<,arrowsize=0.15]
 (8,-1.2)(10,-0.75)
 \psline[linewidth=0.01,linestyle=none,linecolor=black,showpoints=false,arrows=-<,arrowsize=0.15]
 (11,-0.5)(13,-0.22)
 }

\def\Box
 {
 \psline[linewidth=0.01,linestyle=dashed,linecolor=gray,showpoints=false]
 (14,-0.6)(9.3,-0.6)(9.3,0.6)(14,0.6)
 }
 
\CurveA
\CurveB
\CurveC
\CurveD
\CurveDdefomed
\CurveE
\CurveFPlus
\CurveFMinus
\coordinates
\HorizontalForJb
\HorizontalForJa
\ArrowA
\ArrowB
\ArrowC
\ArrowD
\ArrowDdeformed
\ArrowE
\ArrowFPlus	
\ArrowFMinus
\Box
\put(-0.7,2.3){$J_b$}
\put(-0.7,-1.9){$J_a$}
\put(-0.8,-0.1){$c_{\infty}$}
\put(8,-1.2){\psscalebox{0.1}{ \small{\tiny{Flow deformation near infinity}}}}
\end{pspicture}

\caption{Suitable deformation to avoid infinity}\label{Figure2}
\end{figure}

  However, if we do not limit ourselves to use pseudo gradient flows  with just the purpose of finding solutions of the prescribed scalar curvature problem, it is of its own interest to describe the asymptotic behaviour of flow lines qualitatively - those converging to critical points and those diverging to  critical points at infinity.  And this is the aim of this work within its restrictive setting.

\smallskip
We would like to point out, that blowing up flow lines are not an unusual feature of the prescribed scalar curvature problem. On the contrary only under very restrictive assumptions blowing up flow lines can be excluded. 
\subsection{Exposition}
\label{subsec:Exposition}
We wish to give a quick overview on our main arguments.

\smallskip
In subsection \ref{subsec:Preliminaries} we provide the setting of this work, introduce the pseudo gradient flow to be considered, its basic properties and state two theorems, that provide full flow convergence and solubility of the prescribed scalar curvature problem under sufficient conditions on the function $K$ to be prescribed.

\smallskip
Section \ref{sec:LongtimeExistenceAndStrongConvergence} is devoted to prove long time existence and weak convergence of the first variation $\partial J$ along a flow line $u$ in a sense to be made precise. The arguments, we use, are straight forward adaptations from the Yamabe setting \cite{BrendleArbitraryEnergies}, \cite{StruweLargeEnergies};  cf. \cite{ChenXu} for a similar reasoning.

\smallskip
Section \ref{sec:TheFlowNearInfinity} describes the flow near infinity. Since a flow line $u$ restricted to any time sequence tending to infinity is a Palais-Smale sequence, 
well known blow up and concentration compactness arguments \cite{StruweConcentrationCompactness} provide a suitable parametrization. Namely $u$ can up to a small error term $v$  be written as a linear combination of a solution $\omega$ and finitely many bubbles
\begin{align*}
u = \alpha \omega+\alpha^{i}\delta_{a_{i}, \lambda_{i}}+v, \, i=1, \ldots,p,
\end{align*}
where locally around $a_{i}$ the bubble $\delta_{a_{i}, \lambda_{i}}$ has the form
\begin{align*}
\delta_{a_{i}, \lambda_{i}}(x)=(\frac{\lambda_{i}}{1+\lambda_{i}^{2}d(a_{i},x)^{2}})^{\frac{n-2}{2}}.
\end{align*}
Thus a blow up corresponds to $\lambda_{i}\- \infty$.

We then refine the representation by choosing more suitable bubbles $\varphi_{a_{i}, \lambda_{i}}$ instead of $\delta_{a_{i}, \lambda_{i}}$ and take care of a possible degeneracy of the representation in the spirit of \cite{BrendleArbitraryEnergies}. Degeneracy in this context refers to the degeneracy of $\partial^{2}J(\omega)$.
Subsequently the representation is made unique by means of a Lyapunow-Schmidt reduction, that implies some orthogonality properties of the error term $v$ with respect to the solution $\omega$ and the bubbles $\varphi_{a_{i}, \lambda_{i}}$. In particular we obtain smallness of linear interactions of $v$ with $\omega$ and $\varphi_{a_{i}, \lambda_{i}}$ - a crucial aspect, that will enable us to identify the principal forces, that move $\lambda_{i}$ for instance or $a_{i}$.

Finally we show by Lojasiewicz inequality type arguments \cite{SlowConvergence}, \cite{Lojasiewicz}, that, if a flow line is precompact, it is fully compact, thus convergent and this generically with exponential speed.

\smallskip
In section \ref{sec:w=0} we then consider the case, that a flow line $u$ near infinity can up to a small error term $v$ be thought of as a linear combination of bubbles
\begin{align*}
u=\alpha^{i}\var_{a_{i}, \lambda_{i}}+v,
\end{align*}
so no solution $\omega$ is there. By suitable testing of the pseudo gradient flow equation in the spirit of  
\cite{BahriCriticalPointsAtInfinity} we analyse the movement of the bubbles by establishing explicit evolution
equations of those three parameters, that constitute the bubbles, namely the scaling parameter $\alpha_{i}$, height $\lambda_{i}$ and position $a_{i}$.
At this point the special choice of the Lyapunow-Schmidt reduction implies, that the evolution equations of the aforementioned parameters are independent of the time derivative of the error term $v$, which is difficult to control. 

Using the fact, that the second variation $\partial^{2}J(u)$ is positive  definite in this case, when applied to the error term $v$, we are able to give a suitable a priori estimate on $v$ - indeed $\partial J(u)$ is square integrable in time, since we are dealing with a pseudo gradient flow and $\partial J(\alpha^{i}\var_{a_{i}})$ is small. 

In conclusion  we obtain a precise description of the behaviour of the flow line in terms of $\lambda_{i}$  as the only non compact variable and $a_{i}$.

\smallskip
Section \ref{sec:omega>0} deals analogously to section \ref{sec:w=0} with the case, that a flow line $u$ near infinity can be written as a linear combination of a non trivial solution $\omega>0$ and finitely many bubbles - up to a small error term.
We then follow the same scheme as in the previous section. The main difference is, 
that there are more parameters to be considered beyond the scaling factor, height and position of the bubbles. Namely we have to deal with a scaling factor $\alpha$ for the solution $\omega$ plus finitely many parameters $\beta_{i}$ to describe the degenerate space of the solution $\omega$ and the implicit function theorem yields a suitable parametrization $u_{\alpha, \beta}=\alpha u_{1, \beta}$ for this purpose. So
\begin{align*}
u= u_{\alpha, \beta}+\alpha^{i}\var_{a_{i}, \lambda_{i}}+v.
\end{align*}
We would like to point out, that generically a solution $\omega$  is non degenerate, in which case $u_{\alpha, \beta}$ reduces to $\alpha \omega$.
Moreover the second variation  $\partial^{2}J(u)$ is not necessarily positive definite. But, since we have taken care of the degenerate space, the second variation is sort of non degenerate, when applied to the space, that the error term $v$ lives on. Thence we still get a sufficient estimate on $v$.

\smallskip
In section \ref{sec:FlowOnVwpe}, subsection \ref{subsec:PrincipalBehaviour} we proceed considering the flow near infinity and, under a suitable assumption on the energy functional, that the flow behaves as one would expect, 
e.g. that a flow line does not only converge to a solution, once this is true for a time sequence as seen at the end of section \ref{sec:TheFlowNearInfinity},
but that the same holds true for a critical point at infinity. This means, that, if for some time sequence the flow line blows up, this is true for the full flow line as well.
Moreover we show, that the critical set $[\nabla K=0]$ attracts the concentration points $a_{i}$ of a flow line near infinity.

\smallskip
The following subsection \ref{subsec:LeavingVwpe} contains the very essence of the proof of the theorem.
Under suitable conditions on $K$, which already imply, that the flow behaves in the sense of the foregoing subsection, we explicitly construct some functions adapted to the dimension and the case, whether $\omega$ is trivial or not,
with the basic property of becoming arbitrarily negative in case the flow line blows up, while on the other hand their time derivative is basically non negative. So they can be thought of as a way to check the compactness of a flow line near infinity.
This idea originates from \cite{BahriCriticalPointsAtInfinity}, where it was used in case $M=\mathbb{S}^{3}$ to exclude a multi bubble blow up, and our constructions are somewhat technical, but natural generalisations to the non spherical situation in dimensions $n=3,4,5$. 

For the construction the explicit evolution equations of the parameters $\lambda_{i}$ and  $a_{i}$ of the bubbles $\varphi_{a_{i}, \lambda_{i}}$ obtained in sections \ref{sec:w=0} and \ref{sec:omega>0} are used. Besides the necessity of controlling the error term $v$ there are two basic features to be considered.

The first one concerns self-interaction phenomen, whereby we mean quantities, which are attributed solely to a one bubble situation. In this case, the question of what moves a bubble is simply answered by saying, a bubble is moved, by what prevents a bubble from being a solution. E.g. on the standard sphere a bubble is a solution of the Yamabe problem, but not of the prescribed scalar curvature problem for $K$ non constant.  Thus we expect a bubble to be moved by the non vanishing derivatives of $K$, for instance the gradient of $K$ moves $a_{i}$ as $\lambda_{i}$ is moved by the laplacian 

If in addition we are dealing with an arbitrary manifold we expect other geometric quantities to move the bubbles as well - thereby the positive mass theorem comes into play.

The second feature is due to interaction quantities arising from the presence of several bubbles or from bubbles and a solution $\omega$. On the standard sphere for example, while each bubble is a solution of  the Yamabe problem, their linear combination is not. Thus the movement of the bubbles is caused solely by the interaction phenomena and in the context of proving flow convergence, one has to ensure, that the interaction terms rather decrease the possibly non compact variables $\lambda_{i}$ instead of increasing them.

\smallskip
In subsection \ref{subsec:ProvingTheTheorem} we put all the previous informations together and show flow convergence by contradiction based on the functions constructed in foregoing subsection \ref{subsec:LeavingVwpe}.
Thus proving  theorem \ref{thm_1}. In order to prove theorem \ref{thm_2} we basically prove the existence of a converging flow line - using the same arguments as for proving theorem \ref{thm_1}.

\smallskip
The final subsection \ref{subsec:DivergingSzenario} exposes a non trivial scenario of a blowing up  flow line.
In this example the function $K$ to be prescribed as the scalar curvature satisfies at one of its maximum points a flatness condition, that due to \cite{Escobar_Schoen} guarantees the existence of a minimizer of $J$ in case $M$ is not conformally equivalent to the standard sphere. On the other hand the flow line constructed blows up at the same maximum point.

\subsection{Preliminaries and statement of the theorems}
\label{subsec:Preliminaries}
We consider a smooth, closed Riemannian manifold 
\begin{equation*}\begin{split}
M=(M^{n},g_{0}), \; n=3,4,5
\end{split}\end{equation*}
with volume measure $\mu_{g_{0}}$ and scalar curvature $R_{g_{0}}$.
The Yamabe invariant 
\begin{equation*}\begin{split}
Y(M,g_{0})
= &
\inf_{\mathcal{A}}
\frac
{\int c_{n}\vert \nabla u \vert_{g_{0}}^{2}+R_{g_{0}}u^{2}d\mu_{g_{0}}}
{(\int u^{\frac{2n}{n-2}}d\mu_{g_{0}})^{\frac{n-2}{n}}},
\end{split}\end{equation*}
where $c_{n}=4\frac{n-1}{n-2}$ and 
\begin{equation*}\begin{split}
\mathcal{A}=
\{
u\in W^{1,2}_{g_{0}}(M)\mid u\geq 0,u\not \equiv 0
\},
\end{split}\end{equation*}
is assumed to be positive,  $Y(M,g_{0})>0$. The conformal laplacian
\begin{equation*}\begin{split} 
L_{g_{0}}=-c_{n}\lap_{g_{0}}+R_{g_{0}}
\end{split}\end{equation*}
then forms a positive,  self-adjoint operator with Green's function
\begin{equation*}
\begin{split}
G_{g_{0}}:M\times M\-\R_{+}
\end{split}
\end{equation*} 
and we may assume for the background metric
\begin{equation*}\begin{split} 
R_{g_{0}}>0\;\text{ and }\; 
\int Kd\mu_{g_{0}}=1.
\end{split}\end{equation*}
Considering a conformal metric $g=g_{u}=u^{\frac{4}{n-2}}g_{0}$ there holds 
\begin{equation*}\begin{split} 
d\mu=d\mu_{g_{u}}=u^{\frac{2n}{n-2}}d\mu_{g_{0}}
\end{split}\end{equation*}
for the volume element and for the scalar curvature 
\begin{equation*}\begin{split}
R=R_{g_{u}}=u^{-\frac{n+2}{n-2}}(-c_{n} \lap_{g_{0}} u+R_{g_{0}}u)
=
u^{-\frac{n+2}{n-2}}L_{g_{0}}u.
\end{split}\end{equation*} 
Let $0<K\in C^{\infty}(M)$ and 
\begin{equation*}\begin{split}
r=r_{u}=\int Rd\mu, \,k=k_{u}=\int Kd\mu, \,
\K=\K_{u}
=
\frac{K}{k}.
\end{split}\end{equation*}
Note, that 
\begin{align*}
c\Vert u \Vert_{W^{1,2}}\leq r_{u}=\int L_{g_{0}}uud\mu_{g_{0}}=\int c_{n}\vert \nabla u \vert^{2}_{g_{0}}+R_{g_{0}}u^{2}d\mu_{g_{0}}\leq C\Vert u \Vert_{W^{1,2}}
\end{align*}
and
\begin{align*}
c\Vert u \Vert_{L^{\frac{2n}{n-2}}}^{\frac{2n}{n-2}}
\leq k_{u}=\int Ku^{\frac{2n}{n-2}}d\mu_{g_{0}}\leq C\Vert u \Vert_{L^{\frac{2n}{n-2}}}^{\frac{2n}{n-2}}.
\end{align*}
In particular we may define 
$$\Vert u \Vert=\int L_{g_{0}}uud\mu_{g_{0}}$$ 
and use $\Vert \cdot \Vert$ as an equivalent norm on $W^{1,2}$.
The aim of this paper is a study of
\begin{equation*}\begin{split} 
\partial_{t}u
=
-\frac{1}{K}(R-r\bar{K})u, \; u(\cdot,0)=u_{0}>0
\end{split}\end{equation*}
as an evolution equation for the conformal factor. Obviously
\begin{equation*}\begin{split}
\partial_{t}k=\partial_{t}\int K u^{\frac{2n}{n-2}}d\mu_{g_{0}}=0.
\end{split}\end{equation*}
Thus, if we choose as an initial value 
\begin{align*}
u(\cdot,0)=u_{0}>0\;\text{ satisfying }\;k_{u_{0}}=\int Ku_{0}^{\frac{2n}{n-2}}=1, 
\end{align*}
then the unit volume $k\equiv 1$ is preserved and in case  
\begin{equation*}\begin{split}
u\- u_{\infty}>0\; \text{ in }\; W^{1,2}_{g_{0}}(M),
\end{split}\end{equation*} 
where $u_{\infty}$ is a stationary point, there necessarily holds
\begin{equation*}\begin{split}
\int Ku_{\infty}^{\frac{2n}{n-2}}d\mu_{g_{0}}=1
\; \text{ and  }\; 
R_{u_{\infty}}=r_{u_{\infty}}K. 
\end{split}\end{equation*}
In what follows we will simply call any maximal solution 
\begin{align*}
u:M\times [0,T)\-\R, \;T\in (0, \infty]
\end{align*}
of
\begin{align*}
\partial_{t}u=-\frac{1}{K}(R-r\K), \; u(\cdot,0)=u_{0}>0\; \text{ with }\; \int Ku_{0}^{\frac{2n}{n-2}}=1 
\end{align*}
a flow line with initial value $u_{0}$. Let us consider the energy
\begin{equation*}
\begin{split}
J(u)
= &
\frac{\int c_{n}\vert \nabla u \vert_{g_{0}}^{2}+R_{g_{0}}u^{2}d\mu_{g_{0}}}{(\int Ku^{\frac{2n}{n-2}}d\mu_{g_{0}})^{\frac{n-2}{n}}}
\; \text{ for }\;  u\in \mathcal{A}.
\end{split}
\end{equation*} 
\begin{proposition}[Derivatives of $J$]\label{prop_derivatives_of_J}$_{}$\\
We have
\begin{enumerate}[label=(\roman*)]
 \item 
\begin{equation*}
\begin{split}
J(u)
=
\frac{r_{u}}{k_{u}^{\frac{n-2}{n}}} 
\end{split}
\end{equation*} 
 \item 
\begin{equation*}\begin{split}\label{first_derivative_of_J}
\frac{1}{2}\partial  J(u)v
= &
\frac{1}{k_{u}^{\frac{n-2}{n}}}
[
\int L_{g_{0}}uv
-
\frac
{
r_{u}
}
{
k_{u}
}
\int Ku^{\frac{n+2}{n-2}}v
]
\\
= & 
\frac{1}{k_{u}^{\frac{n-2}{n}}}\int (R_{u}-\frac{r_{u}}{k_{u}}K)u^{\frac{n+2}{n-2}}v
\end{split}\end{equation*}
\item 
\begin{equation*}\begin{split}\label{second_derivative_of_J}
\frac{1}{2}\partial^{2}  J(u)vw 
=&
\frac{1}{k_{u}^{\frac{n-2}{n}}}
[
\int L_{g_{0}}vw
-
\frac{n+2}{n-2}
\frac
{
r_{u}
}
{
k_{u}
} 
\int Ku^{\frac{4}{n-2}}vw
]
\\ &-
\frac{2}{k_{u}^{\frac{n-2}{n}+1}}
[
\int L_{g_{0}}uv\int Ku^{\frac{n+2}{n-2}}w
+
\int L_{g_{0}}uw\int Ku^{\frac{n+2}{n-2}}v
]\\
& +
4\frac{n-1}{n-2}\frac{r_{u}}{k_{u}^{\frac{n-2}{n}+2}}
\int Ku^{\frac{n+2}{n-2}}v\int Ku^{\frac{n+2}{n-2}}w.
\end{split}\end{equation*}
\end{enumerate}
Moreover $J$ is $C^{2, \alpha}_{loc}$ and uniformly H\"older continuous on each
\begin{align*}
U_{\epsilon}=\{u\in \mathcal{A}\mid  \epsilon<\Vert u \Vert,\,J(u)\leq \epsilon^{-1}\}\subset \mathcal{A}.
\end{align*}
\end{proposition}
\noindent
The derivatives stated above are obtained by straight forward calculation. Moreover note, that 
$u\in U_{\epsilon}$ implies 
\begin{align*}
\epsilon^{2}\leq r_{u}\leq \epsilon^{-2}
\;\text{ and }\;
c\epsilon^{3}\leq k_{u}^{\frac{n-2}{n}}=J(u)^{-1}r_{u}\leq C\epsilon^{-3}.
\end{align*}
Thus uniform H\"older continuity on $U_{\epsilon}$ follows from the pointwise estimates
\begin{align*}
\vert \vert a\vert^{p}-\vert b \vert^{p}\vert
\leq 
C_{p}\vert a-b\vert^{p}\; \text{ in case $0<p<1$ }
\end{align*}
and
\begin{align*}
\vert \vert a \vert^{p}-\vert b \vert^{p}\vert
\leq 
C_{p}\max\{\vert a\vert^{p-1},\vert b \vert^{p-1}\}\vert a-b\vert\; \text{ in case }\; p\geq 1.
\end{align*}

So the problem of prescribing the scalar curvature has a variational structure, since a critical point $\omega>0$ of $J$ satisfies
\begin{align*}
R_{\omega}=\frac{r_{\omega}}{k_{\omega}}K, \;
\text{ where }\;r_{\omega}=\int L_{g_{0}}\omega \omega,k_{\omega}=\int K\omega^{\frac{2n}{2n}},
\end{align*}
whence the scalar curvature $R_{\omega}$ of $g_{\omega}=\omega^{\frac{4}{n-2}}g_{0}$
equals $K$ up to a coefficient.
\noindent
Note, that the standard norm of $\partial J(u)$
\begin{equation*}
\begin{split}
\Vert \partial J(u)\Vert=\Vert \partial J(u)\Vert_{W^{-1,2}_{g_{0}}(M)} 
\end{split}
\end{equation*} 
may be estimated by
\begin{equation*}\begin{split}
\frac{1}{2}\Vert \partial J(u)\Vert
\leq \frac{1}{k^{\frac{n-2}{n}}}\Vert R-r\K \Vert_{L^{\frac{2n}{n+2}}_{\mu}}
\leq  \frac{1}{k^{\frac{n-2}{n}}}\Vert R-r\K \Vert_{L^{2}_{\mu}}.
\end{split}\end{equation*}
We therefore define by a slight abuse of notation
\begin{equation*}
\begin{split}
\vert \delta J(u)\vert =\frac{2}{k^{\frac{n-2}{n}}}\Vert R-r\K\Vert_{L^{2}_{\mu}} 
\end{split}
\end{equation*} 
as a natural majorant of $\Vert \partial J(u)\Vert$. Since $k\equiv 1$ along a flow line, we get
\begin{equation*}
\begin{split}
\partial_{t} J(u)
= &
\partial J(u)\partial_{t}u
=
-
2\int \frac{1}{K}\vert R-r\K \vert^{2}u^{\frac{2n}{n-2}}
\leq 
-\frac{1}{2\max_{M}K}\vert \delta J(u)\vert^{2}
.
\end{split}
\end{equation*}
This justifies the notion of $\partial_{t}u=-\frac{1}{K}(R-r\K)u$ as a pseudo gradient flow related to $J$ and, since $J$ is bounded from below, we have a priori integrability 
\begin{equation*}
\begin{split}
\int^{T}_{0}\vert \delta J(u)\vert^{2} dt<C(K)J(u_{0}).
\end{split}
\end{equation*} 
On the other hand the positivity of the Yamabe invariant implies
\begin{align*}
J(u)>\frac{Y(M,g_{0})}{\max_{M}K^{\frac{n-2}{n}}}>c.
\end{align*}
Thus we may assume, that along a flow line $c<J(u)=r_{u}<C$ due to $k\equiv 1$. 
Recalling proposition \ref{prop_derivatives_of_J} this shows $u\in U_{\epsilon}$ for some $\epsilon>0$ small and fix, whence
$J$ is uniformly H\"older continuous along and close by every flow line.

Consider the following conditions in cases $n=3,4,5$,
which are obviously satisfied, if $M$ is not conformally equivalent to the standard $\mathbb{S}^{n}$ and $K \equiv 1$.
They are scaling invariant with respect to $K$ as one should expect due to the scaling invariance of $J$.

\begin{hypothesis}[Dimensional conditions]\label{def_dimensional_conditions}$_{}$
\begin{enumerate}
 \item [Cond$_{3}$] \;:\; $M$ is not conformally equivalent to the standard sphere $\mathbb{S}^{3}$
 \item [Cond$_{4}$] \;:\; $M$ is not conformally equivalent to the standard sphere $\mathbb{S}^{4}$ and
$$
[\nabla K=0]\subseteq [\frac{\lap K}{K} > -c]
\;\text{ for some }\; c=c(M)>0
$$
\item [Cond$_{5}$] \;:\; $M$ is not conformally equivalent to the standard sphere $\mathbb{S}^{5}$ and
$$
\langle \nabla \lap K, \nabla K\rangle 
>\frac{1}{3} \vert \lap K \vert^{2}
$$
\quad holds on $[\lap K<0]\cap U$ for an open neighbourhood $U$ of 
$[\nabla K=0].$ 
\end{enumerate}
Moreover let $Cond_{n}'$ denote $Cond_{n}$ with $[\nabla K=0]$ replaced by $[K=\max K]$.
\end{hypothesis}
\noindent

Theorem \ref{thm_1} below generalizes the convergence of the Yamabe flow in these dimensions
proven in \cite{BrendleArbitraryEnergies}, however by a different strategy. 

\begin{thm*}\label{thm_1}$_{}$\\
Let $M=(M^{n},g_{0}), \;n=3,4,5$ be a smooth, closed Riemannian manifold of positive  Yamabe-invariant. 
Then for $0<K\in C^{\infty}(M)$ every flow line
$$\partial_{t}u
=
-\frac{1}{K}(R-r\bar{K})u, \; u(\cdot,0)=u_{0}>0\; \text{ with }\; \int Ku_{0}^{\frac{2n}{n-2}}=1
$$
exists for all times and remains positive. 

Moreover we have convergence in the sense, that
\begin{equation*}\begin{split}
u\-u_{\infty}>0\; \text{ in }\; C^{\infty}
\text{ solving }\;R_{u_{\infty}}=r_{u_{\infty}}K,
\end{split}\end{equation*}

provided the dimensional condition $Cond_{n}$ is satisfied.
\end{thm*}
So $Cond_{n}$ implies compactness of the flow, whereas $Cond_{n}'$ is at least sufficient to solve the prescribed scalar curvature problem. 
\begin{thm*}\label{thm_2}$_{}$\\
Let $M=(M^{n},g_{0}), \;n=3,4,5$ be a smooth, closed Riemannian manifold of positive  Yamabe-invariant. 
Then for $0<K\in C^{\infty}(M)$ there exists 
\begin{equation*}\begin{split}
u_{\infty}>0\; \text{ in }\; C^{\infty}
\text{ solving }\;R_{u_{\infty}}=r_{u_{\infty}}K,
\end{split}\end{equation*}
provided the dimensional condition $Cond_{n}'$ is satisfied.
\end{thm*}

\section{Long time existence and weak convergence}
\label{sec:LongtimeExistenceAndStrongConvergence}
In this section
adapted from \cite{BrendleArbitraryEnergies} and \cite{StruweLargeEnergies}
we derive global existence and weak convergence  in the sense, 
that $\Vert R-r\K \Vert_{L^{p}_{\mu}} \- 0$ as  $t\-\infty$.

\subsection{Long time existence}
\label{subsec:LontTimeExistence}
\begin{lemma}[Lower bounding the scalar curvature] \label{lem_bounding_R_from_below}$_{}$\\
Along a flow line the scalar curvature $R$ is uniformly lower bounded. 
\end{lemma}
\begin{proof}[\textbf{Proof of lemma \ref{lem_bounding_R_from_below}}]\label{p_bounding_R_from_below}$_{}$\\
Letting 
\begin{equation}
\tilde R=e^{\frac{4}{n-2}\int^{t}_{0}\frac{r}{k}(\tau) d\tau}R
\end{equation} 
we have in view of lemma \ref{App1}
\begin{equation}
\begin{split}
\partial_{t} \tilde R
= &
e^{\frac{4}{n-2}\int^{t}_{0}\frac{r}{k}(\tau) d\tau}
[
c_{n}\lap_{g}\frac{R}{K}
+
\frac{4}{n-2}(R-r\K)\frac{R}{K}
]
+
\frac{4}{n-2}\frac{r}{k}\tilde R \\
= &
c_{n}\lap_{g}\frac{\tilde R}{K}
+
\frac{4}{n-2}R\frac{\tilde R}{K}
\geq
c_{n}\lap_{g}\frac{\tilde R}{K}.
\end{split}
\end{equation} 
The parabolic maximum principle then shows
\begin{equation}
\min_{\{ t\}\times M}\frac{\tilde R}{K}\geq  \min_{\{0\}\times M} \frac{\tilde R}{K},
\end{equation} 
whence
\begin{equation}\label{R-_decay}
\min_{\{t\}\times M}R
\geq 
C(K)
e^{-\frac{4}{n-2}\int^{t}_{0}\frac{r}{k}(\tau)d\tau}\min_{\{0\}\times M}R. 
\end{equation} 
Since $\frac{r}{k}=r\geq r_{\infty}>0$ along a flow line, the assertion follows.
\end{proof}

Due to Gronwall's lemma this lower bound implies an upper bound on $u$.
\begin{lemma}[Upper bound] \label{lem_bounding_u_from_above}$_{}$\\
Along a flow line $u$ there exists
$
C>0
$
such, that for $0\leq t < T$ we have
\begin{equation*}\begin{split} 
\sup_{M}u(t, \cdot)
\leq
e^{Ct}
.
\end{split}\end{equation*}
\end{lemma}
\begin{proof}[\textbf{Proof of lemma \ref{lem_bounding_u_from_above}}]\label{p_bounding_u_from_above}$_{}$\\
From lemma \ref{lem_bounding_R_from_below} we infer
\begin{equation}\begin{split}
\partial_{t}u=-\frac{1}{K}(R-r\K)u
\leq &
cu.
\end{split}\end{equation}
The claim follows from Gronwall's inequality.
\end{proof}
The Harnack inequality now implies a lower bound on $u$.
\begin{lemma}[Lower bound] \label{lem_bounding_u_from_below}$_{}$\\
Along a flow line $u$ there exists for
$
\Theta>0
$
some
$
C=C(\Theta)>0
$
such, that 
\begin{equation*}\begin{split}
\sup_{M\times [0,T)}u \leq \Theta \= \inf_{M\times [0,T)}u \geq C.
\end{split}\end{equation*}
\end{lemma}
\begin{proof}[\textbf{Proof of lemma \ref{lem_bounding_u_from_below}}]\label{p_bounding_u_from_below}$_{}$\\
Let us choose $c>0$, such that $R+c>0$ according to lemma \ref{lem_bounding_R_from_below}. Then for
\begin{equation}
P= R_{g_{0}}+cu^{\frac{4}{n-2}}
\end{equation} 
we have
\begin{equation}
-c_{n}\lap_{g_{0}}u+Pu=L_{g_{0}}u-R_{g_{0}}u+Pu=Ru^{\frac{n+2}{n-2}}+cu^{\frac{n+2}{n-2}}. 
\end{equation} 
Thus the weak Harnack inequality gives
\begin{equation}\begin{split}
k= &
\int K u^{\frac{2n}{n-2}} 
\leq 
\sup_{M}(Ku^{\frac{n+2}{n-2}}) \int u
\leq 
C\sup_{M}(Ku^{\frac{n+2}{n-2}}) \inf_{M} u,
\end{split}\end{equation}
where $C=C(\Vert P \Vert_{L^{\infty}})$. The claim follows.
\end{proof}
As a consequence of the positivity of the Yamabe invariant we obtain a logarithmic type estimate on the first variation of $J$.
\begin{lemma}[Logarithmic-type estimate on the first variation] \label{lem_logarithmic-type_estimate_on_dJ(u)}$_{}$\\
For $p>\frac{n}{2}$ there exist constants
$$
c=c(p)>0\; \text{ and }\; C=C(p)>0
$$
such, that along a flow line we have
\begin{equation*}\begin{split}
\partial_{t}\int \vert R-r\K \vert^{p}d\mu
& +
c(\int \vert R-r\K \vert^{\frac{pn}{n-2}}d\mu )^{\frac{n-2}{n}} \\
\leq &
C
(\int \vert R-r\K \vert^{p}d\mu)^{\frac{2p+2-n}{2p-n}}
+
C
\int \vert R-r\K \vert^{p}d\mu.
\end{split}\end{equation*}
\end{lemma}
\begin{proof}[\textbf{Proof of lemma \ref{lem_logarithmic-type_estimate_on_dJ(u)}}]\label{p__logarithmic-type_estimate_on_dJ(u)}$_{}$\\
In view of lemma \ref{App1}  we have
\begin{equation}\begin{split}
\partial_{t}\int \vert R& -r\K \vert^{p}d\mu \\
= & 
p\int \partial_{t}(R-r\K)(R-r\K)\vert R-r\K \vert^{p-2}d\mu  
+
\int \vert R-r\K \vert^{p} \partial_{t}d\mu \\
= &
p c_{n}\int\lap_{g} \frac{R-r\K}{K}(R-r\K)\vert R-r\K \vert^{p-2}d\mu \\
& +
\frac{4p}{n-2} \int \frac{R}{K}\vert R-r\K \vert^{p}d\mu
-
\frac{2n}{n-2}\int \vert R-r\K \vert^{p}\frac{R-r\K}{K}d\mu.
\end{split}\end{equation}
Integrating by parts we obtain
\begin{equation}\begin{split}
\partial_{t}\int  \vert R  -  r\K  \vert^{p}d\mu 
\leq &
-c(p)\int \frac{1}{K}\vert \nabla(R- r\K) \vert_{g}^{2}
\vert R-r\K \vert^{p-2} d\mu\\
& +
C(p)
(
\int \vert R-r\K \vert^{p+1}d\mu
+
\int \vert R-r\K \vert^{p}d\mu
).
\end{split}\end{equation}
Using 
$\vert \nabla(R-r\K) \vert_{g}
\overset{a.e.}{=}
\vert \nabla \vert R-r\K \vert\vert_{g}$
this gives
\begin{equation}\begin{split}
\partial_{t}\int  \vert R  -r\K \vert^{p}d\mu 
\leq &
-c(p)
\int 
c_{n}\vert \nabla\vert R-r\K\vert^{\frac{p}{2}}\vert_{g}^{2}
d\mu \\
& + 
C(p)
(
\int \vert R-r\K \vert^{p+1}d\mu
+
\int \vert R-r\K \vert^{p}d\mu
)
\end{split}\end{equation}
Then
$Y(M,g_{0})>0$
implies
\begin{equation}\begin{split}
\partial_{t}\int \vert R-r\K \vert^{p}d\mu
\leq &
-c(p)
(\int \vert R-r\K \vert^{\frac{pn}{n-2}}d\mu)^{\frac{n-2}{n}} \\
& +
C(p)
(
\int \vert R-r\K \vert^{p+1}d\mu
+
\int \vert R-r\K \vert^{p}d\mu
).
\end{split}\end{equation}
Since $p>\frac{n}{2}$, we may apply H\"older's inequality to $f=\vert R-r\K \vert^{p}$ via
\begin{equation}\begin{split}
\Vert f^{\frac{p+1}{p}} \Vert_{L_{g_{0}}^{1}}
= &
\Vert f \Vert^{\frac{p+1}{p}}_{L_{g_{0}}^{\frac{p+1}{p}}}
\leq
\Vert f \Vert^{\lambda\frac{p+1}{p}}_{L_{g_{0}}^{\Lambda}}
\Vert f \Vert^{(1-\lambda)\frac{p+1}{p}}_{L_{g_{0}}^{\Theta}}
\leq
\Vert f \Vert^{\frac{n}{2p}}_{L_{g_{0}}^{\frac{n}{n-2}}}
\Vert f \Vert^{\frac{2p+2-n}{2p}}_{L_{g_{0}}^{1}} \\
\leq &
\eps \Vert f \Vert_{L_{g_{0}}^{\frac{n}{n-2}}}
+
c(p, \eps)\Vert f \Vert^{\frac{2p+2-n}{2p-n}}_{L_{g_{0}}^{1}},
\end{split}\end{equation}
where
$\Lambda = \frac{n}{n-2}, \; \Theta=1, \;\lambda=\frac{n}{2(p+1)}$
to conclude by absorption
\begin{equation}\begin{split}
\partial_{t}\int \vert R & -r\K \vert^{p}d\mu \\
\leq &
-c(p)(\int \vert R-r\K \vert^{\frac{pn}{n-2}}d\mu)^{\frac{n-2}{n}} \\
& +
C(p)
[
(\int \vert R-r\K \vert^{p}d\mu)^{\frac{2p+2-n}{2p-n}}  
+
\int \vert R-r\K \vert^{p}d\mu
]
.
\end{split}\end{equation}
This is the desired result.
\end{proof}
The next proposition is a typical parabolic type estimate.
\begin{proposition}[Main observation for long time existence] \label{prop_main_observation_for_long_time_existence}$_{}$\\
Along a flow line there holds for $1\leq p \leq \frac{n}{2}$
\begin{equation*}
\begin{split}
\partial_{t}\int \frac{R_{+}^{p}}{K^{p-1}}d\mu 
\leq &
-4\frac{p-1}{p}c_{n}\int \vert \nabla (\frac{R_{+}}{K})^{\frac{p}{2}}\vert^{2}_{g}d\mu \\
& -
\frac{2n-4p}{n-2}\int \frac{1}{K^{p}}\vert R_{+}-r\K\vert^{p+1}d\mu.
\end{split}
\end{equation*} 
\end{proposition}
Here $R_{+}=\min\{R,0\}$.
\begin{proof}[\textbf{Proof of proposition \ref{prop_main_observation_for_long_time_existence}}]\label{p_main_observation_for_long_time_existence}$_{}$\\
In view of lemma \ref{App1} we have using 
\begin{equation}
\begin{split}
\partial_{t}\int & \frac{R^{p}_{+}}{K^{p-1}}d\mu 
=
p\int \partial_{t}RR_{+}^{p-1}d\mu+\int R_{+}^{p}\partial_{t}d\mu
\\
= &
pc_{n}\int \lap_{g}\frac{R}{K}(\frac{R_{+}}{K})^{p-1} d\mu 
+
\frac{4p-2n}{n-2}\int (R-r\K)(\frac{R_{+}}{K})^{p}d\mu \\
= &
-4\frac{p-1}{p}c_{n}\int \vert \nabla (\frac{R_{+}}{K})^{\frac{p}{2}}\vert_{g}^{2}d\mu\\
& +
\frac{4p-2n}{n-2}\int (R_{+}-r\K)[(\frac{R_{+}}{K})^{p}-(\frac{r}{k})^{p}]d\mu  \\
& +
\frac{4p-2n}{n-2}(\frac{r}{k})^{p}\int (R_{+}-r\K)d\mu.
\end{split}
\end{equation} 
Due to $(a^{p}-b^{p})(a-b)\geq \vert a-b\vert^{p+1}$ and $\int (R-r\K)d\mu=0$ one obtains
\begin{equation}
\begin{split}
\partial_{t}\int \frac{R_{+}^{p}}{K^{p-1}}d\mu 
\leq &
-4\frac{p-1}{p}c_{n}\int \vert \nabla (\frac{R_{+}}{K})^{\frac{p}{2}}\vert_{g}^{2}d\mu\\
& +
\frac{4p-2n}{n-2}\int \frac{1}{K^{p}}\vert R_{+}-r\K\vert^{p+1}d\mu.
\end{split}
\end{equation} 
This is the desired result.
\end{proof}
The following is by now an easy consequence.
\begin{corollary}\label{cor_easy_consequence}$_{}$\\
Along a flow line there holds
\begin{equation*}\begin{split}
\sup_{0 \leq t < T}\int \frac{R_{+}^{p}}{K^{p-1}}d\mu
& +
4\frac{p-1}{p}c_{n}\int^{T}_{0}\int \vert \nabla(\frac{R_{+}}{K})^{\frac{p}{2}} \vert_{g}^{2}d\mu dt \\
& +
\frac{2n-4p}{n-2}\int^{T}_{0}\int \frac{1}{K^{p}}\vert R_{+}-r\K \vert^{p+1}d\mu dt \\
\leq &
\int \frac{R_{+}^{p}}{K^{p-1}}d\mu\lfloor_{t=0}
\end{split}\end{equation*}
\end{corollary}
This implies via Sobolev embedding higher integrability, which applied to 
lemma \ref{lem_logarithmic-type_estimate_on_dJ(u)} proves the following time dependent bound.

\begin{corollary}[$L^{p}$-bound on the first variation]\label{cor_L_{g_{0}}p_bound_on_the_first_variation}$_{}$\\
For $1\leq p\leq \frac{n^{2}}{2(n-2)}$ and $T>0$ there exists 
$
C=C(p,T)
$
such, that 
\begin{equation*}\begin{split}
\sup_{0\leq t < T}\int \vert R-r\K \vert^{p}d\mu
\leq &
C
\; \text{ along a flow line.}
\end{split}\end{equation*}
\end{corollary}
\begin{proof}[\textbf{Proof of corollary \ref{cor_L_{g_{0}}p_bound_on_the_first_variation}}]
\label{p_L_{g_{0}}p_bound_on_the_first_variation}$_{}$\\
From corollary \ref{cor_easy_consequence} for $p=\frac{n}{2}$ we infer
\begin{equation}\begin{split}
\sup_{0 \leq t < T}\int R_{+}^{\frac{n}{2}}d\mu
+
\int^{T}_{0}\int \vert \nabla(\frac{R_{+}}{K})^{\frac{n}{4}}\vert_{g}^{2}d\mu dt \leq C.
\end{split}\end{equation}
Sobolev's embedding then implies
\begin{equation}\begin{split}
\int^{T}_{0}(\int (\frac{R_{+}}{K})^{\frac{n^{2}}{2(n-2)}}d\mu)^{\frac{n-2}{n}}dt
\leq
C.
\end{split}\end{equation}
Since $R$ is uniformly bounded from below according to lemma \ref{lem_bounding_R_from_below} we get \begin{equation}\begin{split}
\int^{T}_{0}(\int \vert R\vert^{\frac{n^{2}}{2(n-2)}}d\mu)^{\frac{n-2}{n}}dt
\leq
C,
\end{split}\end{equation}
whence
\begin{equation}\begin{split}
\int^{T}_{0}(\int \vert R-r\K\vert^{\frac{n^{2}}{2(n-2)}}d\mu)^{\frac{n-2}{n}}dt
\leq
C.
\end{split}\end{equation}
But from lemma \ref{lem_logarithmic-type_estimate_on_dJ(u)} with
$p=\frac{n^{2}}{2(n-2)}> \frac{n}{2}$ we infer
\begin{equation}\begin{split}
\partial_{t} \ln \int \vert R-r\K \vert^{\frac{n^{2}}{2(n-2)}}d\mu
\leq &
C(\int \vert R-r\K \vert^{\frac{n^{2}}{2(n-2)}}d\mu)^{\frac{n-2}{n}}+C.
\end{split}\end{equation}
This proves the claim.
\end{proof}
With the above bounds at hand one uses Morrey's inequality to prove H\"older regularity.
\begin{proposition}[Time-dependent H\"older regularity]\label{prop_time_dependend_holder_regularity}$_{}$\\
Along a flow line  there exists for $0<\alpha<\min \lbrace \frac{4}{n},1 \rbrace$ and $ T>0$ a constant
$$
C=C(\alpha,T)
$$
 such, that  we have
\begin{equation*}\begin{split}
\vert u(x_{1},t_{1}) -u(x_{2},t_{2}) \vert
\leq
C
(
\vert t_{1}-t_{2}\vert^{\frac{\alpha}{2}}
+
d(x_{1},x_{2})^{\alpha})
\end{split}\end{equation*}
for all $x_{1},x_{2} \in M$ and $0 \leq t_{1},t_{2}<T$ with
$\vert t_{1}-t_{2} \vert \leq 1$
\end{proposition}

\begin{proof}[\textbf{Proof of proposition \ref{prop_time_dependend_holder_regularity}}]
\label{p_time_dependend_holder_regularity} $_{}$\\
Let
$\alpha=2-\frac{n}{p}$
and
$\frac{n}{2}<p<\min \lbrace \frac{n^{2}}{2(n-2)}, \,n \rbrace$.
Lemma \ref{lem_bounding_R_from_below} and  \ref{cor_easy_consequence} show
\begin{equation}\begin{split}
\int \vert R \vert^{p}d\mu \leq C
\end{split}\end{equation}
with $C=C(T)$, whence by conformal invariance and lemmata \ref{lem_bounding_u_from_above}, \ref{lem_bounding_u_from_below}
\begin{equation}\begin{split}
\int  \vert \lap_{g_{0}} u \vert^{p}\leq C.
\end{split}\end{equation}
On the other hand corollary \ref{cor_L_{g_{0}}p_bound_on_the_first_variation} shows
\begin{equation}\begin{split}
\int \vert \frac{\partial_{t}u}{u} \vert^{p}d\mu \leq C, \; \text{ in particular }\;\int \vert \partial_{t}u \vert^{p} \leq C
\end{split}\end{equation}
From this it follows via Morrey
\begin{equation}\begin{split}
\vert u(x,t)-u(y,t) \vert \leq Cd(x,y)^{\alpha} \; \text{ for all } \; x,y\in M,
\end{split}\end{equation}
where $0<\alpha<\min \lbrace \frac{4}{n},1 \rbrace$,
and
\begin{equation}\begin{split}
\vert u(x,t_{1}) & -u(x,t_{2})\vert \\
= &
\vert t_{1}-t_{2}\vert^{-\frac{n}{2}}
\underset{B_{\sqrt{\vert t_{1}-t_{2}\vert}}(x)}{\int}\vert u(x,t_{1})-u(x,t_{2})\vert d\mu_{g_{0}}(y) \\
\leq &
\vert t_{1}-t_{2}\vert^{-\frac{n}{2}}
\underset{B_{\sqrt{\vert t_{1}-t_{2}\vert}}(x)}{\int}\vert u(y,t_{1})-u(y,t_{2})\vert d\mu_{g_{0}}(y) 
+
C\vert t_{1}-t_{2}\vert^{\frac{\alpha}{2}}\\
\leq &
\vert t_{1}-t_{2}\vert^{-\frac{n}{2}+1} \sup_{0 \leq t <T}
\underset{B_{\sqrt{\vert t_{1}-t_{2}\vert}}(x)}{\int}\vert \partial_{t}u(t,y)\vert d\mu_{g_{0}}(y) 
+
C\vert t_{1}-t_{2}\vert^{\frac{\alpha}{2}}\\
\leq &
\vert t_{1}-t_{2}\vert^{-\frac{n-2}{2}}
\vert t_{1}-t_{2}\vert^{\frac{n}{2}\frac{p-1}{p}}
\sup_{0 \leq t<T}(\int \vert \partial_{t}u \vert^{p}d\mu)^{\frac{1}{p}} +
C\vert t_{1}-t_{2}\vert^{\frac{\alpha}{2}}
\end{split}\end{equation}
for all $\vert t_{1}-t_{2} \vert \leq 1$. The claim follows from
$
-
\frac{n-2}{2}
+
\frac{n}{2}\frac{p-1}{p}
=\frac{\alpha}{2}.
$
\end{proof}
With H\"older regularity at hand standard regularity arguments show
\begin{corollary}[Long-time existence]\label{cor_long_time_existence}$_{}$\\
Each flow line
exists for all times.
\end{corollary}
\begin{proof}[\textbf{Proof of corollary \ref{cor_long_time_existence}}]$_{}$\\
This follows from short time existence and proposition \ref{prop_time_dependend_holder_regularity}.
\end{proof}

\subsection{Integrability and weak convergence }
\label{subsec:IntegrabilityAndStrongConvergence}
Now, that long time existence has been established, we examine in which sense the first variation of $J$ vanishes as $t\- \infty$.
\begin{lemma}[Integrability and weak convergence]\label{lem_integrability_and_weak_convergence}$_{}$\\
For $1\leq p<\frac{n}{2}$ we have along a flow line
\begin{equation*}\begin{split}
\int^{\infty}_{0}\int \vert R-r\K \vert^{p+1}d\mu dt \leq C
\; \text{ and } \;
\liminf_{t\nearrow \infty}\int \vert R-r\K \vert^{p+1}d\mu =0.
\end{split}\end{equation*}
\end{lemma}
\begin{proof}[\textbf{Proof of lemma \ref{lem_integrability_and_weak_convergence}}]
\label{p_integrability_and_weak_convergence}$_{}$\\
Clearly the first inequality above implies the second one. Note, that
\begin{equation}
\int^{\infty}_{0}\int \vert R_{+}-r\K \vert^{p+1}d\mu dt \leq C 
\end{equation} 
with time independent $C$ according to corollary \ref{cor_easy_consequence}. Moreover we have
\begin{equation}
\min_{\{t\}\times M}R
\geq 
C(K)
e^{-\frac{4}{n-2}\int^{t}_{0}\frac{r}{k}(\tau)d\tau}\min_{\{0\}\times M}R,
\end{equation} 
cf. \eqref{R-_decay}. Since along a flow line $k=1$ and $r\searrow r_{\infty}>0$ this gives
\begin{equation}
R_{-}\leq Ce^{-ct},R_{-}=-\min\{R,0\}
\end{equation} 
for suitable constants $c,C>0$. From this the assertion follows.
\end{proof}
Interpolating via lemma \ref{lem_logarithmic-type_estimate_on_dJ(u)} we obtain weak convergence.
\begin{proposition}[Weak convergence  of the first variation]
\label{prop_strong_convergence_of_the_first_variation}$_{}$\\
Along a flow line we have for any $1 \leq p < \infty$
\begin{equation*}\begin{split}
\lim_{t \nearrow \infty}\int \vert R-r\K \vert^{p}d\mu =0.
\end{split}\end{equation*}
In particular we have $\vert \delta J(u)\vert\- 0$ as $t\- \infty$.
\end{proposition}

\begin{proof}[\textbf{Proof of proposition \ref{prop_strong_convergence_of_the_first_variation}}
(cf. \cite{StruweLargeEnergies}, Lemma 3.3 and equation (43))]
\label{p_interpolated_integrability_and_strong_convergence}$_{}$\\
Due to lemma \ref{lem_integrability_and_weak_convergence} for any 
$\max\{2, \frac{n}{2}\} < p_{0} < \frac{n+2}{2}$ there holds
\begin{equation}\begin{split}
\int^{\infty}_{0}\int \vert R-r\K \vert^{p_{0}}d\mu dt \leq C
\; \text{ and } \;
\liminf_{t\nearrow \infty}\int \vert R-r\K \vert^{p_{0}}d\mu =0.
\end{split}\end{equation}
Thus we may choose a sequence $\tau^{0}_{k} \nearrow \infty$ satisfying
\begin{equation}\begin{split}
\int \vert R-r\K \vert^{p_{0}}d\mu\lfloor_{\tau^{0}_{k}}\leq \frac{1}{2k}
\; \text{ and } \;
\int^{\infty}_{\tau^{0}_{k}}\int \vert R-r\K \vert^{p_{0}}d\mu dt < \frac{1}{4Ck},
\end{split}\end{equation}
where $C=C(p)$ is the constant appearing in lemma \ref{lem_logarithmic-type_estimate_on_dJ(u)}.
Define
\begin{equation}\begin{split}
\theta^{0}_{k}
= &
\sup \lbrace
\tau > \tau^{0}_{k} \mid
\forall \tau^{0}_{k}<t<\tau\;:\;
\int \vert R-r\K \vert^{p_{0}}d\mu < \frac{2}{k} \rbrace
>
\tau^{0}_{k}.
\end{split}\end{equation}
Then we infer from lemma \ref{lem_logarithmic-type_estimate_on_dJ(u)} for  $\tau^{0}_{k}<t<\theta^{0}_{k}$
\begin{equation}\begin{split}\label{n/(n-2)_argument}
\int \vert R-r\K \vert^{p_{0}}d\mu \lfloor_{t}
& +
c\int^{t}_{\tau^{0}_{k}}(\int \vert R-r\K \vert^{p_{0}\frac{n}{n-2}}d\mu)^{\frac{n-2}{n}}dt \\
\leq &
\int \vert R-r\K \vert^{p_{0}}d\mu \lfloor_{\tau^{0}_{k}} \\
& +
C\int^{t}_{\tau^{0}_{k}}(\int \vert R-r\K \vert^{p_{0}}d\mu)^{1+\frac{2}{2p-n}}dt \\
& +
C\int^{t}_{\tau^{0}_{k}}\int \vert R-r\K \vert^{p_{0}}d\mu dt \\
\leq &
\frac{1}{2k}
+
2C\int^{\infty}_{\tau^{0}_{k}}\int \vert R-r\K \vert^{p_{0}}d\mu dt
\leq \frac{1}{k}.
\end{split}\end{equation}
If $\theta^{0}_{k}< \infty$, then
$
\frac{2}{k}=\int \vert R-r\K \vert^{p_{0}}d\mu \lfloor_{\theta^{0}_{k}} \leq \frac{1}{k},
$
whence $\theta^{0}_{k}=\infty$ and
\begin{equation}\begin{split}
\int \vert R-r\K \vert^{p_{0}}d\mu \leq \frac{2}{k} \; \text{ on }\;[\tau^{0}_{k}, \infty).
\end{split}\end{equation}
We conclude
$
\lim_{t\nearrow \infty}\int \vert R-r\K \vert^{p_{0}}d\mu =0
$
and in particular, cf. \eqref{n/(n-2)_argument},
\begin{equation}\begin{split}
\int^{\infty}_{0}(\int \vert R-r\K \vert^{p_{1}}d\mu)^{\frac{n-2}{n}} dt <\infty
\; \text{ and } \;
\liminf_{t\nearrow \infty}\int \vert R-r\K \vert^{p_{1}}d\mu =0
\end{split}\end{equation}
letting 
\begin{equation}
\begin{split}
p_{1}=\frac{n}{n-2}p_{0}.
\end{split} 
\end{equation} 
As before we may choose a sequence $\tau^{1}_{k} \nearrow \infty$ satisfying
\begin{equation}\begin{split}
(\int \vert R-r\K \vert^{p_{1}}d\mu)^{\frac{n-2}{n}}\lfloor_{\tau^{1}_{k}}\leq \frac{1}{2k}
\end{split}\end{equation}
and
\begin{equation}\begin{split}
\int^{\infty}_{\tau^{1}_{k}}(\int \vert R-r\K \vert^{p_{1}}d\mu)^{\frac{n-2}{n}} dt < \frac{n}{4Ck(n-2)},
\end{split}\end{equation}
where $C=C(p)$ is the constant appearing in lemma \ref{lem_logarithmic-type_estimate_on_dJ(u)}.
Define
\begin{equation}\begin{split}
\theta^{1}_{k}
= &
\sup \lbrace
\tau > \tau^{1}_{k} \mid
\forall \tau^{1}_{k}<t<\tau\;:\;
(\int \vert R-r\K \vert^{p_{1}}d\mu)^{\frac{n-2}{n}} < \frac{2}{k} \rbrace
>
\tau^{1}_{k}.
\end{split}\end{equation}
Then we infer from lemma \ref{lem_logarithmic-type_estimate_on_dJ(u)} for $\tau^{0}_{k}<t<\theta^{0}_{k}$
\begin{equation}\begin{split}\label{n/(n-2)_argument_k>=1}
(\int \vert R-r\K \vert^{p_{1}}d\mu)^{\frac{n-2}{n}} \lfloor_{t}
& +
c\frac{n-2}{n}\int^{t}_{\tau^{1}_{k}}
\frac
{(\int \vert R-r\K \vert^{p_{1}\frac{n}{n-2}}d\mu)^{\frac{n-2}{n}}}
{(\int \vert R-r\K\vert^{p_{1}})^{\frac{2}{n}}}
dt \\
\leq &
(\int \vert R-r\K \vert^{p_{0}}d\mu)^{\frac{n-2}{n}} \lfloor_{\tau^{1}_{k}} \\
& +
C\frac{n-2}{n}\int^{t}_{\tau^{1}_{k}}(\int \vert R-r\K \vert^{p_{1}}d\mu)^{\frac{n-2}{n}+\frac{2}{2p-n}}dt \\
& +
C\frac{n-2}{n}\int^{t}_{\tau^{1}_{k}}(\int \vert R-r\K \vert^{p_{1}}d\mu)^{\frac{n-2}{n}} dt \\
\leq &
\frac{1}{2k}
+
2C\frac{n-2}{n}\int^{\infty}_{\tau^{1}_{k}}\int \vert R-r\K \vert^{p_{0}}d\mu dt
\leq \frac{1}{k}.
\end{split}\end{equation}
If $\theta^{1}_{k}< \infty$, then
$
\frac{2}{k}=\int \vert R-r\K \vert^{p_{1}}d\mu \lfloor_{\theta^{1}_{k}} \leq \frac{1}{k},
$
whence $\theta^{1}_{k}=\infty$ and
\begin{equation}\begin{split}
\int \vert R-r\K \vert^{p_{1}}d\mu \leq \frac{2}{k} \; \text{ on }\;[\tau^{1}_{k}, \infty).
\end{split}\end{equation}
We conclude
$
\lim_{t\nearrow \infty}\int \vert R-r\K \vert^{p_{1}}d\mu =0
$
and in particular, cf. \eqref{n/(n-2)_argument_k>=1},
\begin{equation}\begin{split}
\int^{\infty}_{0}(\int \vert R-r\K \vert^{p_{2}}d\mu)^{\frac{n-2}{n}} dt <\infty
\; \text{ and } \;
\liminf_{t\nearrow \infty}\int \vert R-r\K \vert^{p_{2}}d\mu =0
\end{split}\end{equation}
letting $p_{2}=p_{1}(\frac{n}{n-2})$. Note, that from this we may start an induction yielding
\begin{equation}
\lim_{t\nearrow \infty}\int \vert R-r\K \vert^{p_{k}}d\mu =0
\end{equation} and 
\begin{equation}\begin{split}
\int^{\infty}_{0}(\int \vert R-r\K \vert^{p_{k+1}}d\mu)^{\frac{n-2}{n}} dt <\infty
\; \text{ and } \;
\liminf_{t\nearrow \infty}\int \vert R-r\K \vert^{p_{k+1}}d\mu =0
\end{split}\end{equation}
letting $p_{k+1}=\frac{n}{n-2}p_{k}$ for $k\geq 1$. Thereby the claim is evidently proven.
\end{proof}

\section{The flow near infinity}
\label{sec:TheFlowNearInfinity}
\subsection{Blow-up analysis}
\label{subsec:BlowUpAnalysis}
For a Palais-Smale sequence of decreasing energy, say $u_{k}=u(t_{k})$ for a flow line $u$ and $t_{k}\- \infty$,
the lack of compactness is described as follows.
\begin{proposition}[Concentration-Compactness]\label{prop_concentration_compactness}$_{}$\\
Let
\;$
(u_{m})\subset W^{1,2}_{g_{0}}(M, \R_{>0})$ satisfy  $k_{u_{m}}=\int Ku_{m}^{\frac{2n}{n-2}}d\mu_{g_{0}}=1
$
and 
\begin{equation*}\begin{split}
\sup_{m \in \mathbb{N}}J(u_{m})<\infty
\; \text{ and } \;
\Vert\partial J(u_{m})\Vert\-0.
\end{split}\end{equation*}
Passing to a subsequence we then have
\begin{equation*}\begin{split}
J(u_{m})=r_{u_{m}}\-J_{\infty}=r_{\infty}.
\end{split}\end{equation*}
and there exist\;
$0 \leq  u_{\infty} \in W^{1,2}_{g_{0}}(M)$ with either $u_{\infty} \equiv 0$ or $u_{\infty} >0$
solving
\begin{equation*}\begin{split}
L_{g_{0}}u_{\infty}
=
r_{\infty }Ku_{\infty}^{\frac{n+2}{n-2}}
\end{split}\end{equation*}
and for some $p \in \mathbb{N}_{0}$ sequences $(a_{i_{m}})\subset M, \;(\lambda_{i_{m}})\subset \R_{>0}, \;i=1, \ldots,p$ with
\begin{equation*}\begin{split}
a_{i_{m}}\-a_{i_{\infty}}
\; \text{ and }\;
\lambda_{i_{m}}\-\infty
\; \text{ as }\; m\- \infty
\end{split}\end{equation*}
such, that
\begin{equation*}\begin{split}
\Vert
u_{m}
-
u_{\infty}
-
\sum^{p}_{i=1}\hat \delta_{a_{i_{m}}, \lambda_{i_{m}}}
\Vert\-0,         
\end{split}\end{equation*}
where 
\begin{equation*}\begin{split}
\hat \delta_{a_{i_{m}}, \lambda_{i_{m}}}
=
(\frac{4n(n-1)}{r_{\infty} K(a_{i})})^{\frac{n-2}{4}}\eta_{a_{i_{m}}}
(\frac{\lambda_{i_{m}}}{1+\lambda_{i_{m}}^{2}\vert\exp^{-1}_{a_{i_{m}}}(\cdot)\vert_{g_{0}}^{2}})^{\frac{n-2}{2}}
\end{split}\end{equation*}
with a cut-off function \;
$\eta_{a_{i_{m}}}=\eta(\vert \exp_{a_{i_{m}}}^{-1}(\cdot)\vert_{g_{0}}^{2})$, where 
\begin{equation*}\begin{split}
\eta\in C^{\infty}(B_{2}(0), \R_{\geq 0}), \; \eta\equiv 1\; \text{ on }\;B_{1}(0).
\end{split}\end{equation*}
More precisely there holds for each $i \neq j = 1, \ldots,p$
\begin{equation*}\begin{split}
\frac{\lambda_{i_{m}}}{\lambda_{j_{m}}}
+
\frac{\lambda_{j_{m}}}{\lambda_{i_{m}}}
+
\lambda_{i_{m}}\lambda_{j_{m}}d^{2}_{g_{0}}(a_{i_{m}},a_{j_{m}})
\- \infty
\; \text{ as }\;m\-\infty.
\end{split}\end{equation*}

\end{proposition}
This characterization is classical and we refer to \cite{StruweConcentrationCompactness}.
The proposition is proven by straight forward adaptation.
For the last statement  cf. \cite{Brezis_Coron_H_Systems}.

\subsection{Bubbles and interaction estimates}
\label{subsec:BubblesAndInteractionEstimates}
We refine the definition of blow up functions $\hat \delta_{a, \lambda}$ given in proposition \ref{prop_concentration_compactness}, 
referred to as bubbles, since they
form a spherical geometry around $a$. 

\begin{definition}[Bubbles] \label{def_bubbles}$_{}$\\
For $a\in M$ let $u_{a}$ introduce normal conformal coordinates around $a\in M$ via
\begin{equation*}\begin{split}
g_{a}=u_{a}^{\frac{4}{n-2}}g_{0}.
\end{split}\end{equation*}
Let  $G_{g_{ a }}$ be the Green's function of the conformal laplacian
\begin{equation*}\begin{split}
L_{g_{ a }}=-c_{n}\lap_{g_{ a }}+R_{g_{ a }}, \;
c_{n}=4\frac{n-1}{n-2}.
\end{split}\end{equation*}
For $\lambda>0$ let
\begin{equation*}\begin{split}
\varphi_{a, \lambda }
= &
u_{ a }(\frac{\lambda}{1+\lambda^{2} \gamma_{n}G^{\frac{2}{2-n}}_{ a }})^{\frac{n-2}{2}},
\;G_{ a }=G_{g_{ a }}( a, \cdot), \;
\gamma_{n}=(4n(n-1)\omega _{n})^{\frac{2}{n-2}}.
\end{split}\end{equation*}
One may expand 
\begin{equation*}\begin{split}
G_{ a }=\frac{1}{4n(n-1)\omega _{n}}(r^{2-n}_{a}+H_{ a }), \; r_{a}=d_{g_{a}}(a, \cdot)
, \; 
H_{ a }=H_{r,a }+H_{s, a }.
\end{split}\end{equation*}
There holds $H_{r,a }\in C^{2, \alpha}_{loc}$ and in conformal normal coordinates
\begin{equation*}
\begin{split}
H_{s,a}
=
O
\begin{pmatrix}
0 & \text{ for }\, n=3\\ r_{a}^{2}\ln r_{a} & \text{ for }\, n=4 \\ r_{a}& \text{ for }\, n=5 
\end{pmatrix}
\end{split}
\end{equation*} 
In addition 
it follows from  the positive  mass theorem, that
\begin{equation*}\begin{split}
H_{ a }( a )=0 \;\text{ for }\; M\simeq \mathbb{S}^{n} \; \text{ and }\; H_{a}(a)>0\;\text{ for }\;
 M\not \simeq \mathbb{S}^{n},
\end{split}\end{equation*}
so $H_{a}(a)$ is always non negative  with strict positivity unless $M$ is conformally equivalent to the standard sphere
$\mathbb{S}^{n}$
.
\end{definition}
For the expansion of the Green's function stated cf.  \cite{LeeAndParker}, Theorem 6.5.
Ibidem conformal normal coordinates are introduced in section 5, see also the improvement due to \cite{Guenter}.
Note, that we may and will replace $\hat \delta_{a, \lambda}$ by $\varphi_{a, \lambda}$ in proposition \ref{prop_concentration_compactness}, since
\begin{equation*}\begin{split}
\Vert \varphi_{a, \lambda}-\hat \delta_{a, \lambda}\Vert \- 0\; \text{ as }\; \lambda\- \infty.
\end{split}\end{equation*}
The reason for the above redefinition of bubbles is the simple way
to calculate their conformal laplacian in terms of its Green's function, see the lemma below,
whose proof we delay to the appendix.
\begin{lemma}[Emergence of the regular part]\label{lem_emergence_of_the_regular_part}$_{}$\\
One has 
$
L_{g_{0}}\varphi_{a, \lambda}=
O
(
\varphi_{a, \lambda}^{\frac{n+2}{n-2}}
)
$
and on a geodesic ball $B_{\alpha}( a )$ for $\alpha>0$ small 
\begin{equation*}\begin{split}
L_{g_{0}}\varphi_{a, \lambda}
= &
4n(n-1)\varphi_{a, \lambda}^{\frac{n+2}{n-2}}
-
2nc_{n}
r_{a}^{n-2}((n-1)H_{a}+r_{a}\partial_{r_{a}}H_{a}) \varphi_{a, \lambda}^{\frac{n+2}{n-2}} \\
& 
+
\frac{u_{a}^{\frac{2}{n-2}}R_{g_{a}}}{\lambda}\varphi_{a, \lambda}^{\frac{n}{n-2}}
+
o(r_{a}^{n-2})\varphi_{a, \lambda}^{\frac{n+2}{n-2}},
\end{split}\end{equation*}
where  $r_{a}=d_{g_{a}}(a, \cdot)$. Note, that $R_{g_{a}}=O(r_{a}^{2})$ in geodesic normal coordinates.
\end{lemma}

We would like to point out, that the term $\frac{R_{g_{a}}}{\lambda}\varphi_{a, \lambda}^{\frac{n}{n-2}}$ is negligible for our discussion, whereas it plays a crucial role in higher dimensions.

To abbreviate the  notation we make the following definitions.
\begin{definition}[Relevant quantities]\label{def_relevant_quantities}$_{}$\\
For $k,l=1,2,3$ and $ \lambda_{i} >0, \, a _{i}\in M, \,i= 1, \ldots,p$ define
\begin{enumerate}[label=(\roman*)]
 \item \quad 
$\var_{i}=\var_{a_{i}, \lambda_{i}}$ and $(d_{1,i},d_{2,i},d_{3,i})=(1,-\lambda_{i}\partial_{\lambda_{i}}, \frac{1}{\lambda_{i}}\nabla_{a_{i}})$
 \item \quad
$\phi_{1,i}=\varphi_{i}, \;\phi_{2,i}=-\lambda_{i} \partial_{\lambda_{i}}\varphi_{i}, \;\phi_{3,i}= \frac{1}{\lambda_{i}} \nabla_{ a _{i}}\varphi_{i}$, so
$
\phi_{k,i}=d_{k,i}\varphi_{i}
$
\end{enumerate}
\end{definition}

We collect some useful estimates, which are well known, so we delay their proof to the appendix.
They are essential for the rest of our discussion and will be heavily used. 
\begin{lemma}[Interactions]\label{lem_interactions}$_{}$\\
Let $k,l=1,2,3$ and $i,j = 1, \ldots,p$. We have
\begin{enumerate}[label=(\roman*)]
 \item \quad
$ \vert \phi_{k,i}\vert, \vert \lambda_{i}\partial_{\lambda_{i}}\phi_{k,i}\vert, \vert \frac{1}{\lambda_{i}}\nabla_{a_{i}} \phi_{k,i}\vert\leq C \varphi_{i}$
 \item \quad
$ 
\int \varphi_{i}^{\frac{4}{n-2}} \phi_{k,i}\phi_{k,i}
=
c_{k}\cdot id
+
O(\frac{1}{\lambda_{i}^{n-2}}+\frac{1}{\lambda_{i}^{2}}), \;c_{k}>0$
\item \quad
$
\int \varphi_{i}^{\frac{n+2}{n-2}}\phi_{k,j}
= 
b_{k}d_{k,i}\eps_{i,j}+o_{\varepsilon}(\eps_{i,j})
=
\frac{n+2}{n-2}\int \phi_{k,i}\varphi_{i}^{\frac{4}{n-2}}\varphi_{j}, \; b_{k}>0,
\,i\neq j$
 \item \quad 
$
\int \varphi_{i}^{\frac{4}{n-2}} \phi_{k,i}\phi_{l,i}
= 
O( \frac{1}{\lambda_{i}^{n-2}} +\frac{1}{\lambda_{i}^{2}})$
for $k\neq l$, 
$\int \varphi_{i}^{\frac{2n}{n-2}}=c_{1}+O(\frac{1}{\lambda_{i}^{n-2}})$ and
$$
\int \varphi_{i}^{\frac{n+2}{n-2}} \phi_{k,i}
= 
O( \frac{1}{\lambda_{i}^{n-2}})\; \text{ for }\; k=2,3
$$
 \item \quad
$
\int \varphi_{i}^{\alpha}\varphi_{j}^{\beta} 
=
O(\eps_{i,j}^{\beta})
$
for $i\neq j$ and $\alpha +\beta=\frac{2n}{n-2}, \;\frac{n}{n-2}> \alpha>\beta\geq 1 $
\item \quad
$
\int \varphi_{i}^{\frac{n}{n-2}}\varphi_{j}^{\frac{n}{n-2}} 
=
O(\eps^{\frac{n}{n-2}}_{i,j}\ln \eps_{i,j}), \,i\neq j
$
 \item \quad 
$
(1, \lambda_{i}\partial_{\lambda_{i}}, \frac{1}{\lambda_{i}}\nabla_{a_{i}})\eps_{i,j}=O(\eps_{i,j})
, \,i\neq j$,
\end{enumerate}
where $\epsilon=\min\{\frac{1}{\lambda_{i}}, \frac{1}{\lambda_{j}}, \eps_{i,j}\}$ and
\begin{equation*}\begin{split}
\eps_{i,j}
=
(
\frac{\lambda_{j}}{\lambda_{i}}
+
\frac{\lambda_{i}}{\lambda_{j}}
+
\lambda_{i}\lambda_{j}\gamma_{n}G_{g_{0}}^{\frac{2}{2-n}}(a _{i},a _{j})
)^{\frac{2-n}{2}}.
\end{split}\end{equation*}
\end{lemma}
Here we used and will use later on $a=o_{\epsilon}(b)$  as short hand for
\begin{align*}
\vert a \vert \leq \omega(\varepsilon)\vert b \vert\; \text{ with } \;\omega(\epsilon)\- 0 \text{ as }\; \epsilon\- 0.
\end{align*}

\subsection{Degeneracy and pseudo critical points}
\label{subsec:DegeneracyAndPseudoCriticalPoints}
In order to obtain a precise description of the dynamical behaviour of a flow line
we have to take care of a possible degeneracy of $J$ at a critical point. 
\begin{lemma}[Spectral theorem and degeneracy] \label{lem_spectral_theorem_and_degeneracy}$_{}$\\
Let $\omega > 0$ solve $L_{g_{0}}\omega =K\omega  ^{\frac{n+2}{n-2}}$. 

Then there exists a set of solutions 
\begin{equation*}\begin{split}
L_{g_{0}}w_{i}=\mu_{w_{i}}K\omega  ^{\frac{4}{n-2}}w_{i}, \, \mu_{w_i}\- \infty
\end{split}\end{equation*}

such, that
\begin{equation*}\begin{split}
\langle w_{i},w_{j}\rangle_{L_{g_{0}}}=\delta_{ij},
\langle w_{i}\mid i \in \N \rangle = W^{1,2}_{g_{0}}(M)
\end{split}\end{equation*}

and for any eigenspace $E_{\mu}(\omega)=\langle w_{i}\mid \mu_{w_{i}}=\mu\rangle$ we have $\dim E_{\mu}<\infty$. 
\\
Moreover we have $\partial J(\omega)=0$ and isomorphy
\begin{equation*}\begin{split}
\partial^{2}J(\omega)\lfloor_{H_{0}(\omega)^{\perp_{L_{g_{0}}}}}:
H_{0}(\omega)^{\perp_{L_{g_{0}}}}
\overset{\simeq}{\longrightarrow}
(H_{0}(\omega)^{\perp_{L_{g_{0}}}})^{*}
,
\end{split}\end{equation*}
where 
\begin{equation*}\begin{split}
H_{0}(\omega )=\langle \omega\rangle \oplus \langle \mathrm{e}_{i}\mid i=1, \ldots,m\rangle
\end{split}\end{equation*}
with
\begin{equation*}\begin{split}
\langle \mathrm{e}_{i}\mid i=1, \ldots,m\rangle=E_{\frac{n+2}{n-2}}(\omega), \, \langle \mathrm{e}_{i}, \mathrm{e}_{j}\rangle_{L_{g_{0}}}=\delta_{ij}
\end{split}\end{equation*}
denotes the kernel of $\partial^{2}J$ at $\omega$
and
$H_{0}(\omega)^{\perp_{L_{g_{0}}}}$
is the orthogonal of $H_{0}(\omega)$ with respect $\langle \cdot, \cdot\rangle_{L_{g_{0}}}$.
The case $E_{\frac{n+2}{n-2}}(\omega)=\emptyset$ is generic.
\end{lemma}
Please note, that due to scaling invariance of the functional the kernel always contains $\omega$ itself. 
We may thus call $\omega$ (essentially) non degenerate,  if simply $H_{0}(\omega)=\langle \omega \rangle$, or equivalently,
if $E_{\frac{n+2}{n-2}}(\omega)=\emptyset$. The foregoing lemma asserts, that non degeneracy  is generic.

\begin{proof}[\textbf{Proof of lemma \ref{lem_spectral_theorem_and_degeneracy}}]$_{}$\\
The statement on the basis $\{w_{i}\mid i\in \N\}$ of eigenfunctions is a direct application of the spectral theorem for compact operators. 
Moreover 
\begin{equation}
\begin{split}
r_{\omega}=\int L_{g_{0}}\omega \omega=\int K\omega^{\frac{2n}{n-2}}=k_{\omega}
\end{split}
\end{equation} 
for a solution $L_{g_{0}}\omega =K\omega  ^{\frac{n+2}{n-2}}$. Thus proposition \ref{prop_derivatives_of_J} 
shows
\begin{equation}\label{def_w_ej}
\begin{split}
\partial J(\omega), \partial^{2} J(\omega)\omega, \partial^{2}J(\omega)\mathrm{e}_{j}=0,
\end{split}
\end{equation} 
which is easy to check. Likewise for
$v\perp_{L_{g_{0}}} \omega$ one obtains
\begin{equation}
\begin{split}
\frac{1}{2}\partial^{2}J(\omega)vf
=
k_{\omega}^{\frac{2-n}{n}}
\int (L_{g_{0}}v-\frac{n+2}{n-2}\omega^{\frac{4}{n-2}}v)f.
\end{split}
\end{equation} 
This proves the claim with isomorphy of
\begin{equation}\label{iso1}
\begin{split}
\partial^{2}J(\omega)\lfloor_{H_{0}(\omega)^{\perp_{L_{g_{0}}}}}:
H_{0}(\omega)^{\perp_{L_{g_{0}}}}
\overset{\simeq}{\longrightarrow}
(H_{0}(\omega)^{\perp_{L_{g_{0}}}})^{*}
\end{split}
\end{equation} 
given by
\begin{equation}\label{iso2}
\begin{split}
w_{i}
\- 
2k_{\omega}^{\frac{2-n}{n}}(1-\mu_{w_i}^{-1}\frac{n+2}{n-2})\langle w_{i}, \cdot\rangle_{L_{g_{0}}}.
\end{split}
\end{equation} 

We are left with proving genericity of $E_{\frac{n+2}{n-2}}(\omega)=\emptyset$.
\\
To that end consider the scalar curvature mapping
\begin{equation}
\begin{split}
R:C^{2, \alpha}(M,A_{\epsilon})\- C^{0, \alpha}(M):\omega\- R_{\omega}=\omega^{-\frac{n+2}{n-2}}L\omega,
\end{split}
\end{equation} 
where $A_{\epsilon}=(\epsilon, \epsilon^{-1})$ for some $\epsilon>0$, with derivative
\begin{equation}
\begin{split}
\partial R_{\omega}\cdot v
= &
\omega^{-\frac{n+2}{n-2}}(L_{g_{0}}v-\frac{n+2}{n-2}R_{\omega}\omega^{\frac{4}{n-2}}v).
\end{split}
\end{equation} 
Note, that for $\omega \in C^{2, \alpha}(M,A_{\epsilon})$ fixed we have isomorphy of
\begin{equation}
\begin{split}
C^{2, \alpha}(M)\- C^{0, \alpha}(M):v\- \omega^{-\frac{n+2}{n-2}}L_{g_{0}}v
\end{split}
\end{equation} 
and compactness of 
\begin{equation}
\begin{split}
C^{2, \alpha}(M)\-C^{0, \alpha}(M):v\- R_{\omega}\omega^{\frac{4}{n-2}}v.
\end{split}
\end{equation} 
Thus $\partial R$ is a Fredholm operator and the Smale-Sard lemma gives 
\begin{equation}
\begin{split}
R[\partial R\neq 0]= \cap^{\infty}_{k=1}O_{k}
\end{split}
\end{equation} 
with countably many open and dense subsets $O_{k}\subset Im(R)$. Covering
\begin{equation}
\begin{split}
\R_{>0}=\cup_{k=1}^{\infty}A_{\frac{1}{k}} 
\end{split}
\end{equation} 
we obtain the same result for $R:C^{2, \alpha}(M, \R_{>0})\-C^{0, \alpha}(M)$.

Thus, if $K\in C^{0, \alpha}(M)$ is the scalar curvature of a conformal metric
\begin{equation}
\begin{split}
K=R_{\omega}=\omega^{-\frac{n+2}{n-2}}L\omega, \, \omega\in C^{2, \alpha}(M, \R_{>0}),
\end{split}
\end{equation} 

then obviously $K\in Im(R)$ and generically $K\in R[\partial R\neq 0]$, so
\begin{equation}
\begin{split}
L_{g_{0}}v-\frac{n+2}{n-2}K\omega^{\frac{4}{n-2}}v\neq 0\; \text{ for all }\; 0\neq v\in C^{2, \alpha}(M),
\end{split}
\end{equation} 

whenever $K=R_{\omega}$. Consequently for a solution $L\omega=K\omega^{\frac{n+2}{n-2}}$
\begin{equation}
\begin{split}
\partial^{2}J(\omega)= \frac{2}{k^{\frac{n-2}{n}}}(L_{g_{0}}u-\frac{n+2}{n-2}\omega^{\frac{n+2}{n-2}})
\end{split}
\end{equation} 

is for a generic $K$ invertible, which is equivalent to $E_{\frac{n+2}{n-2}}(\omega)=\emptyset$.
\\
Please note, that we may replace $C^{2, \alpha},C^{0, \alpha}$ by any 
$C^{k+2, \alpha},C^{k, \alpha}$. So $\gamma=1$ is true for 
any $K$ in a dense subset of $C^{k, \alpha}$
\end{proof}
In light of the foregoing lemma the following parametrization is a natural application of the implicit function theorem.
\begin{lemma}[Degeneracy and pseudo critical points]\label{lem_degeneracy_and_pseudo_critical_points}$_{}$\\
For $\omega > 0$ solving $L_{g_{0}}\omega =K\omega  ^{\frac{n+2}{n-2}}$ let  
\begin{equation*}\begin{split}
\Pi=\Pi_{H_{0}(\omega )^{\perp_{L_{g_{0}}}}}
\end{split}\end{equation*}
be the projection on $H_{0}(\omega)^{\perp_{L_{g_{0}}}}$.

Then there exist $\epsilon>0$, an open neighbourhood $U$ of $\omega$
\begin{equation*}\begin{split}
\omega \in U\subset W^{1,2}_{g_{0}}(M)
\end{split}\end{equation*}

and a smooth function 
$
h:B_{\epsilon}^{\R^{m+1}}(0)\- H_{0}(\omega )^{\perp_{L_{g_{0}}}}
$
such, that
\begin{equation*}
\begin{split}
\{
w\in U 
& 
\mid \Pi \nabla J(w)=0
\}
\\
= &
\{
u_{\alpha, \beta}
=
(1+\alpha) \omega +\beta^{i}\mathrm{e}_{i}
+
h(\alpha, \beta)\mid 
(\alpha, \beta)\in B_{\epsilon}^{m+1}(0)\}
\end{split}
\end{equation*} 

with  
$$\Vert h(\alpha, \beta)\Vert =O(\vert \alpha\vert^{2}+\Vert \beta \Vert^{2}),$$

where $\nabla J$ is gradient of $\partial J$ with respect 
to the scalar product $\langle \cdot, \cdot \rangle_{L_{g_{0}}}$.\\\smallskip
We call $w \in U$ a pseudo critical point related to $\omega$, if 
$\Pi_{H_{0}(\omega)^{\perp_{L_{g_{0}}}}} \nabla J(w)=0$. 
\end{lemma}
\noindent
Thus the construction above parametrizes in a neighbourhood of $\omega$ the set 
of pseudo critical points related to $\omega$; and clearly every critical point of 
$J$  is 
a pseudo critical point related to $\omega$ as well.

For the sake of clarity consider $u_{\alpha, \beta}>0$ close to $\omega$ solving 
$$\Pi_{H_{0}(\omega)^{\perp_{L_{g_{0}}}}}\nabla J(u_{\alpha, \beta})=0.$$ 

Then 
\begin{equation*}\begin{split}
\partial J(u_{\alpha, \beta})f
=
2k_{u_{\alpha, \beta}}^{\frac{2-n}{n}}\int (L_{g_{0}}u_{\alpha, \beta}-\frac{r_{u_{\alpha, \beta}}}{k_{u_{\alpha, \beta}}}Ku_{\alpha, \beta}^{\frac{n+2}{n-2}})f,
\end{split}\end{equation*}

so $\nabla J(u_{\alpha, \beta})=\overline u_{\alpha, \beta}$ solves 
$$L_{g_{0}}\overline u_{\alpha, \beta}=2k_{u_{\alpha, \beta}}^{\frac{2-n}{n}}
(L_{g_{0}}u_{\alpha, \beta}-\frac{r_{u_{\alpha, \beta}}}{k_{u_{\alpha, \beta}}}u_{\alpha, \beta}^{\frac{n+2}{n-2}}).$$ 

Thus
$\Pi \overline u_{\alpha, \beta}=0$ implies
\begin{equation*}\begin{split}
L_{g_{0}}u_{\alpha, \beta}& -\frac{r_{u_{\alpha, \beta}}}{k_{u_{\alpha, \beta}}}Ku_{\alpha, \beta}^{\frac{n+2}{n-2}}
=
\frac{k_{u_{\alpha, \beta}}^{\frac{n-2}{n}}}{2}L_{g_{0}}\overline u_{\alpha, \beta} \\
= &
\frac{k_{u_{\alpha, \beta}}^{\frac{n-2}{n}}}{2}\langle \overline u_{\alpha, \beta}, \frac{\omega}{\Vert \omega \Vert} \rangle_{L_{g_{0}}} L_{g_{0}}\frac{\omega}{\Vert \omega \Vert}
+
\frac{k_{u_{\alpha, \beta}}^{\frac{n-2}{n}}}{2}\sum_{j=1}^{m}\langle \overline u_{\alpha, \beta}, \mathrm{e}_{j}\rangle_{L_{g_{0}}} L_{g_{0}}\mathrm{e}_{j} \\
= &
[\int (L_{g_{0}}u_{\alpha, \beta}-\frac{r_{u_{\alpha, \beta}}}{k_{u_{\alpha, \beta}}}Ku_{\alpha, \beta}^{\frac{n+2}{n-2}})
\frac{w}{\Vert w \Vert}
]
L_{g_{0}}\frac{w}{\Vert w \Vert} \\
& +
\sum_{j=1}^{m}[\int (L_{g_{0}}u_{\alpha, \beta}-\frac{r_{u_{\alpha, \beta}}}{k_{u_{\alpha, \beta}}}Ku_{\alpha, \beta}^{\frac{n+2}{n-2}})\mathrm{e}_{j}]
L_{g_{0}}\mathrm{e}_{j}
\end{split}\end{equation*}

\begin{proof}[\textbf{Proof of lemma \ref{lem_degeneracy_and_pseudo_critical_points}}]$_{}$\\
The statement is a mere application of the implicit function theorem to
\begin{equation}
\begin{split}
W^{1,2}(M)=H_{0}(\omega )\oplus_{L_{g_{0}}} H_{0}(\omega )^{\perp_{L_{g_{0}}}}\- H_{0}(\omega )^{\perp_{L_{g_{0}}}}
:
u\- \Pi\nabla  J(u).
\end{split}
\end{equation}
Indeed
$
\Pi\nabla  J(\omega )=0, 
$
since $\nabla J(\omega)=0$. Moreover
\begin{equation}
\begin{split}
\nabla (\Pi \nabla J)(\omega)
=
\Pi \nabla^{2} J(\omega).
\end{split}
\end{equation}
and from \eqref{iso1} and \eqref{iso2}  we have isomorphy
\begin{equation}
\begin{split}
\nabla^{2}J(\omega)\lfloor_{H_{0}(\omega )^{\perp_{L_{g_{0}}}}}
:
H_{0}(\omega )^{\perp_{L_{g_{0}}}}
\overset{\simeq}{\-}
H_{0}(\omega )^{\perp_{L_{g_{0}}}}.
\end{split}
\end{equation}
As $\Pi$ is the identity operator on $H_{0}(\omega )^{\perp_{L_{g_{0}}}}$, we obtain
\begin{equation}
\begin{split}
\nabla_{H_{0}(\omega )^{\perp_{L_{g_{0}}}}} (\Pi \nabla J)(\omega)
= 
\nabla^{2} J(\omega)\lfloor_{H_{0}(\omega )^{\perp_{L_{g_{0}}}}}
\end{split}
\end{equation}
and therefore isomorphy of $\nabla_{H_{0}(\omega )^{\perp_{L_{g_{0}}}}} (\Pi \nabla J)(\omega)$ as well.

Finally the estimate on $h$ follows from \eqref{def_w_ej}.
\end{proof}
Using Moser iteration one may improve this result to a smooth setting.
\begin{proposition}[Smoothness of $u_{\alpha, \beta}$]\label{prop_smoothness_of_u_a_b}$_{}$\\
For any $k\in \N$ we have $w, \mathrm{e}_{i},u_{\alpha, \beta},h_{\alpha, \beta}\in C^{k}$ and 
\begin{equation*}\begin{split}
\Vert h(\alpha, \beta)\Vert_{C^{k}}\- 0\; \text{ as }\;  \vert \alpha\vert+\Vert \beta \Vert\- 0.
\end{split}\end{equation*}
\end{proposition}
\begin{proof}[\textbf{Proof of proposition \ref{prop_smoothness_of_u_a_b}}]$_{}$\\
In view of lemma \ref{lem_spectral_theorem_and_degeneracy} let us write
\begin{equation}
\begin{split}
u_{\alpha, \beta}=(1+\alpha) \omega +\beta^{i}\mathrm{e}_{i}+h(\alpha, \beta).
\end{split}
\end{equation} 
The equation solved by $u_{\alpha, \beta}$ is $\Pi \nabla J=0$, which is equivalent to
\begin{equation}\label{eq_solved}
\begin{split}
L_{g_{0}}u_{\alpha, \beta}-(r\K)_{u_{\alpha, \beta}}u_{\alpha, \beta}^{\frac{n+2}{n-2}}
= &
[
\int 
(
L_{g_{0}}u_{\alpha, \beta}-(r\K)_{u_{\alpha, \beta}}u_{\alpha, \beta}^{\frac{n+2}{n-2}}
)
\frac{\omega}{\Vert \omega \Vert} 
]
L_{g_{0}}\frac{\omega}{\Vert \omega \Vert}  \\
& +
\sum_{i=1}^{m}
[ \int 
(L_{g_{0}}u_{\alpha, \beta}-(r\K)_{u_{\alpha, \beta}}u_{\alpha, \beta}^{\frac{n+2}{n-2}}) \mathrm{e}_{i}]
L_{g_{0}}\mathrm{e}_{i}
\end{split}
\end{equation} 
In particular
$L_{g_{0}}u_{\alpha, \beta}
=
P u_{\alpha, \beta}+v_{\alpha, \beta}
$
with $\Vert v_{\alpha, \beta} \Vert_{W^{-1,2}_{g_{0}}(M)}=O(\vert \alpha \vert + \Vert \beta \Vert)$  and
\begin{equation}\label{P}
\begin{split}
\Vert P \Vert_{L^{\frac{n}{2}}(B_{r}(x_{0}))}\overset{r\to 0}{\-} 0\; \text{ for all }\; x_{0}\in M. 
\end{split}
\end{equation} 
Let $p\geq 1$ and consider  a suitable cut-off function $\eta\in C^{1}_{0}(B_{2r}(x_{0}))$. For 
\begin{equation}
\begin{split}
w_{\alpha, \beta} = u^{2p-1}_{\alpha, \beta}\eta^{2} \; \text{ and }\; \text{w}_{\alpha, \beta} =u^{p}_{\alpha, \beta}\eta
\end{split}
\end{equation} 
one obtains using Young's inequality and absorption
\begin{equation}
\begin{split}
\vert  \nabla  \text{w}_{\alpha, \beta} \vert^{2}_{g_{0}}
\leq 
c_{p}(\langle \nabla  u_{\alpha, \beta}, \nabla  w_{\alpha, \beta}\rangle_{g_{0}}+ u^{2p}_{\alpha, \beta}\vert \nabla \eta \vert_{g_{0}}^{2})
\end{split}
\end{equation} 
and thus 
\begin{equation}
\begin{split}
\int L_{g_{0}} \text{w}_{\alpha, \beta}  \text{w}_{\alpha, \beta}
= &
\int c_{n}\vert \nabla  \text{w}_{\alpha, \beta} \vert^{2}_{g_{0}}+R_{g_{0}} \text{w}^{2}_{\alpha, \beta} \\
\leq &
c_{n,p}\int L_{g_{0}} u_{\alpha, \beta}  w_{\alpha, \beta}
+
 u ^{2p}_{\alpha, \beta}\vert \nabla \eta \vert_{g_{0}}^{2} \\
= &
c_{n,p}\int P u_{\alpha, \beta}  w_{\alpha, \beta}+v_{\alpha, \beta}  w_{\alpha, \beta} 
+ 
u^{2p}_{\alpha, \beta}\vert \nabla \eta \vert^{2}_{g_{0}}.
\end{split}
\end{equation} 
As $ \text{w}^{2}_{\alpha, \beta}= u_{\alpha, \beta} w_{\alpha, \beta}$ 
and $ w_{\alpha, \beta}= \text{w}_{\alpha, \beta} u^{p-1}_{\alpha, \beta}\eta$
one may absorb via 
\eqref{P} 
to get
\begin{equation}
\begin{split}
\int L_{g_{0}} \text{w}_{\alpha, \beta}  \text{w}_{\alpha, \beta}
\leq  &
C_{n,p}
(
\Vert v_{\alpha, \beta}  u ^{p-1}_{\alpha, \beta}\Vert_{L^{\frac{2n}{n+2}}}^{2}
+
\Vert  u^{2p}_{\alpha, \beta}\Vert_{L_{g_{0}}^{1}}
).
\end{split}
\end{equation} 
Suppose $ u_{\alpha, \beta} \in L^{r}, \,r\geq \frac{2n}{n-2}$. We then get for $p=\frac{r}{2}$ using H\"older's inequality
\begin{equation}
\begin{split}
\int L_{g_{0}} \text{w}_{\alpha, \beta}  \text{w}_{\alpha, \beta}
\leq &
C_{n,p}(\Vert v_{\alpha, \beta} \Vert^{2}_{L_{g_{0}}^{\frac{nr}{n+r}}}\Vert  u_{\alpha, \beta} \Vert_{L_{g_{0}}^{r}}^{r-2}+
\Vert  u_{\alpha, \beta} \Vert^{r}_{L_{g_{0}}^{r}}) \\
\leq &
C_{n,p}
(
\Vert v_{\alpha, \beta} \Vert^{r}_{L_{g_{0}}^{\frac{nr}{n+r}}}
+
\Vert  u_{\alpha, \beta} \Vert^{r}_{L_{g_{0}}^{r}}).
\end{split}
\end{equation} 
whence using a suitable covering $M=\sum^{m}_{i=1}B_{r_{i}}(x_{i})$ we get 
\begin{equation}\label{r_interpolation}
\begin{split}
\Vert  u_{\alpha, \beta} \Vert^{2}_{L_{g_{0}}^{\frac{n}{n-2}r}}
\leq 
C_{n,p}
(
\Vert v_{\alpha, \beta} \Vert^{r}_{L_{g_{0}}^{\frac{nr}{n+r}}}
+
\Vert  u_{\alpha, \beta} \Vert^{r}_{L_{g_{0}}^{r}}).
\end{split}
\end{equation} 

Note, that in case $\vert \alpha \vert +\Vert \beta \Vert=0$ we have $u_{\alpha, \beta}=\omega$ 
and $v_{\alpha, \beta}=0$, whence

by iteration of \eqref{r_interpolation} one obtains $w\in L^{p}_{g_{0}}$ for all $1\leq p< \infty$. Due to
\begin{equation*}\begin{split}
L_{g_{0}}\omega=K\omega^{\frac{n+2}{n-2}}
\;\text{ and }\;
L_{g_{0}}\mathrm{e}_{j}=\frac{n+2}{n-2}K\omega^{\frac{4}{n-2}}\mathrm{e}_{j}
\end{split}\end{equation*}

this gives $\omega, \mathrm{e}_{j}\in C^{\infty}$ by standard regularity arguments. \\
Recalling \eqref{eq_solved} this implies $v_{\alpha, \beta}\in C^{k}$ and 
\begin{equation}
\begin{split}
\Vert v_{\alpha, \beta}\Vert_{C^{k}}=O(\vert \alpha \vert +\Vert \beta \Vert) 
\end{split}
\end{equation} 
Thus we obtain by iteration of \eqref{r_interpolation} 
\begin{equation}
\begin{split}
\fa 1\leq q <\infty\;:\; \sup_{\vert \alpha \vert +\Vert \beta \Vert<\epsilon}\Vert  u_{\alpha, \beta}\Vert_{L_{g_{0}}^{q}}<\infty.
\end{split}
\end{equation} 
and therefore $\sup_{\vert \alpha \vert +\Vert \beta \Vert<\epsilon}\Vert u_{\alpha, \beta}\Vert_{C^{k}}<\infty$. Since by the very definition of $u_{\alpha, \beta}$
\begin{equation}
\begin{split}
\Vert h(\alpha, \beta)\Vert \-0 \;\text{ for }\; \vert \alpha \vert + \Vert \beta \Vert\- 0,
\end{split}
\end{equation} 
this convergence generalizes to all $C^{k}$ by compact embedding.
\end{proof}

Note, that due to  scaling invariance
\begin{equation*}\begin{split}
\Pi \nabla J(\omega)=0 \gdw \fa \alpha >0\;:\; \Pi \nabla J(\alpha \omega)=0.
\end{split}\end{equation*}

Thus we may reparametrise the pseudo critical points related to $\omega$ as 
\begin{equation*}\begin{split}
u_{\alpha, \beta}=\alpha(\omega +\beta^{i}\mathrm{e}_{i}+h(\beta)), h(\beta)\perp_{L_{g_{0}}} H_{0}(\omega),
\end{split}\end{equation*}

where
$
\Vert h(\beta)\Vert=O(\Vert \beta \Vert^{2})
$
and $\Vert h(\beta)\Vert_{C^{k}}\- 0$ as $\Vert \beta \Vert \- 0$.

\subsection{Critical points at infinity}
\label{subsec:CriticalPointsAtInfinity}

\begin{definition}[A neighbourhood of critical points at infinity]\label{def_V(omega,p,e)}$_{}$

Let  
$\omega\geq 0$ solve $L_{g_{0}}\omega =K\omega  ^{\frac{n+2}{n-2}}, \,p\in \N$ and $\eps>0$ sufficiently small. 
\\
For $u\in W^{1,2}_{g_{0}}(M)$ we define 
\begin{equation*}
\begin{split}
A_{u}(\omega,p, \eps)
=
\{ &
(\alpha, \beta_{k}, \alpha_{i}, \lambda_{i},a_{i})\in (\R_{+}, \R^{m}, \R^{p}_{+}, \R^{p}_{+},M^{p}) \mid
\\
& \;
\underset{i\neq j}{\fa}\;
 \lambda_{i}^{-1}, \lambda_{j}^{-1}, \eps_{i,j}, \vert 1-\frac{r\alpha_{i}^{\frac{4}{n-2}}K(a_{i})}{4n(n-1)k}\vert,
\\
& \quad\;\;\;
\vert 1-\frac{r\alpha^{\frac{4}{n-2}}}{k}\vert, \Vert \beta \Vert,
\Vert u-u_{\alpha, \beta}-\alpha^{i}\varphi_{a_{i}, \lambda_{i}}\Vert
<\eps\
\},
\end{split}
\end{equation*} 
where recalling lemma \ref{lem_interactions} we have
\begin{equation*}\begin{split}
\eps_{i,j}
=
(
\frac{\lambda_{j}}{\lambda_{i}}
+
\frac{\lambda_{i}}{\lambda_{j}}
+
\lambda_{i}\lambda_{j}\gamma_{n}G_{g_{0}}^{\frac{2}{2-n}}(a _{i},a _{j})
)^{\frac{2-n}{2}}.
\end{split}\end{equation*}
In case $p>0$ we call 
\begin{equation*}\begin{split}
V(\omega, p, \eps)
= 
\{
u\in W^{1,2}_{g_{0}}(M)  
\mid
A_{u}(\omega,p, \eps)\neq \emptyset
\}
\end{split}\end{equation*}
 a neighbourhood of a critical point at infinity.
\end{definition}
Keep in mind, that $k\equiv 1$ and $r\searrow r_{\infty}$ along a flow line. We would like to make a remark on two special cases.
\begin{enumerate}[label=(\roman*)]
 \item 
If $\omega=0$, then
$u_{\alpha, \beta}=0$. So the conditions on $\alpha$ and $\beta_{k}$ are trivial. Thus the sets $A_{u}(0,p, \eps)$ and $V(0,p, \eps)$ naturally reduce to
\begin{equation*}
\begin{split}
A_{u}(p, \eps)
=
\{ &
(\alpha_{i}, \lambda_{i},a_{i})\in (\R^{p}_{+}, \R^{p}_{+},M^{p}) \mid
\\
&\,
\underset{i\neq j}{\fa}\;
 \lambda_{i}^{-1}, \lambda_{j}^{-1}, \eps_{i,j}, \vert 1-\frac{r\alpha_{i}^{\frac{4}{n-2}}K(a_{i})}{4n(n-1)k}\vert,
\Vert u-\alpha^{i}\varphi_{a_{i}, \lambda_{i}}\Vert
<\eps\
\}
\end{split}
\end{equation*} 
and 
$
V(p, \eps)=\{u\in W^{1,2}_{g_{0}}(M)\mid A_{u}(p, \eps)\neq \emptyset\}.
$

 \item
$V(\omega,0, \eps)$ corresponds to a neighbourhood the critical point line 
\begin{equation*}
\begin{split}
 \{\alpha \omega\mid \alpha>0\}.
\end{split}
\end{equation*} 
\end{enumerate}
So proposition \ref{prop_concentration_compactness} states, that every sequence
$u(t_{k})$ is precompact 
with respect to $V(\omega, p, \eps)$ in the sense, that up to a subsequence for any $\eps>0$ 
we find an index $k_{0}$, for which $u_{t_{k}} \in V(w,p, \eps)$ for some $p\geq 0$ and all $k\geq k_{0}$.

The subsequent reduction by minimization, whose prove we postpone to the appendix, 
makes the representation in $V(\omega, p, \eps)$ unique.

\begin{proposition}[Optimal choice]\label{prop_optimal_choice}$_{}$\\
For every $\eps_{0}>0$  there exists $\eps_{1}>0$ such, that  for $u\in V(\omega, p, \eps)$ with $\eps<\eps_{1}$
\begin{equation*}\begin{split}
\inf
_
{
(\tilde \alpha, \tilde\beta_{k}, \tilde\alpha_{i}, \tilde a_{i}, \tilde\lambda_{i})\in A_{u}(\omega,p,2\eps_{0}) 
}
\int 
Ku^{\frac{4}{n-2}}
\vert 
u
-
u_{\tilde \alpha, \tilde \beta}
-
\tilde\alpha^{i}\varphi_{\tilde a_{i}, \tilde \lambda_{i}}
\vert^{2}
\end{split}\end{equation*}
admits an unique minimizer $(\alpha, \beta_{k}, \alpha_{i},a_{i}, \lambda_{i})\in A_{u}(\omega,p, \eps_{0})$ and we define
\begin{equation*}\begin{split}
\varphi_{i}=\varphi_{a_{i}, \lambda_{i}},v=u-u_{\alpha, \beta}-\alpha^{i}\varphi_{i},
\;\, \eps_{i,j}
=
(
\frac{\lambda_{j}}{\lambda_{i}}
+
\frac{\lambda_{i}}{\lambda_{j}}
+
\lambda_{i}\lambda_{j}\gamma_{n}G_{g_{0}}^{\frac{2}{2-n}}(  a _{i},  a _{j})
)^{\frac{2-n}{2}}.
\end{split}\end{equation*}
Moreover $(\alpha, \beta_{k}, \alpha_{i},a_{i}, \lambda_{i})$ depends smoothly on $u$.
\end{proposition} 
\noindent
Thus for a sequence $u_{l}\in V(\omega,p, \eps_{l}), \, \eps_{l}\-0$ we may assume, that for each $u_{l}$ there exists an unique
representation in $A_{u_{l}}(\omega,p, \eps_{0})$, say
\begin{equation*}\begin{split}
u_{l}=u_{\alpha_{l}, \beta_{l}}+\alpha^{i,l}\var_{a_{i,l}, \lambda_{i,l}}+v_{l}
, \;(\alpha_{l}, \beta_{k,l}, \alpha_{i,l},a_{i,l}, \lambda_{i,l})
\in A_{u_{l}}(\omega,p, \eps_{0})
\end{split}\end{equation*}
and we have 
$(\alpha_{l}, \beta_{k,l}, \alpha_{i,l},a_{i,l}, \lambda_{i,l})
\in A_{u_{l}}(\omega,p, \epsilon_{l})$
for suitable $\epsilon_{l}\-0$.

The error term  
$v=u-u_{\alpha, \beta}-\alpha^{i}\var_{i}$ is with respect to the scalar product
\begin{equation*}\begin{split}
\langle \cdot, \cdot\rangle_{Ku^{\frac{4}{n-2}}}
=
\int \cdot \,Ku^{\frac{4}{n-2}}\, \cdot  
\end{split}\end{equation*}

orthogonal to 
\begin{equation*}\begin{split}
\langle 
u_{\alpha, \beta}, \partial_{\beta_{i}}u_{\alpha, \beta}, \varphi_{i},-\lambda_{i}\partial_{\lambda_{i}}\varphi_{i}, \frac{1}{\lambda_{i}}\nabla_{a_{i}}\varphi_{i}
\rangle 
\end{split}\end{equation*}

and due to $\vert \delta J(u)\vert\- 0$ almost orthogonal with respect to 
\begin{equation*}\begin{split}
\langle \cdot, \cdot\rangle_{L_{g_{0}}}=\int \cdot L_{g_{0}} \cdot
\end{split}\end{equation*}

\begin{definition}[The orthogonal bundle $H(\omega, p, \eps)$]\label{def_H(omega,p,e)}$_{}$\\
For $u\in V(\omega, p, \eps)$ let
\begin{equation*}\begin{split}
H_{u}(\omega, p, \eps)
=
\langle 
u_{\alpha, \beta}, \partial_{\beta_{i}}u_{\alpha, \beta}, \varphi_{i},-\lambda_{i}\partial_{\lambda_{i}}\varphi_{i}, \frac{1}{\lambda_{i}}\nabla_{a_{i}}\varphi_{i}
\rangle
^{\perp_{Ku^{\frac{4}{n-2}}}}
\end{split}\end{equation*}
in case $\omega>0$ and in case $\omega=0$
\begin{equation*}\begin{split}
H_{u}(p, \eps)
=
\langle 
\varphi_{i},-\lambda_{i}\partial_{\lambda_{i}}\varphi_{i}, \frac{1}{\lambda_{i}}\nabla_{a_{i}}\varphi_{i}
\rangle
^{\perp_{Ku^{\frac{4}{n-2}}}}
\end{split}\end{equation*}
\end{definition}

Orthogonality of the error term  $v$ implies smallness of linear interactions.
Subsequently we will even show, that essentially 
$v$ is negligible.
\begin{lemma}[Linear $v$-type interactions]\label{lem_v_type_interactions}$_{}$\\
On $V(\omega, p, \eps)$ for $\eps>0$ small we have
\begin{enumerate}[label=(\roman*)]
 \item \quad
$
\int  L_{g_{0}}\phi_{k,i}v 
= 
o( \frac{1}{\lambda_{i}^{\frac{n-2}{2}}} + \sum_{i\neq j=1}^{p}\eps_{i,j})
+
O
(\Vert v \Vert^{2})
$
 \item \quad
$
\int L_{g_{0}}u_{\alpha, \beta}v
=
o(\sum_{r}\frac{1}{\lambda_{r}^{\frac{n-2}{2}}})
+
O(\Vert v \Vert^{2}+\vert \delta J(u)\vert^{2})
$ 
\item \quad
$
\int Ku^{\frac{n+2}{n-2}}\phi_{k,i}
= 
\int K (u_{\alpha, \beta}+\alpha^{j}\varphi_{j})^{\frac{n+2}{n-2}}\phi_{k,i}
+
O(\Vert v \Vert^{2})
$
\item \quad
$
\int Ku^{\frac{n+2}{n-2}}u_{\alpha, \beta}
= 
\int K (u_{\alpha, \beta}+\alpha^{j}\varphi_{j})^{\frac{n+2}{n-2}}u_{\alpha, \beta}
+
O(\Vert v \Vert^{2})
$
\end{enumerate}
and more precisely for $u\in V(p, \eps)$
\begin{equation*}\begin{split}
\int  L_{g_{0}}\phi_{k,i}v 
= 
o
(
\frac{1}{\lambda_{i}^{n-2}} + \sum_{i\neq j=1}^{p}\eps_{i,j})
+
O
(
\frac{\vert \nabla K_{i}\vert^{2}}{\lambda_{i}^{2}}
+
\Vert v \Vert^{2}
).
\end{split}\end{equation*}

\end{lemma}
\noindent
We use $K_{i}$ as a short hand notation for $K(a_{i})$, $\nabla K_{i}$ for $\nabla K(a_{i})$ etc. 
\begin{proof}[\textbf{Proof of lemma \ref{lem_v_type_interactions}}]\label{p_v_type_interactions}$_{}$\\
We first calculate the bubble type interactions. Recall
\begin{equation}\begin{split}
\phi_{k,i}=d_{k,i}\varphi_{i},
\text{ where }\;
(d_{k,i})_{k=1,2,3}=(1,-\lambda_{i}\partial_{\lambda_{i}}, \frac{1}{\lambda_{i}}\nabla_{a_{i}}).
\end{split}\end{equation}
By lemma \ref{lem_emergence_of_the_regular_part}  one obtains
\begin{equation}\begin{split}
\int L_{g_{0}}\phi_{k,i}v
= &
\int d_{k,i}L_{g_{0}}\varphi_{i}v \\
= & 
4n(n-1)\int_{B_{\alpha}( a _{i})} d_{k,i}\varphi_{i}^{\frac{n+2}{n-2}}v 
+
o
(
\frac{1}{\lambda_{i}^{n-2}})+O(\Vert v \Vert^{2}), 
\end{split}\end{equation}
whence with $c_{k}>0$
\begin{equation}\begin{split}
\int L_{g_{0}}\phi_{k,i}v
= & 
c_{k}\int \varphi_{i}^{\frac{4}{n-2}} \phi_{k,i}v
+
o( \frac{1}{\lambda_{i}^{n-2}} )+O(\Vert v \Vert^{2}).
\end{split}\end{equation}
Moreover we have  
\begin{equation}\begin{split}
\int (K-K_{i}) \varphi_{i}^{\frac{4}{n-2}} \phi_{k,i} v 
= &
o(\frac{1}{\lambda_{i}^{n-2}})
+
O(\frac{\vert \nabla K_{i}\vert^{2} }{\lambda_{i}^{2}} +\Vert v \Vert^{2})
\end{split}\end{equation}
and thus 
\begin{equation}\begin{split}\label{L_{g_{0}}Phiki_Phikideltai^(...)}
\int L_{g_{0}}\phi_{k,i}v
= 
c_{k}\int\frac{K}{K_{i}}\varphi_{i}^{\frac{4}{n-2}}\phi_{k,i}v
& +
o( \frac{1}{\lambda_{i}^{n-2}} ) +
O
(\frac{\vert \nabla K_{i}\vert^{2}}{\lambda_{i}^{2}}+\Vert v \Vert^{2}).
\end{split}\end{equation}
Expanding $u^{\frac{4}{n-2}}=(\alpha^{j}\varphi_{j}+v)^{\frac{4}{n-2}}$ in case $u\in V(p, \eps)$ we have
\begin{equation}\begin{split}
0
= &
\int Ku^{\frac{4}{n-2}}\phi_{k,i}v 
= 
\int_{[\alpha^{j}\varphi_{j}\geq v]} K(\alpha^{j}\varphi_{j})^{\frac{4}{n-2}} \phi_{k,i}v
+
O(\Vert v \Vert^{2}),
\end{split}\end{equation}
whence 
\begin{equation}
\begin{split}
\int K(\alpha^{j}\varphi_{j})^{\frac{4}{n-2}}\phi_{k,i}v
=
O(\Vert v \Vert^{2}).
\end{split}
\end{equation} 
Thus we obtain, since $\vert \phi_{k,i}\vert \leq C\varphi_{i}$,
\begin{equation}\begin{split}
O(\Vert v \Vert^{2})
= &
\int K(\alpha^{j}\varphi_{j})^{\frac{4}{n-2}}\phi_{k,i}v
 \\
= &
\underset{[\alpha_{i}\varphi_{i}\geq \sum_{i\neq j=1}^{p}\alpha_{j}\varphi_{j}]}{\int}
K(\alpha_{i}\varphi_{i}+ \sum_{i\neq j=1}^{p}\alpha_{j}\varphi_{j})^{\frac{4}{n-2}}\phi_{k,i} v\\
& +
\underset{[\alpha_{i}\varphi_{i}< \sum_{i\neq j=1}^{p}\alpha_{j}\varphi_{j}]}{\int} 
K(\alpha_{i}\varphi_{i}+\sum_{i\neq j=1}^{p}\alpha_{j}\varphi_{j})^{\frac{4}{n-2}}\phi_{k,i}v
\\
= &
\underset{[\alpha_{i}\varphi_{i}\geq \sum_{i\neq j=1}^{p}\alpha_{j}\varphi_{j}]}{\int}
K(\alpha_{i}\varphi_{i})^{\frac{4}{n-2}}\phi_{k,i} v 
+
O(
\sum_{i\neq j=1}^{p}\int\varphi_{j}^{\frac{4}{n-2}}\varphi_{i}\vert v \vert 
) \\
= &
\int 
K(\alpha_{i}\varphi_{i})^{\frac{4}{n-2}}\phi_{k,i} v 
+
O(
\sum_{i\neq j=1}^{p}\int\varphi_{j}^{\frac{4}{n-2}}\varphi_{i}\vert v \vert 
).
\end{split}\end{equation}
Using lemma \ref{lem_interactions} we have 
$
\Vert \varphi_{j}^{\frac{4}{n-2}}\varphi_{i}\Vert_{L^{\frac{2n}{n+2}}}=O(\varepsilon_{i,j})
$
for $i\neq j$. This gives
\begin{equation}
\begin{split}
\int 
K(\alpha_{i}\varphi_{i})^{\frac{4}{n-2}} \phi_{k,i} v 
= &
o(\eps_{i,j})+O(\Vert v \Vert^{2})
\end{split} 
\end{equation} 
Plugging this into \eqref{L_{g_{0}}Phiki_Phikideltai^(...)} we conclude
\begin{equation}\begin{split}
\int  L_{g_{0}}\phi_{k,i}v 
= &
o( \frac{1}{\lambda_{i}^{n-2}} +\sum_{i\neq j=1}^{p}\eps_{i,j}) 
+
O
(\frac{\vert \nabla K_{i}\vert^{2}}{\lambda_{i}^{2}}+\Vert v \Vert^{2})
.
\end{split}\end{equation}
Expanding $u^{\frac{4}{n-2}}=(u_{\alpha, \beta} +\alpha^{i}\varphi_{i}+v)^{\frac{4}{n-2}}$ in case $u\in V(\omega, p, \eps)$ we have
\begin{equation}
\begin{split}
0
= &
\int Ku^{\frac{4}{n-2}}\phi_{k,i}v
=
\int K(u_{\alpha, \beta}+\alpha^{j}\varphi_{j}+v)^{\frac{4}{n-2}} \phi_{k,i} v\\
= &
\int_{[u_{\alpha, \beta} +\alpha^{j}\varphi_{j}\geq v]} 
K(u_{\alpha, \beta} +\alpha^{j}\varphi_{j})^{\frac{4}{n-2}} \phi_{k,i} v 
+
O(\Vert v \Vert^{2}) 
\end{split} 
\end{equation} 
and thus 
\begin{equation}
\begin{split}
O(\Vert v \Vert^{2})
= &
\int K(u_{\alpha, \beta}+\alpha^{j}\varphi_{j})^{\frac{4}{n-2}}\phi_{k,i}v \\
= &
\int_{[\varphi_{i}\geq u_{\alpha, \beta}+\sum_{i\neq j=1}^{p}\varphi_{j} ]} 
K(u_{\alpha, \beta}+\alpha^{j}\varphi_{j})^{\frac{4}{n-2}}\phi_{k,i} v \\
& +
\int_{[\varphi_{i}< u_{\alpha, \beta}+\sum_{i\neq j=1}^{p}\varphi_{j} ]} 
K(u_{\alpha, \beta} +\alpha^{j}\varphi_{j})^{\frac{4}{n-2}}\phi_{k,i} v \\
= &
\int_{[\varphi_{i}\geq u_{\alpha, \beta}+\sum_{i\neq j=1}^{p}\varphi_{j} ]}  
K(\alpha_{i}\varphi_{i})^{\frac{4}{n-2}}\phi_{k,i} v \\
& +
O
(\int_{[\varphi_{i}\geq u_{\alpha, \beta}+\sum_{i\neq j=1}^{p}\varphi_{j} ]}  
\varphi_{i}^{\frac{4}{n-2}}(u_{\alpha, \beta}+\sum_{i\neq j=1}^{p}\varphi_{j}) \vert v \vert  
\\ &
\quad\quad\quad +
\int_{[\varphi_{i}< u_{\alpha, \beta}+\sum_{i\neq j=1}^{p}\varphi_{j} ]}  
(u_{\alpha, \beta}+\sum_{i\neq j=1}^{p}\varphi_{j})^{\frac{4}{n-2}}\varphi_{i}\vert v \vert 
).
\end{split}
\end{equation} 
This gives
\begin{equation}
\begin{split}
\int  &
K  (\alpha_{i}  \varphi_{i})^{\frac{4}{n-2}}\phi_{k,i} v \\
= &
O
(
\int_{[\varphi_{i}\geq u_{\alpha, \beta} ]}  
\varphi_{i}^{\frac{4}{n-2}}u_{\alpha, \beta} \vert v \vert  
+
\int_{[\varphi_{i}\geq \sum_{i\neq j=1}^{p}\varphi_{j} ]}  
\varphi_{i}^{\frac{4}{n-2}}\sum_{i\neq j=1}^{p}\varphi_{j}\vert v \vert  
\\ &
\quad\; +
\int_{[\varphi_{i}< u_{\alpha, \beta}]}  
(u_{\alpha, \beta})^{\frac{4}{n-2}}\varphi_{i}\vert v \vert 
+
\int_{[\varphi_{i}< \sum_{i\neq j=1}^{p}\varphi_{j} ]}  
(\sum_{i\neq j=1}^{p}\varphi_{j})^{\frac{4}{n-2}}\varphi_{i}\vert v \vert 
),
\end{split}
\end{equation} 
whence by H\"older's inequality, direct integration and lemma \ref{lem_interactions}
\begin{equation}
\begin{split}
\int K\varphi_{i}^{\frac{4}{n-2}}\phi_{k,i} v 
= & 
o(\frac{1}{\lambda_{i}^{\frac{n-2}{2}}}+\sum_{i\neq j=1}^{p}\eps_{i,j})
+
O(\Vert v \Vert^{2}).
\end{split}
\end{equation} 
Plugging this into \eqref{L_{g_{0}}Phiki_Phikideltai^(...)} we conclude
\begin{equation}\begin{split}
\int  L_{g_{0}}\phi_{k,i}v 
= &
o( \frac{1}{\lambda_{i}^{\frac{n-2}{2}}} + \sum_{i\neq j=1}^{p}\eps_{i,j})
+
O(\Vert v \Vert^{2}).
\end{split}\end{equation}
Next we calculate for $u\in V(\omega, p, \eps)$ as before
\begin{equation}\begin{split}\label{intKu^ualphabetav}
0
= & 
\int 
Ku^{\frac{4}{n-2}}u_{\alpha, \beta}v
=
\int 
K(u_{\alpha, \beta} +\alpha^{i}\varphi_{i})^{\frac{4}{n-2}}u_{\alpha, \beta} v +O(\Vert v \Vert^{2})\\
= &
\int K u_{\alpha, \beta}^{\frac{n+2}{n-2}}v \\
& +
O
(
\int_{u_{\alpha, \beta}\geq \alpha^{i}\varphi_{i}}u_{\alpha, \beta}^{\frac{4}{n-2}}\varphi_{i}\vert v\vert
+
\int_{u_{\alpha, \beta}<\alpha^{i}\varphi_{i}}(\alpha^{i}\varphi_{i})^{\frac{4}{n-2}}u_{\alpha, \beta}\vert v \vert
+
\Vert v \Vert^{2})\\
= &
\int K u_{\alpha, \beta}^{\frac{n+2}{n-2}}v 
+
o(\sum_{r}\frac{1}{\lambda_{r}^{\frac{n-2}{2}}})
+
O(\Vert v \Vert^{2}),
\end{split}\end{equation}
whence due to \eqref{intKu^ualphabetav} and $\Pi \nabla J(u_{\alpha, \beta})=0$, cf. the remark on lemma
\ref{lem_degeneracy_and_pseudo_critical_points}
\begin{equation}\begin{split}
\int L_{g_{0}}u_{\alpha, \beta}v
= &
\int (L_{g_{0}}u_{\alpha, \beta}-(r\K)_{u_{\alpha, \beta}}u_{\alpha, \beta}^{\frac{n+2}{n-2}})v \\
& +
o(\sum_{r}\frac{1}{\lambda_{r}^{\frac{n-2}{2}}})
+
O(\Vert v \Vert^{2})
\\
= &
\int (L_{g_{0}}u_{\alpha, \beta}-(r\K)_{u_{\alpha, \beta}}u_{\alpha, \beta}^{\frac{n+2}{n-2}})
\frac{\omega}{\Vert \omega \Vert}
\int L_{g_{0}}\frac{\omega}{\Vert \omega \Vert} v  \\
& +
\sum_{i=1}^{m}\int (L_{g_{0}}u_{\alpha, \beta}-(r\K)_{u_{\alpha, \beta}}u_{\alpha, \beta}^{\frac{n+2}{n-2}})\mathrm{e}_{i}
\int L_{g_{0}}\mathrm{e}_{i}v \\
& 
+
o(\sum_{r}\frac{1}{\lambda_{r}^{\frac{n-2}{2}}})
+
O(\Vert v \Vert^{2}).
\end{split}\end{equation}
This gives
\begin{equation}\begin{split}
\int L_{g_{0}}u_{\alpha, \beta}v
= &
\int (L_{g_{0}}u-(r\K)_{u_{\alpha, \beta}}u^{\frac{n+2}{n-2}})\omega\int L_{g_{0}}\omega v  \\
& +
\sum_{i=1}^{m}\int (L_{g_{0}}u-(r\K)_{u_{\alpha, \beta}}u^{\frac{n+2}{n-2}})\mathrm{e}_{i}
\int L_{g_{0}}\mathrm{e}_{i}v \\
& 
+
o(\sum_{r}\frac{1}{\lambda_{r}^{\frac{n-2}{2}}})
+
O(\Vert v \Vert^{2}) \\
= &
O(\vert (\frac{r}{k})_{u_{\alpha, \beta}}-(\frac{r}{k})_{u}\vert^{2}) \\
& +
o(\sum_{r}\frac{1}{\lambda_{r}^{\frac{n-2}{2}}})
+
O(\Vert v \Vert^{2}+\vert \delta J(u)\vert^{2}).
\end{split}\end{equation}
Note, that
\begin{equation}
\begin{split}
\int (L_{g_{0}}u-(r\K)_{u} u^{\frac{n+2}{n-2}})u_{\alpha, \beta}
= &
\int L_{g_{0}}u_{\alpha, \beta}u_{\alpha, \beta}
-
(r\K)_{u}u_{\alpha, \beta}^{\frac{2n}{n-2}} \\
& +
O
(
\sum_{r}\frac{1 }{\lambda_{r}^{\frac{n-2}{2}}}
+
\Vert v \Vert 
),
\end{split}
\end{equation}
whence as a rough estimate
\begin{equation} \label{r/kuab-r/k_rough_estimate}
\begin{split}
(\frac{r}{k})_{u_{\alpha, \beta}}-(\frac{r}{k})_{u}
= &
O
(
\sum_{r}\frac{1 }{\lambda_{r}^{\frac{n-2}{2}}}
+
\Vert v \Vert 
+
\vert \delta J(u)\vert
).
\end{split}
\end{equation}
This proves 
\begin{equation}\begin{split}
\int L_{g_{0}}u_{\alpha, \beta}v
=
o(\sum_{r}\frac{1}{\lambda_{r}^{\frac{n-2}{2}}})
+
O(\Vert v \Vert^{2}+\vert \delta J(u)\vert^{2}).
\end{split}\end{equation}
Moreover for $u\in V(\omega, p, \eps)$ 
\begin{equation}\begin{split}
\int Ku^{\frac{n+2}{n-2}}\phi_{k,i}
= &
\int K (u_{\alpha, \beta}+\alpha^{j}\varphi_{j})^{\frac{n+2}{n-2}}\phi_{k,i} \\
& +
\frac{n+2}{n-2}\int K(u_{\alpha, \beta}+\alpha^{j}\varphi_{j})^{\frac{4}{n-2}}\phi_{k,i}v 
+
O(\Vert v \Vert^{2})
\end{split}\end{equation}
and we simply estimate
\begin{equation}\begin{split}
0= &
\int Ku^{\frac{4}{n-2}}\phi_{k,i}v
=
\int K(u_{\alpha, \beta}+\alpha^{j}\varphi_{j})^{\frac{4}{n-2}}\phi_{k,i}v
+
O(\Vert v \Vert^{2}).
\end{split}\end{equation}
\end{proof}

\subsection{Convergence versus critical points at infinity}
\label{subsec:ConvergenceVersusCritical}

Due to the Lojasiewicz inequality one has along  a flow line either convergence or a time sequence blowing up.
\begin{proposition}[Unicity of a limiting critical point]\label{prop_unicity_of_a_limiting_critical_point}$_{}$\\
If a sequence 
$
u(t_{k})
$
converges in $L^{\frac{2n}{n-2}}$ to a critical point $u_{\infty}$ of $J$, then 
\begin{equation*}
\begin{split}
u\- u_{\infty} \; \text{ in }\; C^{\infty} \; \text{ as }\; t\- \infty
\end{split}
\end{equation*}  
with at least polynomial, but generically exponential convergence rate in $C^{k, \alpha}$.
\end{proposition}
More precisely genericity arises from the fact, that generically the second variation is  non degenerate,  cf. lemma \ref{lem_spectral_theorem_and_degeneracy}, and exponential speed of convergence holds true, whenever the limiting critical point is  non degenerate.

In particular the proposition  implies, that in order to show flow convergence
we have to exclude the case of blow up, so we may assume the latter case arguing by contradiction.
\begin{proof}[\textbf{Proof of proposition \ref{prop_unicity_of_a_limiting_critical_point}}]
(\cite{BrendleArbitraryEnergies}, proposition 2.6)$_{}$\\
Suppose $\Vert u(\tau_{l})- \omega\Vert_{L^{\frac{2n}{n-2}}}\- 0$ as $\tau_{l}\nearrow \infty$, but 
$
\Vert u- \omega\Vert_{L^{\frac{2n}{n-2}}} \not \hspace{-4pt}\- 0
$
as $t\- \infty$.

For $\eps_{0}>0$ small we then find a decomposition
\begin{equation}
\begin{split}
a_{1}<b_{1}< a_{2}<b_{2}<\ldots<b_{m-1}< a_{m}<b_{m}< a_{m+1}<\ldots
\end{split}
\end{equation} 

such, that
\begin{equation}
\begin{split}
\sum_{m}(a_{m},b_{m})=\{t>0\mid \Vert u-\omega\Vert_{L^{\frac{2n}{n-2}}}<\eps_{0}\} 
\end{split}
\end{equation} 

and for a subsequence $\tau_{l}\in (a_{m_{l}},b_{m_{l}})$. 
\begin{equation}\label{aml-bml_estimate}
\begin{split}
\Vert u(b_{m_{l}}) & -u(\tau_{l})\Vert_{L^{\frac{2n}{n-2}}}^{\frac{n}{n-2}}
= 
(\int \vert u(b_{m_{l}})-u(\tau_{l})\vert^{\frac{2n}{n-2}})^{\frac{1}{2}} \\
\leq &
c (\int \vert u^{\frac{n}{n-2}}(b_{m_{l}})-u^{\frac{n}{n-2}}(\tau_{l})\vert^{2})^{\frac{1}{2}} 
=
c\Vert u^{\frac{n}{n-2}}(b_{m_{l}})-u^{\frac{n}{n-2}}(\tau_{l})\Vert_{L^{2}}
\\
\leq  &
c\int^{b_{m_{l}}}_{\tau_{l}}\Vert \partial_{t}u^{\frac{n}{n-2}}\Vert_{L^{2}}
\leq 
c\int^{b_{m_{l}}}_{a_{m_{l}}}\vert \delta J(u)\vert,
\end{split}
\end{equation}

whence according to proposition \ref{prop_strong_convergence_of_the_first_variation} we may assume
\begin{equation}
\begin{split}
b_{m_{l}}-a_{m_{l}}\-\infty. 
\end{split}
\end{equation} 

Passing to a subsequence we thus may inductively  decompose 
\begin{equation}
\begin{split}
[a_{m_{l_{1}}},b_{m_{l_{1}}})=\sum_{k=1}^{m_{1}}[s_{k},t_{k})
,
\; 2^{k}\leq t_{k}-s_{k}<c2^{k+1}, \, c\in [1,3)
\end{split}
\end{equation} 

and 
\begin{equation}
\begin{split}
[a_{m_{l_{2}}},b_{m_{l_{2}}})
=\sum_{k=m_{1}+1}^{m_{2}}[s_{k},t_{k})
,
\; 2^{k}\leq t_{k}-s_{k}<c2^{k+1}, \, c\in [1,3)
\end{split}
\end{equation} 

and so on.\\
By analyticity of $J$ we may use the Lojasiewicz inequality
\begin{equation}\label{Lojasiewicz}
\begin{split}
\e C>0, \gamma \in (0,1]\fa u\in B_{\eps_{0}}(\omega )\;:\;\vert J(u)-J(\omega )\vert \leq C\Vert \partial J(u)\Vert^{1+\gamma},
\end{split}
\end{equation} 
cf. \cite{Lojasiewicz}, Theorem 4.1. Clearly $J(\omega )=J_{\infty}=r_{\infty}$ and along a flow line we have
\begin{equation}
\begin{split}
\Vert \partial J(u)\Vert \leq C \vert \delta J(u)\vert.
\end{split}
\end{equation} 
Thus for $t\in (s_{k},t_{k})$
\begin{equation}
\begin{split}
\partial_{t}J(u)
\leq &
-c\vert \delta J(u)\vert^{2}
\leq 
-C (J(u)-J_{\infty})^{\frac{2}{\gamma+1}}.
\end{split}
\end{equation} 
Without loss of generality $\gamma<1$, whence 
$\partial_{t} (J(u)-J_{\infty})^{\frac{\gamma-1}{\gamma+1}} \geq c$
and
\begin{equation}
\begin{split}
(J(u(t_{k}))-J_{\infty})^{\frac{\gamma-1}{\gamma+1}}
\geq &
(J(u(s_{k}))-J_{\infty})^{\frac{\gamma-1}{\gamma+1}}
+
c (t_{k}-s_{k})
\end{split}
\end{equation} 
and in particular
$
J(u(t_{k}))-J_{\infty}
\leq 
c(t_{k}-s_{k})^{\frac{\gamma+1}{\gamma-1}}.
$
We conclude
\begin{equation}
\begin{split}
(\int^{t_{k}}_{s_{k}}&\vert \delta J(u)\vert)^{2} \\
\leq &
(t_{k}-s_{k})\int^{t_{k}}_{s_{k}}\vert \delta J(u)\vert^{2} 
\leq 
c(t_{k}-s_{k})(J(u(s_{k}))-J(u(t_{k})) \\
\leq &
c(t_{k}-s_{k})(J(u(s_{k}))-J_{\infty})
\leq 
c(t_{k}-s_{k})(J(u(t_{k-1}))-J_{\infty}) \\
\leq &
c(t_{k}-s_{k})(t_{k-1}-s_{k-1})^{\frac{\gamma+1}{\gamma-1}} 
\leq 
c_{2}2^{k+1}(2^{k-1})^{\frac{\gamma+1}{\gamma-1}}
\leq 
c
(2^{\frac{2\gamma}{\gamma-1}})^{k-1}
\end{split}
\end{equation} 
having used Jensen's inequality. Consequently 
\begin{equation}
\begin{split}
\sum_{m_{l}}\int_{a_{m_{l}}}^{b_{m_{l}}}\vert \delta J(u)\vert
=
\sum_{k}\int_{s_{k}}^{t_{k}}\vert \delta J(u)\vert<\infty,
\end{split}
\end{equation} 
whence $\lim_{l\to\infty}\int_{a_{m_{l}}}^{b_{m_{l}}}\vert \delta J(u)\vert =0$.
This contradicts \eqref{aml-bml_estimate} and we conclude 
\begin{equation}
\begin{split}
u\- \omega\;\text{ in }\; L^{\frac{2n}{n-2}}\; \text{ as }\; t \-\infty.
\end{split}
\end{equation}  

Now let $x_{0}\in M$. Then
$
\Vert R \Vert_{L^{\frac{n}{2}}_{\mu }(B_{r}(x_{0}))}=o(r)
$
by proposition \ref{prop_strong_convergence_of_the_first_variation}, whence
\begin{equation}
\begin{split}
L_{g_{0}}u =R u ^{\frac{n+2}{n-2}}=P u\; \text{ with }\;
\Vert P \Vert_{L_{g_{0}}^{\frac{n}{2}}(B_{r}(x_{0}))}
=
o(r).
\end{split}\end{equation}

Lemma \ref{App3} then shows
$
\sup_{t\geq 0}\Vert u\Vert_{L_{g_{0}}^{p}}<\infty 
$
for all $p\geq 1$ and due to 
\begin{equation}
\begin{split}
-c_{n}\lap_{g_{0}} u 
= &
(R -r \K )u ^{\frac{n+2}{n-2}}+r \K u^{\frac{n+2}{n-2}} -R_{g_{0}}u
\end{split}\end{equation}

and proposition \ref{prop_strong_convergence_of_the_first_variation} it follows, that
$(-\lap u )\subset L^{p}$ and applying Calderon-

Zygmund estimates, 
that $(u )\subset W^{2,p}\hookrightarrow L^{\infty}$ is uniformly bounded. 

Then lemma \ref{lem_bounding_u_from_below} shows
$
0<c<u<C<\infty. 
$
Due to proposition \ref{prop_strong_convergence_of_the_first_variation} we 

have $\int \vert R-r\K\vert^{p} d\mu\-0$ for all $p\geq 1$. 
With this at hand one may 

repeat the arguments proving proposition \ref{prop_time_dependend_holder_regularity} to show 
\begin{equation}\begin{split}
\vert u(x_{1},t_{1}) -u(x_{2},t_{2}) \vert
\leq
C(\alpha)
(
\vert t_{1}-t_{2}\vert^{\frac{\alpha}{2}}
+
d(x_{1},x_{2})^{\alpha}),
\end{split}\end{equation}

for all $x_{1},x_{2} \in M$ and $0 \leq t_{1},t_{2}<\infty, \, \vert t_{1}-t_{2} \vert \leq 1$, where \begin{equation}
\begin{split}
0<\alpha<\min\{\frac{4}{n},1\}.
\end{split}
\end{equation} 

By standard regularity arguments then 
$(u)\subset C^{k, \alpha}$
is uniformly bounded. \\
As for the speed of convergence note, that as before we have
\begin{equation}
\begin{split}
\partial_{t}(J(u)-J_{\infty})^{\frac{\gamma-1}{\gamma+1}}\geq c.
\end{split}
\end{equation} 
From this we obtain polynomial convergence of $J(u)$, namely 
\begin{equation}
\begin{split}
0<J(u)-J_{\infty}<\frac{C}{(1+t)^{\frac{1+\gamma}{1-\gamma}}}.
\end{split}
\end{equation} 
Moreover
\begin{equation}
\begin{split}
\partial_{t}\Vert u^{\frac{n}{n-2}}-\omega^{\frac{n}{n-2}}\Vert_{L^{2}}
\leq &
c\vert \delta J(u)\vert
\end{split}
\end{equation} 
and applying once more the Lojasiewicz inequality \eqref{Lojasiewicz}
\begin{equation}
\begin{split}
\partial_{t}(J(u) & -J_{\infty})^{\frac{\gamma}{1+\gamma}}
\leq 
-c(J(u)-J_{\infty})^{\frac{\gamma}{1+\gamma}-1}\vert \delta J(u)\vert^{2} \\
\leq &
-c (J(u)-J_{\infty})^{-\frac{1}{1+\gamma}}\Vert \partial J(u)\Vert \vert \delta J(u)\vert 
\leq 
-c \vert \delta J(u)\vert,
\end{split}
\end{equation} 
whence 
\begin{equation}
\begin{split}
\partial_{t} \Vert u^{\frac{n}{n-2}}-\omega^{\frac{n}{n-2}}\Vert_{L^{2}}
\leq 
-C\partial_{t}(J(u)-J(_{\infty}))^{\frac{\gamma}{1+\gamma}}.
\end{split}
\end{equation}
We conclude polynomial convergence $u\- \omega$ in $L^{\frac{2n}{n-2}}$ via
\begin{equation}
\begin{split}
\Vert u-\omega\Vert_{L^{\frac{2n}{n-2}}}^{\frac{n}{n-2}}
\leq &
C\Vert u^{\frac{n}{n-2}}-\omega^{\frac{n}{n-2}}\Vert_{L^{2}}
\leq 
C(J(u)-J(\infty))^{\frac{\gamma}{1+\gamma}}\\
\leq &
\frac{C}{(1+t)^{\frac{\gamma}{1-\gamma}}}.
\end{split}
\end{equation} 
With uniform boundedness at hand we may use Sobolev space interpolation
\begin{equation}
\begin{split}
\Vert v \Vert_{W^{k,p}}\leq C(k,p)\Vert v \Vert_{W^{k-1,p}}^{\frac{1}{2}}\Vert v \Vert_{W^{k+1,p}}^{\frac{1}{2}}
\end{split}
\end{equation} 
to conclude  polynomial convergence at least in each Sobolev or H\"older space. 

Note, that in case $\gamma=1$ we have
\begin{equation}
\begin{split}
\partial_{t}(J(u)-J_{\infty})\leq -c\vert \delta J(u)\vert^{2}\leq -C\vert J(u)-J_{\infty}\vert, 
\end{split}
\end{equation} 

whence $J(u)\searrow J_{\infty}$ with convergence at exponential rate. Moreover 
\begin{equation}
\begin{split}
\partial_{t}\Vert u^{\frac{n}{n-2}}-\omega^{\frac{n}{n-2}}\Vert_{L^{2}}\leq c\vert \delta J(u)\vert 
\end{split}
\end{equation} 

and 
\begin{equation}
\begin{split}
\partial_{t}(J(u)-J_{\infty}))^{\frac{1}{2}}
\leq 
-c(J(u)-J_{\infty})^{-\frac{1}{2}}\vert \delta J(u)\vert^{2}
\leq 
-C \vert \delta J(u)\vert.
\end{split}
\end{equation} 

By the same arguments as before we conclude $u\- \omega$ at exponential 

rate in every Sobolev or H\"older space in case $\gamma=1$. 
\\
In the generic case $E_{\frac{n+2}{n-2}}(\omega)=\emptyset$, cf. lemma \ref{lem_spectral_theorem_and_degeneracy}, however
the Lojasiewicz inequality \eqref{Lojasiewicz}
holds with optimal exponent 
$\gamma=1$. 

Indeed  $J(u)=J(\omega)$ for $u\in \langle \omega \rangle=H_{0}(\omega)$ by scaling invariance and 
\begin{equation}
\begin{split}
\vert J(u)-J(\omega)\vert \leq \vert u-\omega\vert^{2}\;\text{ and }\;\vert \delta J(u)\vert \geq c\vert u-\omega\vert
\end{split}
\end{equation} 

for $u\in \langle \omega \rangle^{\perp_{L_{g_{0}}}}=H_{0}(\omega)^{\perp_{L_{g_{0}}}}=kern(\partial^{2}J(\omega))$.
\end{proof}

\section{Case \textomega=0} 
\label{sec:w=0}
The starting point in this section is a flow line 
$ 
u\in V(p, \eps),
$
that we study by analysing the  evolution of 
the parameters $\alpha_{i}, \lambda_{i},a_{i}$ in the representation
\begin{align*}
u=\alpha^{i}\var_{i}+v=\alpha^{i}\var_{a_{i}, \lambda_{i}}+v
\end{align*}
given by proposition \ref{prop_optimal_choice}. To that end we test the flow equation
\begin{align*}
\partial_{t}u=-\frac{1}{K}(R-r\K)
\end{align*}
with $\var_{i}, \lambda_{i}\partial_{\lambda_{i}}\var_{i}$ and $\frac{1}{\lambda_{i}}\nabla_{a_{i}}\var_{i}$, cf. definition \ref{def_relevant_quantities}.
\begin{lemma}[The shadow flow]\label{lem_the_shadow_flow}$_{}$\\
For $u\in V(p, \eps)$ with $\eps>0$ and 
$$ \sigma_{k,i}=-\int (L_{g_{0}}u-r\K u^{\frac{n+2}{n-2}})\phi_{k,i}, \;i= 1, \ldots,p, \; k=1,2,3 $$
we  have
by testing $K\partial_{t}u=-(R-r\K)u$ with $u^{\frac{4}{n-2}}\phi_{k,i}$
\begin{enumerate}[label=(\roman*)]
 \Item $_{}$
\begin{equation*}\begin{split}
\frac{\dot \alpha_{i}}{\alpha_{i}}
= 
\frac{\alpha_{i}^{\frac{n+2}{2-n}}}{c_{1}K_{i}}\sigma_{1,i}
(1+o_{\frac{1}{\lambda_{i}}}(1))
+
R_{1,i}
\end{split}\end{equation*}
 \Item $_{}$
\begin{equation*}\begin{split}
-\frac{\dot \lambda_{i}}{\lambda_{i}}
= 
\frac{\alpha_{i}^{\frac{n+2}{2-n}}}{c_{2}K_{i}}\sigma_{2,i}
(1+o_{\frac{1}{\lambda_{i}}}(1))
+
R_{2,i}
\end{split}\end{equation*}
 \Item $_{}$
\begin{equation*}\begin{split}
\lambda_{i}\dot a_{i}
= 
\frac{\alpha_{i}^{\frac{n+2}{2-n}}}{c_{3}K_{i}}\sigma_{3,i}
(1+o_{\frac{1}{\lambda_{i}}}(1))
+
R_{3,i}
\end{split}\end{equation*}
\end{enumerate}
with constants $c_{k}>0$ given in lemma \ref{lem_interactions} and
\begin{equation*}\begin{split}
R_{k,i}
= 
O
(
\sum_{r\neq s}\eps_{r,s}^{2}
+
\Vert v \Vert^{2}
+
\vert \delta J(u)\vert^{2}
)_{k,i}
.
\end{split}\end{equation*}
\end{lemma}

\begin{proof}[\textbf{Proof of lemma \ref{lem_the_shadow_flow}}]\label{p_the_shadow_flow}$_{}$\\
For each $i,j= 1, \ldots,p, \; k=1,2,3$ let 
\begin{equation}\begin{split}
(\dot{\xi}_{1,j}, \dot{\xi}_{2,j}, \dot{\xi}_{3,j})=(\dot{\alpha}_{j},-\alpha_{j}\frac{\dot \lambda_{j}}{\lambda_{j}}, \alpha_{j}\lambda_{j}\dot a _{j})
\end{split}\end{equation}
and recall
\begin{equation}
\begin{split}
\phi_{k,i}=d_{k,i}\varphi_{i}=
(\varphi_{i},-\lambda_{i} \partial_{\lambda_{i}}\varphi_{i}, \frac{1}{\lambda_{i}} \nabla_{ a _{i}}\varphi_{i}).
\end{split}
\end{equation} 
Testing $K\partial_{t}u=-(R-r\K)u$ with $u^{\frac{4}{n-2}}\phi_{k,i}$ we obtain using $\int Ku^{\frac{4}{n-2}}\phi_{k,i}v=0$
\begin{equation}\begin{split}\label{testing_w=0}
\sigma_{k,i}
= &
\int \partial_{t}uK  u^{\frac{4}{n-2}}\phi_{k,i} 
= 
\int \partial_{t}(\alpha^{j}\varphi_{j}+v)Ku^{\frac{4}{n-2}}\phi_{k,i} \\
= &
\dot{\xi}^{l,j}\int Ku^{\frac{4}{n-2}}\phi_{l,j}\phi_{k,i}
-
\int Kv[\partial_{t}u^{\frac{4}{n-2}}\phi_{k,i}+u^{\frac{4}{n-2}}\partial_{t}\phi_{k,i}].
\end{split}\end{equation}
Note, that 
\begin{equation}\begin{split}\label{expansion_interaction_w=0}
\int Ku^{\frac{4}{n-2}} & \phi_{l,j}\phi_{k,i}
= 
\int K(\alpha^{m}\varphi_{m})^{\frac{4}{n-2}}\phi_{l,j}\phi_{k,i}
+
O(\Vert v\Vert)_{k,i,l,j}\\
= &
c_{k}\alpha_{i}^{\frac{4}{n-2}}K_{i}\delta_{kl}\delta_{ij} 
+
O(\frac{\vert \nabla K_{i}\vert}{\lambda_{i}}+\frac{1}{\lambda_{i}^{2}}
+
\frac{1}{\lambda_{i}^{n-2}} 
)_{k,l}\delta_{ij}
\\
& +
O(
\sum_{i\neq m=1}^{p}\eps_{i,m} +\Vert v\Vert)
_{k,i,l,j}.
\end{split}\end{equation}
Indeed
\begin{equation}
\begin{split}
\int  K & (\alpha^{m}  \varphi_{m})^{\frac{4}{n-2}}\phi_{l,j}\phi_{k,i} \\
= &
\underset{[\varphi_{i}\geq \sum_{i\neq m=1}^{p}\varphi_{m}]}{\int} 
K(\alpha_{i}\varphi_{i})^{\frac{4}{n-2}}\phi_{l,j}\phi_{k,i} \\
& +
\sum_{i\neq m=1}^{p}
O
(
\underset{[\varphi_{i}\geq \sum_{i\neq m=1}^{p}\varphi_{m}]}{\int} 
\varphi_{i}^{\frac{4}{n-2}}\varphi_{j}\varphi_{m} 
+
\underset{[\varphi_{i}<\sum_{i\neq m=1}^{p}\varphi_{m}]}{\int} 
\varphi_{m}^{\frac{4}{n-2}}\varphi_{j}\varphi_{i} 
),
\end{split}
\end{equation} 
whence by means of lemma \ref{lem_interactions} we have
\begin{equation}\label{alpha^m_delta_m_Phi_k,i_Phi_l,j}
\begin{split}
\int  K  (\alpha^{m}  &\varphi_{m})^{\frac{4}{n-2}}\phi_{l,j}\phi_{k,i} \\
= &
\underset{[\varphi_{i}\geq \sum_{i\neq m=1}^{p}\varphi_{m}]}{\int} 
K(\alpha_{i}\varphi_{i})^{\frac{4}{n-2}}\phi_{l,j}\phi_{k,i} 
+
O
(
\sum_{i\neq m=1}^{p}\eps_{i,m}
) \\
= &
\int  
K(\alpha_{i}\varphi_{i})^{\frac{4}{n-2}}\phi_{l,j}\phi_{k,i} \\
& +
O
(
\underset{[\varphi_{i}< \sum_{i\neq m=1}^{p}\varphi_{m}]}{\int} 
\varphi_{i}^{\frac{n+2}{n-2}}\varphi_{j}
+
\sum_{i\neq m=1}^{p}\eps_{i,m}
)\\
= &
\alpha_{i}^{\frac{4}{n-2}}
\int  
K\varphi_{i}^{\frac{4}{n-2}}\phi_{l,j}\phi_{k,i} 
+
O
(
\sum_{i\neq m=1}^{p}\eps_{i,m}
) \\
= &
\alpha_{i}^{\frac{4}{n-2}}\delta_{ij}
\int  
K\varphi_{i}^{\frac{4}{n-2}}\phi_{l,i}\phi_{k,i} 
+
O
(
\sum_{i\neq m=1}^{p}\eps_{i,m}
) \\
= &
\alpha_{i}^{\frac{4}{n-2}}\delta_{ij}\delta_{kl}
\int  
K\varphi_{i}^{\frac{4}{n-2}}\phi^{2}_{k,i} 
+
O
(
\frac{1}{\lambda_{i}^{2}}+\frac{1}{\lambda_{i}^{n-2}}
)\delta_{ij}
+
O
(
\sum_{i\neq m=1}^{p}\eps_{i,m}
).
\end{split}
\end{equation} 
From this \eqref{expansion_interaction_w=0} follows.
Moreover we may write 
\begin{equation}\label{partial_t_phi_ki_estimate_w=0}
\begin{split}
\int Ku^{\frac{4}{n-2}}\partial_{t}\phi_{k,i}v
=
O(\Vert v \Vert)_{i,k,l,j}\dot \xi^{l,j}
\end{split}
\end{equation} 
using 
$\vert \partial _{\alpha}\phi_{k,i}\vert, \vert \lambda_{i}\partial_{\lambda_{i}}\phi_{k,i}\vert, \vert \frac{1}{\lambda_{i}}\nabla_{a_{i}}\phi_{k,i}\vert \leq C\varphi_{i}
$
and estimate 
\begin{equation}\begin{split}
\vert 
\int Kv\partial_{t}& u^{\frac{4}{n-2}}\Phi_{k,i}
\vert
= 
\frac{4}{n-2}\vert 
\int v ( R-r\K)  u^{\frac{4}{n-2}} \phi_{k,i}
\vert \\
\leq & 
C\int  \vert R-r\K \vert u^{\frac{4}{n-2}}\varphi_{i}\vert v \vert 
= 
C\int  \vert R-r\K \vert u^{\frac{4}{n-2}}\vert u-v\vert\vert v \vert \\
\leq &
C\int \vert R-r\K \vert u^{\frac{n+2}{n-2}}\vert v \vert
+
C\int \vert R-r\K \vert u^{\frac{4}{n-2}}\vert v \vert^{2} \\
\leq &
C(\Vert R-r\K \Vert_{L^{\frac{2n}{n+2}}_{\mu}}\Vert v \Vert+\Vert R-r\K \Vert_{L^{\frac{n}{2}}_{\mu}}\Vert v \Vert^{2})
\end{split}\end{equation}
using $\vert \phi_{k,i}\vert \leq C\varphi_{i}$, whence according to proposition \ref{prop_strong_convergence_of_the_first_variation} we obtain
\begin{equation}\begin{split}\label{trick_with_strong_convergence_w=0}
\int Kv\partial_{t}& u^{\frac{4}{n-2}}\Phi_{k,i}
=
O(\vert \delta J(u)\vert^{2}+\Vert v \Vert^{2}).
\end{split}\end{equation}
Thus plugging \eqref{expansion_interaction_w=0}, \eqref{partial_t_phi_ki_estimate_w=0} and \eqref{trick_with_strong_convergence_w=0} into \eqref{testing_w=0} we obtain for
\begin{equation}
\begin{split}
\Xi_{k,i,l,j}
= &
c_{k}\alpha_{i}^{\frac{4}{n-2}}K_{i}\delta_{kl}\delta_{ij} \\
& +
O(
\frac{\vert \nabla K_{i}\vert}{\lambda_{i}}+\frac{1}{\lambda_{i}^{2}}
+
\frac{1}{\lambda_{i}^{n-2}} 
)_{k,l}\delta_{ij}
+
O
(
\sum_{i\neq m=1}^{p}\eps_{i,m} +\Vert v\Vert)
_{k,i,l,j}
\end{split}
\end{equation} 
the identity
\begin{equation}
\begin{split}
\Xi_{k,i,l,j}\dot\xi^{l,j} 
=
\sigma_{k,i}
+
O(\Vert v \Vert^{2} + \vert \delta J(u)\vert^{2})_{k,i}.
\end{split}
\end{equation} 
For the inverse $\Xi^{-1}$ of $\Xi$ we then have
\begin{equation}
\begin{split}
\Xi_{k,i,l,j}^{-1}
= &
\frac{\alpha^{\frac{4}{2-n}}}{c_{k}K_{i}}\delta_{kl}\delta_{ij} \\
& +
O(
\frac{\vert \nabla K_{i}\vert}{\lambda_{i}}+\frac{1}{\lambda_{i}^{2}}
+
\frac{1}{\lambda_{i}^{n-2}} 
)_{k,l}\delta_{ij}
+
O
(
\sum_{i\neq m=1}^{p}\eps_{i,m} +\Vert v\Vert)
_{k,i,l,j}
\end{split}
\end{equation}
and the claim follows, since by definition $\sigma_{k,i}=O(\vert \delta J(u)\vert)$.
\end{proof}
Consequently our task is two folded. We have to carefully evaluate 
$\sigma_{k,i}$ by expansion and 
find suitable estimates on the error term  $v$.
\begin{proposition}[Analysing $\sigma_{k,i}$]\label{prop_analysing_ski}$_{}$\\
On $V(p, \eps)$ for $\eps>0$ small we have
with constants $b_{1}, \ldots,e_{4}>0$
\begin{enumerate}[label=(\roman*)]
 \Item $_{}$
\begin{equation*}\begin{split}
\sigma_{1,i}
= &
4n(n-1)\alpha_{i}
[
\frac{r\alpha_{i}^{\frac{4}{n-2}}K_{i}}{4n(n-1)k}
-
1
]
\int \varphi_{i}^{\frac{2n}{n-2}}\\
& +
4n(n-1)\sum_{i\neq j =1}^{p}\alpha_{j}
[
\frac{r\alpha_{j}^{\frac{4}{n-2}}K_{j}}{4n(n-1)k}-1
] 
b_{1}\eps_{i,j}
 \\
& 
+
d_{1}\alpha_{i}\frac{ H_{i}}{\lambda_{i} ^{n-2}}
+
e_{1}\frac{r\alpha_{i}^{\frac{n+2}{n-2}}}{k}  \frac{\lap K_{i}}{\lambda_{i} ^{2}} 
+
b_{1}
\frac{r\alpha_{i}^{\frac{4}{n-2}}K_{i}}{k}
\sum_{i \neq j=1}^{p}\alpha_{j}
\eps_{i,j} 
+
R_{1,i}
\end{split}\end{equation*}
 \Item $_{}$
\begin{equation*}\begin{split}
\sigma_{2,i}
= &
-4n(n-1)\alpha_{i}
[
\frac{r\alpha_{i}^{\frac{4}{n-2}}K_{i}}{4n(n-1)k}
-
1
]
\int \varphi_{i}^{\frac{n+2}{n-2}}\lambda_{i}\partial_{\lambda_{i}}\varphi_{i}\\
& 
-4n(n-1)b_{2}\sum_{i\neq j =1}^{p}\alpha_{j}
[
\frac{r\alpha_{j}^{\frac{4}{n-2}}K_{j}}{4n(n-1)k}-1
] 
\lambda_{i} \partial_{\lambda_{i}}\eps_{i,j}
+
d_{2}\alpha_{i}\frac{ H_{i}}{\lambda_{i} ^{n-2}} \\
& +
e_{2}\frac{r\alpha_{i}^{\frac{n+2}{n-2}}}{k} \frac{\lap K_{i}}{\lambda_{i}^{2}}  - 
b_{2}
\frac{r\alpha_{i}^{\frac{4}{n-2}}K_{i}}{k}
\sum_{i \neq j=1}^{p}\alpha_{j}
\lambda_{i}\partial_{\lambda_{i}}\eps_{i,j} 
+
R_{2,i}
\end{split}\end{equation*}
 \Item $_{}$
\begin{equation*}\begin{split}
\sigma_{3,i}
= &
4n(n-1)\alpha_{i}
[
\frac{r\alpha_{i}^{\frac{4}{n-2}}K_{i}}{4n(n-1)k}
-
1
]
\int \varphi_{i}^{\frac{n+2}{n-2}}\frac{1}{\lambda_{i}}\nabla_{a_{i}}\varphi_{i}\\
& +
4n(n-1)b_{3}\sum_{i\neq j =1}^{p}\alpha_{j}
[
\frac{r\alpha_{j}^{\frac{4}{n-2}}K_{j}}{4n(n-1)k}-1
] 
\frac{1}{\lambda_{i}} \nabla_{ a _{i}}\eps_{i,j}
 \\
& 
+
\frac{r\alpha_{i}^{\frac{n+2}{n-2}}}{k}
[
e_{3}\frac{\nabla K_{i}}{\lambda_{i}}
+
e_{4}\frac{\nabla \lap K_{i}}{\lambda_{i}^{3}} 
] 
\\ & +
b_{3}\frac{r\alpha_{i}^{\frac{4}{n-2}}K_{i}}{k}
\sum_{i \neq j=1}^{p}
\frac{\alpha_{j}}{\lambda_{i}}\nabla_{a_{i}}\eps_{i,j} +
R_{3,i},
\end{split}\end{equation*}
\end{enumerate}
where 
$
R_{k,i}
= 
o_{\varepsilon}
(
 \frac{1}{\lambda_{i}^{n-2}} 
+
\sum_{i\neq j=1}^{p}\eps_{i,j}
)_{k,i}
+
O
(
\sum_{r\neq s}\eps_{r,s}^{2}
+
\Vert v \Vert^{2}
)_{k,i}
.                               
$
\end{proposition}

\begin{proof}[\textbf{Proof of proposition \ref{prop_analysing_ski}}]\label{p_analysing_ski}$_{}$\\
By definition and conformal invariance
\begin{equation}\begin{split}\label{sigmaki=Ki...+K-Ki}
\sigma_{k,i}
= &
-\int (L_{g_{0}}u-r\K u^{\frac{n+2}{n-2}})\phi_{k,i}
=
-\int (R-r\K )u^{\frac{n+2}{n-2}}\phi_{k,i}.
\end{split}\end{equation}
We start evaluating
\begin{equation}\begin{split}\label{intL_{g_{0}}uPhiki}
\int L_{g_{0}}u\phi_{k,i}
= &
\int L_{g_{0}}(\alpha^{j}\varphi_{j}+v)\phi_{k,i}
=
\alpha^{j}\int L_{g_{0}}\varphi_{j}\phi_{k,i}
+
\int L_{g_{0}}\phi_{k,i}v.
\end{split}\end{equation}
Using lemmata \ref{lem_emergence_of_the_regular_part} and \ref{lem_interactions}  we obtain for $\alpha >0$ small 
\begin{equation}\begin{split}
\alpha^{j}\int L_{g_{0}} & \varphi_{j}\phi_{k,i} 
= 
\alpha_{i}\int L_{g_{0}}\varphi_{i}\phi_{k,i}
+
\sum_{i\neq j =1}^{p}\alpha_{j}\int L_{g_{0}}\varphi_{j}\phi_{k,i} \\
= &
4n(n-1)\alpha_{i}\int_{B_{\alpha}( a _{i})}
\varphi_{i}^{\frac{n+2}{n-2}}
\phi_{k,i} \\
& -
2nc_{n}
\alpha_{i}\int_{B_{\alpha}( a _{i})}
( 
((n-1)H_{i}+r_{i}\partial_{r_{i}}H_{i}) r_{i}^{n-2}\varphi^{\frac{n+2}{n-2}}_{i} 
)
\phi_{k,i}
\\
& +
4n(n-1)\sum_{i\neq j =1}^{p}\alpha_{j}\int_{B_{\alpha}( a _{j})} 
\varphi_{j}^{\frac{n+2}{n-2}}
\phi_{k,i}
+
o_{\varepsilon}( \frac{1}{\lambda_{i}^{n-2}} +\sum_{i\neq j=1}^{p}\eps_{i,j})
 \\
= &
4n(n-1)\alpha_{i}\int_{B_{\alpha}( a _{i})}
\varphi_{i}^{\frac{n+2}{n-2}}
\phi_{k,i}
+
4n(n-1)b_{k}\sum_{i\neq j =1}^{p}\alpha_{j}d_{k,i}\eps_{i,j}
 \\
& -
2nc_{n}
\alpha_{i}\int_{B_{\alpha}( a _{i})}
( 
(n-1)H_{i}+r_{i}\partial_{r_{i}}H_{i}) r_{i}^{n-2}\varphi^{\frac{n+2}{n-2}} _{i}
)
\phi_{k,i}
\\
& +
o_{\varepsilon}
(
 \frac{1}{\lambda_{i}^{n-2}} 
+
\sum_{i\neq j=1}^{p}\eps_{i,j}
).
\end{split}\end{equation}
Indeed the curvature related term arising from lemma \ref{lem_emergence_of_the_regular_part} is of order
\begin{equation}
\begin{split}
\int_{B_{\alpha}(0)}\frac{r^{2}}{\lambda_{i}}(\frac{\lambda_{i}}{1+\lambda_{i}^{2}r^{2}})^{\frac{n-2}{2}(\frac{n}{n-2}+1)}
=
\lambda_{i}^{-4}O(\lambda_{i}, \ln \lambda_{i},1)=o(\frac{1}{\lambda_{i}^{n-2}}).
\end{split}
\end{equation} 
Thus
\begin{equation}\begin{split}
\alpha^{j}&\int L_{g_{0}}\varphi_{j}\phi_{k,i} \\
= &
4n(n-1)
[
\alpha_{i}
\int
\varphi_{i}^{\frac{n+2}{n-2}}
\phi_{k,i}
+
b_{k}\sum_{i\neq j =1}^{p}\alpha_{j}d_{k,i}\eps_{i,j}
] \\
& -
(n-1)(n-2)c_{n}\alpha_{i}H_{i} 
\int_{B_{\alpha}(0)}
 r^{n-2}
(1, -\lambda_{i} \partial_{\lambda_{i} }, \frac{1}{\lambda_{i}} \nabla)
(\frac{\lambda_{i} }{1+\lambda_{i}^2r^{2}})^{n}\\
& -
(n-2)c_{n}\alpha_{i}\nabla H_{i} 
\int_{B_{\alpha}(0)} 
\nabla rr^{n-1}
(1, -\lambda_{i} \partial_{\lambda_{i} }, \frac{1}{\lambda_{i}} \nabla)
(\frac{\lambda_{i} }{1+\lambda_{i}^2r^{2}})^{n}
\\
& +
o_{\varepsilon}
(
 \frac{1}{\lambda_{i}^{n-2}} 
+
\sum_{i\neq j=1}^{p}\eps_{i,j}
)
\end{split}\end{equation}
using
$
\gamma_{n}\nabla_{ a _{i}} G_{ a _{i}}^{\frac{2}{2-n}}
=
2x+O(r^{n-1}).
$
By radial symmetry we then get
\begin{equation}\begin{split}\label{IntL_{g_{0}}djPhiki_expanded}
\alpha^{j}\int L_{g_{0}}  \varphi_{j}\phi_{k,i}
= &
4n(n-1)
[
\alpha_{i}
\int
\varphi_{i}^{\frac{n+2}{n-2}}
\phi_{k,i}
+
b_{k}\sum_{i\neq j =1}^{p}\alpha_{j}d_{k,i}\eps_{i,j}
] \\
& 
-
\alpha_{i}(d_{1}\frac{ H_{i}}{\lambda_{i} ^{n-2}},d_{2}\frac{ H_{i}}{\lambda_{i} ^{n-2}},d_{3}\frac{\nabla H_{i}}{\lambda_{i} ^{n-1}}  ) \\
& +
o_{\varepsilon}
(
 \frac{1}{\lambda_{i}^{n-2}} 
+
\sum_{i\neq j=1}^{p}\eps_{i,j}
)
\end{split}\end{equation}
with  $d_{k}>0$. Inserting this into \eqref{intL_{g_{0}}uPhiki} and applying lemma  \ref{lem_v_type_interactions} gives
\begin{equation}\begin{split}\label{IntL_{g_{0}}uPhiki_expanded}
\int & L_{g_{0}}u\phi_{k,i}
= 
\int L_{g_{0}}(\alpha^{j}\varphi_{j}+v)\phi_{k,i} \\
= &
4n(n-1)
[
\alpha_{i}
\int
\varphi_{i}^{\frac{n+2}{n-2}}
\phi_{k,i}
+
b_{k}\sum_{i\neq j =1}^{p}\alpha_{j}d_{k,i}\eps_{i,j}
] \\
& 
-
\alpha_{i}(d_{1}\frac{ H_{i}}{\lambda_{i} ^{n-2}},d_{2}\frac{ H_{i}}{\lambda_{i} ^{n-2}},d_{3}\frac{\nabla H_{i}}{\lambda_{i} ^{n-1}}  ) 
+
o_{\varepsilon}( \frac{1}{\lambda_{i}^{n-2}} + \sum_{i\neq j=1}^{p}\eps_{i,j})
+
O
(\Vert v \Vert^{2}).
\end{split}\end{equation}
Next from lemma \ref{lem_v_type_interactions}  we infer
\begin{equation}\begin{split}\label{intKu^..Phiki}
\int Ku^{\frac{n+2}{n-2}}\phi_{k,i}
= &
\int K (\alpha^{j}\varphi_{j})^{\frac{n+2}{n-2}}\phi_{k,i}
+
O(\Vert v \Vert^{2}).
\end{split}\end{equation}
Clearly
\begin{equation}
\begin{split}
\int & K  (\alpha^{j}\varphi_{j})^{\frac{n+2}{n-2}}\phi_{k,i} \\
= &
\underset{[\alpha_{i}\varphi_{i}\geq \sum_{i\neq j=1}^{p}\alpha_{j}\varphi_{j}]}{\int}
K
(\alpha_{i}\varphi_{i})^{\frac{n+2}{n-2}} \phi_{k,i} 
+
\frac{n+2}{n-2}(\alpha_{i}\varphi_{i})^{\frac{4}{n-2}}\sum_{i\neq j=1}^{p}\alpha_{j}\varphi_{j}\phi_{k,i} \\
& +
\underset{[\alpha_{i}\varphi_{i}< \sum_{i\neq j=1}^{p}\alpha_{j}\varphi_{j}]}{\int}
K(\sum_{i\neq j=1}^{p}\alpha_{j}\varphi_{j})^{\frac{n+2}{n-2}}\phi_{k,i} \\
& +
O
(
\underset{[\varphi_{i}\geq \epsilon\sum_{i\neq j=1}^{p}\varphi_{j}]}{\int}
\varphi_{i}^{\frac{4}{n-2}}\sum_{i\neq j=1}^{p}\varphi_{j}^{2}
+
\underset{[\epsilon \varphi_{i}< \sum_{i\neq j=1}^{p}\varphi_{j}]}{\int}
\sum_{i\neq j=1}^{p}\varphi_{j}^{\frac{4}{n-2}}\varphi_{i}^{2}
),
\end{split}
\end{equation} 
whence
\begin{equation}
\begin{split}
\int & K  (\alpha^{j}\varphi_{j})^{\frac{n+2}{n-2}}\phi_{k,i} \\
= &
\int 
K
(\alpha_{i}\varphi_{i})^{\frac{n+2}{n-2}} \phi_{k,i} 
+
\frac{n+2}{n-2}(\alpha_{i}\varphi_{i})^{\frac{4}{n-2}}\sum_{i\neq j=1}^{p}\alpha_{j}\varphi_{j}\phi_{k,i} \\
& +
\int
K(\sum_{i\neq j=1}^{p}\alpha_{j}\varphi_{j})^{\frac{n+2}{n-2}}\phi_{k,i} \\
& +
O
(
\underset{[\varphi_{i}\geq \epsilon\sum_{i\neq j=1}^{p}\varphi_{j}]}{\int}
\varphi_{i}^{\frac{4}{n-2}}\sum_{i\neq j=1}^{p}\varphi_{j}^{2}
+
\underset{[\epsilon \varphi_{i}< \sum_{i\neq j=1}^{p}\varphi_{j}]}{\int}
\sum_{i\neq j=1}^{p}\varphi_{j}^{\frac{4}{n-2}}\varphi_{i}^{2}
).
\end{split}
\end{equation} 
Therefore we obtain applying lemma \ref{lem_interactions}
\begin{equation}
\begin{split}
\int  K  (\alpha^{j}&\varphi_{j})^{\frac{n+2}{n-2}}\phi_{k,i} \\
= &
\int 
K
(\alpha_{i}\varphi_{i})^{\frac{n+2}{n-2}} \phi_{k,i} 
+
\frac{n+2}{n-2}(\alpha_{i}\varphi_{i})^{\frac{4}{n-2}}\sum_{i\neq j=1}^{p}\alpha_{j}\varphi_{j}\phi_{k,i} \\
& +
\int
K(\sum_{i\neq j=1}^{p}\alpha_{j}\varphi_{j})^{\frac{n+2}{n-2}}\phi_{k,i} 
+
o_{\varepsilon}(\sum_{i\neq j=1}^{p}\eps_{i,j}).
\end{split}
\end{equation} 
Moreover note, that for $\epsilon >0$ sufficiently small
\begin{equation}
\begin{split}
M=\cup_{i=1}^{p}[\varphi_{i}>\epsilon\sum_{i\neq j=1}^{p}\varphi_{j}]=\cup_{i=1}^{p}A_{i},
\end{split}
\end{equation} 
whence for $B_{i}=A_{i}\setminus \cup_{i\neq j=1}^{p}A_{j}$ we  have
$M=\sum_{i=1}^{p}B_{i}$.
This gives
\begin{equation}
\begin{split}
\int 
K  & (\sum_{i\neq j=1}^{p}\alpha_{j}\varphi_{j})^{\frac{n+2}{n-2}}\phi_{k,i} 
= 
\sum_{i\neq j=1}^{p}\underset{B_{j}}{\int}K(\sum_{i\neq j=1}^{p}\alpha_{j}\varphi_{j})^{\frac{n+2}{n-2}}\phi_{k,i}
+
o(\sum_{i\neq j=1}^{p}\eps_{i,j}) \\
= &
\sum_{i\neq j=1}^{p}
\int K (\alpha_{j}\varphi_{j})^{\frac{n+2}{n-2}}\phi_{k,i} 
+
O(\sum_{\underset{r\neq s}{s\neq i,r\neq i}}\int \varphi_{r}^{\frac{4}{n-2}}\varphi_{s}\varphi_{i})
+
o(\sum_{i\neq j=1}^{p}\eps_{i,j})
\end{split}
\end{equation} 
and we obtain using H\"older's inequality and lemma \ref{lem_interactions}
\begin{equation}
\begin{split}
\int
K(\sum_{i\neq j=1}^{p}\alpha_{j}\varphi_{j})^{\frac{n+2}{n-2}}\phi_{k,i} 
= &
\sum_{i\neq j=1}^{p}
\int K (\alpha_{j}\varphi_{j})^{\frac{n+2}{n-2}}\phi_{k,i} \\
& +
o(\sum_{i\neq j=1}^{p}\eps_{i,j})
+
O(\sum_{r\neq s}\eps_{r,s}^{2}).
\end{split}
\end{equation} 
Therefore
\begin{equation}\begin{split}\label{intK(ajdj^(...)phiki}
\int K (\alpha^{j}\varphi_{j})^{\frac{n+2}{n-2}}\phi_{k,i}
= &
\alpha_{i}^{\frac{n+2}{n-2}}\int K\varphi_{i}^{\frac{n+2}{n-2}}\phi_{k,i}
+
\sum_{i\neq j =1}^{p}\alpha_{j}^{\frac{n+2}{n-2}}\int K \varphi_{j}^{\frac{n+2}{n-2}}\phi_{k,i} \\
& +
\frac{n+2}{n-2}\alpha_{i}^{\frac{4}{n-2}}
\sum_{i \neq j=1}^{p}\alpha_{j}
\int K\varphi_{i}^{\frac{4}{n-2}}\phi_{k,i}\varphi_{j} \\
& +
o_{\varepsilon}(\sum_{i\neq j=1}^{p}\eps_{i,j})
+
O(\sum_{r\neq s}\eps_{r,s}^{2}).
\end{split}\end{equation}
By a simple expansion we then get
\begin{equation}\begin{split}\label{IntK(ajdj)^{...}Phiki}
\int K (\alpha^{j}\varphi_{j})^{\frac{n+2}{n-2}}\phi_{k,i}
= &
\alpha_{i}^{\frac{n+2}{n-2}}K_{i}\int \varphi_{i}^{\frac{n+2}{n-2}}\phi_{k,i}
+
\sum_{i\neq j =1}^{p}\alpha_{j}^{\frac{n+2}{n-2}}K_{j}\int  \varphi_{j}^{\frac{n+2}{n-2}}\phi_{k,i} \\
& +
\frac{n+2}{n-2}\alpha_{i}^{\frac{4}{n-2}}K_{i}
\sum_{i \neq j=1}^{p}\alpha_{j}
\int \varphi_{i}^{\frac{4}{n-2}}\phi_{k,i}\varphi_{j} \\
& +
\alpha_{i}^{\frac{n+2}{n-2}}(e_{1}  \frac{\lap K_{i}}{\lambda_{i} ^{2}},e_{2} \frac{\lap K_{i}}{\lambda_{i}^{2}},
e_{3}\frac{\nabla K_{i}}{\lambda_{i}} +e_{4}\frac{\nabla \lap K_{i}}{\lambda_{i}^{3}} )\\
& +
o_{\varepsilon}( \frac{1}{\lambda_{i}^{n-2}} +\sum_{i\neq j=1}^{p}\eps_{i,j})
+
O
(
\sum_{r\neq s}\eps_{r,s}^{2}
)
.
\end{split}\end{equation}
Indeed using \eqref{Phi1i}, \eqref{Phi2i}, \eqref{Phi3i}
we have in case $k=1$,
\begin{equation}\begin{split}\label{bringoutK_k=1}
\int (K & -  K_{i})\varphi_{i}^{\frac{n+2}{n-2}}\phi_{k,i} \\
= &
\int_{B_{\lambda_{i}\alpha}(0)}
\frac{K( \frac{1}{\lambda_{i}} \, \cdot)-K(0)}{( 1+r^{2}(1+ \frac{1}{\lambda_{i}^{n-2}} r^{n-2}H_{a_{i}}(  \frac{\cdot}{\lambda_{i}}))^{\frac{2}{2-n}})^{n}}
+
O( \frac{1}{\lambda_{i}^{n}})
\\ 
= &
\int_{B_{\lambda_{i}\alpha }(0)}
\frac{K(  \frac{1}{\lambda_{i}} \, \cdot)-K(0)}{( 1+r^{2})^{n}}
+
O( \frac{1}{\lambda_{i}^{n-1}}) = 
e_{1}  \frac{\lap K_{i}}{\lambda_{i} ^{2}}
+
o( \frac{1}{\lambda_{i}^{n-2}} ),
\end{split}\end{equation}
where 
$e_{1}=\frac{1}{2n}\int_{\R^{n}} \frac{r^{2}}{(1+r^{2})^{n}}$.
In case $k=2$ we get
\begin{equation}\begin{split}\label{bringoutK_k=2}
\int (  K & -K_{i})\varphi_{i}^{\frac{n+2}{n-2}}\phi_{k,i} \\
= &
\frac{n-2}{2}  \frac{1}{\lambda_{i}}\int_{B_{\lambda_{i}\alpha }(0)}
\frac
{(K(  \frac{1}{\lambda_{i}} \, \cdot)-K_{i}) ( r^{2}-1)}
{( 1+r^{2})^{n+1}}  
+
O( \frac{1}{\lambda_{i}^{n-1}})\\
= &
e_{2}  \frac{\lap K_{i}}{\lambda_{i} ^{2}}
+
o( \frac{1}{\lambda_{i}^{n-2}} ),
\end{split}\end{equation}
where 
$e_{2}=\frac{(n-2)}{4n}\int_{\R^{n}} \frac{r^{2}(r^{2}-1)}{(1+r^{2})^{n+1}}$
and in case $k=3$ 
\begin{equation}\begin{split}\label{bringoutK_k=3}
\int (K & -K_{i})\varphi_{i}^{\frac{n+2}{n-2}}\phi_{k,i} 
= 
\frac{n-2}{2n}\int (K  -K_{i})\frac{1}{\lambda_{i}}\nabla_{a_{i}}\varphi_{i}^{\frac{2n}{n-2}} \\
= &
\frac{n-2}{2n}\frac{\nabla K_{i}}{\lambda_{i}}\int \varphi_{i}^{\frac{2n}{n-2}} 
+
\frac{n-2}{2n}\frac{\nabla_{a_{i}}}{\lambda_{i}}\int (K  -K_{i})\varphi_{i}^{\frac{2n}{n-2}} 
\\
=&
e_{3}\frac{\nabla K_{i}}{\lambda_{i}}  
+
e_{4}\frac{\nabla \lap K_{i}}{\lambda_{i}^{3}}
+
o( \frac{1}{\lambda_{i}^{n-2}} )
\end{split}\end{equation}
with 
$e_{3}=\frac{n-2}{2n}\int_{\R^{n}} \frac{1}{(1+r^{2})^{n}},
e_{4}=\frac{n-2}{4n^{2}}\int_{\R^{n}}\frac{r^{2}}{(1+r^{2})^{n}}$.

Plugging \eqref{IntK(ajdj)^{...}Phiki} into \eqref{intKu^..Phiki} gives
\begin{equation}\begin{split}\label{intKu^..Phiki_expanded}
\int K  &u^{\frac{n+2}{n-2}}\phi_{k,i} \\
= &
\alpha_{i}^{\frac{n+2}{n-2}}K_{i}\int \varphi_{i}^{\frac{n+2}{n-2}}\phi_{k,i}
+
\sum_{i\neq j =1}^{p}\alpha_{j}^{\frac{n+2}{n-2}}K_{j}\int  \varphi_{j}^{\frac{n+2}{n-2}}\phi_{k,i} \\
& +
\frac{n+2}{n-2}\alpha_{i}^{\frac{4}{n-2}}K_{i}
\sum_{i \neq j=1}^{p}\alpha_{j}
\int \varphi_{i}^{\frac{4}{n-2}}\phi_{k,i}\varphi_{j} \\
& +
\alpha_{i}^{\frac{n+2}{n-2}}(e_{1}  \frac{\lap K_{i}}{\lambda_{i} ^{2}},e_{2} \frac{\lap K_{i}}{\lambda_{i}^{2}},e_{3}\frac{\nabla K_{i}}{\lambda_{i}} +e_{4}\frac{\nabla \lap K_{i}}{\lambda_{i}^{3}} )\\
& +
o_{\varepsilon}( \frac{1}{\lambda_{i}^{n-2}} +\sum_{i\neq j=1}^{p}\eps_{i,j})
+
O
(
\sum_{r\neq s}\eps_{r,s}^{2}
+
\Vert v \Vert^{2}
)
\end{split}\end{equation}

and inserting finally \eqref{IntL_{g_{0}}uPhiki_expanded} and \eqref{intKu^..Phiki_expanded} into \eqref{sigmaki=Ki...+K-Ki}  we conclude 
\begin{equation}\begin{split}
\sigma_{k,i}
= &
-
4n(n-1)
[
\alpha_{i}
\int
\varphi_{i}^{\frac{n+2}{n-2}}
\phi_{k,i}
+
b_{k}\sum_{i\neq j =1}^{p}\alpha_{j}d_{k,i}\eps_{i,j}
] \\
& 
+
\alpha_{i}(d_{1}\frac{ H_{i}}{\lambda_{i} ^{n-2}},d_{2}\frac{ H_{i}}{\lambda_{i} ^{n-2}},d_{3}\frac{\nabla H_{i}}{\lambda_{i} ^{n-1}}  ) 
\\
& + 
\alpha_{i}^{\frac{n+2}{n-2}}\frac{r}{k}K_{i}\int \varphi_{i}^{\frac{n+2}{n-2}}\phi_{k,i}
+
\sum_{i\neq j =1}^{p}\alpha_{j}^{\frac{n+2}{n-2}}\frac{r}{k}K_{j}\int  \varphi_{j}^{\frac{n+2}{n-2}}\phi_{k,i} \\
& +
\frac{n+2}{n-2}\alpha_{i}^{\frac{4}{n-2}}\frac{r}{k}K_{i}
\sum_{i \neq j=1}^{p}\alpha_{j}
\int \varphi_{i}^{\frac{4}{n-2}}\phi_{k,i}\varphi_{j} \\
& +
\alpha_{i}^{\frac{n+2}{n-2}}\frac{r}{k}
(e_{1}  \frac{\lap K_{i}}{\lambda_{i} ^{2}},e_{2} \frac{\lap K_{i}}{\lambda_{i}^{2}},e_{3}\frac{\nabla K_{i}}{\lambda_{i}} 
+
e_{4}\frac{\nabla \lap K_{i}}{\lambda_{i}^{3}} )
\\
& +
o_{\varepsilon}( \frac{1}{\lambda_{i}^{n-2}} +\sum_{i\neq j=1}^{p}\eps_{i,j})
+
O
(
\sum_{r\neq s}\eps_{r,s}^{2}
+
\Vert v \Vert^{2}
).
\end{split}\end{equation}
The claim follows.
\end{proof}
As $\sigma_{1,i}=O(\vert \delta J(u)\vert)$
the equations for $\sigma_{2,i}, \sigma_{3,i}$ simplify significantly.
\begin{corollary}[Simplifying $\sigma_{k,i}$]\label{cor_simplifying_ski}$_{}$\\
On $V(p, \eps)$ for $\eps>0$ small we have with constants $b_{2}, \ldots,e_{4}>0$
\begin{enumerate}[label=(\roman*)]
 \Item $_{}$
\begin{equation*}\begin{split}
\sigma_{2,i}
= &
d_{2}\alpha_{i}\frac{ H_{i}}{\lambda_{i} ^{n-2}}
+
e_{2}\frac{r\alpha_{i}^{\frac{n+2}{n-2}}}{k}\frac{\lap K_{i}}{\lambda_{i}^{2}} 
- 
b_{2}\frac{r\alpha_{i}^{\frac{4}{n-2}}K_{i}}{k}
\sum_{i \neq j=1}^{p}\alpha_{j}
\lambda_{i}\partial_{\lambda_{i}}\eps_{i,j} 
+
R_{2,i}
\end{split}\end{equation*}
 \Item $_{}$
\begin{equation*}\begin{split}
\sigma_{3,i}
= &
\frac{r\alpha_{i}^{\frac{n+2}{n-2}}}{k}
[
 e_{3}\frac{\nabla K_{i}}{\lambda_{i}}
+
e_{4}\frac{\nabla \lap K_{i}}{\lambda_{i}^{3}}
] 
+
b_{3}\frac{r\alpha_{i}^{\frac{4}{n-2}}K_{i}}{k}
\sum_{i \neq j=1}^{p}
\frac{\alpha_{j}}{\lambda_{i}}\nabla_{a_{i}}\eps_{i,j} +
R_{3,i},
\end{split}\end{equation*}
\end{enumerate}
where  
\begin{equation*}\begin{split}
R_{k,i}
= 
o_{\varepsilon}
(
 \frac{1}{\lambda_{i}^{n-2}} 
+
\sum_{i\neq j=1}^{p}\eps_{i,j}
)_{k,i} 
+
O
(
\sum_{r\neq s}\eps_{r,s}^{2} 
+
\Vert v \Vert^{2}
+
\vert \delta J(u)\vert^{2}
)_{k,i}
.                               
\end{split}\end{equation*}
\end{corollary}

\begin{proof}[\textbf{Proof of corollary \ref{cor_simplifying_ski}}]
\label{p_simplifying_ski}$_{}$\\
We have 
\begin{equation}\begin{split}
C\vert \delta J(u)\vert
\geq  \vert\int (R-r\K)u^{\frac{n+2}{n-2}}\varphi_{i}\vert
=
\vert \sigma_{1,i}\vert,
\end{split}\end{equation}
whence due to proposition \ref{prop_analysing_ski} for $k=1$
\begin{equation}\begin{split}\label{rai^{...}/k=...}
\frac{r\alpha_{i}^{\frac{4}{n-2}}K_{i}}{4n(n-1)k}
= &
1 
+
O
(
\frac{1}{\lambda_{i}^{n-2}}+\frac{\vert \lap K_{i}\vert}{\lambda_{i}^{2}}
+
\sum_{i\neq j=1}^{p}\eps_{i,j}
\\
& \quad\quad \quad
+
\sum_{r\neq s}\eps_{r,s}^{2}
+
\Vert v \Vert^{2}
+
\vert \delta J(u)\vert
).
\end{split}\end{equation}
Inserting \eqref{rai^{...}/k=...} into proposition \ref{prop_analysing_ski} for $k=2,3$ proves the claim, since
\begin{equation}
\begin{split}
\frac{\nabla_{a_{i}}}{\lambda_{i}}\int \var_{i}^{\frac{2n}{n-2}}
,
\lambda_{i}\partial_{\lambda_{i}}\int \var_{i}^{\frac{2n}{n-2}}
=
O(\frac{1}{\lambda_{i}^{n-2}}).
\end{split}
\end{equation} 
\end{proof}
We turn to estimate the error term  $v$. To do so we characterize the first two derivatives of 
$J$ at $\alpha^{i}\varphi_{i}=u-v$.
\begin{proposition}[Derivatives on $H(p, \eps)$]\label{prop_derivatives_on_H}$_{}$\\
For $\eps>0$ small let $u=\alpha^{i}\var_{i}+v\in V(p, \eps)$ and $h_{1},h_{2}\in H= H_{u}(p, \eps)$.

We then have
\begin{enumerate}[label=(\roman*)]
 \Item 
 \begin{equation*}\begin{split}
\Vert\partial J( \alpha^{i}\varphi_{i}  )\lfloor_{H}\Vert
= 
O(\sum_{r} \frac{\vert \nabla K_{r}\vert}{\lambda_{r}}  & + \frac{\vert \lap K_{r}\vert}{\lambda_{r}^{2}} 
+
\frac{1}{\lambda_{r}^{n-2}} 
\\ & 
+\sum_{r\neq s}\eps_{r,s}+\Vert v \Vert^{2}+\vert \delta J(u)\vert)
\end{split}\end{equation*}
\Item 
\begin{equation*}\begin{split}
\frac{1}{2}\partial^{2}J( \alpha^{i}\varphi_{i}   )h_{1}h_{2}
= &
k^{\frac{2-n}{n}}_{\alpha^{i}\varphi_{i}}
[\int L_{g_{0}}h_{1}h_{2}
-
c_{n}n(n+2)
\sum_{i}\int\varphi_{i}^{\frac{4}{n-2}}h_{1}h_{2}
] \\
& +
o_{\varepsilon}(\Vert h_{1} \Vert\, \Vert h_{2} \Vert)
\end{split}\end{equation*}
\end{enumerate}
\end{proposition}

\begin{proof}[\textbf{Proof of proposition \ref{prop_derivatives_on_H}}]\label{p_derivatives_on_H}$_{}$\\
Let in addition $h\in H_{u}(p, \eps)$ with $\Vert h \Vert=1$. From proposition \ref{prop_derivatives_of_J} we then infer
\begin{equation}\begin{split}
\frac{1}{2}\partial  J(\alpha^{i}\varphi_{i})h
= &
k_{\alpha^{i}\varphi_{i}}^{\frac{2-n}{n}} 
[
\int L_{g_{0}}(\alpha^{i}\varphi_{i} )h
-
 \int (r\K)_{\alpha^{i}\varphi_{i}}(\alpha^{i}\varphi_{i})^{\frac{n+2}{n-2}}h
]
\end{split}\end{equation}
and 
\begin{equation}\begin{split}
\frac{1}{2} \partial^{2}  J( \alpha^{i}\varphi_{i}   )h_{1}h_{2} 
= &
k_{\alpha^{i}\varphi_{i}}^{\frac{2-n}{n}}
[
\int \hspace{-4pt}L_{g_{0}}h_{1}h_{2}
 -
\frac{n+2}{n-2}
\int (r\K)_{\alpha^{i}\varphi_{i}}(\alpha^{i}\varphi_{i})^{\frac{4}{n-2}}h_{1}h_{2}
] \\
& +
o_{\varepsilon}(\Vert h_{1} \Vert \Vert h_{2} \Vert),
\end{split}\end{equation}
since, when considering the formula for the second variation, we have
\begin{equation}
\begin{split}
\int L_{g_{0}}uh_{i} 
= &
\frac{r}{k}\int Ku^{\frac{n+2}{n-2}}h_{i}
+
O(\vert \delta J(u)\vert \Vert h_{i}\Vert)\\
= &
\frac{r}{k}\int Ku^{\frac{4}{n-2}}vh_{i}+O(\vert \delta J(u)\vert \Vert h_{i}\Vert) \\
= &
O(\Vert v \Vert+\vert \delta J(u)\vert)\Vert h_{i}\Vert.
\end{split}
\end{equation} 
Using $\frac{r\alpha_{i}^{\frac{4}{n-2}}K_{i}}{k}=4n(n-1)+o_{\varepsilon}(1)$ and $c_{n}=4\frac{n-1}{n-2}$ we obtain
\begin{equation}\begin{split}
\frac{1}{2}\partial^{2}J( \alpha^{i}\varphi_{i}   )h_{1}h_{2}
= &
k_{\alpha^{i}\varphi_{i}}^{\frac{2-n}{n}}
[
\int L_{g_{0}}h_{1}h_{2}
 -
c_{n}n(n+2)
\int  \sum_{i}\tilde\varphi_{i}^{\frac{4}{n-2}}h_{1}h_{2}
] \\
& +
o_{\varepsilon}(\Vert h_{1} \Vert \Vert h_{2} \Vert), 
\end{split}\end{equation}
This shows the statement on the second derivative. Moreover by lemma \ref{lem_v_type_interactions} 
\begin{equation}\begin{split}
\frac{r_{\alpha^{i}\varphi_{i}}}{k_{\alpha^{i}\varphi_{i}}}
= &
\frac{r}{k}  
+
o(\sum_{r} \frac{1}{\lambda_{r}^{n-2}} + \sum_{r\neq s}\eps_{r,s})
+
O(\sum_{r}\frac{\vert \nabla K_{r}\vert^{2}}{\lambda_{r}^{2}}+\Vert v \Vert^{2}).
\end{split}\end{equation}
We obtain with $r\K=\frac{r}{k}K=(\frac{r}{k})_{u}K$
\begin{equation}\begin{split}\label{1/2partialdJu}
\frac{1}{2}\partial  J(\alpha^{i}\varphi_{i}  )h
= &
k_{\alpha^{i}\varphi_{i}}^{\frac{2-n}{n}} 
[
\int L_{g_{0}}(\alpha^{i}\varphi_{i})h
-
\int r\K(\alpha^{i}\varphi_{i})^{\frac{n+2}{n-2}}h
]
\\ &
+
o(\sum_{r} \frac{1}{\lambda_{r}^{n-2}} + \sum_{r\neq s}\eps_{r,s})
+
O(\sum_{r}\frac{\vert \nabla K_{r}\vert^{2}}{\lambda_{r}^{2}}
+
\Vert v \Vert^{2}),
\end{split}\end{equation}
where due to lemmata \ref{lem_emergence_of_the_regular_part} and \ref{lem_interactions}
\begin{equation}\begin{split}
\int K(\alpha^{i}\varphi_{i} )^{\frac{n+2}{n-2}}h
= &
\sum_{i}\frac{\alpha_{i}^{\frac{n+2}{n-2}}K_{i}}{4n(n-1)}\int L_{g_{0}}\varphi_{i}h \\
& +
O(\sum_{r} \frac{\vert \nabla K_{r}\vert}{\lambda_{r}} 
+ 
\frac{\vert \lap K_{r}\vert}{\lambda_{r}^{2}} 
+
\frac{1}{\lambda_{r}^{n-2}} +\sum_{r\neq s}\eps_{r,s})
.
\end{split}\end{equation}
This gives
\begin{equation}\begin{split}
\frac{1}{2}\partial J(\alpha^{i}\varphi_{i})h
= &
k_{\alpha^{i}\varphi_{i}}^{\frac{2-n}{n}} 
\alpha^{i}(1-\frac{r\alpha^{\frac{4}{n-2}}_{i}K_{i}}{4n(n-1)k})\int L_{g_{0}}\varphi_{i}h
\\ &
+
O(\sum_{r} \frac{\vert \nabla K_{r}\vert}{\lambda_{r}}
+
\frac{\vert \lap K_{r}\vert}{\lambda_{r}^{2}} 
+
\frac{1}{\lambda_{r}^{n-2}} +\sum_{r\neq s}\eps_{r,s})
+
\Vert v \Vert^{2}.
\end{split}\end{equation}
From this the assertion on the first derivative follows from \eqref{rai^{...}/k=...}.
\end{proof}
The second variation at $\alpha^{i}\var_{i}$ turns out to be positive  definite.
\begin{proposition}[Positivity of the second variation]\label{prop_positivity_of_D2J}$_{}$\\
There exist $\gamma, \eps_{0}>0$ such,  that for any 
\begin{equation}
u=\alpha^{i}\varphi_{i}+v\in V(p,\varepsilon) 
\end{equation} 
with $0<\eps<\eps_{0}$ we have  
\begin{equation*}\begin{split}
\partial^{2}J(\alpha^{i}\varphi_{i})\lfloor_{H}>\gamma, \quad H=H_{u}(p, \eps).
\end{split}\end{equation*}
\end{proposition}

\begin{proof}[\textbf{Proof of proposition \ref{prop_positivity_of_D2J}}]\label{p_positivity_of_D2J}
(Cf.  \cite{BrendleArbitraryEnergies}, proposition 5.4)$_{}$\\
In view of proposition \ref{prop_derivatives_on_H} there would otherwise exist 
\begin{equation}\begin{split}
\epsilon_{k}\searrow 0 
\;\text{ and }\;
(w_{k})\subset H_{u_{k}}(p, \epsilon_{k})                                                                                          
\end{split}\end{equation}
such, that 
\begin{equation}\begin{split}
1
=
\int c_{n}\vert \nabla w_{k}\vert^{2}_{g_{0}}+R_{g_{0}}w_{k}^{2}
\leq 
c_{n}n(n+2)\lim_{k\nearrow \infty}\int\sum_{i}\varphi_{i,k}^{\frac{4}{n-2}}w_{k}^{2}.
\end{split}\end{equation}
We order 
$
\frac{1}{\lambda_{1_{k}}}\leq \ldots\leq \frac{1}{\lambda_{p_{k}}} 
$
and choose 
$\gamma_{k}\nearrow \infty$ tending to infinity slower than
\begin{equation}
\begin{split}
\frac{1}{\lambda_{i_{k}}}, \eps_{i_{k},j_{k}}\- 0
\end{split}
\end{equation} 
does tend to zero in the sense, that for all $i<j$
\begin{equation}\begin{split}
\frac{\lambda_{i_{k}}}{\gamma_{k}}, \;
\frac{\frac{\lambda_{i_{k}}}{\lambda_{j_{k}}} +\lambda_{i_{k}}G^{\frac{1}{2-n}}(a _{i_{k}},a _{j_{k}})}
{\gamma_{k}}\nearrow \infty
\end{split}\end{equation}
as $k\-\infty$. Define inductively
\begin{equation}\begin{split}
\Omega_{j,k}
=
B_{\frac{\gamma_{k}}{\lambda_{j_{k}}}}( a _{j_{k}})
\setminus
\cup_{i<j}B_{\frac{\gamma_{k}}{\lambda_{i_{k}}} }(a _{i_{k}}).
\end{split}\end{equation}
Then there exists  $j= 1, \ldots,p$ such,  that
\begin{equation}\begin{split}
\lim_{k\-\infty}\int  \varphi_{j,k}^{\frac{4}{n-2}}w_{k}^{2}>0
\end{split}\end{equation}
and
\begin{equation}\begin{split}
\lim_{k \- \infty} \int _{\Omega_{j,k}} c_{n}\vert \nabla w_{k}\vert^{2}_{g_{0}}+R_{g_{0}}w_{k}^{2}
\leq 
c_{n}n(n+2)\lim_{k\-\infty}\int\varphi_{j,k}^{\frac{4}{n-2}}w_{k}^{2}.
\end{split}\end{equation}
Blowing up 
on $\Omega_{j,k}$ one obtains $\tilde{w}_{k}\rightharpoondown:\tilde{w}$ locally with 
$\tilde{w}\in W^{1,2}(\R^{n})$ and
\begin{equation}\begin{split}
\int_{\R^{n}}\vert \nabla \tilde{w} \vert^{2}\leq n(n+2)\int_{\R^{n}}(\frac{1}{1+r^{2}})^{2}\tilde{w}^{2}, \;
\int_{\R^{n}} (\frac{1}{1+r^{2}})^{2}\tilde{w}^{2}>0.
\end{split}\end{equation}
In particular $\tilde{w} \neq 0$. But due to orthogonality  $w_{k}\in H_{u_{k}}(p, \eps)$ one finds
\begin{equation}\begin{split}
\int_{\R^{n}} (\frac{1}{1+r^{2}})^{\frac{n+2}{2}}\tilde{w},
\int_{\R^{n}} (\frac{1}{1+r^{2}})^{\frac{n+2}{2}}\frac{1-r^{2}}{1+r^{2}}\tilde{w}
=0
\end{split}\end{equation}
and 
\begin{equation}\begin{split}
\int_{\R^{n}} (\frac{1}{1+r^{2}})^{\frac{n+2}{2}}\frac{x}{1+r^{2}}\tilde{w} (x)
= 0.
\end{split}\end{equation}
This is a contradiction, cf. \cite{ReyRoleOfGreensFunction} Appendix D, pp.49-51.
\end{proof}
Smallness of the first and positivity of the second derivative give a suitable
estimate on the error term  $v$.
\begin{corollary}[A-priori estimate on $v$]\label{cor_a-priori_estimate_on_v}$_{}$\\
On $V(p, \eps)$ for $\eps>0$ small we have 
\begin{equation*}\begin{split}
\Vert v \Vert
=
O(\sum_{r} \frac{\vert \nabla K_{r}\vert}{\lambda_{r}} + \frac{\vert \lap K_{r}\vert}{\lambda_{r}^{2}}
+
\frac{1}{\lambda_{r}^{n-2}} +\sum_{r\neq s}\eps_{r,s}
+
\vert \delta J(u)\vert).
\end{split}\end{equation*}
\end{corollary}

\begin{proof}[\textbf{Proof of corollary \ref{cor_a-priori_estimate_on_v}}]\label{p_a-priori_estimate_on_v}$_{}$\\
Note, that $\partial^{2}J$ is uniformly H\"older continuous on $V(p, \varepsilon)$ by proposition \ref{prop_derivatives_of_J} and the remarks following, whence in view of proposition  \ref{prop_positivity_of_D2J} we have
\begin{equation}\begin{split}
\partial J(u)v
= &
\partial J(\alpha^{i}\varphi_{i}+v)v
=
\partial J(\alpha^{i}\varphi_{i})v+
\partial^{2} J(\alpha^{i}\varphi_{i})v^{2}
+
o(\Vert v \Vert^{2})
\\
\geq &
\partial J(\alpha^{i}\varphi_{i})v
+
\gamma \Vert v\Vert^{2}
+
o(\Vert v \Vert^{2}).
\end{split}\end{equation}
Since $v\in H_{u}(p, \eps)$ the claim follows from proposition \ref{prop_derivatives_on_H} by absorption.
\end{proof}
Thus having analysed $\sigma_{k,i}$ and the error term  $v$ the shadow flow reads as
\begin{corollary}[Simplifying the shadow flow] \label{cor_simplifying_the_shadow_flow}$_{}$\\
For $u\in V(p, \eps)$ with $\eps>0$ small we have
\begin{enumerate}[label=(\roman*)] 
 \Item $_{}$
\begin{equation*}\begin{split}
-\frac{\dot \lambda_{i}}{\lambda_{i}}
= & 
\frac{r}{k}
[
\frac{d_{2}}{c_{2}}
\frac{ H_{i}}{\lambda_{i} ^{n-2}}
+
\frac{e_{2}}{c_{2}}\frac{\lap K_{i}}{K_{i}\lambda_{i} ^{2}} 
-
\frac{b_{2}}{c_{2}} \sum_{i \neq j=1}^{p}\frac{\alpha_{j}}{\alpha_{i}}
\lambda_{i}\partial_{\lambda_{i}}\eps_{i,j} 
]
(1+o_{\frac{1}{\lambda_{i}}}(1)) \\
& +
R_{2,i}
\end{split}\end{equation*}
 \Item $_{}$
\begin{equation*}\begin{split}
\lambda_{i}\dot a_{i}
= &
\frac{r}{k}
[
\frac{e_{3}}{c_{3}}  \frac{\nabla K_{i}}{K_{i}\lambda_{i}}
+
\frac{e_{4}}{c_{3}}  \frac{\nabla \lap K_{i}}{K_{i}\lambda_{i}^{3}}
+
\frac{b_{3}}{c_{3}} \sum_{i \neq j=1}^{p}\frac{\alpha_{j}}{\alpha_{i}}
 \frac{1}{\lambda_{i}}\nabla_{a_{i}}\eps_{i,j}
](1+o_{\frac{1}{\lambda_{i}}}(1)) \\
& +
R_{3,i},
\end{split}\end{equation*}
\end{enumerate}
where
\begin{equation*}\begin{split}
R_{2,i},R_{3,i}
= &
o_{\varepsilon}( \frac{1}{\lambda_{i}^{n-2}} +\sum_{i\neq j=1}^{p}\eps_{i,j}) \\
& +
O
(
\sum_{r} \frac{\vert \nabla K_{r}\vert^{2}}{\lambda_{r} ^{2}}
+
\frac{\vert \lap K_{r}\vert^{2}}{\lambda_{r}^{4}}
+
\frac{1}{\lambda_{r}^{2(n-2)}}
+
\sum_{r\neq s}\eps_{r,s}^{2}
+
\vert \delta J(u)\vert^{2})
.\end{split}\end{equation*}

\end{corollary}
Thus the movement of $a_{i}$ and $\lambda_{i}$ is primarily ruled by quantities arising from self-interaction of  $\varphi_{i}$ and direct interaction of $\varphi_{i}$ with other bubbles $\varphi_{j}$.

\begin{proof}[\textbf{Proof of corollary \ref{cor_simplifying_the_shadow_flow}}]\label{p_simplifying_the_shadow_flow}$_{}$\\
This follows immediately  from corollaries \ref{cor_simplifying_ski}, \ref{cor_a-priori_estimate_on_v} applied to lemma \ref{lem_the_shadow_flow} and using
\eqref{rai^{...}/k=...} for the $H_{i}$ term; we have replaced $\frac{d_{2}}{4n(n-1)}$ by $d_{2}$
\end{proof}

\section{Case \textomega \textgreater 0}
\label{sec:omega>0}

Analogously to the case $\omega =0$ we establish the shadow flow.
\begin{lemma}[The shadow flow]\label{lem_the_shadow_flow_w}$_{}$\\
For $u\in V(\omega, p, \eps)$ with $\eps>0$ small and 
\begin{equation*}\begin{split}
\sigma_{k,i}=-\int (L_{g_{0}}u-r\K u^{\frac{n+2}{n-2}})\phi_{k,i}, \, i=1, \ldots,p, \,k=1,2,3
\end{split}\end{equation*}
we have suitable testing of $K\partial_{t}u=-(R-r\K)u$
\begin{enumerate}[label=(\roman*)]
 \Item $_{}$
\begin{equation*}\begin{split}
\frac{\dot \alpha_{i}}{\alpha_{i}}
= &
\frac{\alpha_{i}^{\frac{n+2}{2-n}}}{c_{1}K_{i}}
\sigma_{1,i}
(1+o_{\frac{1}{\lambda_{i}}}(1))
+
R_{1,i}
.
\end{split}\end{equation*}
 \Item $_{}$ 
\begin{equation*}\begin{split}
-\frac{\dot \lambda_{i}}{\lambda_{i}}
= &
\frac{\alpha_{i}^{\frac{n+2}{2-n}}}{c_{2}K_{i}} \sigma_{2,i}(1+o_{\frac{1}{\lambda_{i}}}(1))
+
R_{2,i}
\end{split}\end{equation*}
 \Item $_{}$ 
\begin{equation*}\begin{split}
\lambda_{i}\dot a_{i}
= &
\frac{\alpha_{i}^{\frac{n+2}{2-n}}}{c_{3}K_{i}} \sigma_{3,i}(1+o_{\frac{1}{\lambda_{i}}}(1))
+
R_{3,i}
\end{split}\end{equation*}
\end{enumerate}
with constants $c_{k}>0$ given in lemma \ref{lem_interactions} and
\begin{equation*}
R_{k,i}
=
O
(
\sum_{r}
\frac{1}{\lambda_{r}^{n-2}}
+
\sum_{r\neq s}\eps_{r,s}^{2}
+
\Vert v \Vert^{2}
+
\vert \delta J(u)\vert^{2})_{k,i}.
\end{equation*} \end{lemma}
One should not be surprised, that in contrast to lemma \ref{lem_the_shadow_flow} there appear $\frac{1}{\lambda_{r}^{n-2}}$ terms in $R_{k,i}$. Indeed, just like $\eps_{i,j}$ measures the interaction of the bubbles $\varphi_{i}$ and $\varphi_{j}$, the interaction of $u_{\alpha, \beta}$ and $\varphi_{i}$ is measured by $\frac{1}{\lambda_{i}^{\frac{n-2}{2}}}$.
\begin{proof}[\textbf{Proof of lemma \ref{lem_the_shadow_flow_w}}]\label{p_the_shadow_flow_w}$_{}$\\
Let 
\begin{equation}
(\dot{\xi}_{1,j}, \dot{\xi}_{2,j}, \dot{\xi}_{3,j})=(\dot{\alpha}_{j},-\alpha_{j}\frac{\dot \lambda_{j}}{\lambda_{j}}, \alpha_{j}\lambda_{j}\dot a _{j}).
\end{equation} 
Testing 
as indicated in the statement we get
\begin{equation}\begin{split}\label{sigma_ki_implizit}
\sigma_{k,i}
= &
\int  K  u^{\frac{4}{n-2}} \partial_{t}u\phi_{k,i} 
= 
\int Ku^{\frac{4}{n-2}}\partial_{t}(u_{\alpha, \beta} +\alpha^{j}\varphi_{j}+v)\phi_{k,i} \\
= &
\dot \alpha \int Ku^{\frac{4}{n-2}}\partial_{\alpha}u_{\alpha, \beta} \phi_{k,i}
+
\dot \beta^{m}\int Ku^{\frac{4}{n-2}}\partial_{\beta_{m}}u_{\alpha, \beta}\phi_{k,i}\\
& +
\dot{\xi}^{l,j}\int Ku^{\frac{4}{n-2}}\phi_{l,j}\phi_{k,i} 
-
\int Kv[\partial_{t}u^{\frac{4}{n-2}}\phi_{k,i}+u^{\frac{4}{n-2}}\partial_{t}\phi_{k,i}].
\end{split}\end{equation}
The first two integrals on the right hand side above may be estimated via
\begin{equation}
\begin{split}
\int u^{\frac{4}{n-2}}\var_{i}
= &
\int (u_{\alpha, \beta}+\alpha^{q}\var_{q})^{\frac{4}{n-2}}\var_{i}+O(\Vert v \Vert)\\
\leq &
C\int \var_{i}+\varphi_{i}^{\frac{n+2}{n-2}}+C\sum_{i\neq q=1}^{p}\int\var_{q}^{\frac{4}{n-2}}\var_{i}+O(\Vert v \Vert)\\
\leq &
C\sum_{i\neq q=1}^{p}\Vert \var_{q}^{\frac{4}{n-2}}\var_{i}\Vert_{L^{\frac{2n}{n+2}}}
+
O(\frac{1}{\lambda_{i}^{\frac{n-2}{2}}}+\Vert v \Vert) \\
= &
O(\frac{1}{\lambda_{i}^{\frac{n-2}{2}}}+\sum_{i\neq q=1}^{p}\eps_{i,q}+\Vert v \Vert)
\end{split}
\end{equation} 
where we made use of lemma \ref{lem_interactions}, yielding
\begin{equation}\label{dot_alpha,dot_beta_for_sigma_k_i}
\begin{split}
\dot \alpha \int Ku^{\frac{4}{n-2}}\partial_{\alpha}u_{\alpha, \beta} \phi_{k,i}
& +
\dot \beta^{m}\int Ku^{\frac{4}{n-2}}\partial_{\beta_{m}}u_{\alpha, \beta}\phi_{k,i} \\
= &
\begin{pmatrix}
O(\frac{1}{\lambda_{i}^{\frac{n-2}{2}}}+\sum_{i\neq q=1}^{p}\eps_{i,q}+\Vert v \Vert)_{k,i} \\
O(\frac{1}{\lambda_{i}^{\frac{n-2}{2}}}+\sum_{i\neq q=1}^{p}\eps_{i,q}+\Vert v \Vert)_{k,i,m}
\end{pmatrix}
\begin{pmatrix}
\frac{\dot \alpha}{\alpha} \\ \dot \beta^{m}
\end{pmatrix}
\end{split}
\end{equation} 
Turning to the third summand on the right hand side of \eqref{sigma_ki_implizit} note, that
\begin{equation}\label{phi_lj*phi_ki_estimate_w>0}
\begin{split}
\int Ku^{\frac{4}{n-2}}\phi_{l,j}\phi_{k,i}
=
\int K & (u_{\alpha, \beta}  +\alpha^{m}\varphi_{m})^{\frac{4}{n-2}}\phi_{l,j}\phi_{k,i}
+
O(\Vert v \Vert)
\end{split}
\end{equation} 
and
\begin{equation}
\begin{split}
\int K & (u_{\alpha, \beta}  +\alpha^{m}\varphi_{m})^{\frac{4}{n-2}}\phi_{l,j}\phi_{k,i} \\
= &
\int_{[\alpha^{m}\varphi_{m}\geq u_{\alpha, \beta}]}
K(\alpha^{m}\varphi_{m})^{\frac{4}{n-2}}\phi_{l,j}\phi_{k,i} \\
& +
O
(
\underset{[\alpha^{m}\varphi_{m}\geq u_{\alpha, \beta}]}{\int}
(\alpha^{m}\varphi_{m})^{\frac{6-n}{n-2}}u_{\alpha, \beta}\varphi_{j}\varphi_{i}
+
\underset{[\alpha^{m}\varphi_{m}< u_{\alpha, \beta}]}{\int}
u_{\alpha, \beta}^{\frac{4}{n-2}}\varphi_{j}\varphi_{i}
)
\\
= &
\int
K(\alpha^{m}\varphi_{m})^{\frac{4}{n-2}}\phi_{l,j}\phi_{k,i} \\
& +
O
(
\underset{[\alpha^{m}\varphi_{m}\geq u_{\alpha, \beta}]}{\int}
(\alpha^{m}\varphi_{m})^{\frac{6-n}{n-2}}u_{\alpha, \beta}\varphi_{j}\varphi_{i}
+
\underset{[\alpha^{m}\varphi_{m}< u_{\alpha, \beta}]}{\int}
u_{\alpha, \beta}^{\frac{4}{n-2}}\varphi_{j}\varphi_{i}
).
\end{split}
\end{equation} 
Using
\begin{equation}
\begin{split}
\int
\varphi_{j}\varphi_{i}
\leq 
C(\int \varphi_{i}+\int \varphi_{j}^{\frac{n+2}{n-2}}\varphi_{i})
=
O(\lambda_{i}^{\frac{n-2}{2}}+\eps_{i,j})
\end{split}
\end{equation} 
and 
\begin{equation}
\begin{split}
\underset{[\alpha^{m}\varphi_{m}\geq u_{\alpha, \beta}]}{\int} &
(\alpha^{m}\varphi_{m})^{\frac{6-n}{n-2}}u_{\alpha, \beta}\varphi_{j}\varphi_{i} \\
\leq  &
C\underset{[\alpha^{m}\varphi_{m}\geq u_{\alpha, \beta}]\cap [\varphi_{i}\geq \sum_{i\neq q=1}^{p}\varphi_{q}]}{\int} 
(\alpha^{m}\varphi_{m})^{\frac{4}{n-2}}u_{\alpha, \beta}\varphi_{i} \\
& +
C\underset{[\alpha^{m}\varphi_{m}\geq u_{\alpha, \beta}]\cap [\varphi_{i}<\sum_{i\neq q=1}^{p}\varphi_{q}]}{\int} 
(\alpha^{m}\varphi_{m})^{\frac{4}{n-2}}u_{\alpha, \beta}\varphi_{i} \\
\leq &
C
(
\int \varphi_{i}^{\frac{n+2}{n-2}}
+
\int (\sum_{i\neq q=1}^{p}\varphi_{q})^{\frac{n+2}{n-2}}\varphi_{i}
)
=
O(\lambda_{i}^{\frac{n-2}{2}}+\sum_{i\neq q=1}^{p}\eps_{i,q})
\end{split}
\end{equation}
we obtain
\begin{equation}
\begin{split}
\int K (u_{\alpha, \beta} & +\alpha^{m}\varphi_{m})^{\frac{4}{n-2}}\phi_{l,j}\phi_{k,i} \\
= &
\int K(\alpha^{m}\varphi_{m})^{\frac{4}{n-2}}\phi_{l,j}\phi_{k,i}
+
O(\lambda_{i}^{\frac{n-2}{2}}+\sum_{i\neq q=1}^{p}\eps_{i,j}) \\
= &
\alpha_{i}^{\frac{4}{n-2}}\int K\varphi_{i}^{\frac{4}{n-2}}\phi_{l,j}\phi_{k,i}
+
O(\lambda_{i}^{\frac{n-2}{2}}+\sum_{i\neq q=1}^{p}\eps_{i,j}),
\end{split}
\end{equation} 
where we made use of \eqref{alpha^m_delta_m_Phi_k,i_Phi_l,j}. Plugging this into \eqref{phi_lj*phi_ki_estimate_w>0} we obtain
\begin{equation}\label{phi_lj*phi_ki_estimate_w>0_better}
\begin{split}
\int & Ku^{\frac{4}{n-2}}\phi_{l,j}\phi_{k,i} \\
= &
\alpha_{i}^{\frac{4}{n-2}}\int K\varphi_{i}^{\frac{4}{n-2}}\phi_{l,j}\phi_{k,i}
+
O(\lambda_{i}^{\frac{n-2}{2}}+\sum_{i\neq q=1}^{p}\eps_{i,j}+\Vert v \Vert)\\
= &
c_{k}\alpha_{i}^{\frac{4}{n-2}}K_{i}\delta_{kl}\delta_{ij}
+
O(\frac{\vert \nabla K_{i}\vert}{\lambda_{i}})\delta_{ij}
+
O(\lambda_{i}^{\frac{n-2}{2}}+\sum_{i\neq q=1}^{p}\eps_{i,j}+\Vert v \Vert).
\end{split}
\end{equation} 
Moreover arguing as for \eqref{partial_t_phi_ki_estimate_w=0} and \eqref{trick_with_strong_convergence_w=0} we have
\begin{equation}\label{partial_t_phi_ki_estimate_w>0}
\int Ku^{\frac{4}{n-2}}\partial_{t}\phi_{k,i}v=O(\Vert v \Vert)_{i,k,l,j}\dot\xi^{l,j},
\end{equation} 
and 
\begin{equation}\label{trick_with_strong_convergence_w>0}
\int Kv\partial_{t}u\phi_{k,i}
=
O(\Vert v \Vert^{2}+\vert \delta J(u)\vert^{2}).
\end{equation} 
Thus plugging \eqref{dot_alpha,dot_beta_for_sigma_k_i}, \eqref{phi_lj*phi_ki_estimate_w>0_better}, \eqref{partial_t_phi_ki_estimate_w>0} and 
\eqref{trick_with_strong_convergence_w>0} into \eqref{sigma_ki_implizit} we conclude
\begin{equation}\begin{split}
\sigma_{k,i}
= &
\begin{pmatrix}
O(\frac{1}{\lambda_{i}^{\frac{n-2}{2}}}+\sum_{i\neq q=1}^{p}\eps_{i,q}+\Vert v\Vert)_{k,i}
\\
O(\frac{1}{\lambda_{i}^{\frac{n-2}{2}}}+\sum_{i\neq q=1}^{p}\eps_{i,q}+\Vert v\Vert)_{k,i,m}
\\
\Xi_{k,i,l,j}
\end{pmatrix}
^{T}
\begin{pmatrix}
\frac{\dot \alpha }{\alpha}
\\
\dot \beta^{m}
\\
\dot{\xi}^{l,j}
\end{pmatrix}
\\ &
+
O(\Vert v \Vert^{2} + \vert \delta J(u)\vert^{2})_{k,i}.
\end{split}\end{equation}
where
\begin{equation}\begin{split}
 \Xi_{k,i,l,j}
=
c_{k}\alpha_{i}^{\frac{4}{n-2}}K_{i}\delta_{kl}\delta_{ij}  
& +
O
(
\frac{\vert \nabla K_{i}\vert}{\lambda_{i}} 
)_{k,l}
\delta_{ij}\\
& +
O
(
\frac{1}{\lambda_{i}^{\frac{n-2}{2}}}
+
\sum_{i\neq q=1}^{p}\eps_{i,q}
+
\Vert v\Vert
)_{k,i,l,j}.
\end{split}\end{equation}
Next let
\begin{equation}
\begin{split}
\sigma=-\int (L_{g_{0}}u-r\K u^{\frac{n+2}{n-2}})u_{\alpha, \beta}.
\end{split}
\end{equation} 
We then have
\begin{equation}\begin{split}
\sigma
= &
\int K  u^{\frac{4}{n-2}}\partial_{t} u u_{\alpha, \beta} 
= 
\int Ku^{\frac{4}{n-2}} \partial_{t}(u_{\alpha, \beta} +\alpha^{i}\varphi_{i}+v)u_{\alpha, \beta} \\
= &
\frac{\dot \alpha}{\alpha} \int Ku^{\frac{4}{n-2}}u_{\alpha, \beta}^{2} 
+
\dot \beta^{m}\int Ku^{\frac{4}{n-2}}\partial_{\beta_{m}}u_{\alpha, \beta} u_{\alpha, \beta} \\ 
& +
\dot{\xi}^{l,j}\int Ku^{\frac{4}{n-2}}\phi_{l,j}u_{\alpha, \beta} 
-
\int Kv\partial_{t} u^{\frac{4}{n-2}}u_{\alpha, \beta}
\end{split}\end{equation}
and therefore recalling $\alpha \partial_{\alpha}u_{\alpha, \beta}=u_{\alpha, \beta}$
\begin{equation}\begin{split}
\sigma
= &
\begin{pmatrix}
\int Ku^{\frac{4}{n-2}}u_{\alpha, \beta}^{2}
\\
\int Ku^{\frac{4}{n-2}}\partial_{\beta_{m}}u_{\alpha, \beta} u_{\alpha, \beta} 
\\
O(\frac{1}{\lambda_{j}^{\frac{n-2}{2}}}+\sum_{j\neq q=1}^{p}\eps_{i,q}+\Vert v \Vert)_{l,j}
\end{pmatrix}^{T}
\begin{pmatrix}
\frac{\dot \alpha}{\alpha} 
\\
\dot \beta^{m}
\\
\dot{\xi}^{l,j} 
\end{pmatrix} \\
& +
O(\Vert v \Vert^{2}+\vert \delta J(u)\vert^{2}).
\end{split}\end{equation}
Likewise we obtain for
$
\sigma_{n}=-\int (L_{g_{0}}u-r\K u^{\frac{n+2}{n-2}})\partial_{\beta_{n}}u_{\alpha, \beta}
$
\begin{equation}\begin{split}
\sigma_{n}
= &
\int K  u^{\frac{4}{n-2}}\partial_{t} u \partial_{\beta_{n}}u_{\alpha, \beta}
= 
\int Ku^{\frac{4}{n-2}}\partial_{t}(u_{\alpha, \beta} +\alpha^{i}\varphi_{i}+v)\partial_{\beta_{n}}u_{\alpha, \beta}\\
= &
\dot \alpha \int Ku^{\frac{4}{n-2}}\partial_{\alpha}u_{\alpha, \beta} \partial_{\beta_{n}}u_{\alpha, \beta}
+
\dot \beta^{m}\int Ku^{\frac{4}{n-2}}\partial_{\beta_{m}}u_{\alpha, \beta}\partial_{\beta_{n}}u_{\alpha, \beta} \\
& +
\dot{\xi}^{l,j}\int Ku^{\frac{4}{n-2}}\phi_{l,j}\partial_{\beta_{n}}u_{\alpha, \beta} \\
& -
\int Kv
[
\partial_{t} u^{\frac{4}{n-2}}\partial_{\beta_{n}}u_{\alpha, \beta}
+
u^{\frac{4}{n-2}}\partial_{t}\partial_{\beta_{n}}u_{\alpha, \beta}
]\\
= &
\begin{pmatrix}
\int Ku^{\frac{4}{n-2}}u_{\alpha, \beta}\partial_{\beta_{n}}u_{\alpha, \beta}+O(\Vert v \Vert)
\\
\int Ku^{\frac{4}{n-2}}\partial_{\beta_{m}}u_{\alpha, \beta}\partial_{\beta_{n}}u_{\alpha, \beta}
+
O(\Vert v \Vert)
\\
O(\frac{1}{\lambda_{j}^{\frac{n-2}{2}}}+\sum_{j\neq l=q}^{p}\eps_{i,q}+\Vert v \Vert)_{n,l,j}
\end{pmatrix}^{T}
\begin{pmatrix}
\frac{\dot \alpha}{\alpha} 
\\
\dot \beta^{m}
\\
\dot{\xi}^{l,j} 
\end{pmatrix} \\
& +
O(\Vert v \Vert^{2}+\vert \delta J(u)\vert^{2})_{n}.
\end{split}\end{equation}
Summing up we conclude 
\begin{equation}
\begin{split}
\left(
A
+
R
\right)^{T}_{i,k,j,l,n,m}
\begin{pmatrix}
\frac{\dot \alpha}{\alpha} 
\\ 
\dot \beta^{m}
\\ 
\dot\xi^{l,j} 
\end{pmatrix} 
=
\begin{pmatrix}
\sigma \\ \sigma_{k,i} \\ \sigma_{n}
\end{pmatrix}
+
O(\Vert v \Vert^{2}+\vert \delta J(u)\vert^{2})_{k,i,n}
,
\end{split}
\end{equation}
where
\begin{equation}
\begin{split}
A_{i, \ldots,m}
=
\begin{pmatrix}
\langle u_{\alpha, \beta},u_{\alpha, \beta}\rangle &  \langle u_{\alpha, \beta}, \partial_{\beta_{m}}u_{\alpha, \beta} \rangle & 0\\
\langle u_{\alpha, \beta}, \partial_{\beta_{n}}u_{\alpha, \beta}\rangle  & \langle\partial_{\beta_{n}}u_{\alpha, \beta}, \partial_{\beta_{m}}u_{\alpha, \beta}\rangle & 0\\
0 & 0 & \tilde \Xi
\end{pmatrix} 
\end{split}
\end{equation} 
with
\begin{equation}
\tilde \Xi
=
a_{k}K_{i}\alpha_{i}^{\frac{4}{n-2}}\delta_{kl}\delta_{ij} 
+
O(\frac{\vert \nabla K_{i}\vert}{\lambda_{i}})_{k,l}\delta_{ij} 
\end{equation} 
and 
\begin{equation}
\begin{split}
R_{i, \ldots,m}=
O(\sum_{r} \frac{1}{\lambda_{r}^{\frac{n-2}{2}}} +\sum_{r\neq s}\eps_{r,s}+\Vert v\Vert)
_{i, \ldots,m}.
\end{split}
\end{equation} 
Using $\sigma, \sigma_{k,i}, \sigma_{n}=O(\vert \delta J(u)\vert)$ we obtain
\begin{equation}
\begin{split}
A
_{i,k,j,l,m,n}
\begin{pmatrix}
\dot \alpha \\ \dot\xi^{l,j} \\ \dot \beta^{m}
\end{pmatrix} 
=
\begin{pmatrix}
\sigma \\ \sigma_{k,i} \\ \sigma_{n}
\end{pmatrix}
+
R_{k,i,n}
\end{split}
\end{equation}
with 
\begin{equation*}\begin{split}
R_{k,i,n}
=
O
(
\sum_{r}
\frac{1}{\lambda_{r}^{n-2}}
+
\sum_{r\neq s}\eps_{r,s}^{2}
+
\Vert v \Vert^{2}
+
\vert \delta J(u)\vert^{2})_{k,i,n}.
\end{split}\end{equation*}
Note, that we may write $A=A_{i,k,j,l,n,m}$ as  
\begin{equation}
\begin{split}
A
=
\begin{pmatrix}
B & C & 0 \\
C & D & 0 \\
0 & 0 & E
\end{pmatrix}
=
\begin{pmatrix}
I & CD^{-1} & 0 \\
CB^{-1} & I & 0 \\
0 &  0 & I
\end{pmatrix}
\begin{pmatrix}
B & 0 & 0 \\
0 & D & 0 \\
0 & 0 & E
\end{pmatrix}
,
\end{split} 
\end{equation} 
whence we obtain via Neumann series 
\begin{equation}
\begin{split}
A^{-1}
=
\begin{pmatrix}
B^{-1} & 0 & 0 \\
0 & D^{-1} & 0 \\
0 & 0 & E^{-1}
\end{pmatrix}
{\sum
_{k=0}^{\infty}}
(-1)^{k}
\begin{pmatrix}
0 & CD^{-1} & 0 \\
CB^{-1} & 0 & 0 \\
0 & 0 & 0
\end{pmatrix}^{k}
\end{split}
.
\end{equation} 
Last note, that the third row of $A^{-1}$ is just $E^{-1}$, where
$E=\tilde \Xi$.
\end{proof}
As before our task is two folded, namely to analyse $\sigma_{k,i}$ and to provide a suitable estimate on $v$. 
\begin{proposition}[Analysing $\sigma_{k,i}$]\label{prop_analysing_ski_f}$_{}$\\
On $V(\omega, p, \eps)$ for $\eps>0$ small we have with constants $b_{1}, \ldots,d_{3}>0$
\begin{enumerate}[label=(\roman*)]
 \Item $_{}$ 
\begin{equation*}\begin{split}
\sigma_{1,i}
= &
4n(n-1)\alpha_{i}
[\frac{r\alpha_{i}^{\frac{4}{n-2}}K_{i}}{4n(n-1)k}-1]
\int
\varphi_{i}^{\frac{2n}{n-2}}
 \\
& +
4n(n-1)b_{1}\sum_{i\neq j =1}^{p}\alpha_{j}[\frac{r\alpha_{j}^{\frac{4}{n-2}}K_{j}}{4n(n-1)k}-1]
\eps_{i,j} \\
& -
\int (L_{g_{0}}u_{\alpha, \beta}-r\K u_{\alpha, \beta}^{\frac{n+2}{n-2}})\varphi_{i}
 \\
& +
b_{1}\frac{r\alpha_{i}^{\frac{4}{n-2}}K_{i}}{k}\sum_{i\neq j =1}^{p}\alpha_{j}
\eps_{i,j}
+
d_{1}\frac{r\alpha_{i}^{\frac{4}{n-2}}}{k}
\frac{\alpha K_{i}\omega _{i}}{\lambda_{i} ^{\frac{n-2}{2}}}
+
R_{1,i}
\end{split}\end{equation*}
 \Item $_{}$ 
\begin{equation*}\begin{split}
\sigma_{2,i}
= &
-4n(n-1)\alpha_{i}
[\frac{r\alpha_{i}^{\frac{4}{n-2}}K_{i}}{4n(n-1)k}-1]
\int
\varphi_{i}^{\frac{n+2}{n-2}}
\lambda_{i}\partial_{\lambda_{i}}\varphi_{i}
\\
& -
4n(n-1)b_{2}\sum_{i\neq j =1}^{p}\alpha_{j}[\frac{r\alpha_{j}^{\frac{4}{n-2}}K_{j}}{4n(n-1)k}-1]
\lambda_{i}\partial_{\lambda_{i}}\eps_{i,j} \\
& -
\int (L_{g_{0}}u_{\alpha, \beta}-r\K u_{\alpha, \beta}^{\frac{n+2}{n-2}})
\lambda_{i}\partial_{\lambda_{i}}\varphi_{i}
 \\
& -
b_{2}\frac{r\alpha_{i}^{\frac{4}{n-2}}K_{i}}{k}\sum_{i\neq j =1}^{p}\alpha_{j}
\lambda_{i}\partial_{\lambda_{i}}\eps_{i,j}
+
d_{2}\frac{r\alpha_{i}^{\frac{4}{n-2}}K_{i}}{k}
\frac{\alpha \omega_{i}}{\lambda_{i} ^{\frac{n-2}{2}}}
+
R_{2,i}
\end{split}\end{equation*}
 \Item $_{}$ 
\begin{equation*}\begin{split}
\sigma_{3,i}
= &
4n(n-1)\alpha_{i}
[\frac{r\alpha_{i}^{\frac{4}{n-2}}K_{i}}{4n(n-1)k}-1]
\int
\varphi_{i}^{\frac{n+2}{n-2}}
\frac{1}{\lambda_{i}}\nabla_{a_{i}}\varphi_{i}
\\
& +
4n(n-1)b_{3}\sum_{i\neq j =1}^{p}\alpha_{j}[\frac{r\alpha_{j}^{\frac{4}{n-2}}K_{j}}{4n(n-1)k}-1]
\frac{1}{\lambda_{i}}\nabla_{a_{i}}\eps_{i,j} \\
& -
\int (L_{g_{0}}u_{\alpha, \beta}-r\K u_{\alpha, \beta}^{\frac{n+2}{n-2}})
\frac{1}{\lambda_{i}}\nabla_{a_{i}}\varphi_{i}
 \\
& +
b_{3}\frac{r\alpha_{i}^{\frac{4}{n-2}}K_{i}}{k}\sum_{i\neq j =1}^{p}\alpha_{j}
\frac{1}{\lambda_{i}}\nabla_{a_{i}}\eps_{i,j}
+
d_{3}\frac{r\alpha_{i}^{\frac{n+2}{n-2}}}{k}
\frac{\nabla K_{i}}{\lambda_{i}}
+
R_{3,i},
\end{split}\end{equation*}
\end{enumerate}
where
$
R_{k,i}
= 
o_{\varepsilon}
(
\frac{1}{\lambda_{i}^{\frac{n-2}{2}}}+\sum_{i\neq j=1}^{p}\eps_{i,j}
)
+
O
(
\sum_{r\neq s}\eps_{r,s}^{2}
+
\Vert v \Vert^{2}
+ 
\vert \delta J(u)\vert^{2}
).
$
\end{proposition}
\noindent
Here and in what follows $\omega_{i}$ is short hand for $\omega(a_{i})$
analogously to $K_{i}=K(a_{i})$.
\begin{proof}[\textbf{Proof of proposition \ref{prop_analysing_ski_f}}]\label{p_analysing_ski_f}$_{}$\\
We evaluate by means of lemma \ref{lem_v_type_interactions}
\begin{equation}\begin{split}\label{L_{g_{0}}u-rK_f}
\int (L_{g_{0}} & u  -r\K u^{\frac{n+2}{n-2}})\phi_{k,i} \\
= &
 \int L_{g_{0}}u_{\alpha, \beta}\phi_{k,i}
+
\alpha^{j}\int L_{g_{0}}\varphi_{j}\phi_{k,i}
-
\int r\K (u_{\alpha, \beta}+ \alpha^{j}\varphi_{j})^{\frac{n+2}{n-2}}\phi_{k,i}\\
& +
o( \frac{1}{\lambda_{i} ^{\frac{n-2}{2}}}+ \sum_{i\neq j=1}^{p}\eps_{i,j})
+
O
(\Vert v \Vert^{2}+\vert \delta J(u)\vert^{2}).
\end{split}\end{equation}
From \eqref{IntL_{g_{0}}djPhiki_expanded} we infer
\begin{equation}\begin{split}\label{IntL_{g_{0}}djPhiki_expanded_f}
\int L_{g_{0}}u_{\alpha, \beta}\phi_{k,i} 
& 
+
\alpha^{j}\int L_{g_{0}}  \varphi_{j}\phi_{k,i}
\\
= &
4n(n-1)
[
\alpha_{i}
\int
\varphi_{i}^{\frac{n+2}{n-2}}
\phi_{k,i}
+
b_{k}\sum_{i\neq j =1}^{p}\alpha_{j}d_{k,i}\eps_{i,j}
] \\
& 
+
\int L_{g_{0}}u_{\alpha, \beta}\phi_{k,i}
+
o_{\varepsilon}
(
\frac{1}{\lambda_{i} ^{\frac{n-2}{2}}}
+
\sum_{i\neq j=1}^{p}\eps_{i,j}
),
\end{split}\end{equation}
where $(d_{1,i},d_{2,i},d_{3,i})=(1,-\lambda_{i}\partial_{\lambda_{i}}, \frac{1}{\lambda_{i}}\nabla_{a_{i}})$.
On the other hand we may expand
\begin{equation}\begin{split}
\int & K  ( u_{\alpha, \beta}  +\alpha^{j}\varphi_{j})^{\frac{n+2}{n-2}}\phi_{k,i}  
= 
\underset{[u_{\alpha, \beta}\geq \alpha^{j}\varphi_{j}]}{\int} K(  u_{\alpha, \beta}  +\alpha^{j}\varphi_{j})^{\frac{n+2}{n-2}}\phi_{k,i} \\
& \quad\quad\quad\quad\quad\quad\quad\quad\quad\quad\quad\quad  +
\underset{[u_{\alpha, \beta}<\alpha^{j}\varphi_{j}]}{\int} K(  u_{\alpha, \beta}  +\alpha^{j}\varphi_{j})^{\frac{n+2}{n-2}}\phi_{k,i} \\
= &
\underset{[u_{\alpha, \beta}\geq \alpha^{j}\varphi_{j}]}{\int} K (u_{\alpha, \beta})^{\frac{n+2}{n-2}} \phi_{k,i}
+
\underset{[u_{\alpha, \beta}<\alpha^{j}\varphi_{j}]}{\int} K(\alpha^{j}\varphi_{j})^{\frac{n+2}{n-2}}\phi_{k,i} \\
& +
\frac{n+2}{n-2}\int_{[u_{\alpha, \beta}<\alpha^{j}\varphi_{j}]} K( \alpha^{j}\varphi_{j})^{\frac{4}{n-2}}u_{\alpha, \beta}\phi_{k,i}\\
& 
+
O(\underset{[u_{\alpha, \beta}\geq \alpha^{j}\varphi_{j}]}{\int}  u_{\alpha, \beta}^{\frac{4}{n-2}} \alpha^{j}\varphi_{j}\varphi_{i})
+
\underset{[u_{\alpha, \beta}<\alpha^{j}\varphi_{j}]}{\int} (\alpha^{j}\varphi_{j})^{\frac{6-n}{n-2}}u_{\alpha, \beta}^{2}\varphi_{i}). 
\end{split}
\end{equation}
This gives
\begin{equation}\begin{split}\label{expansion_K(omega+deltai)^(...)}
\int  K  & (  u_{\alpha, \beta}  +\alpha^{j}\varphi_{j})^{\frac{n+2}{n-2}}\phi_{k,i} \\
= &
\int K (u_{\alpha, \beta})^{\frac{n+2}{n-2}} \phi_{k,i}
+
\int K(\alpha^{j}\varphi_{j})^{\frac{n+2}{n-2}}\phi_{k,i} \\
& +
\frac{n+2}{n-2}\int K( \alpha^{j}\varphi_{j})^{\frac{4}{n-2}}u_{\alpha, \beta}\phi_{k,i}\\
& 
+
O(\underset{[u_{\alpha, \beta}\geq \alpha^{j}\varphi_{j}]}{\int}  u_{\alpha, \beta}^{\frac{4}{n-2}} \alpha^{j}\varphi_{j}\varphi_{i})
+
\underset{[u_{\alpha, \beta}<\alpha^{j}\varphi_{j}]}{\int} (\alpha^{j}\varphi_{j})^{\frac{6-n}{n-2}}u_{\alpha, \beta}^{2}\varphi_{i})
.
\end{split}\end{equation}
Note, that 
$
\underset{[ u_{\alpha, \beta}\geq c_{0}\varphi_{i}]}{\int} \varphi_{i} ^{2}
=
o(\frac{1}{\lambda_{i}^{\frac{n-2}{2}}})
$
and for suitable $\epsilon>0$ we have 
\begin{equation}
\begin{split}
\underset{[u_{\alpha, \beta}<\alpha^{j}\varphi_{j}]}{\int} &(\alpha^{j}\varphi_{j})^{\frac{6-n}{n-2}}(u_{\alpha, \beta})^{2}\varphi_{i} \\
= &
\underset{[u_{\alpha, \beta}<\alpha^{j}\varphi_{j}]\cap [\varphi_{i}\geq \sum_{i\neq j=1}^{p}\varphi_{j}]}{\int} (\alpha^{j}\varphi_{j}^{\frac{6-n}{n-2}}(u_{\alpha, \beta})^{2}\varphi_{i})\\
& +
\underset{[u_{\alpha, \beta}<\alpha^{j}\varphi_{j}]\cap [\varphi_{i}<\sum_{i\neq j=1}^{p}\varphi_{j}]}{\int} (\alpha^{j}\varphi_{j})^{\frac{6-n}{n-2}}(u_{\alpha, \beta})^{2}\varphi_{i}, 
\end{split} 
\end{equation}
whence 
\begin{equation}\label{deltai_f_interaction_estimate_2}
\begin{split}
\underset{[u_{\alpha, \beta}<\alpha^{j}\varphi_{j}]}{\int} &(\alpha^{j}\varphi_{j})^{\frac{6-n}{n-2}}(u_{\alpha, \beta})^{2}\varphi_{i} \\
\leq &
C\int_{B_{\frac{\epsilon}{\sqrt{\lambda_{i}}}}{a_{i}}} \varphi_{i}^{\frac{4}{n-2}}
+
C
\int_{\cup_{i\neq j=1}^{p}B_{\frac{\epsilon}{\sqrt{\lambda_{j}}}}(a_{j})}
(\sum_{i\neq j=1}^{p}\varphi_{j})^{\frac{n+2}{n-2}-\epsilon}\varphi_{i} \\
\leq & 
o(\frac{1}{\lambda_{i}^{\frac{n-2}{2}}})
+
\vert[\cup_{i\neq j=1}^{p}B_{{\epsilon}{\sqrt{\lambda_{j}}}}(0)]\vert^{\frac{\epsilon(n-2)}{2n}}
\sum_{i\neq j=1}^{p} \eps_{i,j}.
\end{split} 
\end{equation}
Plugging thus \eqref{deltai_f_interaction_estimate_2} 
into \eqref{expansion_K(omega+deltai)^(...)} we get
\begin{equation}\begin{split}\label{expansion_K(omega+deltai)^(...)_refined}
\int  K  (  u_{\alpha, \beta} & +\alpha^{j}\varphi_{j})^{\frac{n+2}{n-2}}\phi_{k,i} \\
= &
\int K (u_{\alpha, \beta})^{\frac{n+2}{n-2}} \phi_{k,i}
+
\int K(\alpha^{j}\varphi_{j})^{\frac{n+2}{n-2}}\phi_{k,i} \\
& +
\frac{n+2}{n-2}\int K( \alpha_{i}\varphi_{i})^{\frac{4}{n-2}}u_{\alpha, \beta}\phi_{k,i}
+
o(\frac{1}{\lambda_{i}^{\frac{n-2}{2}}} +\sum_{i\neq j=1}^{p}\eps_{i,j}).
\end{split}\end{equation}
Then \eqref{IntK(ajdj)^{...}Phiki} shows
\begin{equation}\begin{split}
\int  K  & (  u_{\alpha, \beta}  +\alpha^{j}\varphi_{j}  )^{\frac{n+2}{n-2}}\phi_{k,i} \\
= &
\alpha_{i}^{\frac{n+2}{n-2}}K_{i}\int \varphi_{i}^{\frac{n+2}{n-2}}\phi_{k,i}
+
\int Ku_{\alpha, \beta}^{\frac{n+2}{n-2}}\phi_{k,i}
+
\sum_{i\neq j =1}^{p}\alpha_{j}^{\frac{n+2}{n-2}}K_{j}\int \varphi_{j}^{\frac{n+2}{n-2}}\phi_{k,i} \\
& +
\frac{n+2}{n-2}
\alpha_{i}^{\frac{4}{n-2}}K_{i}
\sum_{i \neq j=1}^{p}\alpha_{j}
\int \varphi_{i}^{\frac{4}{n-2}}\phi_{k,i}\varphi_{j} \\
& +
\alpha_{i}^{\frac{n+2}{n-2}}(e_{1}  \frac{\lap K_{i}}{\lambda_{i} ^{2}},e_{2} \frac{\lap K_{i}}{\lambda_{i}^{2}},
e_{3}\frac{\nabla K_{i}}{\lambda_{i}}+e_{4}\frac{\nabla \lap K_{i}}{\lambda_{i}^{3}}  )
\\ & +
\frac{n+2}{n-2}\alpha_{i}^{\frac{4}{n-2}}\int K \varphi_{i}^{\frac{4}{n-2}}u_{\alpha, \beta}\phi_{k,i} \\
& + 
o_{\varepsilon}(\frac{1}{\lambda_{i}^{\frac{n-2}{2}}} +\sum_{i\neq j=1}^{p}\eps_{i,j})
+
O(\sum_{r\neq s}\eps_{r,s}^{2})
\end{split}\end{equation}
and we obtain letting 
$(d_{1,i},d_{2,i},d_{3,i})=(1,-\lambda_{i}\partial_{\lambda_{i}}, \frac{1}{\lambda_{i}}\nabla_{a_{i}})$
\begin{equation}\begin{split}
\int  K & (  u_{\alpha, \beta}   +\alpha^{j}\varphi_{j}  )^{\frac{n+2}{n-2}}\phi_{k,i} \\
= &
\alpha_{i}^{\frac{n+2}{n-2}}K_{i}\int \varphi_{i}^{\frac{n+2}{n-2}}\phi_{k,i}
+
\int K u_{\alpha, \beta}^{\frac{n+2}{n-2}}\phi_{k,i}
+
\sum_{i\neq j =1}^{p}\alpha_{j}^{\frac{n+2}{n-2}}K_{j}b_{k}d_{k,i}\eps_{i,j} \\
& +
\sum_{i \neq j=1}^{p}\alpha_{i}^{\frac{4}{n-2}}\alpha_{j}K_{i}b_{k}d_{k,i}\eps_{i,j}
+
\alpha_{i}^{\frac{n+2}{n-2}}
(0,0,
e_{3}\frac{\nabla K_{i}}{\lambda_{i}}
) 
\\
& +
\alpha_{i}^{\frac{4}{n-2}}\int K d_{k,i}\varphi_{i}^{\frac{n+2}{n-2}}u_{\alpha, \beta} +
o_{\varepsilon}(\frac{1}{\lambda_{i}^{\frac{n-2}{2}}} +\sum_{i\neq j=1}^{p}\eps_{i,j})
+
O(\sum_{r\neq s}\eps_{r,s}^{2}).
\end{split}\end{equation}
Since $u_{\alpha, \beta}(a_{i})=\alpha \omega(a_{i})+o_{\varepsilon}(1)$, we get in cases $k=1,2$ with $d_{k}>0$
\begin{equation}\begin{split}
\int K d_{k,i}\varphi_{i}^{\frac{n+2}{n-2}}u_{\alpha, \beta}
=
d_{k} \frac{\alpha K_{i}\omega _{i}}{\lambda_{i} ^{\frac{n-2}{2}}}
+
o_{\varepsilon}(  \frac{1}{\lambda_{i} ^{\frac{n-2}{2}}}),
\end{split}\end{equation}
and in case $k=3$ by radial symmetry 
\begin{equation}
\int K\omega  \frac{1}{\lambda_{i}}\nabla_{a_{i}}\varphi_{i}^{\frac{n+2}{n-2}}
=
o(\frac{1}{\lambda_{i}^{\frac{n-2}{2}}}).
\end{equation} 
We get
\begin{equation}\begin{split}\label{intK(...)_omega_evaluated}
\int & K  ( u_{\alpha, \beta}  +\alpha^{j}\varphi_{j}  )^{\frac{n+2}{n-2}}\phi_{k,i} \\
= &
\alpha_{i}^{\frac{n+2}{n-2}}K_{i}\int \varphi_{i}^{\frac{n+2}{n-2}}\phi_{k,i}
+
\int K u_{\alpha, \beta}^{\frac{n+2}{n-2}}\phi_{k,i}
+
\sum_{i\neq j =1}^{p}\alpha_{j}^{\frac{n+2}{n-2}}K_{j}b_{k}d_{k,i}\eps_{i,j} \\
& +
\alpha_{i}^{\frac{4}{n-2}}
(d_{1} \frac{\alpha K_{i}\omega _{i}}{\lambda_{i} ^{\frac{n-2}{2}}}, d_{2}\frac{\alpha K_{i}\omega _{i}}{\lambda_{i} ^{\frac{n-2}{2}}},d_{3}\frac{\alpha_{i}\nabla K_{i}}{\lambda_{i}} ) \\
& +
\sum_{i \neq j=1}^{p}\alpha_{i}^{\frac{4}{n-2}}\alpha_{j}K_{i}b_{k}d_{k,i}\eps_{i,j} 
+
o_{\varepsilon}(\frac{1}{\lambda_{i}^{\frac{n-2}{2}}} +\sum_{i\neq j=1}^{p}\eps_{i,j})
+
O(\sum_{r\neq s}\eps_{r,s}^{2}).
\end{split}\end{equation} 
Plugging \eqref{IntL_{g_{0}}djPhiki_expanded_f} and  \eqref{intK(...)_omega_evaluated} into \eqref{L_{g_{0}}u-rK_f} yields
\begin{equation}\begin{split}
\int ( & L_{g_{0}}u   -  r\K  u^{\frac{n+2}{n-2}}) \phi_{k,i} \\
= &
4n(n-1)
[
\alpha_{i}
\int
\varphi_{i}^{\frac{n+2}{n-2}}
\phi_{k,i}
+
b_{k}\sum_{i\neq j =1}^{p}\alpha_{j}d_{k,i}\eps_{i,j}
] +
\int L_{g_{0}}u_{\alpha, \beta}\phi_{k,i}\\
& 
-
\frac{r\alpha_{i}^{\frac{n+2}{n-2}}K_{i}}{k}\int \varphi_{i}^{\frac{n+2}{n-2}}\phi_{k,i}
-
\frac{r}{k}\int K u_{\alpha, \beta}^{\frac{n+2}{n-2}}\phi_{k,i}
 \\
& -
b_{k}\sum_{i\neq j =1}^{p}\frac{r\alpha_{j}^{\frac{n+2}{n-2}}K_{j}}{k}d_{k,i}\eps_{i,j}
-
b_{k}\sum_{i \neq j=1}^{p}\frac{r\alpha_{i}^{\frac{4}{n-2}}\alpha_{j}K_{i}}{k}d_{k,i}\eps_{i,j}
\\ & -
\frac{r\alpha_{i}^{\frac{4}{n-2}}}{k}
(d_{1}\frac{\alpha K_{i}\omega _{i}}{\lambda_{i} ^{\frac{n-2}{2}}},d_{2}\frac{\alpha K_{i}\omega _{i}}{\lambda_{i} ^{\frac{n-2}{2}}},d_{3}\frac{\alpha_{i}\nabla K_{i}}{\lambda_{i}} )\\
& +
o_{\varepsilon}(\frac{1}{\lambda_{i}^{\frac{n-2}{2}}} +\sum_{i\neq j=1}^{p}\eps_{i,j})
+
O
(
\sum_{r\neq s}\eps_{r,s}^{2}+\Vert v \Vert^{2}+\vert \delta J(u)\vert^{2}
).
\end{split}\end{equation}
From this the assertion follows.
\end{proof}
The equation on $\sigma_{1,i}=O(\vert \delta J(u)\vert)$ and the fact, that $u_{\alpha, \beta}$ is almost a solution,
simplify the equations on $\sigma_{2,i}$ and $\sigma_{3,i}$ significantly.

\begin{corollary}[Simplifying $\sigma_{k,i}$]\label{cor_simplifying_ski_f}$_{}$\\
On $V(\omega, p, \eps)$ for $\eps>0$ small  we have 
\begin{enumerate}[label=(\roman*)]
 \Item 
\begin{equation*}\begin{split}
\sigma_{2,i}
= &
d_{2}\frac{r\alpha_{i}^{\frac{4}{n-2}}}{k}\frac{\alpha \omega_{i}}{\lambda_{i} ^{\frac{n-2}{2}}}
-b_{2} \frac{r\alpha_{i}^{\frac{4}{n-2}}K_{i}}{k}
\sum_{i\neq j =1}^{p}\alpha_{j}
\lambda_{i}\partial_{\lambda_{i}}\eps_{i,j}
+
R_{2,i},
\end{split}\end{equation*}
 \Item
\begin{equation*}\begin{split}
\sigma_{3,i}
= &
d_{3}\frac{r\alpha_{i}^{\frac{n+2}{n-2}}}{k}\frac{\nabla K_{i}}{\lambda_{i}}
+ b_{3}\frac{r\alpha_{i}^{\frac{4}{n-2}}K_{i}}{k}
\sum_{i\neq j =1}^{p}\alpha_{j}
\frac{1}{\lambda_{i}}\nabla_{a_{i}}\eps_{i,j}
+
R_{2,i},
\end{split}\end{equation*}
\end{enumerate}
where 
$
R_{k,i}
= 
o_{\varepsilon}
(
\frac{1}{\lambda_{i}^{\frac{n-2}{2}}}+\sum_{i\neq j=1}^{p}\eps_{i,j}
)
+
O
(
\sum_{r\neq s}\eps_{r,s}^{2}
+
\Vert v \Vert^{2}
+ 
\vert \delta J(u)\vert^{2}
).
$
\end{corollary}

\begin{proof}[\textbf{Proof of corollary \ref{cor_simplifying_ski_f}}]\label{p_simplifying_ski_f}$_{}$\\
Note, that
\begin{equation}\label{u_alpha,beta_tested_with_phi_k,i}
\begin{split}
\int (L_{g_{0}}u_{\alpha, \beta}-r\K u_{\alpha, \beta}^{\frac{n+2}{n-2}})\phi_{k,i}
= &
\int (L_{g_{0}}u_{\alpha, \beta}-\frac{r_{u_{\alpha, \beta}}}{k_{u_{\alpha, \beta}}}K u_{\alpha, \beta}^{\frac{n+2}{n-2}})\phi_{k,i} \\
& +
((\frac{r}{k})_{u_{\alpha, \beta}}-(\frac{r}{k})_{u})
\int K u_{\alpha, \beta}^{\frac{n+2}{n-2}}\phi_{k,i}
\end{split}
\end{equation} 
Due to $\Pi \nabla J(u_{\alpha, \beta})=0$, cf. lemma \ref{lem_degeneracy_and_pseudo_critical_points} and the remarks following, we have
\begin{equation}\label{equation_solved_again}
\begin{split}
L_{g_{0}}u_{\alpha, \beta}-(r\K)_{u_{\alpha, \beta}}u_{\alpha, \beta}^{\frac{n+2}{n-2}}
= &
[
\int 
(
L_{g_{0}}u_{\alpha, \beta}-(r\K)_{u_{\alpha, \beta}}u_{\alpha, \beta}^{\frac{n+2}{n-2}}
)
\frac{\omega}{\Vert \omega \Vert} 
]
L_{g_{0}}\frac{\omega}{\Vert \omega \Vert}  \\
& +
\sum_{i=1}^{m}
[ \int 
(L_{g_{0}}u_{\alpha, \beta}-(r\K)_{u_{\alpha, \beta}}u_{\alpha, \beta}^{\frac{n+2}{n-2}}) \mathrm{e}_{i}]
L_{g_{0}}\mathrm{e}_{i}
\end{split}
\end{equation} 
and there holds
\begin{equation}
\begin{split}
\int (L_{g_{0}}u_{\alpha, \beta} & -(r\K)_{u_{\alpha, \beta}}u_{\alpha, \beta}^{\frac{n+2}{n-2}})\omega \\
= &
\int (L_{g_{0}}u_{\alpha, \beta}-(r\K)_{u}u_{\alpha, \beta}^{\frac{n+2}{n-2}})\omega
+
O(\vert (\frac{r}{k})_{u_{\alpha, \beta}}-(\frac{r}{k})_{u}\vert) \\
= &
\int (L_{g_{0}}(u-\alpha^{i}\var_{i})-(r\K)_{u}(u-\alpha^{i}\var_{i})^{\frac{n+2}{n-2}})\omega \\
& +
O(\vert (\frac{r}{k})_{u_{\alpha, \beta}}-(\frac{r}{k})_{u}\vert+\Vert v \Vert).
\end{split}
\end{equation} 
Clearly
\begin{equation}
\int L_{g_{0}}\varphi_{i}\omega=O(\frac{1}{\lambda_{i}^{\frac{n-2}{2}}}) 
\end{equation} 
and we have
\begin{equation}
\begin{split}
\int K  (u & -\alpha^{i} \delta_{i})^{\frac{n+2}{n-2}}\omega \\
= &
\underset{[u>\alpha^{i}\delta_{i}]}{\int}K(u-\alpha^{i}\delta_{i})^{\frac{n+2}{n-2}}\omega
+
\underset{[u<\alpha^{i}\delta_{i}]}{\int}K(u-\alpha^{i}\delta_{i})^{\frac{n+2}{n-2}}\omega \\
= &
\underset{[u>\alpha^{i}\delta_{i}]}{\int}Ku^{\frac{n+2}{n-2}}\omega
+
O(\sum_{i}\frac{\ln^{\frac{n-2}{n}}\lambda_{i}}{\lambda_{i}^{\frac{n-2}{2}}})\\
= &
\int Ku^{\frac{n+2}{n-2}}\omega
+
O(\sum_{i}\frac{\ln^{\frac{n-2}{n}}\lambda_{i}}{\lambda_{i}^{\frac{n-2}{2}}}).
\end{split}
\end{equation} 
We obtain 
\begin{equation}\label{partial_J_u_alpha,beta_tested}
\begin{split}
\int (L_{g_{0}}u_{\alpha, \beta} & -(r\K)_{u_{\alpha, \beta}}u_{\alpha, \beta}^{\frac{n+2}{n-2}})\omega \\
= &
\int (L_{g_{0}}u-(r\K)_{u}u^{\frac{n+2}{n-2}})\omega \\
& +
O(\vert (\frac{r}{k})_{u_{\alpha, \beta}}-(\frac{r}{k})_{u}\vert+\sum_{i}\frac{\ln^{\frac{n-2}{n}}\lambda_{i}}{\lambda_{i}^{\frac{n-2}{2}}}+\Vert v \Vert)\\
= &
O(\vert (\frac{r}{k})_{u_{\alpha, \beta}}-(\frac{r}{k})_{u}\vert+\sum_{i}\frac{\ln^{\frac{n-2}{n}}\lambda_{i}}{\lambda_{i}^{\frac{n-2}{2}}}+\Vert v \Vert+\vert \delta J(u)\vert).
\end{split}
\end{equation}
and the same estimate holds for $\omega$ replaced by $\mathrm{e_{i}}$. Plugging this into 
\eqref{equation_solved_again} we obtain for \eqref{u_alpha,beta_tested_with_phi_k,i} the estimate
\begin{equation}
\begin{split}
\int (L_{g_{0}} & u_{\alpha, \beta}  -(r\K)u_{\alpha, \beta}^{\frac{n+2}{n-2}})\phi_{k,i} \\
= &
O
(
(\vert (\frac{r}{k})_{u_{\alpha, \beta}}-(\frac{r}{k})_{u}\vert+\sum_{r}\frac{\ln^{\frac{n-2}{n}}\lambda_{r}}{\lambda_{r}^{\frac{n-2}{2}}}+\Vert v \Vert+\vert \delta J(u)\vert
)
\frac{1}{\lambda_{i}^{\frac{n-2}{2}}}
),
\end{split}
\end{equation}
whence using \eqref{r/kuab-r/k_rough_estimate} we conclude
\begin{equation}
\begin{split}
\int (L_{g_{0}} u_{\alpha, \beta} & - (r\K)u_{\alpha, \beta}^{\frac{n+2}{n-2}})\phi_{k,i} \\
= &
o(\frac{1}{\lambda_{i}^{\frac{n-2}{2}}})
+
O
(
\sum_{r}\frac{1}{\lambda_{r}^{n-2}}+\Vert v \Vert^{2}+\vert \delta J(u)\vert^{2}
).
\end{split}
\end{equation}
Consequently equation (i) of proposition \ref{prop_analysing_ski_f} shows 
\begin{equation}\label{rai4/n-2_f}
\begin{split}
\frac{r\alpha_{i}^{\frac{4}{n-2}}K_{i}}{4n(n-1)k}
= 
1
& +
O(\frac{1}{\lambda_{i}^{\frac{n-2}{2}}}+\sum_{i\neq j=1}^{p}\eps_{i,j}
+
\sum_{r\neq s}\eps_{r,s}^{2}+\Vert v \Vert^{2}+\vert \delta J(u)\vert).
\end{split}
\end{equation} 
Thus the claim follows from proposition \ref{prop_analysing_ski_f}.
\end{proof}
We turn to estimate the error term  term $v$. To do so we first characterize the first two derivatives of 
$J$ at $u_{\alpha, \beta}+\alpha^{i}\varphi_{i}=u-v$.

\begin{proposition}[Derivatives on $H(\omega, p, \eps)$]\label{prop_derivatives_on_H_f}$_{}$\\
For $\eps>0$ small let $u=u_{\alpha, \beta}+\alpha^{i}\var_{i}+v\in V(p, \eps)$  and 
$h_{1},h_{2}\in H= H_{u}(\omega, p, \eps)$.

We then have
\begin{enumerate}[label=(\roman*)]
 \Item 
 \begin{equation*}\begin{split}
\Vert \partial J( & u_{\alpha, \beta}  + \alpha^{i}\varphi_{i} )\lfloor_{H}\Vert \\
= &
o_{\varepsilon}(\Vert v \Vert) 
+
O(\sum_{r} \frac{\vert \nabla K_{r}\vert}{\lambda_{r}} 
+
\frac{1}{\lambda_{r}^{\frac{n-2}{2}}} +\sum_{r\neq s}\eps_{r,s}
+
\vert \delta J(u)\vert)
\end{split}\end{equation*}
 \Item 
 \begin{equation*}\begin{split}
\frac{1}{2}\partial^{2} & J(  u_{\alpha, \beta}  + \alpha^{i}\varphi_{i} )h_{1}h_{2}  \\
= &
k_{u_{\alpha, \beta}+\alpha^{i}\varphi}^{\frac{2-n}{n}}
[
\int L_{g_{0}}h_{1}h_{2}
-
c_{n}n(n+2)
\int 
(\frac{K\omega  ^{\frac{4}{n-2}}}{4n(n-1)}
+
\sum_{i}\varphi_{i}^{\frac{4}{n-2}}
)
h_{1}h_{2}
] \\
& +
o_{\varepsilon}(\Vert h_{1} \Vert \Vert h_{2} \Vert).
\end{split}\end{equation*}
\end{enumerate}
\end{proposition}

\begin{proof}[\textbf{Proof of proposition \ref{prop_derivatives_on_H_f}}]\label{p_derivatives_on_H_f}$_{}$\\ 
Let in addition $h\in H_{u}(\omega,p, \eps)$ with $\Vert h \Vert=1$. From proposition \ref{prop_derivatives_of_J} we infer
\begin{equation}\begin{split}
\frac{1}{2}\partial J  (u_{\alpha, \beta} & +\alpha^{i}\varphi_{i} )h \\
= & 
k_{u_{\alpha, \beta}+\alpha^{i}\varphi_{i}}^{\frac{2-n}{n}}
[
\int L_{g_{0}}(u_{\alpha, \beta}+\alpha^{i}\varphi_{i} )h  \\ 
& \quad\quad\quad\quad\;\;
-
\int (r\K)_{u_{\alpha, \beta}+\alpha^{i}\varphi_{i}}(u_{\alpha, \beta}+\alpha^{i}\varphi_{i})^{\frac{n+2}{n-2}}h
]
\end{split}\end{equation}
and 
\begin{equation}\begin{split}
\frac{1}{2}\partial^{2} & J (u_{\alpha, \beta} +\alpha^{i}\varphi_{i})h_{1}h_{2} \\
= &
k_{u_{\alpha, \beta}+\alpha^{i}\varphi_{i}}^{\frac{2-n}{n}}
[
\int L_{g_{0}}h_{1}h_{2} \\
& \quad\quad\quad\quad\quad -
\frac{n+2}{n-2} 
\int (r\K)_{u_{\alpha, \beta}+\alpha^{i}\varphi_{i}}(u_{\alpha, \beta}+\alpha^{i}\varphi_{i})^{\frac{4}{n-2}}h_{1}h_{2}
] \\
& +
o_{\varepsilon}(\Vert h_{1}\Vert \Vert h_{2} \Vert),
\end{split}\end{equation}
since, when considering the formula for the second variation, we have
\begin{equation}
\begin{split}
\int L_{g_{0}}uh_{i} 
= &
\frac{r}{k}\int Ku^{\frac{n+2}{n-2}}h_{i}
+
O(\vert \delta J(u)\vert \Vert h_{i}\Vert)\\
= &
\frac{r}{k}\int Ku^{\frac{4}{n-2}}vh_{i}+O(\vert \delta J(u)\vert \Vert h_{i}\Vert) \\
= &
O(\Vert v \Vert+\vert \delta J(u)\vert)\Vert h_{i}\Vert.
\end{split}
\end{equation} 
By \eqref{r/kuab-r/k_rough_estimate}  there holds
\begin{equation}
\begin{split}
(\frac{r}{k})_{u}=(\frac{r}{k})_{u_{\alpha, \beta}}+o_{\varepsilon}(1) 
\end{split}
\end{equation} 
and 
$
\frac{r\alpha_{i}^{\frac{4}{n-2}}K_{i}}{k}=4n(n-1)+o_{\varepsilon}(1)
$
by \eqref{rai4/n-2_f}.
Consequently
\begin{equation}\begin{split}
\frac{1}{2}\partial^{2} & J (u_{\alpha, \beta} +\alpha^{i}\varphi_{i})h_{1}h_{2} \\
= &
k_{u_{\alpha, \beta}+\alpha^{i}\varphi_{i}}^{\frac{2-n}{n}}
[
\int L_{g_{0}}h_{1}h_{2} \\
& \quad\quad\quad\quad\;-
c_{n}n(n+2) 
(
\int \frac{K\omega  ^{\frac{4}{n-2}}}{4n(n-1)}h_{1}h_{2}
-
\sum_{i}
\int \varphi_{i}^{\frac{4}{n-2}})h_{1}h_{2}
)
] 
\\ &
+
o_{\varepsilon}(\Vert h_{1}\Vert \Vert h_{2} \Vert).
\end{split}\end{equation}
This shows the statement on the second derivative. Moreover by lemma \ref{lem_v_type_interactions}
\begin{equation}\begin{split}
(\frac{r}{k})_{u_{\alpha, \beta}+\alpha^{i}\varphi_{i}}
=
\frac{r}{k}
+
o(\sum_{r} \frac{1}{\lambda_{r}^{\frac{n-2}{2}}} + \sum_{r\neq s}\eps_{r,s})
+
O
(\Vert v \Vert^{2}+\vert \delta J(u)\vert^{2}).
\end{split}\end{equation}
We obtain
\begin{equation}\begin{split}
\frac{1}{2}\partial J  (u_{\alpha, \beta} &  +\alpha^{i}\varphi_{i} )h \\
= & 
k_{u_{\alpha, \beta}+\alpha^{i}\varphi_{i}}^{\frac{2-n}{n}}
[
\int L_{g_{0}}(u_{\alpha, \beta}+\alpha^{i}\varphi_{i} )h   -
\int r\K(u_{\alpha, \beta}+\alpha^{i}\varphi_{i})^{\frac{n+2}{n-2}}h
] \\
& +
o_{\varepsilon}(\sum_{r} \frac{1}{\lambda_{r}^{\frac{n-2}{2}}} + \sum_{r\neq s}\eps_{r,s})
+
O
(\Vert v \Vert^{2}+\vert \delta J(u)\vert^{2})
,
\end{split}\end{equation}
whence by estimates familiar by now
\begin{equation}\begin{split}
\frac{1}{2}\partial J  (u_{\alpha, \beta} &  +\alpha^{i}\varphi_{i} )h \\
= & 
k_{u_{\alpha, \beta}+\alpha^{i}\varphi_{i}}^{\frac{2-n}{n}}
[
\int (L_{g_{0}}u_{\alpha, \beta}-r\K u_{\alpha, \beta}^{\frac{n+2}{n-2}})h\\
& \quad\quad\quad\quad\;
+\sum_{i}\alpha_{i}
\int (L_{g_{0}}\varphi_{i} - r\K\alpha_{i}^{\frac{4}{n-2}}\varphi_{i}^{\frac{n+2}{n-2}})h
] \\
& +
O( 
\sum_{r} \frac{1}{\lambda_{r}^{\frac{n-2}{2}}}
+ 
\sum_{r\neq s}\eps_{r,s}
+
\Vert v \Vert^{2}+\vert \delta J(u)\vert^{2})
.
\end{split}\end{equation}
Using \eqref{rai4/n-2_f}  we get
\begin{equation}\begin{split}
\frac{1}{2}\partial J  (u_{\alpha, \beta} &  +\alpha^{i}\varphi_{i} )h \\
= & 
k_{u_{\alpha, \beta}+\alpha^{i}\varphi_{i}}^{\frac{2-n}{n}}
[
\int (L_{g_{0}}u_{\alpha, \beta}-r\K u_{\alpha, \beta}^{\frac{n+2}{n-2}})h\\
& \quad\quad\quad\quad\;
+\sum_{i}\alpha_{i}
\int (L_{g_{0}}\varphi_{i} - 4n(n-1)\varphi_{i}^{\frac{n+2}{n-2}})h
] \\
& +
O( 
\sum_{r} \frac{\vert \nabla K_{r}\vert}{\lambda_{r}}
+
\frac{1}{\lambda_{r}^{\frac{n-2}{2}}}
+ 
\sum_{r\neq s}\eps_{r,s}
+
\Vert v \Vert^{2}+\vert \delta J(u)\vert)
\end{split}\end{equation}
and we deduce using lemma \ref{lem_emergence_of_the_regular_part}
\begin{equation}\begin{split}\label{first_derivative_of_J_uab_0}
\frac{1}{2}\partial J  (u_{\alpha, \beta} &  +\alpha^{i}\varphi_{i} )h \\
= & 
k_{u_{\alpha, \beta}+\alpha^{i}\varphi_{i}}^{\frac{2-n}{n}}
\int (L_{g_{0}}u_{\alpha, \beta}-r\K u_{\alpha, \beta}^{\frac{n+2}{n-2}})h\\
& +
O( 
\sum_{r} \frac{\vert \nabla K_{r}\vert}{\lambda_{r}}
+
\frac{1}{\lambda_{r}^{\frac{n-2}{2}}}
+ 
\sum_{r\neq s}\eps_{r,s}
+
\Vert v \Vert^{2}+\vert \delta J(u)\vert)
.
\end{split}\end{equation}
We proceed estimating
\begin{equation}\label{first_derivative_of_J_uab_1}
\begin{split}
\int (L_{g_{0}}u_{\alpha, \beta}-r\K u_{\alpha, \beta}^{\frac{n+2}{n-2}})h
 = &
\frac{k_{u_{\alpha, \beta}}^{\frac{n}{n-2}}}{2}
\langle \partial J(u_{\alpha, \beta}),h\rangle 
+
O(\vert (\frac{r}{k})_{u_{\alpha, \beta}}-\frac{r}{k}\vert ),
\end{split}
\end{equation} 
to whose end we will improve \eqref{r/kuab-r/k_rough_estimate}. Due to lemma \ref{lem_v_type_interactions} we have
\begin{equation}
\begin{split}
\int (L_{g_{0}}u & -(r\K)_{u} u^{\frac{n+2}{n-2}})u_{\alpha, \beta} \\
= &
\int (L_{g_{0}}(u_{\alpha, \beta}+\alpha^{i}\varphi_{i})
-
(r\K)_{u}(u_{\alpha, \beta}+\alpha^{i}\varphi_{i})^{\frac{n+2}{n-2}}
)
u_{\alpha, \beta}\\
& +
o
(
\sum_{r}\frac{1 }{\lambda_{r}^{\frac{n-2}{2}}}
+
\sum_{r\neq s}\eps_{r,s}
)
+
O
(
\Vert v \Vert^{2}
+
\vert \delta J(u)\vert^{2}
),
\end{split}
\end{equation} 
whence in particular
\begin{equation}
\begin{split}
\int (L_{g_{0}}u  & -(r\K)_{u} u^{\frac{n+2}{n-2}})u_{\alpha, \beta} \\
= &
\int (L_{g_{0}}u_{\alpha, \beta}
-
(r\K)_{u}u_{\alpha, \beta}^{\frac{n+2}{n-2}}
)
u_{\alpha, \beta}\\
& +
O
(
\sum_{r}\frac{1 }{\lambda_{r}^{\frac{n-2}{2}}}
+
\sum_{r\neq s}\eps_{r,s}
+
\Vert v \Vert^{2}
+
\vert \delta J(u)\vert^{2}
)
\end{split}
\end{equation} 
and therefore
\begin{equation}\label{r/kuab-r/k_strong}
\begin{split}
(\frac{r}{k})_{u_{\alpha, \beta}}-(\frac{r}{k})_{u}
= &
O
(
\sum_{r}\frac{1 }{\lambda_{r}^{\frac{n-2}{2}}}
+
\sum_{r\neq s}\eps_{r,s}
+
\Vert v \Vert^{2}
+
\vert \delta J(u)\vert
).
\end{split}
\end{equation} 
Plugging \eqref{r/kuab-r/k_strong} with $\frac{r}{k}=(\frac{r}{k})_{u}$ into \eqref{first_derivative_of_J_uab_1} gives recalling lemma \ref{lem_degeneracy_and_pseudo_critical_points}
\begin{equation}
\begin{split}
\int & (L_{g_{0}} u_{\alpha, \beta}   - r\K u_{\alpha, \beta}^{\frac{n+2}{n-2}})h \\
= &
\frac{k_{u_{\alpha, \beta}}}{2}
\langle \partial J(u_{\alpha, \beta}),h\rangle 
+
O
(
\sum_{r}\frac{1 }{\lambda_{r}^{\frac{n-2}{2}}}
+
\sum_{r\neq s}\eps_{r,s}
+
\Vert v \Vert^{2}
+
\vert \delta J(u)\vert
)
\\
= &
\int (L_{g_{0}}u_{\alpha, \beta} -(r\K)_{u_{\alpha, \beta}} u_{\alpha, \beta}^{\frac{n+2}{n-2}})\omega\int L_{g_{0}}\omega h \\
& +
\sum_{i=1}^{m}\int (L_{g_{0}}u_{\alpha, \beta} -(r\K)_{u_{\alpha, \beta}} u_{\alpha, \beta}^{\frac{n+2}{n-2}})\mathrm{e}_{i}\int L_{g_{0}}\mathrm{e}_{i}h \\
& +
O
(
\sum_{r}\frac{1 }{\lambda_{r}^{\frac{n-2}{2}}}
+
\sum_{r\neq s}\eps_{r,s}
+
\Vert v \Vert^{2}
+
\vert \delta J(u)\vert
).
\end{split}
\end{equation}
Applying \eqref{r/kuab-r/k_strong} we then get
\begin{equation}
\begin{split}
\int (L_{g_{0}}u_{\alpha, \beta} & -r\K u_{\alpha, \beta}^{\frac{n+2}{n-2}})h\\
 = &
\int (L_{g_{0}}u_{\alpha, \beta} -(r\K)_{u} u_{\alpha, \beta}^{\frac{n+2}{n-2}})\omega\int L_{g_{0}}\omega h \\
& +
\sum_{i=1}^{m}\int (L_{g_{0}}u_{\alpha, \beta} -(r\K)_{u} u_{\alpha, \beta}^{\frac{n+2}{n-2}})\mathrm{e}_{i}\int L_{g_{0}}\mathrm{e}_{i}h \\
& +
O
(
\sum_{r}\frac{1 }{\lambda_{r}^{\frac{n-2}{2}}}
+
\sum_{r\neq s}\eps_{r,s}
+
\Vert v \Vert^{2}
+
\vert \delta J(u)\vert
),
\end{split}
\end{equation}
whence
\begin{equation}
\begin{split}
\int ( & L_{g_{0}}  u_{\alpha, \beta}  -r\K u_{\alpha, \beta}^{\frac{n+2}{n-2}})h\\
 = &
\int (L_{g_{0}}(u_{\alpha, \beta}+\alpha^{i}\varphi_{i}) -(r\K)_{u} (u_{\alpha, \beta}+\alpha^{i}\varphi_{i})^{\frac{n+2}{n-2}})\omega\int L_{g_{0}}\omega h \\
& +
\sum_{i=1}^{m}\int (L_{g_{0}}(u_{\alpha, \beta}+\alpha^{i}\varphi_{i}) 
-(r\K)_{u} (u_{\alpha, \beta}+\alpha^{i}\varphi_{i})^{\frac{n+2}{n-2}})\mathrm{e}_{i}\int L_{g_{0}}\mathrm{e}_{i}h \\
& +
O
(
\sum_{r}\frac{1 }{\lambda_{r}^{\frac{n-2}{2}}}
+
\sum_{r\neq s}\eps_{r,s}
+
\Vert v \Vert^{2}
+
\vert \delta J(u)\vert
).
\end{split}
\end{equation}
Since $\int L_{g_{0}}\omega h, \int L_{g_{0}}\mathrm{e}_{i}h=o_{\varepsilon}(1)$ as $h\in H_{u}(\omega, p, \eps)$ and $\vert h \vert=1$, we conclude
\begin{equation}
\begin{split}
\int (L_{g_{0}} & u_{\alpha, \beta}  -r\K u_{\alpha, \beta}^{\frac{n+2}{n-2}})h\\
= &
o_{\varepsilon}(\Vert v \Vert )
+
O
(
\sum_{r}\frac{1 }{\lambda_{r}^{\frac{n-2}{2}}}
+
\sum_{r\neq s}\eps_{r,s}
+
\vert \delta J(u)\vert
).
\end{split}
\end{equation}
Plugging this into \eqref{first_derivative_of_J_uab_0} proves the statement on the first derivative.
\end{proof}

In contrast to the case $\omega =0$ the second variation at $u_{\alpha, \beta}+\alpha^{i}\var_{i}$ is not necessarily positive  definite. 
It is however sufficient to have non degeneracy.
\begin{proposition}[Decomposition of the second variation on $H_{u}(\omega, p, \eps)$]\label{prop_decomposing the second variation_f}$_{}$\\
There exist $\gamma, \eps_{0}>0$ such,  that for any 
\begin{equation}
u=u_{\alpha,\beta}+\alpha^{i}\varphi_{i}+v\in V(\omega,p,\varepsilon) 
\end{equation} 
with $0<\eps<\eps_{0}$ we may decompose 
\begin{equation*}\begin{split}
H_{u}(\omega, p, \eps)=H=H_{+}\oplus_{L_{g_{0}}} H_{-}\;\text{ with }\;\dim H_{-}<\infty
\end{split}\end{equation*}
and for any $h_{+}\in H_{+},h_{-}\in H_{-}$ there holds
\begin{enumerate}[label=(\roman*)]
 \item 
$\partial^{2}J(u_{\alpha, \beta}+\alpha^{i}\varphi_{i})\lfloor_{H_{+}}>\gamma$
 \item 
$\partial^{2}J(u_{\alpha, \beta}+\alpha^{i}\varphi_{i})\lfloor_{H_{-}}<-\gamma$
 \item
$\partial^{2}J(u_{\alpha, \beta}+ \alpha^{i}\varphi_{i})h_{+}h_{-}=o_{\varepsilon}(\Vert h_{+} \Vert \Vert h_{-}\Vert)$.
 \end{enumerate}
\end{proposition}

\begin{proof}[\textbf{Proof of proposition \ref{prop_decomposing the second variation_f}}]\label{p_decomposing the second variation_f}$_{}$\\
Let $H=H_{u}(\omega, p, \eps)$ and 
note, that  $H$ is a closed subspace of $W$, since
\begin{equation}
H=\langle \upsilon, \upsilon_{j}, \upsilon_{k,i}\rangle^{\perp_{L_{g_{0}}}}
\end{equation} 
according to definition \ref{def_H(omega,p,e)} for $\upsilon, \upsilon_{k,i}, \upsilon_{j} \in W^{1,2}_{g_{0}}(M) $ solving 
\begin{equation}\begin{split}\label{def_phiki&phi_1}
L_{g_{0}}\upsilon=Ku^{\frac{4}{n-2}}u_{\alpha, \beta}, \;
L_{g_{0}}\upsilon_{j}=\frac{n+2}{n-2}Ku^{\frac{4}{n-2}}\partial_{\beta_{j}}u_{\alpha, \beta}
\end{split}\end{equation}
and
\begin{equation}\label{def_phiki&phi_2}
\begin{split}
L_{g_{0}}\upsilon_{k,i}=Ku^{\frac{4}{n-2}}\phi_{k,i} 
\end{split}
\end{equation} 
cf. definitions \ref{def_relevant_quantities} and \ref{def_H(omega,p,e)}.
In view of proposition \ref{prop_derivatives_on_H_f} we consider
\begin{equation}\begin{split}
T
:H\times H\-\R:(a,b)\- T(a,b)
\end{split}\end{equation}
with
\begin{equation}\begin{split}
T(h_{1},h_{2})
= &
\int L_{g_{0}}h_{1}h_{2} \\
& \quad -
c_{n}n(n+2)
\int 
\left[
\frac{K\omega  ^{\frac{4}{n-2}}}{4n(n-1)}
+
\sum_{i}\varphi_{i}^{\frac{4}{n-2}}ab
\right]
h_{1}h_{2}.
\end{split}\end{equation}
Due to the spectral theorem for compact operators there exist
\begin{equation}
(h_{i})_{i\in \N}\subset H \; \text{ and }\; (\mu_{h_{i}})\subset \R\; \text{ with }\; \mu_{h_{i}}\-0
\; \text{ as }\; i\- \infty
\end{equation} 
such, that $(h_{i})_{i\in \N}$ forms an orthonormal basis of $H$
\begin{equation}\begin{split}
H=\langle h_{i}\mid i\in \N\rangle\; \text{ and }\; \langle h_{i},h_{j}\rangle_{L_{g_{0}}}=\int L_{g_{0}}h_{i}h_{j}=\delta_{ij},
\end{split}\end{equation}
and we have 
$
K\omega  ^{\frac{4}{n-2}}h_{i}=\mu_{h_{i}}L_{g_{0}}h_{i}$
weakly, so
\begin{equation}
\int K\omega^{\frac{4}{n-2}}h_{i}h=\mu_{i}\int L_{g_{0}}h_{i}h
\; \text{ for all }\; h \in H.
\end{equation} 
Likewise there exists an orthonormal basis of $W=W^{1,2}(M)$
\begin{equation}\begin{split}
W=\langle w_{q}\mid q \in \N\rangle\; \text{ and }\;
\langle w_{p},w_{q}\rangle_{L_{g_{0}}}
=
\int L_{g_{0}}w_{p}w_{q}=\delta_{pq}
\end{split}\end{equation}
satisfying for a sequence $(\mu_{w_{q}})\subset \R$ with $\mu_{w_{q}}\- 0$ as $q\- \infty$
\begin{equation}
K\omega  ^{\frac{4}{n-2}}w_{q}=\mu_{w_{q}}L_{g_{0}}w_{q}.
\end{equation}  
Below we will prove, that for any $q,l\in \N$ there holds
\begin{equation}\begin{split}\label{wq_hl_interaction_small}
(\mu_{w_{q}}-\mu_{h_{l}})\langle w_{q},h_{l}\rangle_{L_{g_{0}}}
\- 0\; \text{ as }\; \varepsilon\- 0.
\end{split}\end{equation}
Moreover recall, that according to proposition \ref{prop_positivity_of_D2J} we have
\begin{equation}\begin{split}
\int L_{g_{0}}hh-c_{n}n(n+2)\sum_{i}\int\varphi_{i}^{\frac{4}{n-2}}h^{2}
\geq 
c\int L_{g_{0}}hh
\end{split}\end{equation}
for some positive  constant $c>0$. Thus for any 
\begin{equation}\begin{split}
\bar h\in H_{1}=\langle h_{i}\mid \frac{n+2}{n-2}\mu_{h_{i}}\leq \frac{c}{2}\rangle
\end{split}\end{equation}
we have $T(\bar  h, \bar h)\geq \frac{c}{2}\Vert \bar h\Vert^{2}$.
Let $\epsilon>0$ such, that
\begin{equation}\begin{split}\label{2epsilon_condition}
\lbrace 
w_{q}\mid 1-2\epsilon\leq \frac{n+2}{n-2}\mu_{w_{q}}\leq 1+2\epsilon
\rbrace
=
\lbrace \mathrm{e}_{j}\mid j=1, \ldots,m
\rbrace,
\end{split}\end{equation}
where $E_{\frac{n+2}{n-2}}(\omega)=\langle \mathrm{e}_{j}\mid j=1, \ldots,m\rangle$, cf. lemma \ref{lem_spectral_theorem_and_degeneracy}, and define
\begin{equation}\begin{split}
H_{2}=\langle h_{i}\mid \frac{c}{2}<\frac{n+2}{n-2}\mu_{h_{i}}<1-\epsilon\rangle
\end{split}\end{equation}
and 
\begin{equation}\begin{split}
W_{2}=\langle w_{q}\mid \frac{c}{2}<\frac{n+2}{n-2}\mu_{w_{q}}< 1-\epsilon\rangle.
\end{split}\end{equation}
Then for $0\neq \tilde h\in H_{2}$ we have due \eqref{wq_hl_interaction_small} 
\begin{equation}\begin{split}
\Vert \tilde h \Vert^{2}=\Vert \Pi_{W_{2}}\tilde h\Vert^{2}+\Vert \Pi_{W_{2}^{\perp}}\tilde h\Vert^{2},
\;\Vert\Pi_{W_{2}^{\perp}}\tilde h\Vert=o_{\varepsilon}(\Vert \tilde h \Vert),
\end{split}\end{equation}
whence for $\bar h+\tilde h\in H_{1}\oplus H_{2}$ we obtain 
\begin{equation}\begin{split}
T(\bar h+\tilde h, \bar h+\tilde h)
= &
T(\bar h, \bar h)+2T(\bar h, \tilde h)+T(\tilde h, \tilde h) \\
\geq &
\frac{c}{2}\Vert \bar h\Vert^{2}
-2\frac{n+2}{n-2}\int\sum_{i}\tilde\varphi_{i}^{\frac{4}{n-2}}\bar h (\Pi_{W_{2}}\tilde h)\\
& +
T((\Pi_{W_{2}}\tilde h),(\Pi_{W_{2}}\tilde h))
+
o_{\varepsilon}(\Vert \bar h\Vert^{2}+\Vert \tilde h\Vert^{2}).
\end{split}\end{equation}
Since $W_{2}$ is fix and finite dimensional, we get
\begin{equation}\begin{split}\label{finite_dimensional_interaction_small}
\int \tilde\varphi_{i}^{\frac{4}{n-2}}\bar h(\Pi_{W_{2}}\tilde h)=o_{\varepsilon}(\Vert \bar h\Vert^{2}+\Vert \tilde h\Vert^{2})
\end{split}\end{equation}
and 
\begin{equation}\begin{split}
T(( \Pi_{W_{2}} & \tilde h), (\Pi_{W_{2}}\tilde h)) \\
= &
\int L_{g_{0}}(\Pi_{W_{2}}\tilde h)(\Pi_{W_{2}}\tilde h)
-
\frac{n+2}{n-2}\int K\omega  ^{\frac{4}{n-2}}(\Pi_{W_{2}}\tilde h)^{2}
+
o_{\varepsilon}(\Vert \overline{h}_{2}\Vert^{2}) \\
\geq &
\epsilon \Vert (\Pi_{W_{2}}\tilde h)\Vert^{2}
=
\epsilon (\Vert \tilde h \Vert^{2}-\Vert \Pi_{W_{2}^{\perp}}\tilde h\Vert^{2})
\end{split}\end{equation}
Thus $T$ is positive   on $H_{1}\oplus H_{2}$.
Let
\begin{equation}\begin{split}
H_{3}=
\langle h_{i}\mid 1-\eps\leq \frac{n+2}{n-2}\mu_{h_{i}}\leq 1+\eps \rangle
\end{split}\end{equation}
and
\begin{equation}\begin{split}
W_{3}=
\langle w_{q}\mid 1-\eps\leq \frac{n+2}{n-2}\mu_{w_{q}}\leq 1+\eps \rangle
=
\langle \mathrm{e}_{j}\mid j=1, \ldots,m\rangle.
\end{split}\end{equation} 
Then for $0\neq \hat h\in H_{3}$ we have due to \eqref{wq_hl_interaction_small} and \eqref{2epsilon_condition}
\begin{equation}\begin{split}
\Vert \hat h \Vert^{2}=\Vert \Pi_{W_{3}}\hat h\Vert^{2}+\Vert \Pi_{W_{3}^{\perp}}\hat h\Vert^{2},
\;\Vert \Pi_{W_{3}^{\perp}}\hat h\Vert=o_{\varepsilon}(\Vert \hat h \Vert).
\end{split}\end{equation}
Since $\Pi_{W_{3}}\hat h=\sum^{m}_{j=1}\langle \mathrm{e}_{j}, \hat h\rangle_{L_{g_{0}}} \mathrm{e}_{j}$ and 
\begin{equation}\begin{split}
\langle \upsilon_{j}, \hat h\rangle_{L_{g_{0}}}
=
0
\end{split}\end{equation}
we obtain
\begin{equation}
\Vert \Pi_{W_{3}}\hat h\Vert=o_{\varepsilon}(\Vert \hat h \Vert), 
\end{equation} 
once we know $\Vert \upsilon_{j}-\mathrm{e}_{j}\Vert=o_{\varepsilon}(1)$ and we will show this below, cf \eqref{upsilon_close_by}.

Thus $H_{3}=\lbrace 0 \rbrace$
is trivial for $\varepsilon>0$ sufficiently small. 
\\
Finally let
\begin{equation}\begin{split}
H_{4}=\langle h_{i}\mid \frac{n+2}{n-2}\mu_{h_{i}} \geq 1+\epsilon\rangle=(H_{1}\oplus H_{2})^{\perp_{L_{g_{0}}}}
\end{split}\end{equation}
and 
\begin{equation}
\begin{split}
W_{4}=\langle w_{q}\mid \frac{n+2}{n-2}\mu_{w_{q}} \geq 1+\epsilon\rangle.
\end{split}
\end{equation} 
$W_{4}$ is fixed and finite dimensional. Arguing as for $H_{2}$ one obtains, that $T$ is strictly negative on $H_{4}$.
We conclude for $H=\tilde H_{1}\oplus \tilde H_{2}$, where 
\begin{equation}\begin{split} 
\tilde H_{1}=H_{1}\oplus H_{2}\;\text{ and }\;\tilde H_{2}=H_{4}, \; \dim \tilde H_{2}<\infty,
\end{split}\end{equation}
that
$
T\lfloor _{\tilde H_{1}}>\gamma
\;\text{ and }\;
T\lfloor_{\tilde H_{2}}<-\gamma
$
for some $\gamma>0$ small, whence
\begin{equation}\begin{split} 
\partial^{2} J(u_{\alpha, \beta}+\alpha^{i}\varphi_{i})\lfloor _{\tilde H_{1}}>\tilde \gamma
\;\text{ and }\;
\partial^{2} J(u_{\alpha, \beta} +\alpha^{i}\varphi_{i})\lfloor _{\tilde H_{2}}<-\tilde \gamma
\end{split}\end{equation}  
for some $\tilde \gamma>0$  by proposition \ref{prop_derivatives_on_H_f}. Moreover for $\tilde h_{1} \in \tilde H_{1}, \; \tilde h_{2} \in \tilde H_{2}$
\begin{equation}\begin{split}
\int L_{g_{0}}\tilde h_{1}\tilde h_{2}
=
\int K\omega^{\frac{4}{n-2}}\tilde h_{1}\tilde h_{2}=0, \;
\end{split}\end{equation}
whence 
\begin{equation}
\begin{split}
T(\tilde h_{1}, \tilde h_{2})
= &
-c_{n}n(n+2)\sum_{i}\int \varphi^{\frac{4}{n-2}}_{i}\tilde h_{1}\tilde h_{2}.
\end{split}
\end{equation} 
Thus arguing as for \eqref{finite_dimensional_interaction_small} we get
\begin{equation}\begin{split}
\partial^{2} J(u_{\alpha, \beta} +\alpha^{i}\varphi_{i})\tilde h_{1} \tilde h_{2}=o_{\varepsilon}(\Vert \tilde  h_{1}\Vert\Vert \tilde h_{2}\Vert).
\end{split}\end{equation}
\\
We are left with proving \eqref{wq_hl_interaction_small} and \eqref{upsilon_close_by}. First observe, that by definition
\begin{equation}
\begin{split}
L_{g_{0}}\omega=K\omega^{\frac{n+2}{n-2}}, \,L_{g_{0}}\mathrm{e}_{j}=\frac{n+2}{n-2}K\omega^{\frac{n+2}{n-2}}\mathrm{e}_{j}
\end{split}
\end{equation} 
and
\begin{equation}
u_{\alpha, \beta}=\alpha(\omega+\beta^{j}\mathrm{e}_{j})+O(\Vert \beta \Vert^{2}).
\end{equation} 
Consequently \eqref{def_phiki&phi_1} implies 
\begin{equation}
\begin{split}
\Vert L_{g_{0}}(\upsilon-\alpha^{\frac{n+2}{n-2}}\omega)\Vert_{L^{\frac{2n}{n+2}}}\,,
\Vert L_{g_{0}}(\upsilon_{j}-\alpha^{\frac{n+2}{n-2}}\mathrm{e}_{j})\Vert_{L^{\frac{2n}{n+2}}}
=
o_{\varepsilon}(1).
\end{split}
\end{equation} 
Likewise one obtains recalling
definition \ref{def_relevant_quantities} and lemma \ref{lem_interactions}
\begin{equation}
\begin{split}
\Vert L_{g_{0}}(\upsilon_{k,i}-c_{k}\alpha_{i}^{\frac{4}{n-2}}K_{i}\phi_{k,i})\Vert_{L^{\frac{2n}{n+2}}}
=
o_{\varepsilon}(1).
\end{split}
\end{equation} 
Therefore we obtain with $o_{\varepsilon}(1)\- 0$ in $W^{1,2}$ as $\varepsilon\- 0$ 
\begin{equation}\label{upsilon_close_by}
\begin{split}
\upsilon
=
\alpha \omega +o_{\varepsilon}(1),
\upsilon_{j}
=
\alpha \mathrm{e}_{j}+o_{\varepsilon}(1)
\; \text{ and }\;
\upsilon_{k,i}=c_{k}\alpha_{i}^{\frac{4}{n-2}}K_{i} \phi_{k,i}+o_{\varepsilon}(1).
\end{split}
\end{equation} 
Let us write now
\begin{equation}\begin{split}
w_{q}=\langle w_{q},h^{i}\rangle_{L_{g_{0}}}h_{i}+\alpha_{q}\upsilon+\alpha^{k,i}_{q}\upsilon_{k,i}+\alpha^{j}_{q}\upsilon_{j}.
\end{split}\end{equation}
Then on the one hand 
\begin{equation}\begin{split}\label{lhs}
\int K\omega  ^{\frac{4}{n-2}}w_{q}h_{l}
= &
\mu_{w_{q}}\langle w_{q},h_{l}\rangle_{L_{g_{0}}}, 
\end{split}\end{equation}
while on the other one
\begin{equation}\begin{split}\label{rhs}
\int K\omega  ^{\frac{4}{n-2}}w_{q}h_{l}
= &
\langle w_{q},h^{i}\rangle_{L_{g_{0}}}\int K\omega  ^{\frac{4}{n-2}}h_{i}h_{l} 
+
\alpha_{q}\int K\omega  ^{\frac{4}{n-2}}\upsilon h_{l} \\
& +
\alpha_{q}^{j}\int K\omega  ^{\frac{4}{n-2}}\upsilon_{j}h_{l}
+
\alpha_{q}^{k,i}\int K\omega  ^{\frac{4}{n-2}}\upsilon_{k,i}h_{l} 
\\
= &
\mu_{h_{l}}\langle w_{q},h_{l}\rangle_{L_{g_{0}}}
+
o_{\varepsilon}(\vert \alpha_{q}\vert+\sum_{j}\vert \alpha^{j}_{q}\vert+\sum_{k,i}\vert \alpha^{k,i}_{q}\vert )_{l}.
\end{split}\end{equation}
The last equality above follows easily from \eqref{upsilon_close_by} and the orthogonal properties of $H_{u}(\omega,p, \varepsilon)$. Combining \eqref{lhs} and \eqref{rhs} we get
\begin{equation}\begin{split}\label{lambdaej-lambdahl}
(\mu_{w_{q}}-\mu_{h_{l}})\langle w_{q},h_{l}\rangle_{L_{g_{0}}}
= &
o_{\varepsilon}(\vert \alpha_{q}\vert+\sum_{j}\vert \alpha^{j}_{q}\vert+\sum_{k,i}\vert \alpha^{k,i}_{q}\vert)_{l}.
\end{split}\end{equation}
Moreover
\begin{equation}\begin{split}
\langle w_{q}, \upsilon\rangle_{L_{g_{0}}}
= &
\alpha_{q}\langle \upsilon, \upsilon\rangle_{L_{g_{0}}}
+
\alpha_{q}^{j}\langle\upsilon_{j}, \upsilon\rangle
+
\alpha^{l,p}_{q}\langle \upsilon_{l,p}, \upsilon_{l,p}\rangle_{L_{g_{0}}}
\\
\simeq & 
\alpha_{q}
+
o_{\varepsilon}(\sum_{j}\vert \alpha^{j}_{q}\vert+\sum_{l,p}\vert \alpha_{q}^{l,p}\vert),
\end{split}\end{equation}
likewise
\begin{equation}\begin{split}
\langle w_{q}, \upsilon_{j}\rangle_{L_{g_{0}}}
= &
\alpha_{q}\langle \upsilon, \upsilon_{j}\rangle_{L_{g_{0}}}
+
\alpha_{q}^{p}\langle \upsilon_{p}, \upsilon_{j}\rangle
+
\alpha^{l,p}_{q}\langle \upsilon_{l,p}, \upsilon_{j}\rangle_{L_{g_{0}}}
\\
\simeq & 
\alpha_{q}^{p}\delta_{p,j}
+
o_{\varepsilon}(\vert \alpha_{q}\vert +\sum_{j}\vert \alpha_{q}^{j}\vert+\sum_{l,p}\vert \alpha_{q}^{l,p}\vert)
\end{split}\end{equation}
and
\begin{equation}\begin{split}
\langle w_{q}, \upsilon_{k,i}\rangle_{L_{g_{0}}}
= &
\alpha_{q}\langle \upsilon, \upsilon_{k,i}\rangle_{L_{g_{0}}}
+
\alpha_{q}^{j}\langle \upsilon_{j}, \upsilon_{k,i}\rangle
+
\alpha^{l,p}_{q}\langle \upsilon_{l,p}, \upsilon_{k,i}\rangle_{L_{g_{0}}}
\\
\simeq & 
\alpha^{l,p}_{j}\delta_{l,k}\delta_{p,i}
+
o_{\varepsilon}(\vert \alpha_{q}\vert +\sum_{j}\vert \alpha_{q}^{j}\vert+\sum_{l,p}\vert \alpha_{q}^{l,p}\vert)_{k,i}.
\end{split}\end{equation}
Summing up we obtain by Parseval's identity 
\begin{equation}\begin{split}\label{lambdaei-lambdahl_summability}
\Vert \upsilon\Vert^{2}+\sum_{k,i} \Vert \upsilon_{k,i}\Vert^{2} & +\sum_{j}\Vert \upsilon_{j} \Vert^{2} \\
= &(1+o_{\varepsilon}(1))
[
\sum_{q}\vert \alpha_{q}\vert^{2}
+
\sum_{q,k,i}\vert \alpha^{k,i}_{q}\vert^{2}
+
\sum_{q,j}\vert \alpha_{q}^{j}\vert^{2}
]
\end{split}\end{equation}
and the left hand side is uniformly bounded. Thus \eqref{lambdaej-lambdahl} gives
\begin{equation}\begin{split}
(\mu_{w_{q}}-\mu_{h_{l}})\langle w_{q},h_{l}\rangle_{L_{g_{0}}}
= &
o_{\varepsilon}(1).
\end{split}\end{equation}
The proof is thereby complete.
\end{proof}
As before smallness of the first and definiteness of the second variation provide 
an appropriate estimate on the error term  $v$.
\begin{corollary}[A-priori estimate on $v$]\label{cor_a-priori_estimate_on_v_f}$_{}$\\
On $V(\omega, p, \eps)$ for $\eps>0$ small we have 
\begin{equation*}\begin{split}
\Vert v \Vert
= &
O
(
\sum_{r} \frac{\vert \nabla K_{r}\vert}{\lambda_{r}}
+
\frac{1}{\lambda_{r}^{\frac{n-2}{2}}} 
+
\sum_{r\neq s}\eps_{r,s}
+
\vert \delta J(u)\vert).
\end{split}\end{equation*}
\end{corollary}

\begin{proof}[\textbf{Proof of corollary \ref{cor_a-priori_estimate_on_v_f}}]\label{p_a_priori_estimates_on_v_f}$_{}$\\
Note, that $\partial^{2}J$ is uniformly H\"older continuous on $V(\omega,p, \varepsilon)$ according to proposition \ref{prop_derivatives_of_J} and the remarks following.
Decomposing
$
v=v_{+}+v_{-}\in H_{+}\oplus H_{-}
$
according to proposition \ref{prop_decomposing the second variation_f} we readily have
\begin{enumerate}[label=(\roman*)]
 \Item 
\begin{equation}\begin{split}
\partial J(u)v_{+}
\geq  &
\partial J (u_{\alpha, \beta}+\alpha^{i}\varphi_{i})v_{+} 
+
\gamma \Vert v_{+}\Vert ^{2}
+
o_{\varepsilon}(\Vert v_{+}\Vert\Vert v_{-}\Vert)
\end{split}\end{equation}
 \Item
\begin{equation}\begin{split}
\partial J(u)v_{-}
\leq  &
\partial J(u_{\alpha, \beta}+\alpha^{i}\varphi_{i})v_{-} 
-
\gamma \Vert v_{-}\Vert ^{2}
+
o_{\varepsilon}(\Vert v_{+}\Vert\Vert v_{-}\Vert).
\end{split}\end{equation}
\end{enumerate}
This gives
$
\Vert v \Vert^{2}
=
O(
\vert \delta J(u)\vert^{2}
+
\vert \delta J(u_{\alpha, \beta} +\alpha^{i}\varphi_{i})\lfloor_{H}\vert^{2}
)
$
and the claim follows from proposition \ref{prop_derivatives_on_H_f}
\end{proof}

Next we combine lemma \ref{lem_the_shadow_flow_w} and corollaries \ref{cor_simplifying_ski_f}, \ref{cor_a-priori_estimate_on_v}. 

\begin{corollary}[The simplified shadow flow] \label{cor_simplifying_the_shadow_flow_w}$_{}$\\
For $u\in V(\omega, p, \eps)$ with $\eps>0$ we have 
\begin{enumerate}[label=(\roman*)]
\Item
\begin{equation*}\begin{split}
-\frac{\dot \lambda_{i}}{\lambda_{i}}
= &
\frac{r}{k}
[
\frac{d_{2}}{c_{2}}
\frac{\alpha \omega_{i}}{\alpha_{i}K_{i}\lambda_{i} ^{\frac{n-2}{2}}}
-
\frac{b_{2}}{c_{2}}\sum_{i\neq j =1}^{p}\frac{\alpha_{j}}{\alpha_{i}}
\lambda_{i}\partial_{\lambda_{i}}\eps_{i,j}
]
(1+o_{\frac{1}{\lambda_{i}}}(1))
+
R_{2,i}
\end{split}\end{equation*}
\Item
\begin{equation*}\begin{split}
\lambda_{i}\dot a_{i}
= &
\frac{r}{k}
[
\frac{d_{3}}{c_{3}}\frac{\nabla K_{i}}{K_{i}\lambda_{i}}
+
\frac{b_{3}}{c_{3}}
\sum_{i\neq j =1}^{p}\frac{\alpha_{j}}{\alpha_{i}}
\frac{1}{\lambda_{i}}\nabla_{a_{i}}\eps_{i,j}
]
(1+o_{\frac{1}{\lambda_{i}}}(1))
+
R_{3,i},
\end{split}\end{equation*}
\end{enumerate}
where
\begin{equation*}\begin{split}
R_{k,i}
= 
o_{\varepsilon}(\frac{1}{\lambda_{i}^{\frac{n-2}{2}}}+\sum_{i\neq j=1}^{p}\eps_{i,j})
+
O
(
\sum_{r}\frac{\vert \nabla K_{r}\vert^{2}}{\lambda^{2}}
+
\frac{1}{\lambda_{r}^{n-2}}
+
\sum_{r\neq s}\eps_{r,s}^{2}
+
\vert \delta J(u)\vert^{2}
).
\end{split}\end{equation*}
\end{corollary}
\begin{proof}[\textbf{Proof of proposition \ref{cor_simplifying_the_shadow_flow_w}}]$_{}$\\
This follows from  lemma \ref{lem_the_shadow_flow_w} and corollaries \ref{cor_simplifying_ski_f}, 
\ref{cor_a-priori_estimate_on_v_f}.
\end{proof}

\section{The flow on V(\textomega, p, \textepsilon)}
\label{sec:FlowOnVwpe}
\subsection{Principal behaviour}
\label{subsec:PrincipalBehaviour}
For $u\in V(\omega, p, \eps)$ corollaries \ref{cor_simplifying_ski} and \ref{cor_simplifying_ski_f} give a hint on the principal terms
of $\partial J(u)$. The following definition assumes these quantities
to give a lower bound on the first variation of $J$.
\begin{definition}[Principal lower bound of the first variation]\label{def_principally_lower_bounded_of_the_first_variation}$_{}$\\ 
We call $\partial J$ principally lower bounded, 

if for every $p\geq 1$ there exist $c, \eps>0$ such, that
\begin{equation*}\begin{split}
\vert \delta J(u)\vert 
\geq
c
(
\sum_{r}
\frac{\vert \nabla K_{r}\vert}{K_{r}\lambda_{r}}
+
\frac{\vert \lap K_{r}\vert }{K_{r}\lambda_{r}^{2}}
+
\frac{1}{\lambda_{r}^{n-2}}
+
\sum_{r\neq s}\eps_{r,s}
)
\; \text{ for all }\; u\in V(p, \eps).
\end{split}\end{equation*}

and
\begin{equation*}\begin{split}
\vert \delta J(u)\vert \geq
c
(
\sum_{r}
\frac{\vert \nabla K_{r}\vert}{K_{r}\lambda_{r}}
+
\frac{1}{\lambda_{r}^{\frac{n-2}{2}}}
+
\sum_{r\neq s}\eps_{r,s}
)
\; \text{ for all }\; u\in V(\omega, p, \eps).
\end{split}\end{equation*}
\end{definition}

Under this mild assumption  we have uniformity in $V(\omega, p, \eps)$ as follows.

 \begin{proposition}[Uniformity in $V(\omega, p, \eps)$]\label{prop_uniformity_in_V(omega,p,e)}$_{}$
 
Assume $\partial J$ to be principally lower bounded.\\
For $u=u_{\alpha, \beta}+\alpha^{i}\var_{i}+v\in V(\omega, p, \eps)$ with $k_{u}=\int Ku^{\frac{2n}{n-2}}\equiv 1$ we then have
\begin{enumerate}[label=(\roman*)]
 \Item 
\begin{equation*}\begin{split}
\lambda_{i}^{-1}, \eps_{i,j}, \vert 1-\frac{r_{\infty}\alpha_{i}^{\frac{4}{n-2}}K_{i}}{4n(n-1)}\vert
,
\Vert v \Vert
\- 0
\end{split}\end{equation*}
 \Item
\begin{equation*}\begin{split}
\vert(\frac{r}{k})_{u_{1, \beta}}-r_{\infty}\alpha^{\frac{4}{n-2}}\vert, \vert \delta J(u_{1, \beta})\vert
\-
0
\end{split}\end{equation*} 
\end{enumerate}
uniformly as $\vert \delta J(u)\vert \- 0$ and $J(u)=r\- J_{\infty}=r_{\infty}$.
\end{proposition}
In view of (i) above and definition \ref{def_V(omega,p,e)} we would expect to have as well
\begin{equation}
\vert 1-r_{\infty}\alpha^{\frac{4}{n-2}}\vert, \Vert \beta \Vert\- 0
\end{equation}

as $\vert \delta J(u)\vert \- 0$ and $J(u)=r\- J_{\infty}=r_{\infty}$. 
\\
But, since critical points of $J$ are not necessarily isolated, some $u_{\alpha, \beta}$ with 
$0\neq \Vert \beta \Vert <\varepsilon$ could be a critical point of $J$ itself.

\begin{proof}[\textbf{Proof of proposition \ref{prop_uniformity_in_V(omega,p,e)} }]$_{}$\\
Of course $\frac{1}{\lambda_{i}}, \eps_{i,j}\- 0$ as $\vert \delta J(u)\vert \- 0$ by assumption
and the same holds true for $\Vert v \Vert $ due to corollaries \ref{cor_a-priori_estimate_on_v}, \ref{cor_a-priori_estimate_on_v_f}.
Then due to \eqref{rai^{...}/k=...} and \eqref{rai4/n-2_f}
\begin{equation}
\begin{split}
1-\frac{r\alpha_{i}^{\frac{4}{n-2}}K_{i}}{4n(n-1)}\- 0 \; \text{ as }\; \vert \delta J(u)\vert \- 0
\end{split}
\end{equation} 
as well and $(\frac{r}{k})_{u_{\alpha, \beta}}-(\frac{r}{k})_{u}\- 0$ as $\vert \delta J(u)\vert \- 0$ due to \eqref{r/kuab-r/k_strong}. From \eqref{equation_solved_again}  and \eqref{partial_J_u_alpha,beta_tested} we infer $\vert \delta J(u_{\alpha, \beta})\- 0$
as $\vert \delta J(u)\vert \- 0$ and we have $\partial J(u_{\alpha, \beta})=\alpha J(u_{1, \beta})$, since
$u_{\alpha, \beta}=\alpha u_{1, \beta}$ and scaling invariance of $J$. Thereby
\begin{equation}
\begin{split}
(\frac{r}{k})_{u_{\alpha, \beta}}=(\frac{r}{k})_{u_{1, \beta}}\alpha^{-\frac{4}{n-2}},
\end{split}
\end{equation} 
whence due to $(\frac{r}{k})_{u}=r_{u}\- r_{\infty}$ we have $(\frac{r}{k})_{u_{1, \beta}}-r_{\infty}\alpha^{\frac{4}{n-2}}\- 0$.
\end{proof}
As indicated above
$\Vert \beta \Vert \- 0$ is not necessary. On the other hand we may assume due to proposition \ref{prop_concentration_compactness}, that along a flow line
$$u=u_{\alpha, \beta}+\alpha^{i}\var_{i}+v\in V(\omega,p, \varepsilon)$$
we have $\Vert \beta_{t_{k}}\Vert\-0$ for a time sequence $t_{k}\- \infty$.

We then have to show $\vert 1-r_{\infty}\alpha^{\frac{4}{n-2}}\vert, \, \Vert \beta \Vert\-0$ along the full flow line.
\\
For $p=0$ this is true due to  the unicity of a limiting critical point, cf. proposition \ref{prop_unicity_of_a_limiting_critical_point}.
The following proposition yields the same result for $p\geq 1$.

\begin{proposition}[Unicity of a limiting critical point at infinity]\label{prop_unicity_of_a_limiting_critical_point_at_infinity}$_{}$

Assume $\partial J$ to be principally lower bounded.\\
If a sequence $u(t_{k})$ converges to a critical point at infinity of $J$

in the sense, that
\begin{equation*}\begin{split}
\e p>1, \eps_{k}\searrow 0\;:\; u(t_{k})\in V(\omega, p, \eps_{k}),
\end{split}\end{equation*}
then $u$ converges as well

in the sense, that
\begin{equation*}\begin{split}
\e  p>1 \fa \eps>0 \e T>0 \fa t>T\;:\; u(t)\in V(\omega, p, \eps).
\end{split}\end{equation*}
\end{proposition}
\begin{proof}[\textbf{Proof of proposition \ref{prop_unicity_of_a_limiting_critical_point_at_infinity}}]$_{}$\\
Since 
\begin{equation}
k\equiv 1,J(u)=r\searrow r_{\infty}\;\text{ and }\;\partial J(u)\- 0
\end{equation} 
along a flow line we have on $V(\omega,p, \varepsilon)$ according to proposition \ref{prop_uniformity_in_V(omega,p,e)} 
\begin{equation}\label{Ju_expansion_along_a_flow_line}
\begin{split}
J(u)
= &
\int L_{g_{0}}uu 
= 
\int L_{g_{0}}u_{\alpha, \beta}u_{\alpha, \beta}
+
\sum_{i}\alpha_{i}^{2}\int L_{g_{0}}\varphi_{i}\varphi_{i}
+
o(1) \\
= &
\alpha^{2}(c_{\omega}+\Vert \beta \Vert^{2}+o(\Vert \beta \Vert^{2}))
+
c_{0}r_{\infty}^{\frac{2-n}{2}}\sum_{i}K_{i}^{\frac{2-n}{2}}
+
o(1),
\end{split}
\end{equation} 
where $c_{\omega }=\int L_{g_{0}}\omega \omega $. On the other hand 
\begin{equation}
\begin{split}
(\frac{r}{k})_{u_{1, \beta}}
= &
\frac{\int L_{g_{0}}u_{1, \beta}u_{1, \beta}}{\int Ku_{1, \beta}^{\frac{2n}{n-2}}} \\
= &
\frac
{\int L_{g_{0}}\omega \omega +L_{g_{0}}\beta^{i}\mathrm{e}_{i}\beta^{j}\mathrm{e}_{j}+o(\Vert \beta \Vert^{2})}
{\int K\omega  ^{\frac{2n}{n-2}}+\frac{2n}{n-2}\frac{n+2}{n-2}K\omega  ^{\frac{4}{n-2}}\beta^{i}\mathrm{e}_{i}\beta^{j}\mathrm{e}_{j}+o(\Vert \beta \Vert^{2})} \\
= &
\frac
{c_{\omega}+\Vert \beta \Vert^{2}}
{c_{\omega}+\frac{2n}{n-2}\Vert \beta \Vert^{2}}
+
o(\Vert \beta \Vert^{2})
=
1-\frac{n+2}{n-2}\frac{\Vert \beta \Vert^{2}}{c_{\omega}}+o(\Vert \beta\Vert^{2})
\end{split}
\end{equation} 
whence still according to proposition \ref{prop_uniformity_in_V(omega,p,e)}
\begin{equation}
\begin{split}
\alpha^{-\frac{4}{n-2}}(1-\frac{n+2}{n-2}\frac{\Vert \beta\Vert^{2}}{c_{\omega}}+o(\Vert \beta \Vert^{2}))
=
r_{\infty}+o(1).
\end{split}
\end{equation} 
In particular $\alpha$ is fixed in terms of $\Vert \beta \Vert^{2}$ by
\begin{equation}\label{alpha_in_terms_of_beta}
\begin{split}
\alpha^{2}
=
(\frac{c_{\omega}-\frac{n+2}{n-2}\Vert \beta \Vert ^{2}+o(\Vert \beta \Vert^{2})}{c_{\omega}r_{\infty}})^{\frac{n-2}{2}}.
\end{split}
\end{equation} 
Plugging this into \eqref{Ju_expansion_along_a_flow_line} we obtain, since $J(u)=r_{\infty}+o(1)$
\begin{equation}
\begin{split}
c_{\omega}^{\frac{n-2}{2}}r_{\infty}^{\frac{n}{2}}
= &
(c_{\omega}-\frac{n+2}{n-2}\Vert \beta\Vert^{2} )^{\frac{n-2}{2}}(c_{\omega}+\Vert \beta \Vert^{2})
\\ & +
c_{0}c_{\omega}^{\frac{n-2}{2}}\sum_{i}K_{i}^{\frac{2-n}{2}}
+
o(1)+o(\Vert \beta \Vert^{2}) \\
= &
c_{\omega}^{\frac{n}{2}}-\frac{n}{2}c_{\omega}^{\frac{n-2}{2}}(1+o(1))\Vert \beta\Vert^{2} 
+
c_{0}c_{\omega}^{\frac{n-2}{2}}\sum_{i}K_{i}^{\frac{2-n}{2}}
+
o(1).
\end{split}
\end{equation}
Thus, if
$\Vert \beta \Vert^{2}$ increases significantly, then $\sum_{i}K_{i}^{\frac{2-n}{2}}$ has to increase significantly as well. But 
\begin{equation}
\begin{split}
\partial_{t}K_{i}^{\frac{2-n}{2}}
= &
\frac{2-n}{2}K_{i}^{-\frac{n}{2}}\frac{\nabla K_{i}}{\lambda_{i}}\lambda_{i}\dot a_{i} 
\\
\leq &
-c \frac{\vert \nabla K_{i}\vert^{2}}{\lambda_{i}^{2}}
+
O(\sum_{i}\frac{1}{\lambda_{i}^{2(2-n)}}+\sum_{r\neq s}\eps_{r,s}^{2}+\vert \delta J(u)\vert^{2})
\end{split}
\end{equation} 
due to corollaries \ref{cor_simplifying_the_shadow_flow}, \ref{cor_simplifying_the_shadow_flow_w}, whence 
\begin{equation}
\partial_{t}K_{i}^{\frac{2-n}{2}}
\leq 
O(\vert \delta J(u)\vert^{2})
\end{equation} 
due to definition \ref{def_principally_lower_bounded_of_the_first_variation}.
If the proposition were false, there would exist
\begin{equation*}
s_{0}<s_{0}'<s_{1}<s_{1}'<\ldots <s_{n}<s_{n}'<\ldots
\end{equation*} 
such, that $u\lfloor _{[s_{k},s_{k}']}\in V(\omega, p, \eps_{0})$ and 
\begin{equation}
\begin{split}
u(s_{k})\in V(\omega, p, \eps_{k}), \eps_{k}\- 0, \, u(s_{k}') \in \partial V(\omega, p, \eps_{0}).
\end{split}
\end{equation} 
However due to proposition \ref{prop_uniformity_in_V(omega,p,e)} we may assume
\begin{equation}
\begin{split}
\frac{1}{\lambda_{i}}, \eps_{i,j},1-\frac{r_{\infty}\alpha_{i}^{\frac{4}{n-2}}K_{i}}{4n(n-1)}, \Vert v \Vert
\leq \eps_{k}
\; \text{ during }\; (s_{k},s_{k}').
\end{split}
\end{equation} 
Thus by the very definition \ref{def_V(omega,p,e)} of $V(\omega, p, \eps)$ the only possibility for $u$ to escape from $V(\omega,p, \varepsilon_{0})$ during $(s_{k},s_{k}')$ is, that $\vert 1-r_{\infty}\alpha^{\frac{4}{n-2}}\vert$ or $\Vert \beta \Vert$ has 
to increase during $(s_{k},s_{k}')$ for at least a quantity $\eps_{0}-\eps_{k}$.
This possibility has already been ruled out for $\Vert \beta\Vert $ and is thus as well for 
$\vert 1-r_{\infty}\alpha^{\frac{4}{n-2}}\vert $
by \eqref{alpha_in_terms_of_beta}.
\end{proof}
The only lack in the discussion so far is a missing compactness result on the blow up points.
A straight forward use of the evolution equations given by corollaries 
\ref{cor_simplifying_the_shadow_flow} and \ref{cor_simplifying_the_shadow_flow_w} 
provides at least a weak form of convergence.

\begin{lemma}[Critical points of $K$ as attractors]\label{lem_critical_points_of_K_as_attractors}$_{}$

Suppose $\partial J$ to be principally lower bounded. \\
We then have
\begin{equation*}\begin{split}
K(a_{i})\-K_{i_{\infty}}\; \text{ and }\; \vert \nabla K(a_{i})\vert\-0\;\text{ for all }\; i=1, \ldots,p
\end{split}\end{equation*}
for every flow line $u\in V(\omega, p, \eps)$ converging to a critical point at infinity.
\end{lemma}
\begin{proof}[\textbf{Proof of lemma \ref{lem_critical_points_of_K_as_attractors}}]$_{}$\\
In case $\partial J$ is principally lower bounded
lemmata \ref{cor_simplifying_the_shadow_flow} and \ref{cor_simplifying_the_shadow_flow_w} show
\begin{equation}\begin{split}
\partial_{t}\sum_{i}K_{i}
=
\sum_{i}\frac{\nabla K_{i}}{\lambda_{i}}\lambda \dot a_{i}=O(\vert \delta J(u)\vert^{2})
\end{split}\end{equation}
 As a consequence
\begin{equation}
\begin{split}
K_{i}=K(a_{i})\-K_{i_{\infty}}\; \text{ for all }\; i=1, \ldots,p.
\end{split}
\end{equation} 
Then still according to lemmata \ref{cor_simplifying_the_shadow_flow} and \ref{cor_simplifying_the_shadow_flow_w} we observe
\begin{equation*}\begin{split}
\partial_{t}\vert \nabla K_{i}\vert^{2}
=
2\frac{\nabla^{2}K_{i}(\nabla K_{i}, \lambda_{i}\dot a_{i})}{\lambda_{i}}
=
O( \vert \delta J(u)\vert^{2}),
\end{split}\end{equation*}
whence $\vert \nabla K_{i}\vert \- c_{i_{\infty}}$. 
Letting
\begin{equation}
P=\{1, \ldots,p\}, Q=\{i\in P\mid c_{i_{\infty}}\neq 0\}\; \text{ and }\; q=\sharp\{Q\}
\end{equation} 
we may assume without loss of generality, that 
\begin{equation}
Q=\{1, \ldots,q\}\;\text{ and }\;\min_{i\in Q,j\in P\setminus q}{d(a_{i},a_{j})}>\epsilon_{0}>0
\end{equation} 
We then reorder, if necessary, the elements of $q$ by
\begin{equation}
\begin{split}
\frac{1}{K_{1}}\ln \frac{1}{\lambda_{1}}\geq \ldots \geq  \frac{1}{K_{q}}\ln\frac{1}{\lambda_{q}} 
.
\end{split}
\end{equation} 
In case $u\in V(p, \eps)$ we consider 
$\psi 
=
\sum_{i=1}^{q}\frac{C^{i}}{K_{i}}\ln\frac{1}{\lambda_{i}}.
$
Then corollary \ref{cor_simplifying_the_shadow_flow} gives
\begin{equation}
\begin{split}
\psi'
= &
\sum_{i=1}^{q}\frac{C^{i}}{K_{i}}
[
\frac{\ln\lambda_{i}}{ \lambda_{i}}\frac{\nabla K_{i}}{K_{i}}\lambda_{i}\dot a_{i}
-
\frac{\dot \lambda_{i}}{\lambda_{i}}
] \\
\geq &
c\sum_{i=1}^{q}\frac{C^{i}}{K_{i}}
[
\gamma_{1}\frac{\vert \nabla K_{i}\vert^{2}}{K_{i}^{2}}\frac{\ln \lambda_{i}}{\lambda_{i}^{2}}
+
\gamma_{2}\frac{H_{i}}{\lambda_{i}^{n-2}}
+
\gamma_{3}\frac{\lap K_{i}}{\lambda_{i}^{2}} \\
& \quad\quad\quad\quad\quad\quad\quad\quad\quad\quad -
\gamma_{4}
\sum_{i\neq j=1}^{p}\frac{\alpha_{j}}{\alpha_{i}}\lambda_{i}\partial_{\lambda_{i}}\eps_{i,j}
] (1+o_{\frac{1}{\lambda_{i}}}(1))\\
& +
o_{\varepsilon}(\sum_{i=1}^{q}\frac{1}{\lambda_{i}^{n-2}}+\sum_{i=1}^{q}\sum_{i\neq j=1}^{p}\eps_{i,j})
+
O(\vert \delta J(u)\vert)^{2},
\end{split}
\end{equation} 
where we made use of the principal lower boundedness of $\partial J$. 
We obtain
\begin{equation}
\begin{split}
\psi'
\geq &
-c(1+o_{\frac{1}{\lambda_{i}}}(1))\sum_{\underset{i\in Q}{i\neq j}}\frac{C^{i}}{K_{i}}
\frac{\alpha_{j}}{\alpha_{i}}\lambda_{i}\partial_{\lambda_{i}}\eps_{i,j}
+
o
(
\sum_{\underset{i\in Q}{i\neq j}}\eps_{i,j})
+
O(\vert \delta J(u)\vert^{2})
\end{split}
\end{equation} 
by definition of $q$.
Note, that for $i\in Q$ and $j\in P\setminus Q$ we may assume
\begin{equation}
\begin{split}
-\lambda_{i}\partial_{\lambda_{i}}\eps_{i,j}
=
\frac{n-2}{2}
\frac
{
\frac{\lambda_{i}}{\lambda_{j}}
-
\frac{\lambda_{j}}{\lambda_{i}}
+
\lambda_{i}\lambda_{j}\gamma_{n}G_{g_{0}}^{\frac{2}{2-n}}(a_{i},a_{j})
}
{
(
\frac{\lambda_{i}}{\lambda_{j}}
+
\frac{\lambda_{j}}{\lambda_{i}}
+
\lambda_{i}\lambda_{j}\gamma_{n}G_{g_{0}}^{\frac{2}{2-n}}(a_{i},a_{j})
)^{\frac{n}{2}}
}
\geq
\frac{n-2}{4}\eps_{i,j},
\end{split}
\end{equation} 
since in that case $d(a_{i},a_{j})\geq \varepsilon_{0} >0$, and we obtain
\begin{equation}\label{psi'_attractors}
\begin{split}
\psi'
\geq &
-c(1+o_{\frac{1}{\lambda_{i}}}(1))\sum_{\underset{i\in Q}{i\neq j}}\frac{C^{i}}{K_{i}}
\frac{\alpha_{j}}{\alpha_{i}}\lambda_{i}\partial_{\lambda_{i}}\eps_{i,j}
+
o(\sum_{\underset{i\in Q}{i\neq j}}\eps_{i,j})
+
O(\vert \delta J(u)\vert^{2}).
\end{split}\end{equation}
Moreover  for sufficiently small $\varepsilon>0$ and $C>1$ large we have 
\begin{equation}\label{eij_positive _argument_attractors}
\begin{split}
-
\sum_{\underset{i,j\in Q}{i\neq j}}\frac{C^{i}}{K_{i}}\frac{\alpha_{j}}{\alpha_{i}}
\lambda_{i}\partial_{\lambda_{i}}\eps_{i,j}
\geq c\sum_{\underset{i,j\in Q}{i> j}}\eps_{i,j}.
\end{split}
\end{equation} 

To prove \eqref{eij_positive _argument_attractors} note, that by definition  we have
\begin{equation}
\begin{split}
(C^{i}-C^{j})\frac{\ln\frac{1}{\lambda_{i}}}{K_{i}}
\leq 
(C^{i}-C^{j})\frac{\ln\frac{1}{\lambda_{j}}}{K_{j}}
\end{split}
\end{equation}

for any  $i>j$ with $i,j\in Q$ or equivalently
\begin{equation}
\begin{split}
& \frac{C^{i}-C^{j}}{K_{i}}\ln  \frac{1}{\lambda_{i}} 
+
\frac{C^{j}-C^{i}}{K_{j}}\ln  \frac{1}{\lambda_{j}} 
\leq  
0.
\end{split}
\end{equation}

We then have
\begin{equation}\label{li/lj=o(...)_remark_attractors}
\begin{split}
\frac{\lambda_{j}}{\lambda_{i}}
=
o
(
\frac{\lambda_{i}}{\lambda_{j}}
+
\lambda_{i}\lambda_{j}\gamma_{n}G^{\frac{2}{2-n}}(a_{i},a_{j})
).
\end{split}
\end{equation} 

Otherwise we may assume for some $c>0$
\begin{equation}
\begin{split}
\frac{\lambda_{j}}{\lambda_{i}}
\geq c
(
\frac{\lambda_{i}}{\lambda_{j}}
+
\lambda_{i}\lambda_{j}\gamma_{n}G^{\frac{2}{2-n}}(a_{i},a_{j})
).
\end{split}
\end{equation} 

This implies $\frac{\lambda_{j}}{\lambda_{i}}\gg 1 \gg \frac{\lambda_{i}}{\lambda_{j}}$ and
 $d(a_{i},a_{j})=O(\frac{1}{\lambda_{i}})$. Consequently
\begin{equation}
\begin{split}
& \frac{C^{i}-C^{j}}{K_{j}}\ln  \frac{\lambda_{j}}{\lambda_{i}} 
\leq  
O(\frac{\ln \lambda_{i}}{\lambda_{i}}),
\end{split}
\end{equation}

yielding a contradiction. Thus \eqref{li/lj=o(...)_remark_attractors}
for $i>j$ is established. Write 
\begin{equation}\begin{split}
- & \sum_{\underset{i,j\in Q}{i\neq j}} C^{i}  \frac{\alpha_{j}}{\alpha_{i}}\lambda_{i}\partial_{\lambda_{i}}\eps_{i,j} 
= 
-
\sum_{\underset{i,j\in Q}{i> j}}C^{i}\frac{\alpha_{j}}{\alpha_{i}}\lambda_{i}\partial_{\lambda_{i}}\eps_{i,j}
+
\sum_{\underset{i,j\in Q}{i< j}}C^{i}\frac{\alpha_{j}}{\alpha_{i}}\lambda_{j}\partial_{\lambda_{j}}\eps_{i,j}\\
& \quad\quad\quad\quad\quad\quad\quad\quad\;\;\; -
\sum_{\underset{i,j\in Q}{i< j}}C^{i}\frac{\alpha_{j}}{\alpha_{i}}\lambda_{i}\partial_{\lambda_{i}}\eps_{i,j}
-
\sum_{\underset{i,j\in Q}{i< j}}C^{i}\frac{\alpha_{j}}{\alpha_{i}}\lambda_{j}\partial_{\lambda_{j}}\eps_{i,j}
\\
& = 
-
\sum_{\underset{i,j\in Q}{i> j}}
[
C^{i}\frac{\alpha_{j}}{\alpha_{i}}
-
C^{j}\frac{\alpha_{i}}{\alpha_{j}}
]
\lambda_{i}\partial_{\lambda_{i}}\eps_{i,j} 
-
\sum_{\underset{i,j\in Q}{i< j}}C^{i}\frac{\alpha_{j}}{\alpha_{i}}
[
\lambda_{i}\partial_{\lambda_{i}}\eps_{i,j}
+
\lambda_{j}\partial_{\lambda_{j}}\eps_{i,j}
].
\end{split}\end{equation}

We have
\begin{equation}\begin{split}
-
\lambda_{i}\partial_{\lambda_{i}}\eps_{i,j}
-
\lambda_{j}\partial_{\lambda_{j}}\eps_{i,j}=(n-2)\eps_{i,j}^{\frac{n}{n-2}}\lambda_{i}\lambda_{j}\gamma_{n}G^{\frac{2}{2-n}}( a _{i}, a _{j})
>0
\end{split}\end{equation}

and for $i>j$ due to \eqref{li/lj=o(...)_remark_attractors}
\begin{equation}\begin{split}
-\lambda_{i}\partial_{\lambda_{i}}\eps_{i,j}
= &
\frac{n-2}{2}\eps_{i,j}^{\frac{n}{n-2}}
(
\frac{\lambda_{i} }{ \lambda_{j} }
-
\frac{ \lambda_{j} }{ \lambda_{i} }
+
\lambda_{i}\lambda_{j}\gamma_{n}G^{\frac{2}{2-n}}( a _{i}, a _{j}))\geq 
\frac{n-2}{4}\eps_{i,j}.
\end{split}\end{equation}

This shows \eqref{eij_positive _argument_attractors}.
\\
Thus plugging \eqref{eij_positive _argument_attractors} into
\eqref{psi'_attractors} shows
$\psi'\geq O(\vert \delta J(u)\vert^{2})$ for $C>1$ sufficiently large, whereas $\psi\- -\infty$ 
by definition 
as a continuous, piecewise differentiable function in time;
a contradiction.

The case $u\in V(\omega,p, \varepsilon)$ is proven analogously.
\end{proof}

The following lemma assures $\partial J$ to be principally lower bounded in the case the dimensional conditions $Cond_{n}$, on which theorem \ref{thm_1} relies, hold true.
\begin{proposition}[Principal lower bound of the first variation under $Cond_{n}$]\label{prop_princ_lower_bounded_under_Cond_n}$_{}$
$\partial J$ is principally lower bounded, 
if $Cond_{n}$ as in definition \ref{def_dimensional_conditions} is satisfied.
\end{proposition}

\begin{proof}[\textbf{Proof of proposition \ref{prop_princ_lower_bounded_under_Cond_n}}]$_{}$\\
In case $\omega =0$ corollaries \ref{cor_simplifying_ski}, \ref{cor_a-priori_estimate_on_v} and  \eqref{rai^{...}/k=...} show, that
\begin{enumerate}[label=(\roman*)]
 \Item $_{}$
\begin{equation}\begin{split}
\sigma_{2,i}
= &
\tilde \gamma_{1}\alpha_{i}\frac{ H_{i}}{\lambda_{i} ^{n-2}}
+
\gamma_{2}\alpha_{i}\frac{\lap K_{i}}{K_{i}\lambda_{i}^{2}} 
- 
\tilde \gamma_{5}b_{2}
\sum_{i \neq j=1}^{p}\alpha_{j}
\lambda_{i}\partial_{\lambda_{i}}\eps_{i,j} 
+
R_{2,i}
\end{split}\end{equation}
 \Item $_{}$
\begin{equation}\begin{split}
\sigma_{3,i}
= &
\tilde \gamma_{3}\alpha_{i}\frac{\nabla K_{i}}{K_{i}\lambda_{i}}
+
\gamma_{4}\alpha_{i}\frac{\nabla \lap K_{i}}{K_{i}\lambda_{i}^{3}} 
+
\gamma_{6}
\sum_{i \neq j=1}^{p}
\frac{\alpha_{j}}{\lambda_{i}}\nabla_{a_{i}}\eps_{i,j} +
R_{3,i},
\end{split}\end{equation}
\end{enumerate}
where  
\begin{equation}
\begin{split}
R_{k,i}
= &
o_{\varepsilon}
(
\frac{1}{\lambda_{i}^{n-2}} 
+
\sum_{i\neq j=1}^{q}\eps_{i,j}
)
\\
& +
O
(
\sum_{r} \frac{\vert \nabla K_{r}\vert^{2}}{\lambda_{r}^{2}} 
+
\frac{\vert \lap K_{r}\vert^{2}}{\lambda_{r}^{4}}
+
\frac{1}{\lambda_{r}^{2(n-2)}} +\sum_{r\neq s}\eps_{r,s}^{2}
+
\vert \delta J(u)\vert^{2}
)
.                               
\end{split}
\end{equation} 
Letting $0<\underline \kappa\leq \kappa_{i} \leq \overline\kappa <\infty$ for $\vert \nabla K_{i}\vert \neq 0$
and $\kappa_{i}=0$ for $\vert \nabla K_{i}\vert=0$ we get
\begin{equation}\label{testfunction_unicity>...}
\begin{split}
\sum_{i} & C^{i}(\sigma_{2,i}
+\kappa_{i}\langle \sigma_{3,i}, \frac{\nabla K_{i}}{\vert \nabla K_{i}\vert}\rangle ) \\
\geq &
\sum_{i}\alpha_{i}C^{i}
[
\gamma_{1}\frac{H_{i}}{\lambda_{i}^{n-2}}
+
\gamma_{2}\frac{\lap K_{i}}{K_{i}\lambda_{i}^{2}}
+
\gamma_{3}\kappa_{i}\frac{\vert \nabla K_{i}\vert}{K_{i}\lambda_{i}}
+
\gamma_{4}\kappa_{i}\frac{\langle \nabla \lap K_{i}, \nabla K_{i}\rangle}{K_{i}\vert \nabla K_{i}\vert\lambda_{i}^{3}}
] \\
& -
\tilde \gamma_{5}\sum_{i\neq j}C^{i}\alpha_{j}
\lambda_{i}\partial_{\lambda_{i}} \eps_{i,j}
+
o_{\varepsilon}(\sum_{r\neq s}\eps_{r,s})
+
O
(
\sum_{i\neq j}\frac{C^{i}}{\lambda_{i}}\vert \nabla_{a_{i}}\eps_{i,j}\vert
)\\
& +
O
(
\frac{\vert \lap K_{r}\vert^{2}}{\lambda_{r}^{4}}
+
\vert \delta J(u)\vert^{2}
)
.
\end{split}
\end{equation} 
Note, that we do not try to construct a continuous pseudo gradient, so there is no need to choose 
$\kappa_{i}$ continuously.
As before we order 
\begin{equation}
\frac{1}{\lambda_{1}}\geq \ldots\geq \frac{1}{\lambda_{p}}.
\end{equation} 
We then have for sufficiently small
$\varepsilon>0$ and $C>1$  large
\begin{equation}\label{eij_large_unicity}
\begin{split}
\sum_{i\neq j}C^{i}\alpha_{j}\lambda_{i}\partial_{\lambda_{i}}\eps_{i,j}
\geq c \sum_{i>j}C^{i}\eps_{i,j}
\end{split}
\end{equation}
and
\begin{equation}\label{eij_small_unicity}
\begin{split}
\sum_{i\neq j}\frac{C^{i}}{\lambda_{i}}\vert \nabla_{a_{i}}\eps_{i,j}\vert =O(\sum_{i>j}C^{j}\eps_{i,j})
\end{split}
\end{equation}

To prove \eqref{eij_large_unicity} and \eqref{eij_small_unicity} note, that
\begin{equation}\begin{split}
- & \sum_{i\neq j} C^{i}  \frac{\alpha_{j}}{\alpha_{i}}\lambda_{i}\partial_{\lambda_{i}}\eps_{i,j} 
= 
-
\sum_{i>j}C^{i}\frac{\alpha_{j}}{\alpha_{i}}\lambda_{i}\partial_{\lambda_{i}}\eps_{i,j}
+
\sum_{i<j}C^{i}\frac{\alpha_{j}}{\alpha_{i}}\lambda_{j}\partial_{\lambda_{j}}\eps_{i,j}\\
& \quad\quad\quad\quad\quad\quad\quad\quad\;\;\; -
\sum_{i<j}C^{i}\frac{\alpha_{j}}{\alpha_{i}}\lambda_{i}\partial_{\lambda_{i}}\eps_{i,j}
-
\sum_{i<j}C^{i}\frac{\alpha_{j}}{\alpha_{i}}\lambda_{j}\partial_{\lambda_{j}}\eps_{i,j}
\\
& = 
-
\sum_{i>j}
[
C^{i}\frac{\alpha_{j}}{\alpha_{i}}
-
C^{j}\frac{\alpha_{i}}{\alpha_{j}}
]
\lambda_{i}\partial_{\lambda_{i}}\eps_{i,j} 
-
\sum_{i<j}C^{i}\frac{\alpha_{j}}{\alpha_{i}}
[
\lambda_{i}\partial_{\lambda_{i}}\eps_{i,j}
+
\lambda_{j}\partial_{\lambda_{j}}\eps_{i,j}
].
\end{split}\end{equation}

One has
\begin{equation}\begin{split}
-
\lambda_{i}\partial_{\lambda_{i}}\eps_{i,j}
-
\lambda_{j}\partial_{\lambda_{j}}\eps_{i,j}=(n-2)\eps_{i,j}^{\frac{n}{n-2}}\lambda_{i}\lambda_{j}\gamma_{n}G^{\frac{2}{2-n}}( a _{i}, a _{j})
>0
\end{split}\end{equation}

and for $i>j$
\begin{equation}\begin{split}
-\lambda_{i}\partial_{\lambda_{i}}\eps_{i,j}
= &
\frac{n-2}{2}\eps_{i,j}^{\frac{n}{n-2}}
(
\frac{\lambda_{i} }{ \lambda_{j} }
-
\frac{ \lambda_{j} }{ \lambda_{i} }
+
\lambda_{i}\lambda_{j}\gamma_{n}G^{\frac{2}{2-n}}( a _{i}, a _{j}))\geq 
\frac{n-2}{4}\eps_{i,j}.
\end{split}\end{equation}

Thus \eqref{eij_large_unicity} is proven. We are left with estimating 
\begin{equation}
\begin{split}
\sum_{i\neq j}\frac{C^{i}}{\lambda_{i}}\vert \nabla_{a_{i}}\eps_{i,j}\vert 
= &
\frac{n-2}{2}\sum_{i< j}C^{i}\eps_{i,j}
\vert
\frac
{(\frac{\lambda_{j}}{\lambda_{i}})^{\frac{1}{2}}(\lambda_{i}\lambda_{j})^{\frac{1}{2}}\gamma_{n}\nabla_{ a_{i}}G^{\frac{2}{2-n}}(a_{i},a_{j})}
{\frac{\lambda_{i}}{\lambda_{j}}+\frac{\lambda_{j}}{\lambda_{i}}
+
\lambda_{i}\lambda_{j}\gamma_{n}G^{\frac{2}{2-n}}(a_{i},a_{j})} 
\vert \\
& +
o(\sum_{i\neq j}\eps_{i,j}),
\end{split} 
\end{equation}

whence we immediately obtain \eqref{eij_small_unicity}.
\\
Plugging \eqref{eij_small_unicity} and \eqref{eij_small_unicity} into 
\eqref{testfunction_unicity>...} we obtain for $C>1$ sufficiently large
\begin{equation}
\begin{split}
\sum_{i} & C^{i}(\sigma_{2,i}
+\kappa_{i} \langle \sigma_{3,i}, \frac{\nabla K_{i}}{\vert \nabla K_{i}\vert}\rangle ) \\
\geq &
\sum_{i}\alpha_{i}C^{i}
[
\gamma_{1}\frac{H_{i}}{\lambda_{i}^{n-2}}
+
\gamma_{2}\frac{\lap K_{i}}{K_{i}\lambda_{i}^{2}}
+
\gamma_{3}\kappa_{i}\frac{\vert \nabla K_{i}\vert}{K_{i}\lambda_{i}}
+
\gamma_{4}\kappa_{i}\frac{\langle \nabla \lap K_{i}, \nabla K_{i}\rangle}{K_{i}\vert \nabla K_{i}\vert \lambda_{i}^{3}\lambda_{i}^{3}}
] \\
& +
\gamma_{5}\sum_{i> j}C^{i}\eps_{i,j}
+
O
(
\frac{\vert \lap K_{r}\vert^{2}}{\lambda_{r}^{4}}
+
\vert \delta J(u)\vert^{2}
).
\end{split}
\end{equation} 
In case $\lap K_{i}\geq 0$ or $\vert \nabla K_{i}\vert >\epsilon $ for $\epsilon >0$ small 
we immediately obtain
\begin{equation}\label{non_degeneracy_under_cond_n}
\begin{split}
\gamma_{i}\frac{H_{i}}{\lambda_{i}^{n-2}}
+
\gamma_{2}\frac{\lap K_{i}}{K_{i}\lambda_{i}^{2}}
+
\gamma_{3}\kappa_{i}\frac{\vert \nabla K_{i}\vert}{K_{i}\lambda_{i}}
& +
\gamma_{4}\kappa_{i}\frac{\langle \nabla \lap K_{i}, \nabla K_{i}\rangle}{K_{i}\vert \nabla K_{i}\vert \lambda_{i}^{3}} \\
\geq & 
c
[
\frac{H_{i}}{\lambda_{i}^{n-2}}
+
\frac{\vert \lap K_{i}\vert}{K_{i}\lambda_{i}^{2}}
+
\frac{\vert \nabla K_{i}\vert}{K_{i}\lambda_{i}}
]
\end{split}
\end{equation} 
for some $c>0$ and all $\lambda_{i}>0$ sufficiently large choosing $\kappa_{i}$ such, that
\begin{equation}
\begin{split}
\gamma_{i}\frac{H_{i}}{\lambda_{i}^{n-2}} 
+
\gamma_{4}\kappa_{i}\frac{\langle \nabla \lap K_{i}, \nabla K_{i}\rangle}{K_{i}\vert \nabla K_{i}\vert \lambda_{i}^{3}} 
\geq 
c\frac{H_{i}}{\lambda_{i}^{n-2}} 
\end{split}
\end{equation} 
Moreover \eqref{non_degeneracy_under_cond_n} holds true as well for $n=3$ and by $Cond_{4}$ for $n=4.$ For
\begin{equation}
\begin{split}
n=5, \; \lap K_{i}<0\; \text{ and }\; \vert \nabla K_{i} \vert<\eps 
\end{split}
\end{equation}
we may according to $Cond_{5}$ assume, that $\langle\nabla \lap K_{i}, \nabla K_{i}\rangle>\frac{1}{3}\vert \lap K_{i}\vert^{2}$. Thus
\begin{equation}
\begin{split}
\frac{\lap K_{i}}{K_{i}\lambda_{i}^{2}}
> &
-\frac{3}{2}\frac{\vert \nabla K_{i}\vert}{K_{i}\lambda_{i}}
-\frac{3}{2}
\frac
{
\langle \nabla \lap K_{i}, \nabla K_{i}\rangle
}
{
K_{i}\vert \nabla K_{i}\vert \lambda_{i}^{3}
}.
\end{split}
\end{equation} 
Choosing therefore $\kappa_{i}$ such, that 
$
\frac{3}{2}\gamma_{2}<\gamma_{3}\kappa_{i}, \frac{3}{2}\gamma_{2}<\gamma_{4}\kappa_{i},
$
 then \eqref{non_degeneracy_under_cond_n} holds true as well
and thus in any case. We conclude
\begin{equation}
\begin{split}
\sum_{i} & C^{i}(\sigma_{2,i}
+\kappa_{i} \langle \sigma_{3,i}, \frac{\nabla K_{i}}{\vert \nabla K_{i}\vert}\rangle ) \\
\geq &
c\sum_{i}
[
\frac{H_{i}}{\lambda_{i}^{n-2}}
+
\frac{\vert \lap K_{i}\vert}{K_{i}\lambda_{i}^{2}}
+
\frac{\vert \nabla K_{i}\vert}{K_{i}\lambda_{i}}
] 
+
c\sum_{i> j}\eps_{i,j}
+
O(\vert \delta J(u)\vert^{2}).
\end{split}
\end{equation} 
Since $\sigma_{k,i}=O(\vert \delta J(u)\vert)$ by definition, the claim follows.

In case $\omega >0$ we have
due to corollaries \ref{cor_simplifying_ski_f}, \ref{cor_a-priori_estimate_on_v_f} and 
\eqref{rai4/n-2_f}
\begin{enumerate}[label=(\roman*)]
 \Item 
 \begin{equation}
\begin{split}
\sigma_{2,i}
= &
\tilde \gamma_{1}\alpha\frac{\omega_{i}}{K_{i}\lambda_{i}^{\frac{n-2}{2}}}
-
\tilde \gamma_{3}\sum_{i\neq j=1}^{p}\alpha_{j}\lambda_{i}\partial_{\lambda_{i}}\eps_{i,j}
+
R_{2,i}
\end{split}
\end{equation} 
 \Item 
\begin{equation}
\begin{split}
\sigma_{3,i}
=
\tilde \gamma_{2}\alpha_{i}\frac{\nabla K_{i}}{K_{i}\lambda_{i}}
+
\gamma_{4}
\sum_{i\neq j=1}^{p}\alpha_{j}\frac{1}{\lambda_{i}}\nabla_{a_{i}}\eps_{i,j}
+
R_{3,i}
\end{split} 
\end{equation} 
\end{enumerate}

where
\begin{equation}
\begin{split}
R_{k,i}
=
o_{\varepsilon}(\sum_{r}\frac{1}{\lambda_{r}^{\frac{n-2}{2}}}+\sum_{r\neq s}\eps_{r,s})
+
O(\vert \delta J(u)\vert^{2})
\end{split}
\end{equation} 

and the same arguments apply in a simpler way.
\end{proof}

\subsection{Leaving V(\textomega, p, \textepsilon)}
\label{subsec:LeavingVwpe}
In this subsection we consider a flow line 
$$
u=u_{\alpha, \beta}+\alpha^{i}\varphi_{i}+v\in V(\omega,p, \eps)
$$
and we wish to define piecewise differentiable continuous 
function in time
$$
\psi:(a_{i}, \lambda_{i})_{i=1, \ldots,p}\-\psi((a_{i}, \lambda_{i})_{i=1, \ldots,p})
$$
with the fundamental properties
\begin{enumerate}[label=(\roman*)]
 \Item 
$$
\psi\--\infty\; \text{ as }\;\lambda_{i}\- \infty\;\text{ for some }\;  i=1, \ldots,p
$$
 \Item 
$$
\psi'\in L^{1}(\R_{+})\; \text{ is integrable in time}.
$$
\end{enumerate}
The existence of such a function implies, that a flow line cannot at once remain in $V(\omega,p, \varepsilon)$ for all times  
and concentrate in the sense, that $\lambda_{i}\-\infty$.

The subsequent propositions are devoted to prove their existence under the 

dimensional conditions $Cond_{n}$, cf. definition
\ref{def_dimensional_conditions}.

\begin{proposition}[Case $n=3, \, \omega=0$]\label{prop_n=3}$_{}$\\
Let $n=3$ and $Cond_{3}$ hold true. 
Ordering
\begin{equation*}
\begin{split}
\frac{1}{\lambda_{1}}\geq \ldots \geq \frac{1}{\lambda_{p}} 
\end{split}
\end{equation*} 
the piecewise differentiable continuous function
$
\psi=\sum_{i}C^{i}\ln  \frac{1}{\lambda_{i}}
$
satisfies 
\begin{equation*}\begin{split}
\psi'
\geq &
\sum_{i} \frac{H_{i}}{\lambda_{i}} 
+
\sum_{i> j}\eps_{i,j}
+
O(\vert \delta J(u)\vert^{2}),
\end{split}\end{equation*}
provided $C>1$ is sufficiently large
\end{proposition}
In view of corollary \ref{cor_simplifying_the_shadow_flow} the positive  sign of the mass  
related terms $\frac{H_{i}}{\lambda_{i}}$ is rather obvious and the
ordering $\frac{1}{\lambda_{1}}\geq \ldots \geq \frac{1}{\lambda_{p}}$ and choice of $C\gg 1$ ensure, 
that the interaction related terms are of positive  sign as well. 

\begin{proof}[\textbf{Proof of proposition \ref{prop_n=3}}]\label{p_n=3}$_{}$\\
As $M$ is not conformally equivalent to the standard sphere $\mathbb{S}^{3}$, the positive  mass theorem holds. Thus $H_{i}>0$ in the statement of corollary \ref{cor_simplifying_the_shadow_flow}
\begin{equation}\begin{split}
-\frac{\dot \lambda_{i} }{\lambda_{i} }
= &
\frac{r}{k}
[
\gamma_{0}\frac{H_{i}}{\lambda_{i}}
-
\gamma_{1}\sum_{i \neq j=1}^{p}\frac{\alpha_{j}}{\alpha_{i}}
\lambda_{i}\partial_{\lambda_{i}}\eps_{i,j}]
(1+o_{\frac{1}{\lambda_{i}}}(1))
\\
& +
o_{\varepsilon}(\sum_{r} \frac{1}{\lambda_{r}} +\sum_{r\neq s}\eps_{r,s})
+
O(\vert \delta J(u)\vert^{2})
\end{split}\end{equation}
for suitable $\gamma_{0}, \gamma_{1}>0$. Then for 
$
\psi=\sum_{i}C^{i}\ln  \frac{1}{\lambda_{i}}, \; C>1
$
there holds
\begin{equation}
\begin{split}
\psi'
= &
\frac{r}{k}
[
\tilde \gamma_{0}\sum_{i}C^{i}\frac{H_{i}}{\lambda_{i}}
-
\gamma_{1}\sum_{i\neq j}C^{i}\frac{\alpha_{j}}{\alpha_{i}}
\lambda_{i}\partial_{\lambda_{i}}\eps_{i,j}
] (1+o_{\frac{1}{\lambda_{i}}}(1))\\
& +
o_{\varepsilon}(\sum_{r\neq s }\eps_{r,s})
+
O(\vert \delta J(u)\vert^{2}).
\end{split}
\end{equation} 
We complete the definition of $\psi$ by ordering
\begin{equation}
\begin{split}
\ln \frac{1}{\lambda_{1}}\geq \ldots \geq \ln \frac{1}{\lambda_{p}}
\end{split}
\end{equation} 
and claim, that there exists $c >0$ such, that for any $C>1$ sufficiently large
\begin{equation}\label{eij_positive _argument}
\begin{split}
-
\gamma_{1}\sum_{i\neq j}C^{i}\frac{\alpha_{j}}{\alpha_{i}}
\lambda_{i}\partial_{\lambda_{i}}\eps_{i,j}
\geq c \sum_{i>j}C^{i}\eps_{i,j}.
\end{split}
\end{equation} 
Readily the statement of the proposition follows from this fact.

To prove \eqref{eij_positive _argument} note, that
\begin{equation}\begin{split}
- & \sum_{i\neq j} C^{i}  \frac{\alpha_{j}}{\alpha_{i}}\lambda_{i}\partial_{\lambda_{i}}\eps_{i,j} 
= 
-
\sum_{i>j}C^{i}\frac{\alpha_{j}}{\alpha_{i}}\lambda_{i}\partial_{\lambda_{i}}\eps_{i,j}
+
\sum_{i<j}C^{i}\frac{\alpha_{j}}{\alpha_{i}}\lambda_{j}\partial_{\lambda_{j}}\eps_{i,j}\\
& \quad\quad\quad\quad\quad\quad\quad\quad\;\;\; -
\sum_{i<j}C^{i}\frac{\alpha_{j}}{\alpha_{i}}\lambda_{i}\partial_{\lambda_{i}}\eps_{i,j}
-
\sum_{i<j}C^{i}\frac{\alpha_{j}}{\alpha_{i}}\lambda_{j}\partial_{\lambda_{j}}\eps_{i,j}
\\
& = 
-
\sum_{i>j}
[
C^{i}\frac{\alpha_{j}}{\alpha_{i}}
-
C^{j}\frac{\alpha_{i}}{\alpha_{j}}
]
\lambda_{i}\partial_{\lambda_{i}}\eps_{i,j} 
-
\sum_{i<j}C^{i}\frac{\alpha_{j}}{\alpha_{i}}
[
\lambda_{i}\partial_{\lambda_{i}}\eps_{i,j}
+
\lambda_{j}\partial_{\lambda_{j}}\eps_{i,j}
].
\end{split}\end{equation}

One has
\begin{equation}\begin{split}
-
\lambda_{i}\partial_{\lambda_{i}}\eps_{i,j}
-
\lambda_{j}\partial_{\lambda_{j}}\eps_{i,j}=(n-2)\eps_{i,j}^{\frac{n}{n-2}}\lambda_{i}\lambda_{j}\gamma_{n}G^{\frac{2}{2-n}}( a _{i}, a _{j})
>0
\end{split}\end{equation}

and for $ \frac{\lambda_{j}}{\lambda_{i}} \leq 1,$ so for $i>j$, and $\eps>0$ sufficiently small
\begin{equation}\begin{split}
-\lambda_{i}\partial_{\lambda_{i}}\eps_{i,j}
= &
\frac{n-2}{2}\eps_{i,j}^{\frac{n}{n-2}}
(
\frac{\lambda_{i} }{ \lambda_{j} }
-
\frac{ \lambda_{j} }{ \lambda_{i} }
+
\lambda_{i}\lambda_{j}\gamma_{n}G^{\frac{2}{2-n}}( a _{i}, a _{j}))\geq 
\frac{n-2}{4}\eps_{i,j}.
\end{split}\end{equation}

Thus \eqref{eij_positive _argument} follows.
\end{proof}

\begin{proposition}[Case $n=4, \, \omega=0$]\label{prop_n=4}$_{}$\\
Let $n=4$ and $Cond_{4}$ hold true.
Ordering
\begin{equation*}
\begin{split}
\frac{1}{K_{1}}\ln\frac{1}{\lambda_{1}}\geq \ldots \geq \frac{1}{K_{p}}\ln\frac{1}{\lambda_{p}} 
\end{split}
\end{equation*} 
the piecewise differentiable continuous function
$
\psi=\sum_{i}\frac{C^{i}}{K_{i}}\ln  \frac{1}{\lambda_{i}}
$
satisfies 
\begin{equation*}\begin{split}
\psi'
\geq &
\sum_{i} \frac{H_{i}}{\lambda_{i}^{2}} 
+
\frac{\vert \nabla K_{i}\vert^{2}\ln \lambda_{i}}{K_{i}^{2}\lambda_{i}^{2}}
+
\sum_{i> j}\eps_{i,j}
+
O(\vert \delta J(u)\vert^{2}),
\end{split}\end{equation*}
provided $C>1$ is sufficiently large.
\end{proposition}
The interaction terms are of correct sign again. Differentiating $\frac{1}{K_{i}}$ in time leads
to the quantity $\frac{\vert \nabla K_{i}\vert^{2}\ln \lambda_{i}}{K_{i}^{2}\lambda_{i}^{2}}$, which
enforces a blow up point $a_{i}$ to come close to $[\nabla K=0]$. $Cond_{4}$ then ensures
the 
$\frac{\lap K_{i}}{\lambda_{i}^{2}}$ terms to be controlled by the positive  mass related terms $\frac{H_{i}}{\lambda_{i}^{2}}$.

\begin{proof}[\textbf{Proof of proposition \ref{prop_n=4}}]\label{p_n=4}$_{}$\\
As $M$ is not conformally equivalent to the standard sphere $\mathbb{S}^{4}$, the positive  mass theorem holds. Thus $H_{i}>0$ in the statement of corollary \ref{cor_simplifying_the_shadow_flow}
\begin{enumerate}[label=(\roman*)]
 \Item 
\begin{equation}\begin{split}
-\frac{\dot \lambda_{i} }{\lambda_{i} }
= &
\frac{r}{k}
[
\gamma_{0} \frac{H_{i}}{\lambda_{i}^{2}}
+
\gamma_{1}\frac{\lap K_{i}}{K_{i}\lambda_{i}^{2}} 
-
\gamma_{3}\sum_{i \neq j=1}^{p}\frac{\alpha_{j}}{\alpha_{i}}
\lambda_{i}\partial_{\lambda_{i}}\eps_{i,j}
](1+o_{\frac{1}{\lambda_{i}}}(1)) \\
& +
o_{\varepsilon}(\sum_{r} \frac{1}{\lambda_{r}^{2}} +\sum_{r\neq s}\eps_{r,s})
+
O(\sum_{r} \frac{\vert \nabla K_{r}\vert^{2}}{\lambda_{r} ^{2}}+\vert \delta J(u)\vert^{2})
\end{split}\end{equation}
 \Item
\begin{equation}\begin{split}
\dot K_{i}
= &
\frac{r}{k}
\gamma_{2}\frac{\vert \nabla K_{i}\vert^{2}}{K_{i}\lambda_{i}^{2}} (1+o_{\frac{1}{\lambda_{i}}}(1)) \\
& +
 \frac{\nabla K_{i}}{\lambda_{i}}O (\sum_{r} \frac{1}{\lambda_{r}^{2}} +\sum_{r\neq s}\eps_{r,s}+\vert \delta J(u)\vert^{2})
\end{split}\end{equation}
\end{enumerate}
for suitable $\gamma_{0}, \gamma_{1}, \gamma_{2}, \gamma_{3}>0$. 
Then for
$
\psi=\sum_{i}\frac{C^{i}}{K_{i}}\ln  \frac{1}{\lambda_{i}}, \; C>1
$
there holds
\begin{equation}\begin{split}
\psi '
\geq &
\frac{r}{k}
\sum_{i}
\frac{C^{i}}{K_{i}\lambda_{i}^{2}}
(
\tilde \gamma_{0}H_{i}
+
\gamma_{1}\frac{\lap K_{i}}{K_{i}} 
+
\tilde \gamma_{2}
\frac{\vert \nabla K_{i}\vert^{2}}{K_{i}^{2}}\ln \lambda_{i} 
) (1+o_{\frac{1}{\lambda_{i}}}(1))\\
& -
\gamma_{3}\frac{r}{k}\sum_{i\neq j}\frac{C^{i}}{K_{i}}\frac{\alpha_{j}}{\alpha_{i}}\lambda_{i}\partial_{\lambda_{i}}\eps_{i,j} (1+o_{\frac{1}{\lambda_{i}}}(1))\\
& +
o_{\varepsilon}(\sum_{r\neq s}\eps_{r,s})
+
O(\vert \delta J(u)\vert^{2}).
\end{split}\end{equation}
We complete the definition of $\psi$ by ordering
\begin{equation*}\begin{split}
\frac{1}{K_{1}}\ln\frac{1}{\lambda_{1}}\geq \ldots \geq \frac{1}{K_{p}}\ln\frac{1}{\lambda_{p}}
\end{split}\end{equation*}
and claim, that there exists $c >0$ such, that for any $C>1$ sufficiently large
\begin{equation}\begin{split}\label{eij_n=4_estimate}
-\sum_{i\neq j}\frac{C^{i}}{K_{i}}\frac{\alpha_{j}}{\alpha_{i}}\lambda_{i}\partial_{\lambda_{i}}\eps_{i,j}
\geq 
c\sum_{i> j} C^{i}\eps_{i,j}.
\end{split}\end{equation}

To prove \eqref{eij_n=4_estimate} note, that by definition for any pair $i>j$ we have 
\begin{equation}
\begin{split}
(C^{i}-C^{j})\frac{\ln \frac{1}{\lambda_{i}}}{K_{i}}
\leq
(C^{i}-C^{j})\frac{\ln \frac{1}{\lambda_{j}}}{K_{j}}
\end{split}
\end{equation}

or equivalently
\begin{equation}
\begin{split}
\frac{C^{i}-C^{j}}{K_{i}}\ln  \frac{1}{\lambda_{i}} 
+
\frac{C^{j}-C^{i}}{K_{j}}\ln  \frac{1}{\lambda_{j}} 
\leq  &
0
\end{split}
\end{equation}

We then have 
\begin{equation}
\begin{split}
\frac{\lambda_{j}}{\lambda_{i}}
=
o
(
\frac{\lambda_{i}}{\lambda_{j}}
+
\lambda_{i}\lambda_{j}\gamma_{n}G^{\frac{2}{2-n}}(a_{i},a_{j})
),
\end{split}
\end{equation} 

from which the claim follows as when proving \eqref{eij_positive _argument}.
Otherwise we have
\begin{equation}
\begin{split}
\frac{\lambda_{j}}{\lambda_{i}}
\geq c
(
\frac{\lambda_{i}}{\lambda_{j}}
+
\lambda_{i}\lambda_{j}\gamma_{n}G^{\frac{2}{2-n}}(a_{i},a_{j})
)
\end{split}
\end{equation} 

for some $c>0$. This implies $\frac{\lambda_{j}}{\lambda_{i}}\gg 1 \gg \frac{\lambda_{i}}{\lambda_{j}}$ and
 $d((a_{i},a_{j}))=O(\frac{1}{\lambda_{i}})$. Thus
\begin{equation}
\begin{split}
\frac{C^{i}-C^{j}}{K_{j}}\ln  \frac{\lambda_{j}}{\lambda_{i}} 
\leq  &
O(\frac{\ln \lambda_{i}}{\lambda_{i}}),
\end{split}
\end{equation}

yielding a contradiction.
\\
We conclude
\begin{equation}\begin{split}
\psi '
\geq &
\frac{r}{k}
\sum_{i}
\frac{C^{i}}{K_{i}\lambda_{i}^{2}}
(
\tilde \gamma_{0} H_{i}
+
\gamma_{1}\frac{\lap K_{i}}{K_{i}} 
+
\tilde \gamma_{2}\frac{\vert \nabla K_{i}\vert^{2}}{K_{i}^{2}}\ln  \lambda_{i} )
(1+o_{\frac{1}{\lambda_{i}}}(1)) \\
& +
\tilde \gamma_{3}\sum_{i> j}C^{i}\eps_{i,j} 
+
O(\vert \delta J(u)\vert^{2}).
\end{split}\end{equation}
Thereby the assertion follows immediately due to $Cond_{4}$.
\end{proof}

\begin{proposition}[Case $n=5, \, \omega=0$]\label{prop_n=5}$_{}$\\
Let $n=5$ and $Cond_{5}$ hold true. 
For $\underline{\epsilon}>0$ small let $\eta_{\underline\epsilon}\in C^{\infty}_{0}(\R,[0,1])$ with 
\begin{equation*}\begin{split}
\eta_{\underline\epsilon}(r)\equiv 0\;\text{ for }\;r\leq \epsilon, \; \eta_{\underline\epsilon}(r)\equiv 1\;\text{ for } \;r\geq 2
\; \text{ and }\;
0\leq \eta_{\underline\epsilon}'\leq \frac{2}{\underline\epsilon}
\end{split}\end{equation*}
and 
\begin{equation*}\begin{split}
\theta_{i}=\eta_{\underline\epsilon}(-\lambda_{i}\lap K_{i})\ln \frac{-\lambda_{i}\lap K_{i}}{\underline\epsilon  }\geq 0.
\end{split}\end{equation*}
Ordering for some $\kappa>0$
\begin{equation*}
\begin{split}
\frac{\ln  \frac{1}{\lambda_{1}}}{K_{1}} 
-
\kappa \theta_{1}
\geq 
\ldots
\geq
\frac{\ln  \frac{1}{\lambda_{p}}}{K_{p}} 
-
\kappa \theta_{p}
\end{split}
\end{equation*} 
the piecewise differentiable continuous function
\begin{equation*}\begin{split}
\psi
=
\sum_{i}
\left(
\frac{C^{i}}{K_{i}}\ln  \frac{1}{\lambda_{i}} 
-
\kappa C^{i}
\theta_{i}
\right)
\end{split}\end{equation*}
satisfies for $C\gg 1$ and a suitable choice of $\kappa$
\begin{equation*}
\begin{split}
\psi'\geq &
\sum_{i} \frac{H_{i}}{\lambda_{i}^{3}} +\frac{\vert \nabla K_{i}\vert^{2}}{K_{i}^{2}\lambda_{i}^{2}}\ln \lambda_{i}
+
\sum_{i\neq j}\eps_{i,j}
+
O(\vert \delta J(u)\vert^{2}),
\end{split}
\end{equation*} 
provided $d(a_{i},[\nabla K=0])\ll 1$ is sufficiently small for all $i=1, \ldots,p$.
\end{proposition}
Note, that closeness of the blow up points to the critical set $[\nabla K=0]$ 
is not a serious restriction, cf. lemma \ref{lem_critical_points_of_K_as_attractors} and proposition \ref{prop_princ_lower_bounded_under_Cond_n}

The interaction terms however are of correct sign again and one is left
with comparing $\frac{H_{i}}{\lambda_{i}^{3}}$ to $\frac{\lap K_{i}}{\lambda_{i}^{2}}$.
$Cond_{5}$ then ensures by differentiating in time, that
$\lambda_{i}\lap K_{i}$ can be absorbed.

\begin{proof}[\textbf{Proof of proposition \ref{prop_n=5}}]\label{p_n=5}$_{}$\\
As $M$ is not conformally equivalent to the standard sphere $\mathbb{S}^{5}$, the positive  mass theorem holds. Thus $H_{i}>0$ in the statement of corollary \ref{cor_simplifying_the_shadow_flow}
\begin{enumerate}[label=(\roman*)]
 \Item 
\begin{equation}\begin{split}\label{dotepsi/epsi_n=5}
-\frac{\dot \lambda_{i} }{\lambda_{i} }
= &
\frac{r}{k}
[\gamma_{1}\frac{H_{i}}{\lambda_{i}^{3}}
+
\gamma_{2}\frac{\lap K_{i}}{K_{i}\lambda_{i}^{2}} 
-
\gamma_{4}\sum_{i \neq j=1}^{p}\frac{\alpha_{j}}{\alpha_{i}}
\lambda_{i}\partial_{\lambda_{i}}\eps_{i,j} 
](1+o_{\frac{1}{\lambda_{i}}}(1)) \\
& +
o_{\varepsilon}(\sum_{r} \frac{1}{\lambda_{r}^{3}}+\sum_{r\neq s}\eps_{r,s})
+
O
(
\vert \delta J(u)\vert^{2}
)
\end{split}\end{equation}
 \Item
\begin{equation}\begin{split}\label{dot_Ki_n=5}
\dot K_{i}
= &
\gamma_{3}\frac{r}{k}
\frac{\vert \nabla K_{i}\vert^{2}}{K_{i}\lambda_{i}^{2}}(1+o_{\frac{1}{\lambda_{i}}}(1))
\\
& +
O
(
\sum_{r} \frac{1}{\lambda_{r}^{3}}+\sum_{r\neq s}\eps_{r,s}
+
\vert \delta J(u)\vert^{2}
)
\frac{\vert \nabla K_{i}\vert}{\lambda_{i}}
\end{split}\end{equation}
 \Item
\begin{equation}\begin{split}\label{dotlapKi_n=5}
(\lap K_{i})'
= &
\frac{r}{k}
[
\gamma_{3}
\frac{\langle\nabla \lap K_{i}, \nabla K_{i}\rangle}{K_{i}\lambda_{i}^{2}} 
+
\gamma_{5}\frac{\vert \nabla \lap K_{i}\vert^{2}}{K_{i}\lambda_{i}^{4}} \\
& \quad\quad\quad\quad +
\gamma_{6}\sum_{i\neq j=1}^{p}\frac{\alpha_{j}}{\alpha_{i}}\frac{\nabla \lap K_{i}\nabla_{a_{i}}\eps_{i,j}}{\lambda_{i}^{2}}
] (1+o_{\frac{1}{\lambda_{i}}}(1))\\
& +
o_{\varepsilon}(\frac{1}{\lambda_{i}}(\sum_{r} \frac{1}{\lambda_{r}^{3}}+\sum_{r\neq s}\eps_{r,s}))  
+
O(\frac{1}{\lambda_{i}}\vert \delta J(u)\vert^{2})
\end{split}\end{equation}
\end{enumerate}
with suitable constants $\gamma_{1}, \ldots, \gamma_{6}>0$. Here we have used
\begin{equation*}\begin{split}
\sum_{r}
\frac{\vert \nabla K_{r}\vert}{K_{r}\lambda_{r}}
+
\frac{\vert \lap K_{r}\vert }{K_{r}\lambda_{r}^{2}}
+
\frac{1}{\lambda_{r}^{n-2}}
+
\sum_{r\neq s}\eps_{r,s}
\leq 
C\vert \delta J(u)\vert 
\end{split}\end{equation*}
according to lemma \ref{prop_princ_lower_bounded_under_Cond_n}.
In view of \eqref{dotepsi/epsi_n=5} we wish to compare 
$\lap K_{i}$ to $\frac{H_{i}}{\lambda_{i}}$ in a neighbourhood of a
critical point with non positive   laplacian and this is done as follows.
For $\eta_{\underline \epsilon}$ as in statement of the proposition consider
\begin{equation}\begin{split}\label{thetai_definition}
\theta_{i}=\eta_{\underline\epsilon}(-\lambda_{i}\lap K_{i})\ln \frac{-\lambda_{i}\lap K_{i}}{\underline\epsilon  }\geq 0.
\end{split}\end{equation}
Letting  $s_{i}=-\lambda_{i}\lap K_{i}$ we calculate
\begin{equation}\begin{split}
\theta_{i}'
= &
\eta_{\underline\epsilon}'(s_{i})s_{i}'\ln \frac{s_{i}}{\underline\epsilon  } 
+
\eta_{\underline\epsilon}(s_{i})(\ln \frac{s_{i}}{\underline\epsilon  })'\\
= &
[
\underline\epsilon\eta_{\underline\epsilon}'(s_{i})
\frac{s_{i}}{\underline\epsilon  }\ln \frac{s_{i}}{\underline\epsilon  } 
+
\eta_{\underline\epsilon}(s_{i})
]
(\ln \frac{s_{i}}{\underline\epsilon  })'
= 
\vartheta_{\underline\epsilon,i}\cdot(\ln \frac{s_{i}}{\underline\epsilon  })',
\end{split}\end{equation}
where readily 
\begin{enumerate}[label=(\roman*)]
 \item \quad
$
\vartheta_{\underline\epsilon,i}
= 0 
\; \text{ for }\; \frac{s_{i}}{\underline\epsilon}\leq 1
$
 \item \quad
$
0\leq \vartheta_{\underline\epsilon,i}
\leq 4\ln 2+1
\; \text{ for }\; 1\leq \frac{s_{i}}{\underline\epsilon}\leq 2
$
 \item \quad 
$
\vartheta_{\underline\epsilon,i}
= 1
\; \text{ for }\; \frac{s_{i}}{\underline\epsilon}\geq 2.
$
\end{enumerate}
From \eqref{dotepsi/epsi_n=5} and \eqref{dotlapKi_n=5} we  infer
\begin{equation}\begin{split}
\theta_{i}'
= &
\frac{r}{k}\vartheta_{\underline\epsilon,i}
[
-\gamma_{1}\frac{H_{i}}{\lambda_{i}^{3}}
+
(
-
\gamma_{2}
+
\frac{\gamma_{3}\langle \nabla \lap K_{i}, \nabla K_{i}\rangle }{\vert \lap K_{i} \vert^{2}}
)
\frac{\lap K_{i}}{K_{i}\lambda_{i}^{2}}  \\
& \quad \quad\;+
\gamma_{4}\sum_{i\neq j=1}^{p}\frac{\alpha_{j}}{\alpha_{i}}\lambda_{i}\partial_{\lambda_{i}}\eps_{i,j} \\
& \quad \quad\;+
\gamma_{5}
\frac{\vert \nabla \lap K_{i}\vert^{2}}{K_{i}\lap K_{i}\lambda_{i}^{4}}
+
\gamma_{6}\sum_{i\neq j=1}^{p}\frac{\alpha_{j}}{\alpha_{i}}
\frac{\nabla \lap K_{i}}{\lap K_{i}\lambda_{i}^{2}}
\nabla_{a_{i}}\eps_{i,j}
](1+o_{\frac{1}{\lambda_{i}}}(1))
\\
& +
o_{\varepsilon}(\sum_{r} \frac{1}{\lambda_{r}^{3}}+\sum_{r\neq s}\eps_{r,s})
+
O
(
\vert \delta J(u)\vert^{2}
).
\end{split}\end{equation}
Note, that  we have  
$-\lambda_{i}\lap K_{i}\geq \underline \epsilon$ for $\vartheta_{\underline\epsilon,i}\neq 0$,
whence 
\begin{equation}
\begin{split}
\vartheta_{\underline\epsilon,i}
\frac{\vert \nabla \lap K_{i}\vert^{2}}{K_{i}\lap K_{i}\lambda_{i}^{4}}
\leq 0
\; \text{ and }\;
\vartheta_{\underline\epsilon,i}\sum_{i\neq j=1}^{p}
\frac{\nabla \lap K_{i}}{\lap K_{i}\lambda_{i}^{2}}
\nabla_{a_{i}}\eps_{i,j}
=
O(\frac{1}{\lambda_{i}}\nabla_{a_{i}}\eps_{i,j}).
\end{split}
\end{equation} 
This gives
\begin{equation}\begin{split}\label{thetai'=}
\theta_{i}'
\leq  &
\frac{r}{k}\vartheta_{\underline\epsilon,i}
[
-\gamma_{1}\frac{H_{i}}{\lambda_{i}^{3}}
+
(
-
\gamma_{2}
+
\frac{\gamma_{3}\langle \nabla \lap K_{i}, \nabla K_{i}\rangle}{\vert \lap K_{i} \vert^{2}}
)
\frac{\lap K_{i}}{K_{i}\lambda_{i}^{2}}  \\
& \quad \quad\;+
\gamma_{4}\sum_{i\neq j=1}^{p}\frac{\alpha_{j}}{\alpha_{i}}\lambda_{i}\partial_{\lambda_{i}}\eps_{i,j} 
+
O(\sum_{i\neq j=1}^{p}
\frac{1}{\lambda_{i}}
\nabla_{a_{i}}\eps_{i,j}
)
](1+o_{\frac{1}{\lambda_{i}}}(1))
\\
& +
o_{\varepsilon}(\sum_{r} \frac{1}{\lambda_{r}^{3}}+\sum_{r\neq s}\eps_{r,s})
+
O
(
\vert \delta J(u)\vert^{2}
).
\end{split}\end{equation}
Consider for some $\kappa>0$ to be defined later on
\begin{equation}\begin{split}
\psi
=
\tilde \psi
-
\sum_{i}\kappa C^{i}
\theta_{i}
=
\sum_{i}
\frac{C^{i}}{K_{i}}\ln  \frac{1}{\lambda_{i}} 
-
\sum_{i}
\kappa C^{i}
\theta_{i}.
\end{split}\end{equation}
By \eqref{dotepsi/epsi_n=5} and \eqref{dot_Ki_n=5} we have
\begin{equation}
\begin{split}
(\frac{\ln \frac{1}{\lambda_{i}}}{K_{i}})'
= &
\frac{r}{kK_{i}}
[
\gamma_{1}\frac{H_{i}}{\lambda_{i}^{3}}
+
\gamma_{2}\frac{\lap K_{i}}{K_{i}\lambda_{i}^{2}} 
+
\gamma_{3}
\frac{\vert \nabla K_{i}\vert^{2}\ln \lambda_{i}}{K_{i}^{2}\lambda_{i}^{2}} \\
& \quad\quad\quad\quad\quad\quad\quad\quad\;\, \, \, \,
-
\gamma_{4}\sum_{i \neq j=1}^{p}\frac{\alpha_{j}}{\alpha_{i}}
\lambda_{i}\partial_{\lambda_{i}}\eps_{i,j} 
](1+o_{\frac{1}{\lambda_{i}}}(1)) \\
&
+
o_{\varepsilon}(\sum_{r} \frac{1}{\lambda_{r}^{3}}+\sum_{r\neq s}\eps_{r,s})
+
O(\vert \delta J(u)\vert^{2}),
\end{split}
\end{equation}
whence in conjunction with \eqref{thetai'=} there holds
\begin{equation}
\begin{split}
 \psi'
\geq &
\frac{r}{k}
\sum_{i}
\frac{C^{i}}{K_{i}\lambda_{i}^{2}}
(
\tilde {\gamma}_{1}\frac{H_{i}}{\lambda_{i}}
+
\gamma_{2}\frac{\lap K_{i}}{K_{i}} 
+
\tilde{\gamma}_{3}
\frac{\vert \nabla K_{i}\vert^{2}}{K_{i}^{2}}\ln \lambda_{i} 
)(1+o_{\frac{1}{\lambda_{i}}}(1)) \\
& -
\gamma_{4}\frac{r}{k}\sum_{i\neq j}C^{i}\frac{\alpha_{j}}{\alpha_{i}}
[
\frac{1}{K_{i}}+\kappa \vartheta_{\underline\epsilon,i}
]
\lambda_{i}\partial_{\lambda_{i}}\eps_{i,j} (1+o_{\frac{1}{\lambda_{i}}}(1))
 \\
& -
\kappa \frac{r}{k}\sum_{i} C^{i}
\vartheta_{\underline\epsilon,i}
(
-
\gamma_{2}
+
\frac{\gamma_{3}\langle\nabla \lap K_{i}, \nabla K_{i}\rangle}{\vert \lap K_{i} \vert^{2}}
)
\frac{\lap K_{i}}{K_{i}\lambda_{i}^{2}} (1+o_{\frac{1}{\lambda_{i}}}(1)) \\
& +
O(\sum_{i\neq j}\frac{C^{i}}{\lambda_{i}}\vert \nabla_{a_{i}}\eps_{i,j}\vert)
+
o_{\varepsilon}(\sum_{r\neq s}\eps_{r,s})
+
O(\vert \delta J(u)\vert^{2}).
\end{split}
\end{equation}
We complete the definition of $\psi$ by ordering
\begin{equation}
\begin{split}
\frac{\ln  \frac{1}{\lambda_{1}}}{K_{1}} 
-
\kappa \theta_{1}
\geq 
\ldots
\geq
\frac{\ln  \frac{1}{\lambda_{p}}}{K_{p}} 
-
\kappa \theta_{p}
\end{split}
\end{equation} 
and claim, that there exists $c >0$ such, that for any $C>1$ sufficiently large
\begin{equation}\label{eij_n=5_estimate}
\begin{split}
 -
\sum_{i\neq j}C^{i}\frac{\alpha_{j}}{\alpha_{i}}
[
\frac{1}{K_{i}}+\kappa \vartheta_{\underline\epsilon,i}
]
\lambda_{i}\partial_{\lambda_{i}}\eps_{i,j} 
\geq \epsilon\sum_{i>j }C^{i}\eps_{i,j}
\end{split}
\end{equation} 
and
\begin{equation}\label{eij_n=5_small}
\begin{split}
\sum_{i\neq j}\frac{C^{i}}{\lambda_{i}}\vert\nabla_{a_{i}}\eps_{i,j}\vert
= &
O(\sum_{i>j}C^{j}\eps_{i,j}).
\end{split}
\end{equation} 

To prove \eqref{eij_n=5_estimate}, \eqref{eij_n=5_small} note, that by definition for any  $i>j$ we have
\begin{equation}
\begin{split}
(C^{i}-C^{j})(\frac{\ln\frac{1}{\lambda_{i}}}{K_{i}}-\kappa \theta_{i})
\leq 
(C^{i}-C^{j})(\frac{\ln\frac{1}{\lambda_{j}}}{K_{j}}-\kappa \theta_{j})
\end{split}
\end{equation}

or equivalently
\begin{equation}
\begin{split}
\frac{C^{i}-C^{j}}{K_{i}}\ln  \frac{1}{\lambda_{i}} 
& +
\frac{C^{j}-C^{i}}{K_{j}}\ln  \frac{1}{\lambda_{j}} \\
& +
\kappa
(C^{j}- C^{i})
\theta_{i}  
+
\kappa (C^{i}-C^{j})
\theta_{j}
\leq  
0.
\end{split}
\end{equation}

We then have 
\begin{equation}\label{li/lj=o(...)_remark}
\begin{split}
\frac{\lambda_{j}}{\lambda_{i}}
=
o
(
\frac{\lambda_{i}}{\lambda_{j}}
+
\lambda_{i}\lambda_{j}\gamma_{n}G^{\frac{2}{2-n}}(a_{i},a_{j})
).
\end{split}
\end{equation} 

Otherwise we may assume for some $c>0$
\begin{equation}
\begin{split}
\frac{\lambda_{j}}{\lambda_{i}}
\geq c
(
\frac{\lambda_{i}}{\lambda_{j}}
+
\lambda_{i}\lambda_{j}\gamma_{n}G^{\frac{2}{2-n}}(a_{i},a_{j})
).
\end{split}
\end{equation} 

This implies $\frac{\lambda_{j}}{\lambda_{i}}\gg 1 \gg \frac{\lambda_{i}}{\lambda_{j}}$ and
 $d(a_{i},a_{j})=O(\frac{1}{\lambda_{i}})$. Consequently
\begin{equation}
\begin{split}
& \frac{C^{i}-C^{j}}{K_{j}}\ln  \frac{\lambda_{j}}{\lambda_{i}} 
+
\kappa
(C^{j}- C^{i})
\theta_{i}  
+
\kappa (C^{i}-C^{j})
\theta_{j} 
\leq  
O(\frac{\ln \lambda_{i}}{\lambda_{i}}),
\end{split}
\end{equation}

whence due to the definition of $\theta_{i}$, see \eqref{thetai_definition}, there necessarily holds
\begin{equation}
\begin{split}
\ln \frac{-\lambda_{i}\lap K_{i}}{\underline\epsilon  }\gg 1
, \; \text{ so }\;
-\lambda_{i}\lap K_{i}\gg 1
\end{split}
\end{equation} 

and we get
\begin{equation}
\begin{split}
& \frac{C^{i}-C^{j}}{K_{j}}\ln  \frac{\lambda_{j}}{\lambda_{i}} 
+
\kappa
(C^{j}- C^{i})
\ln \frac{-\lambda_{i}\lap K_{i}}{\underline\epsilon  } \\
&  
+
\kappa (C^{i}-C^{j})
\eta_{\underline\epsilon}(-\lambda_{j}\lap K_{j})\ln \frac{-\lambda_{j}\lap K_{j}}{\underline\epsilon  } 
\leq  
O(\frac{\ln \lambda_{i}}{\lambda_{i}}).
\end{split}
\end{equation}

On the other hand $d(a_{i},a_{j})=O(\frac{1}{\lambda_{i}})$ and therefore 
\begin{equation}
\begin{split}
1
\ll 
-\lambda_{i}\lap K_{i}
=
-\lambda_{i}\lap K_{j}+O(1)
=
-\frac{\lambda_{i}}{\lambda_{j}}\lambda_{j}\lap K_{j}+O(1).
\end{split} 
\end{equation} 

This shows at once
$
1\ll - \lambda_{i}\lap K_{i} \ll -\lambda_{j}\lap K_{j}
$
and we conclude
\begin{equation}
\begin{split}
\frac{C^{i}-C^{j}}{K_{j}}\ln  \frac{\lambda_{j}}{\lambda_{i}}  
+
\kappa (C^{i}-C^{j})
\ln \frac{-\lambda_{j}\lap K_{j}}{- \lambda_{i}\lap K_{i} } 
\leq  
O(\frac{\ln \lambda_{i}}{\lambda_{i}})
\end{split}
\end{equation}

yielding a contradiction. Thus \eqref{li/lj=o(...)_remark} is established, whence  \eqref{eij_n=5_estimate}
follows

as when proving \eqref{eij_positive _argument}. We are left with estimating 
\begin{equation}
\begin{split}
\sum_{i\neq j}\frac{C^{i}}{\lambda_{i}}\vert \nabla_{a_{i}}\eps_{i,j}\vert 
= &
\frac{n-2}{2}\sum_{i< j}C^{i}\eps_{i,j}
\vert
\frac
{(\frac{\lambda_{j}}{\lambda_{i}})^{\frac{1}{2}}(\lambda_{i}\lambda_{j})^{\frac{1}{2}}\gamma_{n}\nabla_{ a_{i}}G^{\frac{2}{2-n}}(a_{i},a_{j})}
{\frac{\lambda_{i}}{\lambda_{j}}+\frac{\lambda_{j}}{\lambda_{i}}
+
\lambda_{i}\lambda_{j}\gamma_{n}G^{\frac{2}{2-n}}(a_{i},a_{j})} 
\vert \\
& +
o(\sum_{i\neq j}\eps_{i,j}),
\end{split} 
\end{equation}

whence we immediately obtain \eqref{eij_n=5_small}.\\
We conclude for $C>1$ sufficiently large
\begin{equation}
\begin{split}
 \psi'
\geq &
\frac{r}{k}
\sum_{i}
\frac{C^{i}}{K_{i}\lambda_{i}^{2}}
(
\tilde {\gamma}_{1}\frac{H_{i}}{\lambda_{i}}
+
\gamma_{2}\frac{\lap K_{i}}{K_{i}} 
+
\tilde{\gamma}_{3}
\frac{\vert \nabla K_{i}\vert^{2}}{K_{i}^{2}}\ln \lambda_{i} 
)(1+o_{\frac{1}{\lambda_{i}}}(1)) \\
& -
\kappa \frac{r}{k}\sum_{i} C^{i}
\vartheta_{\underline\epsilon,i}
(
-
\gamma_{2}
+
\frac{\gamma_{3}\langle \nabla \lap K_{i}, \nabla K_{i}\rangle }{\vert \lap K_{i} \vert^{2}}
)
\frac{\lap K_{i}}{K_{i}\lambda_{i}^{2}}(1+o_{\frac{1}{\lambda_{i}}}(1))  \\
& +
\tilde \gamma_{4}\sum_{i>j}C^{i}\eps_{i,j}
+
O(\vert \delta J(u)\vert^{2}).
\end{split}
\end{equation}
This gives
\begin{equation}
\begin{split}
 \psi'
\geq &
\frac{r}{k}\sum_{i} C^{i}
[
\frac{\gamma_{2}}{K_{i}}
-
\kappa\vartheta_{\underline\epsilon,i}
(
-
\gamma_{2}
+
\frac{\gamma_{3}\langle \nabla \lap K_{i}, \nabla K_{i}\rangle }{\vert \lap K_{i} \vert^{2}}
)
]
\frac{\lap K_{i}}{K_{i}\lambda_{i}^{2}}(1+o_{\frac{1}{\lambda_{i}}}(1))  \\
& +
c
(
\sum_{i}
\frac{H_{i}}{\lambda_{i}^{3}}
+
\frac{\vert \nabla K_{i}\vert^{2}}{K_{i}^{2}\lambda_{i}^{2}}\ln \lambda_{i} 
+
\sum_{i> j}\eps_{i,j} 
)
+
O(\vert \delta J(u)\vert^{2}).
\end{split}
\end{equation}
We now decompose $P=\{1, \ldots,p\}=P_{1}+P_{2}+P_{3}$ with
\begin{enumerate}[label=(\roman*)]
 \item \quad
$P_{1}=\{i\in \{1, \ldots,p\} \mid -\lap K_{i}<  \frac{\underline\epsilon}{\lambda_{i}} \}$
 \item \quad
$P_{2}=\{i\in \{1, \ldots,p\} \mid  \frac{\underline\epsilon }{\lambda_{i}} \leq -\lap K_{i} \leq 2 \frac{\underline\epsilon}{\lambda_{i}} \}$
 \item \quad
$P_{3}=\{i\in \{1, \ldots,p\} \mid -\lap K_{i}>  2\frac{\underline \epsilon}{\lambda_{i}} \}$.
\end{enumerate}
Note, that for $i\in P_{2}\cup P_{3}$ we have 
$
\lap K_{i}<0,
$
whence according to $Cond_{5}$
\begin{equation}
\begin{split}
\langle \nabla \lap K_{i}, \nabla K_{i}\rangle >\frac{1}{3}\vert \lap K_{i}\vert^{2},
\end{split}
\end{equation} 
in particular $\nabla K_{i}\neq 0$ for $i\in P_{2}\cup P_{3}$.

For $i\in P_{1}$ there holds $\vartheta_{\underline\epsilon,i}=0$, thus
\begin{equation}\begin{split}
(
\frac{\gamma_{2}}{K_{i}}
-
\kappa\vartheta_{\underline\epsilon,i}
(
-
\gamma_{2}
+
\frac{\gamma_{3}\langle\nabla \lap K_{i}, \nabla K_{i}\rangle }{\vert \lap K_{i} \vert^{2}}
)
)
\frac{\lap K_{i}}{K_{i}\lambda_{i}^{2}} 
\geq 
-\frac{\gamma_{2}}{K_{i}^{2}}  \frac{\underline\epsilon}{\lambda_{i}^{3}}.
\end{split}\end{equation}

For $i\in P_{2}$ 
\begin{equation}\begin{split}
(
\frac{\gamma_{2}}{K_{i}}
-
\kappa\vartheta_{\underline\epsilon,i}
(
-
\gamma_{2}
+
\frac{\gamma_{3}\langle \nabla \lap K_{i}, \nabla K_{i}\rangle }{\vert \lap K_{i} \vert^{2}}
)
)
\frac{\lap K_{i}}{K_{i}\lambda_{i}^{2}} 
\geq 
-2\frac{\gamma_{2}}{K_{i}^{2}}  \frac{\underline\epsilon}{\lambda_{i}^{3}},
\end{split}\end{equation}

since indeed $Cond_{5}$ imposed on $K$ can be rewritten as
\begin{equation}\begin{split}
-\gamma_{2}
+
\frac{\gamma_{3}\langle \nabla \lap K_{i}, \nabla K_{i}\rangle }{\vert \lap K_{i} \vert^{2}}
\geq c_{0}>0
\; 
\text{ for }\;a_{i}\in U(\mathcal{N})\cap[\lap K_{i}<0],
\end{split}\end{equation}

as 
$
\frac{\gamma_{3}}{\gamma_{2}}=3 
$ 
by precise calculation, see below.\\
Choosing therefore
\begin{equation}
\underline \epsilon\leq  c\min_{a\in M}H(a)\; \text{ with }\; c=c(K)   
\end{equation} 
we get
\begin{equation}
\begin{split}
 \psi'
\geq &
\frac{r}{k}\sum_{i\in P_{3}} C^{i}
[
\frac{\gamma_{2}}{K_{i}}
-
\kappa c_{0}
]
\frac{\lap K_{i}}{K_{i}\lambda_{i}^{2}} (1+o_{\frac{1}{\lambda_{i}}}(1)) \\
&
+
c
(
\sum_{i}
\frac{H_{i}}{\lambda_{i}^{3}}
+
\frac{\vert \nabla K_{i}\vert^{2}}{K_{i}^{2}\lambda_{i}^{2}}\ln \lambda_{i} 
+
\sum_{i> j}\eps_{i,j} 
)
+
O(\vert \delta J(u)\vert^{2}).
\end{split}
\end{equation}
since $\vartheta_{\underline\epsilon,i}=1$ on $i\in P_{3}$.  
Letting 
$
\kappa=\frac{\gamma_{2}}{c_{0}\cdot\min_{M}K}
$
we get as $\lap K_{i}<0$ for $i\in P_{3}$
\begin{equation}\begin{split}\label{phi'_n=5}
\psi'
\geq &
c
(
\sum_{i}
\frac{H_{i}}{\lambda_{i}^{3}}
+
\frac{\vert \nabla K_{i}\vert^{2}}{K_{i}^{2}\lambda_{i}^{2}}\ln \lambda_{i} 
+
\sum_{i> j}\eps_{i,j} 
)
+
O(\vert \delta J(u)\vert^{2}).
\end{split}\end{equation}
We are left with checking $\frac{\gamma_{3}}{\gamma_{2}}=3$. $\gamma_{2}$ and $\gamma_{3}$ arise from differentiating
\begin{equation}
\begin{split}
-\frac{\dot \lambda_{i}}{\lambda_{i}}
=
\gamma_{2}\frac{r}{k}\frac{\lap K_{i}}{K_{i}\lambda_{i}^{2}}+\ldots 
\; \text{ and }\;
\lambda_{i}\dot a_{i} 
=
\gamma_{3}\frac{r}{k}\frac{\nabla K_{i}}{K_{i}\lambda_{i}}+\ldots,
\end{split}
\end{equation} 
where
$
\gamma_{2}=\frac{e_{2}}{c_{2}}, \; \gamma_{3}=\frac{e_{3}}{c_{3}},
$
cf. corollary \ref{cor_simplifying_the_shadow_flow}.
According to \eqref{a2} and \eqref{a3} 
\begin{equation}
\begin{split}
c_{2}
= 
\frac{(n-2)^{2}}{4}\int_{\R^{n}} \frac{\vert r^{2}-1\vert^{2}}{(1+r^{2})^{n+2}}
, \;
c_{3}
=
\frac{(n-2)^{2}}{n}\int_{\R^{n}} \frac{r^{2}}{(1+r^{2})^{n+2}},
\end{split}
\end{equation} 
whereas according to \eqref{bringoutK_k=2} and \eqref{bringoutK_k=3}
\begin{equation}
\begin{split}
e_{2}=\frac{(n-2)}{4n}\int_{\R^{n}} \frac{r^{2}(r^{2}-1)}{(1+r^{2})^{n+1}}
\; \text{ and }\; 
e_{3}=\frac{n-2}{n}\int_{\R^{n}} \frac{r^{2}}{(1+r^{2})^{n+1}} 
\end{split}
\end{equation} 
One obtains
\begin{equation}\label{g3/g2_n=5}
\frac{\gamma_{3}}{\gamma_{2}}
= 
 \frac{c_{2}e_{3}}{c_{3}e_{2}}
=3.
\end{equation} The proof is thereby complete.
\end{proof}
The strategy in case $\omega>0$ is independent of the dimension the same as when proving proposition \ref{prop_n=3}. Note, that in comparison to propositions \ref{prop_n=3}, \ref{prop_n=4} and \ref{prop_n=5}
the contribution of the positive  mass related term $\frac{H_{i}}{\lambda^{n-2}}$ is replaced by the positive terms
$\frac{\omega_{i}}{\lambda_{i}^{\frac{n-2}{2}}}$.
\begin{proposition}[Case $\omega>0$]\label{prop_w}$_{}$\\
Let $n=3,4,5$. 
Ordering
\begin{equation*}
\begin{split}
\frac{1}{\lambda_{1}}\geq \ldots \geq \frac{1}{\lambda_{p}} 
\end{split}
\end{equation*} 
the function
$
\psi=\sum_{i}C^{i}\ln  \frac{1}{\lambda_{i}}
$
satisfies
\begin{equation*}\begin{split}
\psi'
\geq &
\sum_{i} \frac{\omega_{i}}{\lambda_{i}^{\frac{n-2}{2}}} 
+
\sum_{i> j}\eps_{i,j}
+
O(\vert \delta J(u)\vert^{2}),
\end{split}\end{equation*}
provided $C>1$ is sufficiently large.
\end{proposition}

\begin{proof}[\textbf{Proof of proposition \ref{prop_w}}]\label{p_case_omega>0_flow_outside_V(omega,p,e)}$_{}$\\
This follows analogously to the proof of proposition \ref{prop_n=3}.
\end{proof}

\subsection{Proving the theorems}
\label{subsec:ProvingTheTheorem}
\subsubsection{Proof of theorem \ref{thm_1}}
Let us consider a flow line, which is a solution of the evolution equation
\begin{equation}
\partial_{t}u=-\frac{1}{K}(R-r\K)u, \; u_{0}=u(\cdot,0)>0\;\text{ with }\; \int Ku_{0}^{\frac{2n}{n-2}}=1.
\end{equation}
The flow line exists for all times   according to corollary \ref{cor_long_time_existence} and we know
\begin{equation}
J(u)=r\searrow J_{\infty}=r_{\infty}
\;
\text{ and }\; 
\vert \delta J(u)\vert\- 0
\; \text{ as }\; t\- \infty.
\end{equation}
due to proposition \ref{prop_strong_convergence_of_the_first_variation}. 

Thus a flow line is of Palais-Smale type and due to the 
concentration-compactness principle, cf. proposition \ref{prop_concentration_compactness}, the flow line is
precompact in some $V(\omega,p, \eps)$, cf. definition \ref{def_V(omega,p,e)} and the remarks following.

Taking the unicity result on a limiting critical point into account, cf. proposition \ref{prop_unicity_of_a_limiting_critical_point}, we obtain convergence of the flow line to a critical point of $J$, once the flow line is precompact in $V(\omega,0, \eps)$. In other words the flow line converges strongly, if and only if it converges along a sequence in time, and in this case we are done.

Thus we wish to lead to a contradiction the scenario, that for some $p\geq 1$ the flow line is precompact in some $V(\omega,p, \eps)$.

By assumption of  theorem \ref{thm_1} the dimensional condition $Cond_{n}$ hold true, so
$ \partial J $ is principally lower bounded, cf. proposition \ref{prop_princ_lower_bounded_under_Cond_n}. 
Taking the unicity result on a limiting critical point at infinity into account, cf. proposition \ref{prop_unicity_of_a_limiting_critical_point_at_infinity}, we may assume, that the flow line
remains for all times   in $V(\omega,p, \eps)$ and goes deeper and deeper in the sense, that
\begin{equation}\begin{split}
\fa 0<\epsilon<\varepsilon \e T>0 \fa t>T\;:\; u(t)\in V(\omega, p, \epsilon).
\end{split}\end{equation}
In particular the unique representation
$
u=u_{\alpha, \beta}+\alpha^{i}\var_{a_{i}, \lambda_{i}}+v
$
given by proposition \ref{prop_optimal_choice} is well defined for all times    and we have
$\lambda_{i}\- \infty\; \text{ as }\; t\- \infty.$
Moreover the blow up points $a_{i}$ converge to  $[\nabla K=0]$, cf. lemma \ref{lem_critical_points_of_K_as_attractors}. 
Recalling the explanatory introduction of the previous subsection the functions given by propositions \ref{prop_n=3},\ref{prop_n=4},\ref{prop_n=5} and \ref{prop_w} then yield the desired contradiction. 
\subsubsection{Proving theorem \ref{thm_2}}
First of all note, that on $V(p,\varepsilon)$ we have according to definition \ref{def_V(omega,p,e)}
\begin{equation}
J(u)
=
\frac
{\sum \alpha_{i}^{2}\int L_{g_{0}}\varphi_{i}\varphi_{i}}
{
(
\sum_{i}\alpha_{i}^{\frac{2n}{n-2}}\int K\varphi_{i}^{\frac{2n}{n-2}}
)^{\frac{n-2}{n}}
}
+
o_{\varepsilon}(1)
=
c_{0}\frac{\sum_{i}\alpha_{i}^{2}}{(\sum_{i}\alpha_{i}^{\frac{2n}{n-2}}K_{i})^{\frac{n-2}{n}}}
+
o_{\varepsilon}(1)
\end{equation}
with $\alpha_{i}^{\frac{4}{n-2}}=\frac{4n(n-1)k}{rK_{i}}+o_{\varepsilon}(1)$. Therefore
\begin{equation}\label{energy_J_rough}
J(u)
=
c_{0}(\sum_{i}\frac{1}{K_{i}^{\frac{n-2}{2}}})^{\frac{2}{n}}+o_{\varepsilon}(1).
\end{equation} 
From this it is clear, that the least critical energy level at infinity is 
\begin{equation}
J_{\infty,\min}=\frac{c_{0}}{(\max K)^{\frac{n-2}{n}}} 
\end{equation} 
Thus, if we start a flow line $u$ with $u(0,\cdot)=u_{0}$, where
\begin{align*}
u_{0}=\alpha_{0}\varphi_{a_{0},\lambda_{0}}\in V(1,\varepsilon),\,d(a_{0},[K=\max K])<\varepsilon
\end{align*}
and $\varepsilon>0$ is sufficiently small, we may assume, that $u$ remains in $V(1,\varepsilon)$ for all times and $d(a,[K=\max K])=o_{\varepsilon}(1)$.

Indeed according to definition 
\ref{def_H(omega,p,e)} and the remarks following $u$ is precompact 

with respect to $V(\omega,p,\varepsilon)$. Since we want to prove the existence of a non 

trivial solution $\omega>0$, we may argue by contradiction and assume, that no 

non trivial solution exists, that is $\omega=0$. 
So $u$ is precompact with respect to 

$V(p,\varepsilon)$. Moreover, 
if for some time sequence $t_{k}\- \infty$ we had $u_{t_{k}}\in V(p,\varepsilon_{k})$ 

with $\varepsilon_{k}\searrow 0$ and 
$p\geq 2$, then  \eqref{energy_J_rough} would imply 
\begin{equation}
J(u_{t_{k}})=c_{0}(\sum_{i}\frac{1}{K_{i}^{\frac{n-2}{2}}})^{\frac{2}{n}}+o_{\varepsilon_{k}}(1)
\geq 
c_{0}\frac{p^{\frac{2}{n}}}{(\max K)^{\frac{n-2}{n}}}+o_{\varepsilon_{k}}(1),
\end{equation} 

whence without loss of generality $J(u_{t_{k}})>J(u_{0})$; contradicting $\partial_{t}J(u)\leq 0$.

Therefore $u$ is precompact with respect to $V(1,\varepsilon)$. Likewise we obtain 

$d(a,[K=\max K])=o_{\varepsilon}(1)$, since otherwise $J(u_{t_{k}})>J(u_{0})$. \\
Repeating now the arguments for proposition \ref{prop_princ_lower_bounded_under_Cond_n} it is obvious, that
$\partial J$ is principally lower bounded along the flow line $u$, since due to $Cond_{n}'$ the dimensional conditions $Cond_{n}$, cf. definition \ref{def_dimensional_conditions} are satisfied at the critical level $[K=\max K]$, to which $a$ is close. Therefore the results on the principal behaviour proven in subsection 
\ref{subsec:PrincipalBehaviour} hold true for the flow line $u$, in particular $d(a,[\nabla K=0])\- 0$. 
On the other hand we have
\begin{equation}
\max K-K\lfloor_{[\nabla K=0]\setminus [K=\max K]} >\delta
\end{equation} 
for some $\delta>0$ and $d(a,[K=\max K])=o_{\varepsilon}(1)$. Thus we may assume
\begin{equation}
a\-[K=\max K].
\end{equation} 
Finally note, that the statement of propositions \ref{prop_n=3},\ref{prop_n=4} and \ref{prop_n=5} remain 
valid for the functions constructed there, since as before $Cond_{n}'$ implies, that $Cond_{n}$ is satisfied at the critical level $[K=\max K]$, to which $a$ is close. 
Thus we arrive at the same contradiction as before, whence  
$u$ has to be precompact in some $V(\omega,p,\varepsilon)$ with $w>0$ being a non trivial solution. The proof is thereby complete.

\subsection{A diverging scenario}
\label{subsec:DivergingSzenario}
We give a non trivial example of a non compact flow line.
\begin{lemma}[Non-compact flow line with flatness]$_{}$

Let $n=5$ and $u_{0}=\alpha_{0}\varphi_{a_{0}, \lambda_{0}}$
with $a_{0}$ close to $0\in M$, where
\begin{equation*}
K(x)=1-\sum_{i=1, \ldots,5}\vert x_{i}\vert^{4}
\end{equation*}

in local normal conformal coordinates. 
\\
Then for $\varepsilon>0$ small there exists $0<\varepsilon_{0}<\varepsilon$ such, that the flow line 
$u$ with initial data $u_{0}$ remains in $V(1, \varepsilon)$ for all times, provided   
\begin{enumerate}[label=(\roman*)]
 \item 
$\alpha_{0}\varphi_{a_{0}, \lambda_{0}}\in V(1, \varepsilon_{0})$ and $k_{u_{0}}=\int Ku_{0}^{\frac{2n}{n-2}}=1$
 \item 
$\Vert a_{0}\Vert <\varepsilon_{0}$ and $\lambda_{0}\Vert a_{0}\Vert^{2} >\varepsilon^{-1}_{0}$
 \item 
$(a_{i})_{0}=(a_{j})_{0}>0$ for all $i,j=1, \ldots,5$.
\end{enumerate}
Moreover $u$ converges to a critical point at infinity in the sense, that
\begin{align*}
\lambda\- \infty\;\text{ and }\;\Vert a \Vert\- 0\;\text{ as }\;t\- \infty.
\end{align*}
\end{lemma}
Note, that $K$ does not satisfy condition $Cond_{5}$, cf. definition \ref{def_dimensional_conditions}, since
\begin{equation*}
 \langle\nabla \Delta K, \nabla K\rangle=\frac{7}{9}\vert \lap K\vert^{2}\; \text{ on }\; B_{\varepsilon}(0),
\end{equation*} 

but $K$ satisfies the flatness condition of Theorem 0.1 in \cite{YanYanLi_Part1}, cf. \cite{YanYanLi_Part2}, \cite{Escobar_Schoen}.

\begin{proof}
In order to prove, that $u$ remains in $V(1, \varepsilon)$ for all times let us define
\begin{equation}
\begin{split}
T
=
\sup\{\tau>0\mid \fa 0\leq t <\tau
\;:\;&
u\in V(1, \varepsilon), \;
\Vert a\Vert< \varepsilon, \;
\lambda\Vert a\Vert^{2}> \varepsilon^{-1} \\
& \frac{a_{i}}{a_{j}}<\sqrt[4]{\frac{5}{2}}
\; \text{ for all }\; i,j=1, \ldots,n
\}.
\end{split}
\end{equation} 
We then have to show $T=\infty$.

Note, that we may assume $J(u_{0})\leq C$ independent of $0<\varepsilon_{0}\ll1 $, whence
\begin{equation}\label{divergence_integrability}
\int^{\infty}_{0}\vert \delta J(u)\vert^{2}\leq c(K)
\end{equation} 

independent of the smallness of $0<\varepsilon\ll 1$.\\
According to corollary \ref{cor_simplifying_the_shadow_flow} the relevant evolution equations are 
\begin{equation}\begin{split}
-\frac{\dot{\lambda}}{\lambda}
=
\frac{r}{k}(\gamma_{1}\frac{H(a)}{\lambda^{3}}& +\gamma_{2}\frac{\Delta K(a)}{K(a)\lambda^{2}})
(1+o_{\frac{1}{\lambda}}(1))\\
& +
o(\frac{1}{\lambda^{3}})+O(\frac{\vert\nabla K(a)\vert^{2}}{\lambda^{2}}+\vert \delta J(u)\vert^{2})
\end{split}\end{equation}
and 
\begin{equation}\begin{split}
\lambda\dot{a}=\frac{r}{k}
(\gamma_{3}\frac{\nabla K(a)}{K(a)\lambda} & +\gamma_{4}\frac{\nabla \Delta K(a)}{K(a)\lambda^{3}}) (1+o_{\frac{1}{\lambda}}(1))\\
& +
o(\frac{1}{\lambda^{3}})+O(\frac{\vert\nabla K(a)\vert^{2}}{\lambda^{2}}+\vert \delta J(u)\vert^{2}),
\end{split}\end{equation}
where $\frac{r}{k}=4n(n-1)+o(1)=80(1+o_{\varepsilon}(1))$ according to \eqref{rai^{...}/k=...}.
Moreover
\begin{equation}\begin{split}\label{derivatives_of_K}
\nabla K(a)=-4\Vert a\Vert^{2}a,
\,
\Delta K(a)=-12\Vert a \Vert^{2}\; \text{ and }\;\nabla\Delta K(a)=-24a.
\end{split}\end{equation}
We obtain during $(0,T)$ the simplified evolution equations
\begin{equation}\begin{split}\label{lambda_dot_simplified}
-
\frac{\dot{\lambda}}{\lambda}=80\gamma_{2}\frac{\Delta K(a)}{\lambda^{2}}(1+o_{\varepsilon}(1))
+
O(\vert \delta J(u)\vert^{2})
\end{split}\end{equation}
and
\begin{equation}\begin{split}\label{a_dot_simplified}
\lambda\dot{a}=80\gamma_{3}\frac{\nabla K(a)}{\lambda}(1+o_{\varepsilon}(1))+O(\vert \delta J(u)\vert^{2}).
\end{split}\end{equation}
First note, that during $(0,T)$
\begin{equation}\label{Norma'_divergence}
\begin{split}
\partial_{t}\Vert a\Vert^{2} 
= & 
\frac{2}{\lambda}\langle a, \lambda\dot{a}\rangle \\
= &
c\frac{\langle\nabla K(a),a\rangle}{\lambda^{2}}(1+o_{\varepsilon}(1))
+
O(\frac{\Vert a \Vert\vert \delta J(u)\vert^{2}}{\lambda})\\ 
\leq & 
O(\frac{\Vert a \Vert\vert \delta J(u)\vert^{2}}{\lambda}),
\end{split}
\end{equation} 
whence
\begin{align*}
\partial_{t}\ln \Vert a \Vert^{2}\leq O(\frac{\vert \delta J(u)\vert^{2}}{\lambda\Vert a \Vert}).
\end{align*}
But 
$\lambda \Vert a \Vert =\lambda^{\frac{1}{2}}(\lambda \Vert a \Vert^{2})^{\frac{1}{2}}>c\varepsilon^{-1}$
during $(0,T)$ by definition.
Therefore $\Vert a \Vert $ remains uniformly small, e.g.$\Vert a\Vert\leq C\varepsilon_{0}$. 
Let us calculate 
\begin{equation}
\begin{split}
(\lambda\Delta K(a))' 
= & 
\frac{\dot{\lambda}}{\lambda}\lambda\Delta K(a)+\langle\nabla\Delta K(a), \lambda\dot{a}\rangle\\
= & 
-80
\gamma_{2}\frac{\vert\Delta K(a)\vert^{2}}{\lambda}(1+o_{\varepsilon}(1))
\\ & +
80\gamma_{3}\frac{\langle\nabla\Delta K(a), \nabla K(a)\rangle}{\lambda}(1+o_{\varepsilon}(1))\\
  & +O((\vert\lambda\Delta K(a)\vert+\vert\nabla\Delta K(a)\vert)\vert \delta J(u)\vert^{2}).
\end{split}
\end{equation} 
Since $\vert \lambda \Delta K(a)\vert=12 \lambda \Vert a \Vert^{2}\geq c\varepsilon^{-1}$  during $(0,T)$, we obtain  
\begin{equation}
\begin{split}
\frac{(\lambda\Delta K(a))' }{80} 
= & 
(
-
12^{2}\gamma_{2}(\sum_{i=1}^{5}\vert a_{i}\vert^{2})^{2}
+
4\cdot 24\gamma_{3}\sum_{i=1}^{5}\vert a_{i}\vert^{4}
)
\frac{1+o_{\varepsilon}(1)}{\lambda}\\
& +
O(\vert\lambda\Delta K(a)\vert\vert \delta J(u)\vert^{2})\\
\leq & 
(-12^{2}\cdot 5^{2}\, \gamma_{2}\, a^{4}_{\min}
+
4\cdot 5\cdot 24 \, \gamma_{3}\,a^{4}_{\max})\frac{1+o_{\varepsilon}(1)}{\lambda}\\
& +
O(\vert\lambda\Delta K(a)\vert\vert \delta J(u)\vert^{2}),
\end{split}
\end{equation}
where we used $a_{i}>0$ during $(0,T)$ and let
\begin{equation}
a_{\min}=\min\{a_{i}\mid i=1, \ldots,n\}
\;
\text{ and }\;
a_{\max}=\max\{a_{i}\mid i=1, \ldots,n\}.
\end{equation} 
Due to  $\frac{\gamma_{3}}{\gamma_{2}}=3$, cf. \eqref{g3/g2_n=5}, 
and $ a^{4}_{\max}<\frac{5}{2}a_{\min}^{4}$ during $(0,T)$ we get
\begin{equation}
\begin{split}
(\lambda\Delta K(a))' 
\leq & 
O(\vert\lambda\Delta K(a)\vert\vert \delta J(u)\vert^{2}).
\end{split}
\end{equation}
Therefore
\begin{equation}\begin{split}
\partial_{t}\ln(-\lambda\Delta K(a))\geq O(\vert \delta J(u)\vert^{2})
\end{split}\end{equation}
and we conclude using \eqref{divergence_integrability}, that 
\begin{equation}\begin{split}
12 \lambda \Vert a \Vert^{2}
=
-\lambda\Delta K(a)
\geq
-\lambda_{0}\Delta K(a_{0})e^{-C\int_{0}^{\infty}\vert \delta J(u)\vert^{2}}
=
12 \lambda_{0}\Vert a_{0}\Vert^{2}
\end{split}\end{equation}
remains during $(0,T)$ uniformly large, say $\lambda \Vert a \Vert^{2}\geq c\varepsilon_{0}^{-1}$.
Moreover
\begin{equation}
\begin{split}
-\frac{\dot{\lambda}}{\lambda} 
= & 
80\gamma_{2}\frac{\Delta K(a)}{\lambda^{2}}(1+o_{\varepsilon}(1))+O(\vert \delta J(u)\vert^{2})\\
= &
-c\frac{\Vert a \Vert^{2}}{\lambda^{2}}+O(\vert \delta J(u)\vert^{2}) \\
\leq &
-c\frac{\varepsilon^{-1}}{\lambda^{3}}+O(\vert \delta J(u)\vert^{2}),
\end{split}
\end{equation} 
whence
\begin{equation}\begin{split}
\partial_{t}\lambda^{3}+\lambda^{3}O(\vert \delta J(u)\vert^{2})
\geq
C\varepsilon^{-1}.
\end{split}\end{equation}
Letting $\vartheta=\lambda^{3}$ this becomes 
\begin{equation}\begin{split}
\dot{\vartheta}+\vartheta O(\vert \delta J(u)\vert^{2})
\geq
C\varepsilon^{-1}
.
\end{split}\end{equation}
Thus for $\tau(t)=\vartheta(t)e^{\int_{0}^{t}O(\vert \delta J(u)\vert^{2})}$
there holds
\begin{equation}
\begin{split}
\dot{\tau}(t) 
= & (\dot{\vartheta}+\vartheta O(\vert \delta J(u)\vert^{2}))(t)
e^{\int_{0}^{t}O(\vert \delta J(u)\vert^{2})}
\geq 
C\varepsilon^{-1}
e^{\int_{0}^{t}O(\vert \delta J(u)\vert^{2})}
\end{split}
\end{equation} 
and therefore
\begin{equation}\label{tau_evolution_divergence}
\begin{split}
\dot{\tau}(t) 
\geq c\varepsilon^{-1},
\end{split}
\end{equation}
whence
\begin{equation}\begin{split}\label{theta_estimate_divergence}
\vartheta(0)=\tau(0)\leq\tau(t)=\vartheta(t)e^{\int_{0}^{t}O(\vert \delta J(u)\vert^{2})}\leq
C\vartheta(t),
\end{split}\end{equation}
so $\vartheta$ and thereby $\lambda$ remain uniformly large, say $\lambda \geq c\varepsilon_{0}^{-1}$.
Finally note, that 
\begin{equation}
\begin{split}
(\frac{a_{i}}{a_{j}})^{'} 
= & 
\frac{\lambda\dot{a}_{i}}{\lambda a_{j}}-\frac{a_{i}}{a_{j}}\frac{\lambda\dot{a}_{j}}{\lambda a_{j}}\\
= & 
-c(\frac{\vert a_{i}\vert^{2}a_{i}}{\lambda^{2}a_{j}}
-
\frac{a_{i}}{a_{j}}\frac{\vert a_{j}\vert^{2}a_{j}}{\lambda^{2}a_{j}})(1+o_{\varepsilon}(1))
+
O(\frac{\vert \delta J(u)\vert^{2}}{\lambda a_{j}})\\
= & 
-c\frac{a_{i}}{a_{j}}\frac{1}{\lambda^{2}}(\vert a_{i}\vert^{2}-\vert a_{j}\vert^{2})(1+o_{\varepsilon}(1))
+
O(\frac{\vert \delta J(u)\vert^{2}}{\lambda a_{j}}),
\end{split}
\end{equation} 
whence without loss of generality we may assume
\begin{align*}
(\frac{a_{\max}}{a_{\min}})^{'}
\leq 
C\frac{\vert \delta J(u)\vert^{2}}{\lambda a_{\min}}
\; \text{ in case }\;
\frac{a_{\max}}{a_{\min}}\geq \sqrt[4]{\frac{5}{4}}.
\end{align*}
But during $(0,T)$ we have
\begin{equation}
\lambda a_{\min}\geq \sqrt[4]{\frac{2}{5}}\lambda  a_{\max} \geq c\lambda \Vert a \Vert
\end{equation} 
and 
$\lambda \Vert a \Vert
=
\lambda^{\frac{1}{2}}(\lambda \Vert a \Vert^{2})^{\frac{1}{2}}
>\varepsilon^{-1}
$, whence 
\begin{equation}
\begin{split}
\partial_{t}\ln(\frac{a_{\max}}{a_{\min}})
\leq 
C\varepsilon\vert \delta J(u)\vert^{2}
\; \text{ in case }\;
\frac{a_{\max}}{a_{\min}}\geq \sqrt[4]{\frac{5}{4}}.
\end{split}
\end{equation} 
Consequently we may assume 
\begin{equation}
\begin{split}
\frac{a_{i}}{a_{j}}<\sqrt[4]{\frac{5}{3}} 
\; \text{ during }\; (0,T).
\end{split}
\end{equation} 
So far we have seen, that during $(0,T)$ we may assume
\begin{equation}
\begin{split}
\Vert a \Vert <C\varepsilon_{0}, \lambda \Vert a \Vert^{2}>c\varepsilon_{0}^{-1}, \, \lambda >c\varepsilon_{0}^{-1}
\; \text{ and }\; 
\frac{a_{i}}{a_{j}}<\sqrt[4]{\frac{5}{3}}.
\end{split}
\end{equation} 
In order to show $T=\infty$ it remains to prove
\begin{equation}
u\in V(1, \frac{\varepsilon}{2})\; \text{ during }\; (0,T).
\end{equation} 
By definition \ref{def_V(omega,p,e)} and the remarks thereafter this is equivalent to showing
\begin{equation}
\vert 1-\frac{r\alpha^{\frac{4}{n-2}} K(a)}{4n(n-1)k}\vert, \Vert u-\alpha \var_{a, \lambda}\Vert = \Vert v \Vert<\frac{\varepsilon}{2}.
\end{equation} 
To that end let us expand using $k\equiv 1$
\begin{equation}
\begin{split}
J(u)
= & 
r
=
\int L_{g_{0}}uu
=
\int L_{g_{0}}(\alpha \var_{a, \lambda}+v)(\alpha \var_{a, \lambda}+v) \\
= &
\alpha^{2}\int L_{g_{0}}\var_{a, \lambda}\var_{a, \lambda} 
+
2\alpha \int L_{g_{0}}\var_{a, \lambda}v
+
\int L_{g_{0}}vv.
\end{split}
\end{equation} 
Due to lemmata  \ref{lem_emergence_of_the_regular_part} and  \ref{lem_interactions} we have with $n=5$
\begin{equation}
\int L_{g_{0}}\varphi_{a, \lambda}\varphi_{a, \lambda}
=
4n(n-1)c_{0}+o_{\frac{1}{\lambda}}(1).
\end{equation} 
Moreover from lemma \ref{lem_emergence_of_the_regular_part} we get 
\begin{equation}\label{divergence_v_interaction}
\begin{split}
\frac{\int L_{g_{0}}\varphi_{a, \lambda}v }{4n(n-1)}
= &
\int \varphi_{a, \lambda}^{\frac{n+2}{n-2}}v
+
o_{\frac{1}{\lambda}}(1) 
= 
\int K \varphi_{a, \lambda}^{\frac{n+2}{n-2}}v
+
o_{\frac{1}{\lambda}}(1) \\
= &
\alpha^{-\frac{4}{n-2}}\int K (u-v)^{\frac{4}{n-2}}\varphi_{a, \lambda}v
+
o_{\frac{1}{\lambda}}(1) \\
= &
-\frac{4}{n-2}\alpha^{-\frac{4}{n-2}}\int Ku^{\frac{6-n}{n-2}}\varphi_{a, \lambda} v^{2}
+
o(\Vert v \Vert^{2})+o_{\frac{1}{\lambda}}(1)\\
= &
-\frac{4}{n-2}\alpha^{-1}\int K\varphi_{a, \lambda}^{\frac{4}{n-2}} v^{2}
+
o(\Vert v \Vert^{2})+o_{\frac{1}{\lambda}}(1).
\end{split}
\end{equation} 
We conclude
\begin{equation}\label{J(u)_estimate_divergence}
\begin{split}
J(u)
= & 
4n(n-1)c_{0}\alpha^{2} 
+
\int L_{g_{0}}vv
-
\frac{32n(n-1)}{n-2}\int \varphi_{a, \lambda}^{\frac{4}{n-2}}v^{2}\\
& +
o_{\frac{1}{\lambda}}(1)+o(\Vert v \Vert^{2})
\end{split}
\end{equation} 
On the other hand we have
\begin{equation}
\begin{split}
1
\equiv &
k
=
\int Ku^{\frac{2n}{n-2}}
=
\int K(\alpha \varphi_{a, \lambda}+v)^{\frac{2n}{n-2}}\\
= &
\alpha^{\frac{2n}{n-2}}\int K\varphi_{a, \lambda}^{\frac{2n}{n-2}}
+
\frac{2n}{n-2}\alpha^{\frac{n+2}{n-2}}\int K\varphi_{a, \lambda}^{\frac{n+2}{n-2}}v \\
& +
\frac{n}{n-2}\frac{n+2}{n-2}\alpha^{\frac{4}{n-2}}\int K\varphi_{a, \lambda}^{\frac{4}{n-2}}v^{2}
+
o(\Vert v \Vert^{2}).
\end{split}
\end{equation}
Considering the second summand above we obtain
using  \eqref{divergence_v_interaction} 
\begin{equation}
\begin{split}
1
= &
\alpha^{\frac{2n}{n-2}}c_{0}
+
\frac{n(n-6)}{(n-2)^{2}}
\alpha^{\frac{4}{n-2}}\int \varphi_{a, \lambda}^{\frac{4}{n-2}}v^{2}
+
o_{\frac{1}{\lambda}}(1)+o(\Vert v \Vert^{2}),
\end{split}
\end{equation}
whence
\begin{equation}\label{alpha_estimate_divergence}
\alpha
=
c_{0}^{-\frac{n-2}{2n}}
+
\frac{6-n}{2(n-2)}c_{0}^{-\frac{n+2}{2n}}\int \varphi_{a, \lambda}^{\frac{4}{n-2}}v^{2}
+
o_{\frac{1}{\lambda}}(1)+o(\Vert v \Vert^{2})
\end{equation} 
and therefore
\begin{equation}
c_{0}\alpha^{2}
=
c_{0}^{-\frac{2}{n}}
+
\frac{6-n}{n-2}\int \varphi_{a, \lambda}^{\frac{4}{n-2}}v^{2}
+
o_{\frac{1}{\lambda}}(1)+o(\Vert v \Vert^{2}).
\end{equation} 
We conclude
\begin{equation}
\begin{split}
J(u)
= &
4n(n-1)c_{0}^{\frac{2}{n}} \\
& +
\int L_{g_{0}}vv
-
4n(n-1)\frac{n+2}{n-2}
\int \varphi_{a, \lambda}^{\frac{4}{n-2}}v^{2} 
+
o_{\frac{1}{\lambda}}(1)+o(\Vert v \Vert^{2}) \\
\geq &
4n(n-1)c_{0}^{\frac{2}{n}} \\
& +
c_{n}\int 
\left(\vert \nabla v \vert^{2}_{g_{0}}
-
n(n+2)\int \varphi_{a, \lambda}^{\frac{4}{n-2}}v^{2} 
\right)
+
o_{\frac{1}{\lambda}}(1)+o(\Vert v \Vert^{2})
\end{split}
\end{equation} 
and thus by means of proposition \ref{prop_positivity_of_D2J},
\begin{equation}
\begin{split}
J(u)
\geq 
4n(n-1)c_{0}^{\frac{2}{n}} 
+
o_{\frac{1}{\lambda}}(1)+c\Vert v \Vert^{2}
.
\end{split}
\end{equation} 
But $J(u)\leq J(u_{0})=4n(n-1)c_{0}^{\frac{2}{n}}+O(\frac{1}{\lambda_{0}})$ and therefore
\begin{equation}
\Vert v \Vert^{2}=o_{\frac{1}{\lambda}+\frac{1}{\lambda_{0}}}(1)
\end{equation} 
remains uniformly small during $(0,T)$.
Finally we infer from \eqref{alpha_estimate_divergence}, that $\alpha$ remains 
uniformly close to $c_{0}^{-\frac{n-2}{2n}}$, in particular
\begin{equation}
\begin{split}
\frac{r\alpha^{\frac{4}{n-2}}K(a)}{4n(n-1)k}
= &
\frac{\alpha^{\frac{4}{n-2}} J(u)}{4n(n-1)}
+
O(\Vert a \Vert) \\
= &
\alpha^{\frac{4}{n-2}}c^{\frac{2}{n}}_{0}+O(\Vert a \Vert +\Vert v \Vert)+o_{\frac{1}{\lambda}}(1) \\
= &
1+O(\Vert a \Vert +\Vert v \Vert)+o_{\frac{1}{\lambda}}(1),
\end{split}
\end{equation} 
whence 
\begin{equation}
\vert 1-\frac{r\alpha^{\frac{4}{n-2}}K(a)}{4n(n-1)k}\vert
\end{equation} remains uniformly small. This completes the proof of $T=\infty$, which is to say, that $u$ remains in $V(1, \varepsilon)$. 
Turning back to \eqref{tau_evolution_divergence} we then get $\tau\geq t$ as $t\- \infty$, whence 
according to \eqref{theta_estimate_divergence} 
\begin{equation}
\vartheta=\lambda^{3}\geq ct.
\end{equation} 
Finally \eqref{derivatives_of_K} and \eqref{Norma'_divergence} show
\begin{equation}
\begin{split}
\partial_{t} \Vert a \Vert^{2} \leq -c\frac{ \Vert a \Vert^{4}}{\lambda^{2}} 
+
O(\frac{ \Vert a \Vert\vert \delta J(u)\vert^{2}}{\lambda})
=
\Vert a \vert^{2}
(
-
\frac{\Vert a \Vert^{2}}{\lambda^{2}}
+
O(\frac{ \vert \delta J(u)\vert^{2}}{\Vert a \Vert\lambda})
)
.
\end{split}
\end{equation}
Since $\lambda \Vert a \Vert^{2}$ and therefore $\lambda \Vert a \Vert$ as well remain large we obtain
\begin{equation}
\partial_{t}\ln \Vert a \Vert^{2}
\leq 
-c\frac{\Vert a \Vert^{2}}{\lambda^{2}}+O( \vert \delta J(u)\vert^{2}),
\end{equation} 
whence due to \eqref{derivatives_of_K} and \eqref{lambda_dot_simplified}
\begin{equation}
\partial_{t}\ln \Vert a \Vert^{2}
\leq 
-c\frac{\dot \lambda}{\lambda}+O( \vert \delta J(u)\vert^{2})
=
-
c\partial_{t}\ln \lambda+O(\vert \delta J(u)\vert^{2})
\end{equation} 
Therefore $\lambda\- \infty$ implies $\Vert a \Vert\- 0$.
\end{proof}

\section{Appendix}
\label{sec:Appendix}
\begin{lemma} \label{App1}$_{}$\\
Let $(M^{n},g_{0})$ be a Riemannian manifold,
$g(t)=u^{\frac{4}{n-2}}(t)g_{0}, \;u>0$.
There holds
\begin{enumerate}[label=(\roman*)]
\item 
$$
d\mu_{g} 
= 
u^{\frac{2n}{n-2}}d\mu_{g_{0}}
$$
\item \quad
$$
_{g}\Gamma^{k}_{i,j} 
= 
_{\hspace{3pt}g_{0}}\hspace{-3pt}\Gamma^{k}_{i,j}
+
\frac{2}{n-2}u^{-1}
(
\partial_{i}u \delta^{k}_{j}
+
\partial_{j}u \delta^{k}_{i}
-
\partial_{l}u g^{k,l}g_{i,j}
)
$$
\item 
\begin{equation*}\begin{split}
\tilde{R}^{l}_{i,j,k}
= &
R^{l}_{i,j,k} 
+
\frac{2}{n-2}u^{-1}
[
\nabla^{2}_{i,k}u
\delta^{l}_{j}
-
\nabla^{2}_{i,p}u
g^{l,p}g_{j,k} \\
& \quad\quad\quad\quad \quad\quad \quad\;\,
-
\nabla^{2}_{j,k}u
\delta^{l}_{i}
+
\nabla^{2}_{j,p}u
g^{l.p}g_{i,k}
] \\
& -
\frac{2n}{(n-2)^{2}}u^{-2}
[
\nabla_{i}u\nabla_{k} u \delta^{l}_{j}
-
\nabla_{i}u\nabla_{p}ug^{l,p}g_{j,k}
 \\
& \quad\quad\quad\quad \quad\quad\;\;-
\nabla_{j}u \nabla_{k}u \delta^{l}_{i}
+
\nabla_{j}u \nabla_{p}u g ^{l,p}g_{i,k} \\
& \quad\quad\quad\quad \quad\quad\;\; 
+
\frac{2}{n}\vert \nabla u \vert^{2} g_{j,k}\delta^{l}_{i}
-
\frac{2}{n}\vert \nabla u \vert^{2} g_{i,k}\delta^{l}_{j}
]  
\end{split}\end{equation*}
\item 
\begin{equation*}\begin{split}
\tilde{R}_{i,k} 
= 
R_{i,k}
& +
\frac{2}{n-2}u^{-1}
[
(n-2)\nabla^{2}_{i,k}u
-
\lap u g_{i,k}
] \\
& -
\frac{2}{n-2}u^{-2}
[
n\nabla_{i}u\nabla_{k}u
-
\vert \nabla u \vert^{2}g_{i,k}
] 
\end{split}\end{equation*}
\item 
$
R=R_{g}
=
u^{-\frac{n+2}{n-2}}[-c_{n}\lap u + R_{g_{0}}u]
=
u^{-\frac{n+2}{n-2}}L_{g_{0}} u
, \;\text{i.e.}\quad \hspace{-0.46pt}
$
\begin{equation*}
\begin{split}
 u^{-\frac{n+2}{n-2}}L_{g_{0}}(uv)=L_{g}(v)
\end{split}\end{equation*}

\item 
and for $\partial_{t}u=-\frac{1}{K}(R-r\K)u$ we have  
\begin{equation*}\begin{split}
\partial_{t} R
= 
c_{n}\lap_{g}\frac{R}{K}
+
\frac{4}{n-2}(R-r\K)\frac{R}{K}.
\end{split}\end{equation*}
\end{enumerate}

\end{lemma}

\begin{lemma}\label{App3}[Local bound and higher integrability, cf. \cite{StruweLargeEnergies}, Theorem A.1.]$_{}$\\
Let $P\in C^{\infty}(M), \, p>\frac{2n}{n-2}$ and $r>0$ small.

There exists $C=C(p,r)$ such, that for $u>0$ solving
$L_{g_{0}}u=Pu$ with
\begin{equation*}\begin{split}
\Vert P \Vert_{L_{g_{0}}^{\frac{n}{2}}(B_{2r}(x_{0}))}<\frac{2n}{n-2}\frac{Y(M,g_{0})}{p}
\end{split}\end{equation*}

we have
\begin{equation*}
\begin{split}
\Vert u \Vert_{L_{g_{0}}^{p}(B_{r}(x_{0}))}\leq C\Vert u\Vert_{L_{g_{0}}^{\frac{2n}{n-2}}(B_{2r}(x_{0}))}.
\end{split}
\end{equation*}
\end{lemma}

\begin{proof}[\textbf{Proof of lemma \ref{lem_emergence_of_the_regular_part}}]\label{p_emergence_of_the_regular_part}$_{}$\\
A straight forward calculation shows
\begin{equation}\begin{split}
\lap_{g_{ a }}(\frac{\lambda}{1+\lambda^{2}\gamma_{n}G_{ a }^{\frac{2}{2-n}}})^{\frac{n-2}{2}}
= &
\frac{n}{2-n}\gamma_{n}(\frac{\varphi_{a, \lambda}}{u_{a}})^{\frac{n+2}{n-2}}
\vert \nabla G_{ a } \vert^{2}_{g_{a}} 
G_{a}^{2\frac{n-1}{2-n}}\\
& +
\gamma_{n}\lambda(\frac{\varphi_{a, \lambda}}{u_{a}})^{\frac{n}{n-2}}G_{ a }^{\frac{n}{2-n}}\lap_{g_{a}} G_{ a },
\end{split}\end{equation}
which is due to
\begin{equation}
\begin{split}
\vert \nabla G_{ a } \vert^{2}_{g_{a}} 
G_{a}^{2\frac{n-1}{2-n}}
=
(n-2)^{2}\vert \nabla G_{ a }^{\frac{1}{2-n}} \vert^{2}_{g_{a}}
\;
\text{ and }\;
\lap_{g_{a}}G_{a}=\frac{-\delta_{a}+R_{g_{a}}G_{a}}{c_{n}},
\end{split} 
\end{equation}
where $\delta_{a}$ denotes the Dirac measure at $a$, equivalent to
\begin{equation}\begin{split}
\lap_{g_{ a }}(\frac{\lambda}{1+\lambda^{2}\gamma_{n}G_{ a }^{\frac{2}{2-n}}})^{\frac{n-2}{2}}
= &
n(2-n)\gamma_{n}(\frac{\varphi_{a, \lambda}}{u_{a}})^{\frac{n+2}{n-2}}
\vert \nabla G_{ a }^{\frac{1}{2-n}} \vert^{2}_{g_{a}}\\
& +
\frac{R_{g_{ a }}\gamma_{n}}{c_{n}}\lambda(\frac{\varphi_{a, \lambda}}{u_{a}})^{\frac{n}{n-2}}G_{ a }^{\frac{2}{2-n}}
.
\end{split}\end{equation}
Since $L_{g_{a}}=-c_{n}\lap_{g_{a}}+R_{g_{a}}$ with $c_{n}=4\frac{n-1}{n-2}$ we obtain
\begin{equation}\begin{split}
L_{g_{ a }}\frac{\varphi_{a, \lambda}}{u_{a}}
= &
4n(n-1)(\frac{\varphi_{a, \lambda}}{u_{a}})^{\frac{n+2}{n-2}}\gamma_{n}
\vert \nabla G_{ a }^{\frac{1}{2-n}} \vert^{2}_{g_{a}} \\
& +
R_{g_{a}}\frac{\varphi_{a, \lambda}}{u_{a}}
(
1-\lambda \gamma_{n}G_{a}^{\frac{2}{2-n}}(\frac{\varphi_{a, \lambda}}{u_{a}})^{\frac{2}{2-n}}
)
\\
= &
4n(n-1)(\frac{\varphi_{a, \lambda}}{u_{a}})^{\frac{n+2}{n-2}}\gamma_{n}
\vert \nabla G_{ a }^{\frac{1}{2-n}} \vert^{2}_{g_{a}} 
+
\frac{R_{g_{ a }}}{\lambda}(\frac{\varphi_{a, \lambda}}{u_{a}})^{\frac{n}{n-2}}.
\end{split}\end{equation}
By conformal invariance, cf. lemma \ref{App1}, we conclude
\begin{equation}\begin{split}
L_{g_{0}} \varphi_{a, \lambda}
= &
u_{a}^{\frac{n+2}{n-2}}L_{g_{ a }}\frac{\varphi_{a, \lambda}}{u_{a}} \\
= &
4n(n-1)\varphi_{a, \lambda}^{\frac{n+2}{n-2}}\gamma_{n}\vert \nabla G_{ a }^{\frac{1}{2-n}} \vert^{2}_{g_{a}}
+
\frac{u_{a}^{\frac{2}{n-2}}R_{g_{a}}}{\lambda}\varphi_{a, \lambda}^{\frac{n}{n-2}},
\end{split}\end{equation}
in particular $L_{g_{0}}\varphi_{a, \lambda}=O(\varphi_{a, \lambda}^{\frac{n+2}{n-2}})$.
Expanding 
\begin{equation}
\begin{split}
G_{ a }=\frac{1}{4n(n-1)\omega _{n}}(r_{a}^{2-n}+H_{ a }), \, r_{ a }=d_{g_{a}}(a, \cdot)
\end{split}
\end{equation} we derive
\begin{equation}\begin{split}
\gamma_{n} \vert \nabla G_{ a }^{\frac{1}{2-n}} \vert^{2}_{g_{a}}
= &
\vert \nabla(r_{a}(1+r_{a}^{n-2} H_{ a })^{\frac{1}{2-n}})\vert^{2}_{g_{a}} \\
= &
\vert \nabla r_{a}(1+\frac{1}{2-n}r_{a}^{n-2} H_{ a }+O(\vert r_{a}^{n-2} H_{ a }\vert ^{2})) \\
 & \, +
r_{a}(-r_{a}^{n-3}\nabla r_{a}   H_{ a }+\frac{1}{2-n}r_{a}^{n-2}\nabla  H_{ a } \\
& \quad\quad\quad\quad\quad\quad\quad\,+
O(\vert r_{a}^{n-2} H_{ a }\vert \vert \nabla (r_{a}^{n-2} H_{ a })\vert))\vert^{2}_{g_{a}}, 
\end{split}\end{equation}
whence 
\begin{equation}\begin{split}
\gamma_{n} \vert \nabla G_{ a }^{\frac{1}{2-n}}\vert^{2}_{g_{a}}
= &
1
-
\frac{2}{n-2}
( (n-1)H_{ a }+r_{a}\partial_{r_{a}}H_{ a }) r_{a}^{n-2}
+
o(r_{a}^{n-2}).
\end{split}\end{equation}
Thus we conclude
\begin{equation}\begin{split}
L_{g_{0}} \varphi_{a, \lambda}
= &
4n(n-1)\varphi_{a, \lambda}^{\frac{n+2}{n-2}}
-
2nc_{n}
((n-1)H_{ a }+r_{a}\partial_{r_{a}}H_{ a }) r_{a}^{n-2}\varphi_{a, \lambda}^{\frac{n+2}{n-2}} \\
& +
o(r_{a}^{n-2}\varphi_{a, \lambda}^{\frac{n+2}{n-2}})
+
\frac{u_{a}^{\frac{2}{n-2}}R_{g_{a}}}{\lambda}\varphi_{a, \lambda}^{\frac{n}{n-2}}.
\end{split}\end{equation}
For $R_{g_{a}}=O(r_{a}^{2})$ and $\lap R_{g_{a}}=-\frac{1}{6}\vert W(a)\vert^{2}$ cf. \cite{LeeAndParker}.
\end{proof}
\begin{proof}[\textbf{Proof of lemma \ref{lem_interactions}}]\label{p_interactions} $_{}$\\
These kind of expansions are well known, cf.  \cite{BahriCriticalPointsAtInfinity}. 
Using but just slightly modified bubbles we nonetheless repeat their proves. 
\begin{enumerate}[label=(\roman*)]
 \item 
We have 
\begin{equation}\begin{split}
(\phi_{k,i})_{k=1,2,3}
= &
(\varphi_{i},-\lambda_{i}\partial_{\lambda_{i}}\varphi_{i}, \frac{1}{\lambda_{i}}\nabla_{a_{i}}\varphi_{i}),
\end{split}\end{equation}
so
\begin{equation}\begin{split}\label{Phi1i}
\phi_{1,i}
= & 
u_{a_{i}}
(\frac{\lambda_{i}}{1+\lambda_{i}^{2}\gamma_{n}G^{\frac{2}{2-n}}_{a_{i}}}
)^{\frac{n-2}{2}}
\end{split}\end{equation}
and
\begin{equation}\begin{split}\label{Phi2i}
\phi_{2,i}
= &
\frac{n-2}{2}
\frac
{\lambda_{i}^{2}\gamma_{n}G^{\frac{2}{2-n}}_{a_{i}}-1}
{\lambda_{i}^{2}\gamma_{n}G^{\frac{2}{2-n}}_{a_{i}}+1}
\varphi_{i}
\end{split}\end{equation}
and
\begin{equation}\begin{split}\label{Phi3i}
\phi_{3,i}
= &
-\frac{n-2}{2}
u_{a_{i}}
\frac
{\lambda_{i}\gamma_{n}\nabla_{a_{i}}G^{\frac{2}{2-n}}_{a_{i}}}
{1+\lambda_{i}^{2}\gamma_{n}G^{\frac{2}{2-n}}_{a_{i}}} 
\varphi_{i}
+
\frac{\nabla_{a_{i}}u_{a_{i}}}{u_{a_{i}}\lambda_{i}}
\varphi_{i}.
\end{split}\end{equation}
Note, that in $x\simeq \exp_{g_{a_{i}}}x$ coordinates
\begin{equation}\begin{split}
\gamma_{n}G_{a_{i}}^{\frac{2}{2-n}}(x)
= 
r^{2}+O(r^{n}),
\end{split}\end{equation}
cf. definition \ref{def_bubbles} and 
\begin{equation}\begin{split}
\gamma_{n}(\nabla_{a_{i}}G^{\frac{2}{2-n}}_{a_{i}})(x)
=
-2x+O(r^{n-1}).
\end{split}\end{equation}
Moreover $u_{a_{i}}=1+O(r^{2})$. The assertion readily follows.
\item
\begin{enumerate}
 \item[$(\alpha)$] \quad\quad  Case $k=1$ \\
We have $\phi_{k,i}=\varphi_{i}$ for $k=1$ and thus for $c>0$ small
\begin{equation}\begin{split}
\int \varphi_{i}^{\frac{2n}{n-2}}
= &
\int_{B_{c}(a_{i})}(\frac{\lambda_{i}}{1+\lambda_{i}^{2}\gamma_{n}G^{\frac{2}{2-n}}_{a_{i}}})^{n}d\mu_{g_{a_{i}}}
+
O(\frac{1}{\lambda_{i}^{n}}).
\end{split}\end{equation}
By definition \ref{def_bubbles} one has
passing to $x\simeq \exp_{g_{a_{i}}}x$ coordinates
\begin{equation}\begin{split}
\int \varphi_{i}^{\frac{2n}{n-2}}
= &
\underset{B_{c}(0)}{\int}
(\frac{\lambda_{i}}{1+\lambda_{i}^{2}r^{2}})^{n}
(
1
+
O(\frac{\lambda_{i}^{2}O(r^{n}H_{a})}{1+\lambda_{i}^{2}r^{2}})
)
+
O(\frac{1}{\lambda_{i}^{n}})
\\
= &
\underset{B_{c\lambda_{i}}(0)}{\int}
\frac{1}{(1+r^{2})^{n}}
+
O(\frac{1}{\lambda_{i}^{n-2}})
= 
c_{1}+O(\frac{1}{\lambda_{i}^{n-2}}).
\end{split}\end{equation}
 \item[$(\beta)$] \quad\quad  Case $k=2$ \\
The proof runs analogously to the one of case $k=1$ above yielding
\begin{equation}\begin{split}\label{a2}
c_{2}
= &
\frac{(n-2)^{2}}{4}\int \frac{\vert r^{2}-1\vert^{2}}{(1+r^{2})^{n+2}}
\end{split}\end{equation}

 \item[$(\gamma)$] \quad\quad  Case $k=3$ \\
We have
\begin{equation}\begin{split}
\phi_{k,i}
= &
\frac{2-n}{2}
\frac
{\lambda_{i}\gamma_{n}\nabla_{a_{i}}G^{\frac{2}{2-n}}_{a_{i}}}
{1+\lambda_{i}^{2}\gamma_{n}G^{\frac{2}{2-n}}_{a_{i}}}\varphi_{i}
+
\frac{\nabla_{a_{i}}u_{a_{i}}}{\lambda_{{i}}u_{a_{i}}}\varphi_{i},
\end{split}\end{equation}
whence using $\gamma_{n}(\nabla_{a_{i}}G^{\frac{2}{2-n}}_{a_{i}})(x)
=
-2x+O(r^{n-1})$ and $u_{a_{i}}=O(r^{2})$
\begin{equation}\begin{split}\label{a3}
\int \vert \phi_{k,i}\vert^{2}\varphi_{i}^{\frac{4}{n-2}}
= &
\frac{(n-2)^{2}}{n}\int \frac{r^{2}}{(1+r^{2})^{n+2}}
+
O(\frac{1}{\lambda_{i}^{n-2}}+\frac{1}{\lambda_{i}^{2}})\\
= &
c_{3}
+
O(\frac{1}{\lambda_{i}^{n-2}}+\frac{1}{\lambda_{i}^{2}})
\end{split}\end{equation}
\end{enumerate}
 \item  
\begin{enumerate}
 \item[$(\alpha)$] \quad\quad  Case $k=1$ \\$_{}$\\
Due to lemma \ref{lem_emergence_of_the_regular_part} and case $(v)$ we have for $c>0$ small
\begin{equation}\begin{split}
\int \varphi_{i}^{\frac{n+2}{n-2}}\varphi_{j}
= &
\underset{B_{c}( a _{i})}{\int}
\frac{L_{g_{0}}\varphi_{i}\varphi_{j}
}{4n(n-1)}
+
o(\eps_{i,j}) \\
= &
\int
\frac
{
L_{g_{0}}\varphi_{i}
\varphi_{j}
}
{4n(n-1)}+
o(\eps_{i,j}),
\end{split}\end{equation}
whence by $\int L_{g_{0}} \varphi_{i}\varphi_{j}=\int \varphi_{i}L_{g_{0}}\varphi_{j}$ and backward calculation
\begin{equation}
\begin{split}
 \int \varphi_{i}^{\frac{n+2}{n-2}}\varphi_{j}
= &
\int \varphi_{i}\varphi_{j}^{\frac{n+2}{n-2}}+o(\eps_{i,j}).
\end{split}
\end{equation} 
Thus we may assume $ \frac{1}{\lambda_{i}} \leq  \frac{1}{\lambda_{j}}$. We get
\begin{equation}\begin{split}\label{di^...dj=didj^...+o(eij)}
\int \varphi_{i}^{\frac{n+2}{n-2}} \varphi_{j}
= &
\underset{B_{c}( a _{i})}{\int}
(\frac{ \lambda_{i} }{ 1+\lambda_{i}^{2}\gamma_{n}G_{ a _{i}}^{\frac{2}{2-n}}})^{\frac{n+2}{2}} \\
& \quad\quad\quad\quad\quad\quad
\frac{u_{ a _{j}}}{u_{ a _{i}}}(\frac{ \lambda_{j} }{ 1+\lambda_{j}^{2}\gamma_{n}G_{ a _{j}}^{\frac{2}{2-n}}})^{\frac{n-2}{2}}d\mu_{g_{ a _{i}}}
\\ & +
O(  \frac{1}{\lambda_{i}^{\frac{n+2}{2}}}  \frac{1}{\lambda_{j}^{\frac{n-2}{2}}}).
\end{split}\end{equation}
Clearly 
$ \lambda_{i}^{-\frac{n+2}{2}}\lambda_{j}^{-\frac{n-2}{2}}=o(\eps_{i,j})$
and in
$
x\simeq \exp_{g_{ a _{i}}}(x)
$
coordinates
\begin{equation}
\begin{split}
u_{ a _{i}}(x)=1+O(r^{2})
\;\text{ and }\;
\gamma_{n}G^{\frac{2}{2-n}}_{ a _{i}}(x)=r^{2}+O(r^{n}),
\end{split}
\end{equation}  
whence using case (v)
\begin{equation}\begin{split}\label{expanding_epsij}
\int \varphi_{i}^{\frac{n+2}{n-2}} \varphi_{j}
= &
\underset{B_{ c\lambda_{i} }(0)}{\int}
\frac{u_{ a _{j}}( a _{i})}{(1+r^{2})^{\frac{n+2}{2}}} \\
& \quad\;\;\,
(\frac{1}{\frac{\lambda_{i} }{ \lambda_{j} }+\lambda_{i}\lambda_{j}\gamma_{n}G_{ a _{j}}^{\frac{2}{2-n}}(\exp_{g_{ a _{i}}} \frac{x}{\lambda_{i}})})^{\frac{n-2}{2}}
+
o(\eps_{i,j}).
\end{split}\end{equation}
Due to $\frac{1}{\lambda_{i}}\leq \frac{1}{\lambda_{j}}$ we have
\begin{equation}\begin{split}
\eps_{i,j}^{\frac{2}{2-n}}
\sim 
\lambda_{i}\lambda_{j}\gamma_{n}G_{g_{0}}^{\frac{2}{2-n}}( a _{i}, a _{j})
\; \text{ or }\;
\eps_{i,j}^{\frac{2}{2-n}}
\sim
\frac{\lambda_{i}}{\lambda_{j}}
\end{split}\end{equation}
and may expand on
\begin{equation}\begin{split}
\mathcal{A}
=
\left[
\vert  \frac{x}{\lambda_{i}}\vert \leq \epsilon\sqrt{\gamma_{n}G^{\frac{2}{2-n}}_{ a _{j}}( a _{i})}
\right]
\cup
\left[
\vert  \frac{x}{\lambda_{i}}\vert \leq \epsilon  \frac{1}{\lambda_{j}} 
\right]
\subset
B_{c \lambda_{i}}(0)
\end{split}\end{equation}
for $\epsilon>0$ sufficiently small
\begin{equation}\begin{split}
(
\frac{\lambda_{i} }{ \lambda_{j} }
& +
\lambda_{i}\lambda_{j}\gamma_{n}G_{ a _{j}}^{\frac{2}{2-n}}(\exp_{g_{ a _{i}}} \frac{x}{\lambda_{i}}))^{\frac{2-n}{2}} \\
= &
(
\frac{\lambda_{i} }{ \lambda_{j} }
+
\lambda_{i}\lambda_{j}
\gamma_{n}G_{ a _{j}}^{\frac{2}{2-n}}( a _{i})
)^{\frac{2-n}{2}} \\
& +
\frac{2-n}{2}
\frac
{
\gamma_{n}\nabla G^{\frac{2}{2-n}}_{ a _{j}}( a _{i})\lambda_{j}x
+
O(\frac{ \lambda_{j} }{ \lambda_{i} }\vert x\vert^{2})
}
{
(
\frac{\lambda_{i} }{ \lambda_{j} }
+
\lambda_{i}\lambda_{j}\gamma_{n}G_{ a _{j}}^{\frac{2}{2-n}}( a _{i})
)^{\frac{n}{2}}
}.
\end{split}\end{equation}
 Thus by \eqref{expanding_epsij}
 \begin{equation}\begin{split}
 \int \varphi_{i}^{\frac{n+2}{n-2}}\varphi_{j}
 = &
 \sum_{k=1}^{4}I_{k}+o(\eps_{i,j})
 \end{split}\end{equation}
 with 
 \begin{equation}\begin{split}
 I_{1}
 = &
 \frac
 {u_{ a _{j}}( a _{i})}
 {(\frac{\lambda_{i} }{ \lambda_{j} }
 +
 \lambda_{i}\lambda_{j}
 \gamma_{n}G_{ a _{j}}^{\frac{2}{2-n}}( a _{i})
 )^{\frac{n-2}{2}}
 }
\int_{\mathcal{A}}
 \frac{1}{(1+r^{2})^{\frac{n+2}{2}}}
 \end{split}\end{equation}
 and
 \begin{equation}\begin{split}
 I_{2}
 = &
 -
 \frac
 {\frac{n-2}{2}u_{ a _{j}}( a _{i})\gamma_{n}}
 {(\frac{\lambda_{i} }{ \lambda_{j} }
 +
 \lambda_{i}\lambda_{j}
 \gamma_{n}G_{ a _{j}}^{\frac{2}{2-n}}( a _{i})
 )^{\frac{n}{2}}
 }
\int_{\mathcal{A}} \frac{\nabla G^{\frac{2}{2-n}}_{ a _{j}}( a _{i})\lambda_{j}x}{(1+r^{2})^{\frac{n+2}{2}}}
 \end{split}\end{equation}
 and 
 \begin{equation}\begin{split}
 I_{3}
 = &
 \frac
 {u_{ a _{j}}( a _{i})}
 {(\frac{\lambda_{i} }{ \lambda_{j} }
 +
 \lambda_{i}\lambda_{j}
 \gamma_{n}G_{ a _{j}}^{\frac{2}{2-n}}( a _{i})
 )^{\frac{n}{2}}
 }
\int_{\mathcal{A}}
 \frac{O(\frac{ \lambda_{j} }{ \lambda_{i} }\vert x \vert^{2})}{(1+r^{2})^{\frac{n+2}{2}}}
 \end{split}\end{equation}
 and
 \begin{equation}\begin{split}
 I_{4}
 = &
\underset{\mathcal{A}^{c}}{\int} \frac{ u_{ a _{j}}( a _{i})}{(1+r^{2})^{\frac{n+2}{2}}}
 (\frac{1}{\frac{\lambda_{i} }{ \lambda_{j} }+\lambda_{i}\lambda_{j}\gamma_{n}G_{ a _{j}}^{\frac{2}{2-n}}(\exp_{g_{ a _{i}}} \frac{x}{\lambda_{i}})})^{\frac{n-2}{2}}.
 \end{split}\end{equation}
Note, that since $\lambda_{i}\geq \lambda_{j}$, $\mathcal{A}$ tends to cover $\R^{n}$ as $\eps_{i,j}\-0$. 
Thus 
\begin{equation}\begin{split}
I_{1}
= &
b_{1}\frac
{u_{ a _{j}}( a _{i})}
{(\frac{\lambda_{i} }{ \lambda_{j} }
+
\lambda_{i}\lambda_{j}
\gamma_{n}G_{ a _{j}}^{\frac{2}{2-n}}( a _{i})
)^{\frac{n-2}{2}}
}
+
o(\eps_{i,j}),
\end{split}\end{equation}
whereas 
$I_{2}=0$ by radial symmetry and 
$
I_{3}
= 
o(\eps_{i,j}).
$
Moreover  
\begin{equation}
\begin{split}
I_{4}=o(\eps_{i,j})
\end{split}
\end{equation} 
in case $\eps_{i,j}^{\frac{2}{2-n}}\sim \frac{\lambda_{i}}{\lambda_{j}}$. 
Otherwise we decompose 
\begin{equation}
\begin{split}
\mathcal{A}^{c}
\subseteq &
\mathcal{B}_{1}\cup \mathcal{B}_{2},
\end{split}
\end{equation}
where for a sufficiently large constant $E>0$
\begin{equation}
\begin{split}
\mathcal{B}_{1}
= &
[\eps \sqrt{\gamma_{n}G^{\frac{2}{2-n}}_{a_{j}}(a_{i})}
\leq 
\vert \frac{x}{\lambda_{i}}\vert 
\leq 
E\sqrt{\gamma_{n}G^{\frac{2}{2-n}}_{a_{j}}(a_{i})}] \\
\end{split}
\end{equation} 
and
\begin{equation}
\begin{split}
\mathcal{B}_{2}
= &
[
E\sqrt{\gamma_{n}G^{\frac{2}{2-n}}_{a_{j}}(a_{i})}
\leq 
\vert \frac{x}{\lambda_{i}}\vert 
\leq c
].
\end{split}
\end{equation} 
We then may estimate
\begin{equation}
\begin{split}
I_{4}^{1}
= &
\underset{\mathcal{B}_{1}}{\int} \frac{ u_{ a _{j}}( a _{i})}{(1+r^{2})^{\frac{n+2}{2}}}
 (\frac{1}{\frac{\lambda_{i} }{ \lambda_{j} }+\lambda_{i}\lambda_{j}\gamma_{n}G_{ a _{j}}^{\frac{2}{2-n}}(\exp_{g_{ a _{i}}} \frac{x}{\lambda_{i}})})^{\frac{n-2}{2}} \\
\leq &
\frac{C(\frac{\lambda_{j}}{\lambda_{i}})^{\frac{n+2}{2}}}{(1+\lambda_{i}^{2}\gamma_{n}G^{\frac{2}{2-n}}_{a_{j}}(a_{i}))^{\frac{n+2}{2}}} \\
& \quad
\underset{[\vert \frac{x}{\lambda_{j}}\vert \leq E\sqrt{\gamma_{n}G^{\frac{2}{2-n}}_{a_{j}}(a_{i})}]}{\int} 
 (\frac{1}{1+\lambda_{j}^{2}\gamma_{n}G_{ a _{j}}^{\frac{2}{2-n}}(\exp_{g_{ a _{i}}} \frac{x}{\lambda_{j}})})^{\frac{n-2}{2}}.
\end{split}
\end{equation} 
Changing coordinates via $d_{i,j}=\exp_{g_{a_{i}}}^{-1}\exp_{g_{a_{j}}}$ we get
\begin{equation}
\begin{split}
I_{4}^{1}
\leq &
\frac{C}{(\frac{\lambda_{i}}{\lambda_{j}}+\lambda_{i}\lambda_{j}G^{\frac{2}{2-n}}_{a_{j}}(a_{i}))^{\frac{n+2}{2}}}
\underset{[\vert \frac{x}{\lambda_{j}}\vert \leq \tilde  E d(a_{i},a_{j})]}{\int} 
 (\frac{1}{1+r^{2}})^{\frac{n-2}{2}}
\end{split}
\end{equation} 
and thus $I_{4}^{1}=o(\eps_{i,j})$, since we may assume $\lambda_{j}\ll \lambda_{i}$. Moreover
\begin{equation}
\begin{split}
I_{4,2}
= &
\underset{\mathcal{B}_{2}}{\int} \frac{ u_{ a _{j}}( a _{i})}{(1+r^{2})^{\frac{n+2}{2}}}
 (\frac{1}{\frac{\lambda_{i} }{ \lambda_{j} }+\lambda_{i}\lambda_{j}\gamma_{n}G_{ a _{j}}^{\frac{2}{2-n}}(\exp_{g_{ a _{i}}} \frac{x}{\lambda_{i}})})^{\frac{n-2}{2}} \\
 \leq &
\frac{C}{(\frac{\lambda_{i} }{ \lambda_{j} }+\lambda_{i}\lambda_{j}\gamma_{n}G_{ a _{j}}^{\frac{2}{2-n}}(a_{i}))^{\frac{n-2}{2}}} \\
& \quad\quad\quad\quad 
\underset{[\vert x \vert \geq \sqrt{\lambda_{i}^{2}\gamma_{n}G^{\frac{2}{2-n}}_{a_{j}}(a_{i})}]}{\int}
\frac{1}{(1+r^{2})^{\frac{n+2}{2}}} \\
= &
o(\eps_{i,j}),
\end{split}
\end{equation} 
since $\lambda_{i}^{2}\gamma_{n}G^{\frac{2}{2-n}}_{a_{j}}(a_{i})\gg 1$ in this case. Therefore
\begin{equation}
\begin{split}
I_{4}\leq I_{4}^{1}+I_{4}^{2}=o(\eps_{i,j}).
\end{split}
\end{equation} 
Collecting terms we get 
\begin{equation}
\begin{split}
\int \varphi_{i}^{\frac{n+2}{n-2}}\varphi_{j}
= &
I_{1}+o(\eps_{i,j}).               
\end{split}
\end{equation}
Due to conformal invariance there holds
\begin{equation}\begin{split}
G_{a_{j}}(a_{j},a_{i})=u_{a_{j}}^{-1}(a_{i})u_{a_{j}}^{-1}(a_{j})G_{g_{0}}(a_{i},a_{j})
\end{split}\end{equation}
and we conclude
\begin{equation}\begin{split}
\int \varphi_{i}^{\frac{n+2}{n-2}}\varphi_{j}
= &
\frac
{b_{1}}
{(\frac{\lambda_{i} }{ \lambda_{j} }
+
\lambda_{i}\lambda_{j}
\gamma_{n}G_{g_{0}}^{\frac{2}{2-n}}( a _{i}, a _{j})
)^{\frac{n-2}{2}}
}
+
o(\eps_{i,j}).
\end{split}\end{equation}
The claim follows.
 \item[$(\beta)$] \quad\quad  Case $k=2$ \\$_{}$\\
First we deal with the case $\frac{1}{\lambda_{i}}\leq \frac{1}{\lambda_{j}}$. For $c>0$ small we get
\begin{equation}\begin{split}
-\lambda_{j}\int & \varphi_{i}^{\frac{n+2}{n-2}}\partial_{\lambda_{j}}\varphi_{j}\\
= &
\frac{n-2}{2}
\underset{B_{c}( a _{i})}{\int}
(\frac{ \lambda_{i} }{ 1+\lambda_{i}^{2}\gamma_{n}G_{ a _{i}}^{\frac{2}{2-n}}})^{\frac{n+2}{2}}\frac{u_{ a _{j}}}{u_{ a _{i}}} \\
& \quad\quad\quad\quad\;\;
(\frac{ \lambda_{j} }{ 1+\lambda_{j}^{2}\gamma_{n}G_{ a _{j}}^{\frac{2}{2-n}}})^{\frac{n-2}{2}}
\frac{\lambda_{j}^{2}\gamma_{n}G_{a_{j}}^{\frac{2}{2-n}}-1}{\lambda_{j}^{2}\gamma_{n}G_{a_{j}}^{\frac{2}{2-n}}+1}
d\mu_{g_{ a _{i}}}
\\ & +
O(  \frac{1}{\lambda_{i}^{\frac{n+2}{2}}}  \frac{1}{\lambda_{j}^{\frac{n-2}{2}}}).
\end{split}\end{equation}
Clearly $ \lambda_{i}^{-\frac{n+2}{2}}\lambda_{j}^{-\frac{n-2}{2}}=o(\eps_{i,j})$,
whence as before
\begin{equation}\begin{split}\label{expanding_lambdaipartiallambdaiepsij}
-\lambda_{j}\int & \varphi_{i}^{\frac{n+2}{n-2}}\partial_{\lambda_{j}}\varphi_{j} \\
= &
\frac{n-2}{2}
\underset{B_{ c\lambda_{i} }(0)}{\int}
\frac{u_{ a _{j}}( a _{i})}{(1+r^{2})^{\frac{n+2}{2}}} 
\frac{\lambda_{j}^{2}\gamma_{n}G_{a_{j}}^{\frac{2}{2-n}}(\exp_{g_{ a _{i}}}\frac{x}{\lambda_{i}})-1}
{\lambda_{j}^{2}\gamma_{n}G_{a_{j}}^{\frac{2}{2-n}}(\exp_{g_{ a _{i}}} \frac{x}{\lambda_{i}})+1}
\\
& \quad\quad\quad\quad\quad
(\frac{1}{\frac{\lambda_{i} }{ \lambda_{j} }+\lambda_{i}\lambda_{j}\gamma_{n}G_{ a _{j}}^{\frac{2}{2-n}}
(\exp_{g_{ a _{i}}} \frac{x}{\lambda_{i}})})^{\frac{n-2}{2}} 
\\
& +
o(\eps_{i,j}).
\end{split}\end{equation}
Due to $\frac{1}{\lambda_{i}}\leq \frac{1}{\lambda_{j}}$ we have
\begin{equation}\begin{split}
\eps_{i,j}^{\frac{2}{2-n}}
\sim 
\lambda_{i}\lambda_{j}\gamma_{n}G^{\frac{2}{2-n}}( a _{i}, a _{j})
\; \text{ or }\;
\eps_{i,j}^{\frac{2}{2-n}}
\sim
\frac{\lambda_{i}}{\lambda_{j}}
\end{split}\end{equation}
and may expand on
\begin{equation}\begin{split}
\mathcal{A}
= &
[\vert  \frac{x}{\lambda_{i}}\vert \leq \epsilon\sqrt{\gamma_{n}G^{\frac{2}{2-n}}_{ a _{j}}( a _{i})}]
\cup
[\vert  \frac{x}{\lambda_{i}}\vert \leq \epsilon  \frac{1}{\lambda_{j}} ]
\end{split}\end{equation}
for $\epsilon>0$ sufficiently small

\begin{equation}\begin{split}
& 
\frac{1}
{
(\frac{\lambda_{i} }{ \lambda_{j} }
+
\lambda_{i}\lambda_{j}\gamma_{n}G_{ a _{j}}^{\frac{2}{2-n}}(\exp_{g_{ a _{i}}} \frac{x}{\lambda_{i}}))^{\frac{n-2}{2}} }\\
& \quad\quad\quad\quad\quad\quad\quad\quad\quad\quad\quad\quad 
\frac{\lambda_{j}^{2}\gamma_{n}G_{a_{j}}^{\frac{2}{2-n}}(\exp_{g_{ a _{i}}} \frac{x}{\lambda_{i}})-1}
{\lambda_{j}^{2}\gamma_{n}G_{a_{j}}^{\frac{2}{2-n}}(\exp_{g_{ a _{i}}} \frac{x}{\lambda_{i}})+1}
\\
& \quad = 
(
\frac{\lambda_{i} }{ \lambda_{j} }
+
\lambda_{i}\lambda_{j}
\gamma_{n}G_{ a _{j}}^{\frac{2}{2-n}}( a _{i})
)^{\frac{2-n}{2}} 
\frac
{\lambda_{j}^{2}\gamma_{n}G_{a_{j}}^{\frac{2}{2-n}}(a_{i})-1}
{\lambda_{j}^{2}\gamma_{n}G_{a_{j}}^{\frac{2}{2-n}}(a_{i})+1}\\
& \quad \;\;\;\; +
\frac{2-n}{2}
\frac
{
\gamma_{n}\nabla G^{\frac{2}{2-n}}_{ a _{j}}( a _{i})\lambda_{j}x
}
{
(
\frac{\lambda_{i} }{ \lambda_{j} }
+
\lambda_{i}\lambda_{j}\gamma_{n}G_{ a _{j}}^{\frac{2}{2-n}}( a _{i})
)^{\frac{n}{2}}
}
\frac
{\lambda_{j}^{2}\gamma_{n}G_{a_{j}}^{\frac{2}{2-n}}(a_{i})-1}
{\lambda_{j}^{2}\gamma_{n}G_{a_{j}}^{\frac{2}{2-n}}(a_{i})+1}
\\
& \quad \;\;\;\;+
\frac
{
2
}
{
(
\frac{\lambda_{i} }{ \lambda_{j} }
+
\lambda_{i}\lambda_{j}\gamma_{n}G_{ a _{j}}^{\frac{2}{2-n}}( a _{i})
)^{\frac{n}{2}}
}
\frac
{\gamma_{n}\nabla G_{a_{j}}^{\frac{2}{2-n}}(a_{i})\lambda_{j}x}
{1+\lambda_{j}^{2}\gamma_{n}G_{a_{j}}^{\frac{2}{2-n}}(a_{i})}
\\
& \quad \;\;\;\;+
\frac
{
O(\frac{ \lambda_{j} }{ \lambda_{i} }\vert x\vert^{2})
}
{
(
\frac{\lambda_{i} }{ \lambda_{j} }
+
\lambda_{i}\lambda_{j}\gamma_{n}G_{ a _{j}}^{\frac{2}{2-n}}( a _{i})
)^{\frac{n}{2}}
}.
\end{split}\end{equation}
By radial symmetry we then get with $b_{2}=\frac{n-2}{2}\int \frac{1}{(1+r^{2})^{\frac{n+2}{2}}}$
\begin{equation}\begin{split}
-&\lambda_{j}\int  \varphi_{i}^{\frac{n+2}{n-2}}\varphi_{j} \\
& \hspace{-8pt}
=
\frac
{b_{2}u_{ a _{j}}( a _{i})}
{(\frac{\lambda_{i} }{ \lambda_{j} }
+
\lambda_{i}\lambda_{j}
\gamma_{n}G_{ a _{j}}^{\frac{2}{2-n}}( a _{i})
)^{\frac{n-2}{2}}
}
\frac
{\lambda_{j}^{2}\gamma_{n}G_{a_{j}}^{\frac{2}{2-n}}(a_{i})-1}
{\lambda_{j}^{2}\gamma_{n}G_{a_{j}}^{\frac{2}{2-n}}(a_{i})+1} 
+
o(\eps_{i,j})
\end{split}\end{equation}
and thus by conformal invariance
\begin{equation}\begin{split}\label{didj_first_two_cases}
-\lambda_{j}\int \varphi_{i}^{\frac{n+2}{n-2}}\partial_{\lambda_{j}}\varphi_{j}
= &
\frac
{
b_{2}(\lambda_{i}\lambda_{j}\gamma_{n}G_{g_{0}}^{\frac{2}{2-n}}(a_{i},a_{j})-\frac{\lambda_{i}}{\lambda_{j}})
}
{(\frac{\lambda_{i} }{ \lambda_{j} }
+
\lambda_{i}\lambda_{j}
\gamma_{n}G^{\frac{2}{2-n}}_{g_{0}}( a _{i}, a _{j})
)^{\frac{n}{2}}
} \\
& +
o(\eps_{i,j}).
\end{split}\end{equation}
We turn to the case
$
\frac{1}{\lambda_{i}}\geq \frac{1}{\lambda_{j}}.
$
By the same reasoning as for \eqref{di^...dj=didj^...+o(eij)} 
\begin{equation}\begin{split}
-\lambda_{j}\int \varphi_{i}^{\frac{n+2}{n-2}}\partial_{ \lambda_{j}}\varphi_{j}
= &
-\lambda_{j}\int \varphi_{i}\partial_{\lambda_{j}} \varphi_{j}^{\frac{n+2}{n-2}}
+
o(\eps_{i,j}).
\end{split}\end{equation}
For $c>0$ small we get
\begin{equation}\begin{split}
-\lambda_{j}\int &\varphi_{i} \partial_{\lambda_{j}}\varphi_{j}^{\frac{n+2}{n-2}}\\
= &
\frac{n+2}{2}\underset{B_{c}( a _{i})}{\int}
(\frac{ \lambda_{i} }{ 1+\lambda_{i}^{2}\gamma_{n}G_{ a _{i}}^{\frac{2}{2-n}}})^{\frac{n-2}{2}} \\
& \quad\quad\quad
\frac{u_{ a _{i}}}{u_{ a _{j}}}(\frac{ \lambda_{j} }{ 1+\lambda_{j}^{2}\gamma_{n}G_{ a _{j}}^{\frac{2}{2-n}}})^{\frac{n+2}{2}}
\frac{\lambda_{j}^{2}\gamma_{n}G_{a_{j}}^{\frac{2}{2-n}}-1}{\lambda_{j}^{2}\gamma_{n}G_{a_{j}}^{\frac{2}{2-n}}+1}
d\mu_{g_{ a _{j}}}
\\ & +
O(  \frac{1}{\lambda_{j}^{\frac{n+2}{2}}}  \frac{1}{\lambda_{i}^{\frac{n-2}{2}}}),
\end{split}\end{equation}
whence
\begin{equation}\begin{split}
-\lambda_{j}&\int  \varphi_{i}\partial_{\lambda_{j}}\varphi_{j}^{\frac{n+2}{n-2}} \\
= &
\frac{n+2}{2}
\underset{B_{ c\lambda_{i} }(0)}{\int}
\frac
{r^{2}-1}
{r^{2}+1}
(\frac{1}{1+r^{2}})^{\frac{n+2}{2}}
\\
& \quad\quad\quad\quad
\frac
{
u_{ a _{i}}( a _{j})
}
{
(\frac{\lambda_{j}}{\lambda_{i}}+\lambda_{i}\lambda_{j}\gamma_{n}G^{\frac{2}{2-n}}_{a_{i}}(\exp_{g_{a_{j}}}\frac{x}{\lambda_{j}}))^{\frac{n+2}{2}}
} 
+
o(\eps_{i,j}).
\end{split}\end{equation}
We may expand on
\begin{equation*}\begin{split}
\mathcal{A}
= &
[
\vert  \frac{x}{\lambda_{j}}\leq \varepsilon \sqrt{\gamma_{n}G^{\frac{2}{2-n}}_{a_{i}}(a_{j})}
]
\cup
[\vert  \frac{x}{\lambda_{j}}\vert \leq \epsilon  \frac{1}{\lambda_{i}} ]
\end{split}\end{equation*}
for $\epsilon>0$ sufficiently small
\begin{equation}\begin{split}
(
\frac{\lambda_{j} }{ \lambda_{i} }
& +
\lambda_{i}\lambda_{j}\gamma_{n}G_{ a_{i}}^{\frac{2}{2-n}}(\exp_{g_{ a_{j}}} \frac{x}{\lambda_{i}}))^{\frac{2-n}{2}} 
\\
= & 
(
\frac{\lambda_{j} }{ \lambda_{i} }
+
\lambda_{i}\lambda_{j}
\gamma_{n}G_{ a_{i}}^{\frac{2}{2-n}}( a_{j})
)^{\frac{2-n}{2}} 
\\ & +
\frac{2-n}{2}
\frac
{
\gamma_{n}\nabla G^{\frac{2}{2-n}}_{ a_{i}}( a_{j})\lambda_{i}x
+
O(\frac{ \lambda_{i} }{ \lambda_{j} }\vert x\vert^{2})
}
{
(
\frac{\lambda_{j} }{ \lambda_{i} }
+
\lambda_{i}\lambda_{j}\gamma_{n}G_{ a_{i}}^{\frac{2}{2-n}}( a_{j})
)^{\frac{n}{2}}
}.
\end{split}\end{equation}
This gives with indeed
$
b_{2}=\frac{n+2}{2}\int \frac{r^{2}-1}{r^{2}+1}(\frac{1}{1+r^{2}})^{\frac{n+2}{2}}
=
\frac{n-2}{2}\int (\frac{1}{1+r^{2}})^{\frac{n+2}{2}}
$
\begin{equation}\begin{split}
-\lambda_{j}\int \varphi_{i}^{\frac{n+2}{n-2}}\partial_{\lambda_{j}}\varphi_{j}
= &
b_{2}\frac
{u_{ a_{i}}( a_{j})}
{(\frac{\lambda_{i} }{ \lambda_{j} }
+
\lambda_{i}\lambda_{j}
\gamma_{n}G_{ a_{i}}^{\frac{2}{2-n}}( a_{j})
)^{\frac{n-2}{2}}
} \\
& +
o(\eps_{i,j})
\end{split}\end{equation}
and we conclude by conformal invariance
\begin{equation}\begin{split}\label{didj_third_case}
-\lambda_{j}\int \varphi_{i}^{\frac{n+2}{n-2}}\partial_{\lambda_{j}}\varphi_{j}
= &
\frac
{
b_{2}
}
{(\frac{\lambda_{i} }{ \lambda_{j} }
+
\lambda_{i}\lambda_{j}
\gamma_{n}G_{g_{0}}^{\frac{2}{2-n}}( a _{i}, a _{j})
)^{\frac{n-2}{2}}
} \\
& +
o(\eps_{i,j}).
\end{split}\end{equation}
\eqref{didj_first_two_cases} and \eqref{didj_third_case} then prove the claim.
 \item[$(\gamma)$] \quad\quad  Case $k=3$ \\$_{}$\\
First we consider the case $\frac{1}{\lambda_{i}}\leq \frac{1}{\lambda_{j}}$. For $c>0$ small we get
\begin{equation}\begin{split}
\frac{1}{\lambda_{j}}\int & \varphi_{i}^{\frac{n+2}{n-2}}\nabla_{a_{j}}\varphi_{j}\\
= &
\frac{2-n}{2}
\underset{B_{c}( a _{i})}{\int}
(\frac{ \lambda_{i} }{ 1+\lambda_{i}^{2}\gamma_{n}G_{ a _{i}}^{\frac{2}{2-n}}})^{\frac{n+2}{2}}\frac{u_{ a _{j}}}{u_{ a _{i}}} \\
& \quad\quad\quad\quad\;\;
(\frac{ \lambda_{j} }{ 1+\lambda_{j}^{2}\gamma_{n}G_{ a _{j}}^{\frac{2}{2-n}}})^{\frac{n-2}{2}}
\frac
{\lambda_{j}\gamma_{n}\nabla_{a_{j}}G_{a_{j}}^{\frac{2}{2-n}}}
{1+\lambda_{j}^{2}\gamma_{n}G^{\frac{2}{2-n}}_{a_{j}}}
d\mu_{g_{ a _{i}}}
\\ & +
O
(
\underset{B_{c}( a _{i})}{\int}
\varphi_{i}^{\frac{n+2}{n-2}}
\vert\frac{\nabla_{a_{j}}u_{ a _{j}}}{u_{ a _{i}}\lambda_{j}} \vert \varphi_{j}
+
\frac{1}{\lambda_{i}^{\frac{n+2}{2}}}  \frac{1}{\lambda_{j}^{\frac{n-2}{2}}}
).
\end{split}\end{equation}
Due to case  $(v)$ we obtain passing to 
$
x\simeq \exp_{g_{ a _{i}}}(x)
$
coordinates
\begin{equation}\begin{split}\label{expanding_1/lambdainablaaiepsij}
\frac{1}{\lambda_{j}}& \int \varphi_{i}^{\frac{n+2}{n-2}}\nabla_{a_{j}}\varphi_{j} \\
= &
\frac{2-n}{2}
\underset{B_{ c\lambda_{i} }(0)}{\int}
\frac{u_{ a _{j}}( a _{i})}{(1+r^{2})^{\frac{n+2}{2}}} 
\frac
{\lambda_{j}\gamma_{n}\nabla_{a_{j}}G_{a_{j}}^{\frac{2}{2-n}}(\exp_{g_{ a _{i}}}\frac{x}{\lambda_{i}})}
{1+\lambda_{j}^{2}\gamma_{n}G_{a_{j}}^{\frac{2}{2-n}}(\exp_{g_{ a _{i}}} \frac{x}{\lambda_{i}})}
\\
& \quad\quad\quad\quad\quad
(\frac{1}{\frac{\lambda_{i} }{ \lambda_{j} }+\lambda_{i}\lambda_{j}\gamma_{n}G_{ a _{j}}^{\frac{2}{2-n}}
(\exp_{g_{ a _{i}}} \frac{x}{\lambda_{i}})})^{\frac{n-2}{2}} 
+
o(\eps_{i,j}).
\end{split}\end{equation}
Since
\begin{equation}\begin{split}
\eps_{i,j}^{\frac{2}{2-n}}
\sim 
\lambda_{i}\lambda_{j}\gamma_{n}G^{\frac{2}{2-n}}( a _{i}, a _{j})
\; \text{ or }\;
\eps_{i,j}^{\frac{2}{2-n}}
\sim
\frac{\lambda_{i}}{\lambda_{j}}
\end{split}\end{equation}
we may expand on 
\begin{equation}\begin{split}
\mathcal{A}
= &
[\vert  \frac{x}{\lambda_{i}}\vert \leq \epsilon\sqrt{\gamma_{n}G^{\frac{2}{2-n}}_{ a _{j}}( a _{i})}]
\cup
[\vert  \frac{x}{\lambda_{i}}\vert \leq \epsilon  \frac{1}{\lambda_{j}} ]
\end{split}\end{equation}
for $\epsilon>0$ sufficiently small as before to obtain
\begin{equation}\begin{split}
\frac{1}{\lambda_{j}}\int & \varphi_{i}^{\frac{n+2}{n-2}}\nabla_{a_{j}}\varphi_{j}\\
= &
-b_{3}
\frac
{u_{a_{j}}(a_{i})}
{(\frac{\lambda_{i}}{\lambda_{j}}+\lambda_{i}\lambda_{j}\gamma_{n}G^{\frac{2}{2-n}}_{a_{j}}(a_{i}))^{\frac{n-2}{2}}}
\frac
{\lambda_{j}\gamma_{n}\nabla_{a_{j}}G^{\frac{2}{2-n}}_{a_{j}}(a_{i})}
{1+\lambda_{j}^{2}\gamma_{n}G^{\frac{2}{2-n}}_{a_{j}}(a_{i})}
\\ & +
o(\eps_{i,j})
\end{split}\end{equation}
with $b_{3}=\frac{n-2}{2}\int \frac{1}{(1+r^{2})^{\frac{n+2}{2}}}$.
This gives
\begin{equation}\begin{split}
\frac{1}{\lambda_{j}}\int & \varphi_{i}^{\frac{n+2}{n-2}}\nabla_{a_{j}}\varphi_{j}
= 
-b_{3}
\frac
{u_{a_{j}}(a_{i})\lambda_{i}\gamma_{n}\nabla_{a_{j}}G^{\frac{2}{2-n}}_{a_{j}}(a_{i})}
{(\frac{\lambda_{i}}{\lambda_{j}}+\lambda_{i}\lambda_{j}\gamma_{n}G^{\frac{2}{2-n}}_{a_{j}}(a_{i}))^{\frac{n}{2}}} \\
& +
o(\eps_{i,j}),
\end{split}\end{equation}
whence by conformal invariance 
\begin{equation}\begin{split}\label{deltai^...*1/lambdaj*nablaajdeltaj}
\frac{1}{\lambda_{j}}\int & \varphi_{i}^{\frac{n+2}{n-2}}\nabla_{a_{j}}\varphi_{j}
= 
\frac
{-b_{3}\lambda_{i}\gamma_{n}\nabla_{a_{j}}G_{g_{0}}^{\frac{2}{2-n}}(a_{i},a_{j})}
{
(
\frac{\lambda_{i}}{\lambda_{j}}
+
\lambda_{i}\lambda_{j}\gamma_{n}G_{g_{0}}^{\frac{2}{2-n}}(a_{i},a_{j})
)
^{\frac{n}{2}}
}
+
o(\eps_{i,j}).
\end{split}\end{equation}
We turn to the case $\frac{1}{\lambda_{i}}\geq \frac{1}{\lambda_{j}}$.
As before
\begin{equation}\begin{split}
\frac{1}{\lambda_{j}}\int  \varphi_{i}^{\frac{n+2}{n-2}}\nabla_{a_{j}}\varphi_{j}
= &
\frac{1}{\lambda_{j}}\int \varphi_{i}\nabla_{a_{j}}\varphi_{j}^{\frac{n+2}{n-2}}
+
o(\eps_{i,j})
\end{split}\end{equation}
and for $c>0$ small we obtain by arguments familiar by now
\begin{equation}\begin{split}
\frac{1}{\lambda_{j}}\int & \varphi_{i}\nabla_{a_{j}}\varphi_{j}^{\frac{n+2}{n-2}}\\
= &
-\frac{n+2}{2}
\underset{B_{c\lambda_{j}}( a _{j})}{\int}
(
\frac
{ 1 }
{ 
\frac{\lambda_{j}}{\lambda_{i}}
+
\lambda_{i}\lambda_{j}\gamma_{n}G_{ a _{i}}^{\frac{2}{2-n}}(\exp_{g_{a_{j}}}\frac{x}{\lambda_{j}})})^{\frac{n-2}{2}} \\
& \quad\quad\quad\quad\quad\quad\;\;\;
\frac
{u_{ a _{i}}(a_{j})\lambda_{j}\gamma_{n}\nabla_{a_{j}}G_{a_{j}}^{\frac{2}{2-n}}(\exp_{g_{a_{j}}}\frac{x}{\lambda_{j}})}
{( 1+r^{2})^{\frac{n+4}{2}}}
\\ & +
o(\eps_{i,j}),
\end{split}\end{equation}
whence
\begin{equation}\begin{split}
\frac{1}{\lambda_{j}}\int & \varphi_{i}\nabla_{a_{j}}\varphi_{j}^{\frac{n+2}{n-2}}\\
= &
(n+2)
\underset{B_{c\lambda_{j}}( a _{j})}{\int}
\frac
{u_{ a _{i}}(a_{j})x}
{( 1+r^{2})^{\frac{n+4}{2}}}
\\
& \quad\quad\quad\quad\quad\quad
(
\frac
{ 1 }
{ 
\frac{\lambda_{j}}{\lambda_{i}}
+
\lambda_{i}\lambda_{j}\gamma_{n}G_{ a _{i}}^{\frac{2}{2-n}}(\exp_{g_{a_{j}}}\frac{x}{\lambda_{j}})})^{\frac{n-2}{2}} 
\\ & +
o(\eps_{i,j}).
\end{split}\end{equation}
Expanding on
\begin{equation}
\begin{split}
\mathcal{A}
= &
[
\vert \frac{x}{\lambda_{j}}\vert
\leq \epsilon
\sqrt{\gamma_{n}G^{\frac{2}{2-n}}_{a_{i}}(a_{j})}
]
\cup
[\vert  \frac{x}{\lambda_{j}}\vert \leq \epsilon  \frac{1}{\lambda_{i}} ]
\end{split}
\end{equation} for $\epsilon>0$ sufficiently small we  derive
\begin{equation}\begin{split}
\frac{1}{\lambda_{j}}\int \varphi_{i}\nabla_{a_{j}}\varphi_{j}^{\frac{n+2}{n-2}}
= &
b_{3}
\frac
{
u_{ a _{i}}( a _{j})
\lambda_{i}\gamma_{n}\nabla_{a_{j}}G^{\frac{2}{2-n}}_{a_{i}}(a_{j})
}
{(\frac{\lambda_{j} }{ \lambda_{i} }
+
\lambda_{i}\lambda_{j}
\gamma_{n}G_{ a _{j}}^{\frac{2}{2-n}}( a _{i})
)^{\frac{n}{2}}
}
 +
o(\eps_{i,j})
\end{split}\end{equation}
with indeed
$
b_{3}=\frac{(n+2)(n-2)}{2n}\int \frac{r^{2}}{(1+r^{2})^{\frac{n+4}{2}}}
=
\frac{n-2}{2}\int (\frac{1}{1+r^{2}})^{\frac{n+2}{2}}.
$
Thus
\begin{equation}\begin{split}\label{deltai*1/lambdaj*nablaajdeltaj^...}
\frac{1}{\lambda_{j}}\int  \varphi_{i}\nabla_{a_{j}}\varphi_{j}^{\frac{n+2}{n-2}}
= &
b_{3}
\frac
{
\lambda_{i}\gamma_{n}\nabla_{a_{j}}G_{g_{0}}^{\frac{2}{2-n}}(a_{i},a_{j})
}
{(\frac{\lambda_{j} }{ \lambda_{i} }
+
\lambda_{i}\lambda_{j}
\gamma_{n}G_{g_{0}}^{\frac{2}{2-n}}( a _{i},a_{j})
)^{\frac{n}{2}}
} \\
& +
o(\eps_{i,j})
\end{split}\end{equation}
by conformal invariance. 
From
\eqref{deltai^...*1/lambdaj*nablaajdeltaj}, \eqref{deltai*1/lambdaj*nablaajdeltaj^...}
the claim follows.
\end{enumerate}
\item
Due to \eqref{Phi1i}, \eqref{Phi2i}, \eqref{Phi3i} and
\begin{equation}\begin{split}
\gamma_{n}G_{a_{i}}^{\frac{2}{2-n}}
= 
r^{2}+O(r^{n}),
\;
\gamma_{n}\nabla_{a_{i}}G^{\frac{2}{2-n}}_{a_{i}}
=
-2x+O(r^{n-1}),
\end{split}\end{equation}
cf. definition \ref{def_bubbles} we have on $B_{c}(a_{i})$ for $c>0$ small
\begin{enumerate}
\item [$(\alpha)$]
 \begin{equation}\begin{split}
\phi_{1,i}
= &
u_{a_{i}}(\frac{\lambda_{i}}{1+\lambda_{i}^{2}r^{2}})^{\frac{n-2}{2}}
+
O(r^{n-2}(\frac{\lambda_{i}}{1+\lambda_{i}^{2}r^{2}})^{\frac{n-2}{2}})
\end{split}\end{equation}
\item [$(\beta)$]
 \begin{equation}\begin{split}
\phi_{2,i}
= &
\frac{n-2}{2}u_{a_{i}}\frac{\lambda_{i}^{2}r^{2}-1}{\lambda_{i}^{2}r^{2}+1}
(\frac{\lambda_{i}}{1+\lambda_{i}^{2}r^{2}})^{\frac{n-2}{2}} \\
& +
O(r^{n-2}(\frac{\lambda_{i}}{1+\lambda_{i}^{2}r^{2}})^{\frac{n-2}{2}})
\end{split}\end{equation}
\item [$(\gamma)$]
\begin{equation}\begin{split}
\phi_{3,i}
= &
-\frac{n-2}{2}
u_{a_{i}}
\frac
{\lambda_{i}x}
{1+\lambda_{i}^{2}r^{2}} 
(\frac{\lambda_{i}}{1+\lambda_{i}^{2}r^{2}})^{\frac{n-2}{2}} \\
& +
O(r^{n-2}(\frac{\lambda_{i}}{1-\lambda_{i}^{2}r^{2}})^{\frac{n-2}{2}})
+
O(\frac{r}{\lambda_{i}}(\frac{\lambda_{i}}{1+\lambda_{i}^{2}r^{2}})^{\frac{n-2}{2}})
\end{split}\end{equation}  
\item [$(\varphi)$]
\begin{equation}\begin{split}
\varphi_{i}^{\frac{4}{n-2}}
= &
u^{\frac{4}{n-2}}_{a_{i}}(\frac{\lambda_{i}}{1+\lambda_{i}^{2}r^{2}})^{2}
+
O(r^{n-2}(\frac{\lambda_{i}}{1+\lambda_{i}^{2}r^{2}})^{2}).
\end{split}\end{equation}
\end{enumerate}
Consequently
\begin{equation}\begin{split}
\int \phi_{1,i}\varphi_{i}^{\frac{4}{n-2}}\phi_{2,i}
= &
\int_{B_{c}(a_{i})} \phi_{1,i}\varphi_{i}^{\frac{4}{n-2}}\phi_{2,i}
+
O(\frac{1}{\lambda_{i}^{n}})\\
= &
\int_{B_{\frac{c}{\lambda_{i}}}(0)} \frac{r^{2}-1}{r^{2}+1}(\frac{1}{1+r^{2}})^{n}
+
O(\frac{1}{\lambda_{i}^{n-2}})
\\
= &
\int \frac{r^{2}-1}{r^{2}+1}(\frac{1}{1+r^{2}})^{n}
+
O(\frac{1}{\lambda_{i}^{n-2}})=O(\frac{1}{\lambda_{i}^{n-2}}),
\end{split}\end{equation}
since the integral above vanishes. Alike using radial symmetry
\begin{equation}\begin{split}
\int \phi_{1,i}\varphi_{i}^{\frac{4}{n-2}}\phi_{3,i},
\int \phi_{2,i}\varphi_{i}^{\frac{4}{n-2}}\phi_{3,i}
= &
O(\frac{1}{\lambda_{i}^{n-2}}+\frac{1}{\lambda_{i}^{2}}).
\end{split}\end{equation}
Moreover we have readily have
\begin{equation}
\int \var_{i}^{\frac{n+2}{n-2}}\phi_{k,i} 
=
(1,-\lambda_{i} \partial_{i}, \frac{1}{\lambda_{i}}\nabla_{a_{i}})\int \varphi_{i}^{\frac{2n}{n-2}}
=
\delta_{1k}+O(\frac{1}{\lambda_{i}}^{n-2})
.\end{equation} 
 \item 
Let $\alpha'=\frac{n-2}{2}\alpha, \, \beta'=\frac{n-2}{2}\beta$, so $\alpha'+\beta'=n$.
We distinguish 
\begin{enumerate}
 \item [($\alpha$)] \quad 
$\eps_{i,j}^{\frac{2}{2-n}}\sim \frac{\lambda_{i} }{ \lambda_{j} }
\; \vee \;
\eps_{i,j}^{\frac{2}{2-n}} \sim \lambda_{i}\lambda_{j}\gamma_{n}G^{\frac{2}{2-n}}( a _{i}, a _{j})$
\smallskip \\
We estimate for $c>0$ small
\begin{equation}\begin{split}
\int &\varphi_{i}^{\alpha} \varphi_{j}^{\beta}\\
\leq &
C \underset{B_{c}(0)}{\int}(\frac{\lambda_{i} }{1+\lambda_{i}^2r^{2}})^{\alpha'}(\frac{ \lambda_{j} }{1 +\lambda_{j}^{2}\gamma_{n}G_{ a _{j}}^{\frac{2}{2-n}}(\exp_{ a _{i}}x)})^{\beta'}\\
& +
C  \frac{1}{\lambda_{i}^{\alpha'}}\underset{B_{c}(0)}{\int}(\frac{ \lambda_{j} }{1+\lambda_{j}^2r^{2}})^{\beta'}
+
O(  \frac{1}{\lambda_{i}^{\alpha'}} \frac{1}{\lambda_{j}^{\beta'}} ) \\
= &
C \underset{B_{ c\lambda_{i} }(0)}{\int}(\frac{1}{1+r^{2}})^{\alpha'}
(\frac{1}{\frac{\lambda_{i} }{ \lambda_{j} }
+
\lambda_{i}\lambda_{j}\gamma_{n}G_{ a _{j}}^{\frac{2}{2-n}}(\exp_{ a _{i}}\frac{x}{\lambda_{i}})})^{\beta'}\\
& +
C  \frac{1}{\lambda_{i}^{\alpha'}} \frac{1}{\lambda_{j}^{n-\beta'}} 
\underset{B_{ c\lambda_{j} }(0)}{\int}(\frac{1}{1+r^{2}})^{\beta'}
+
o(\eps_{i,j}^{\beta}).
\end{split}\end{equation}
Thus by 
$
\int_{B_{ c\lambda_{j} }(0)}(\frac{1}{1+r^{2}})^{\beta'}
\leq 
C \frac{1}{\lambda_{j} ^{2\beta'-n}}
$ 
we get
\begin{equation}\begin{split}
\int  & \varphi_{i}^{\alpha}  \varphi_{j}^{\beta}  \\
\leq &
C\underset{B_{ c\lambda_{i} }(0)}{\int}(\frac{1}{1+r^{2}})^{\alpha'}(\frac{1}{\frac{\lambda_{i} }{ \lambda_{j} }+\lambda_{i}\lambda_{j}\gamma_{n}G_{ a _{j}}^{\frac{2}{2-n}}(\exp_{ a _{i}} \frac{x}{\lambda_{i}})})^{\beta'} \\
& +
o(\eps_{i,j}^{\beta}).
\end{split}\end{equation}
This shows the claim in cases
\begin{equation}\begin{split}
\frac{ \lambda_{j} }{ \lambda_{i} }+\frac{\lambda_{i} }{ \lambda_{j} }
+
\lambda_{i}\lambda_{j}\gamma_{n}G_{g_{0}}^{\frac{2}{2-n}}( a _{i}, a _{j})
\sim \frac{\lambda_{i}}{\lambda_{j}}
\; \text{ or }\;
d( a _{i}, a _{j})>3c.
\end{split}\end{equation}
Else we may assume $d( a _{i}, a _{j})<3c$ and 
\begin{equation}\begin{split}
\frac{ \lambda_{j} }{ \lambda_{i} }+\frac{\lambda_{i} }{ \lambda_{j} }+\lambda_{i}\lambda_{j}\gamma_{n}G_{g_{0}}^{\frac{2}{2-n}}( a _{i}, a _{j}) 
& \sim 
\lambda_{i}\lambda_{j}d^{2}( a _{i}, a _{j}).
\end{split}\end{equation}
We then get with $\mathcal{B}=[\frac{1}{2}d( a _{i}, a _{j})\leq \vert  \frac{x}{\lambda_{i}}\vert \leq 2 d( a _{i}, a _{j})]$
\begin{equation}\begin{split}
\int &\varphi_{i}^{\alpha}\varphi_{j}^{\beta} \\
\leq &
C \underset{\mathcal{B}}{\int}
(\frac{1}{1+r^{2}})^{\alpha'}(\frac{1}{\frac{\lambda_{i} }{ \lambda_{j} }+\lambda_{i}\lambda_{j}d^{2}( a _{j}, \exp_{ a _{i}}( \frac{x}{\lambda_{i}}))})^{\beta'} +
O(\eps_{i,j}^{\beta})\\
\leq &
C
(\frac{1}{1+\vert\lambda_{i}d( a _{i}, a _{j})\vert^{2}})^{\alpha'}
\underset{ [\vert  \frac{x}{\lambda_{i}}\vert \leq 4 d( a _{i}, a _{j})]}{\int}
(\frac{1}{\frac{\lambda_{i} }{ \lambda_{j} }+\frac{ \lambda_{j} }{ \lambda_{i} }r^{2}})^{\beta'} \\
& +
O(\eps_{i,j}^{\beta})\\
\leq &
C
\frac{(\frac{ \lambda_{j} }{ \lambda_{i} })^{\beta'-n}}{(1+\vert\lambda_{i}d( a _{i}, a _{j})\vert^{2})^{\alpha'}}
 \underset{ [r \leq 4 \lambda_{j}d( a _{i}, a _{j})]}{\int}
(\frac{1}{1+r ^{2}})^{\beta'}
+
O(\eps_{i,j}^{\beta}).
\end{split}\end{equation}
Note, that in case 
$\lambda_{j}d( a _{i}, a _{j})$ remains bounded, we are done. Else
\begin{equation}\begin{split}
\int \varphi_{i}^{\alpha}\varphi_{j}^{\beta}
\leq &
C
\frac{(\frac{ \lambda_{j} }{ \lambda_{i} })^{\beta'-n}(\lambda_{j}d( a _{i}, a _{j}))^{n-2\beta'}}{(1+\vert\lambda_{i}d( a _{i}, a _{j})\vert^{2})^{\alpha'}}
+
O(\eps_{i,j}^{\beta}) \\
\leq &
C
(\frac{1}{1+\vert\lambda_{i}d( a _{i}, a _{j})\vert^{2}})^{\alpha'-\frac{n}{2}+\beta'}
(\frac{\lambda_{i} }{ \lambda_{j} })^{\beta'}+
O(\eps_{i,j}^{\beta}),
\end{split}\end{equation}
whence due to $\alpha'>\frac{n}{2}$ the claim follows.

 \item [$(\beta)$]\quad $\eps_{i,j}^{\frac{2}{2-n}}\sim \frac{ \lambda_{j} }{ \lambda_{i} }$. 
\smallskip \\
We estimate for $c>0$ small
\begin{equation}\begin{split}
\int & \varphi_{i}^{\alpha}\varphi_{j}^{\beta} \\
\leq &
C
\int_{B_{ c\lambda_{j} }(0)}
(\frac{1}{\frac{ \lambda_{j} }{ \lambda_{i} }+\lambda_{i}\lambda_{j}\gamma_{n}G_{ a _{i}}^{\frac{2}{2-n}}( \exp_{ a _{j}}( \frac{x}{\lambda_{j}} ))})^{\alpha'} 
(\frac{1}{1+r^{2}})^{\beta'} \\
& +
C \frac{1}{\lambda_{j}^{\beta'}} \int_{B_{c}(0)}(\frac{\lambda_{i} }{1+\lambda_{i}^2r^{2}})^{\alpha'}
+
O(  \frac{1}{\lambda_{i}^{\alpha'}} \frac{1}{\lambda_{j}^{\beta'}} ),
\end{split}\end{equation}
which by
$
\int_{B_{c}(0)}(\frac{\lambda_{i} }{1+\lambda_{i}^2r^{2}})^{\alpha'}
\leq
C\frac{1}{\lambda_{i}^{n-\alpha'}} =  C\frac{1}{\lambda_{i} ^{\beta'}}
$
gives

\begin{equation}\begin{split}
 \int & \varphi_{i}^{\alpha}\varphi_{j}^{\beta} \\
\leq &
C
\int_{B_{ \lambda_{j}c }(0)}
(\frac{1}{\frac{ \lambda_{j} }{ \lambda_{i} }+\lambda_{i}\lambda_{j}\gamma_{n}G_{a_{i}}^{\frac{2}{2-n}}(\exp_{ a _{j}}( \frac{x}{\lambda_{j}}))})^{\alpha'} 
(\frac{1}{1+r^{2}})^{\beta'} \\
& +
o(\eps_{i,j}^{\beta}).
\end{split}\end{equation}
By assumption $d( a _{i}, a _{j})\leq  \frac{2}{\lambda_{i}} $, whence we may replace as before
\begin{equation}\begin{split}
\gamma_{n}G_{ a _{i}}^{\frac{2}{2-n}}(\exp_{ a _{j}}( \frac{x}{\lambda_{j}}))\sim d^{2}( a _{i}, \exp_{ a _{j}}( \frac{x}{\lambda_{j}}))
\;\text{ on }\;
B_{ \lambda_{j}c }(0).
\end{split}\end{equation}
Thus for $\gamma>3$
\begin{equation}\begin{split}
\int & \varphi_{i}^{\alpha}\varphi_{j}^{\beta} \\
\leq &
C
\underset{[\gamma\frac{ \lambda_{j} }{ \lambda_{i} }\leq \vert x \vert\leq  \lambda_{j}c ]}{\int}
(\frac{1}{\frac{ \lambda_{j} }{ \lambda_{i} }+\lambda_{i}\lambda_{j}d^{2}( a _{i}, \exp_{ a _{j}}( \frac{x}{\lambda_{j}}))})^{\alpha'} 
(\frac{1}{1+r^{2}})^{\beta'} \\
& +C
(\frac{\lambda_{i} }{ \lambda_{j} })^{\alpha'}
\int_{[\vert x \vert<\gamma\frac{ \lambda_{j} }{ \lambda_{i} }]}
(\frac{1}{1+r^{2}})^{\beta'} 
+
o(\eps_{i,j}^{\beta}) \\
\leq &
C
\int_{ [\gamma\frac{ \lambda_{j} }{ \lambda_{i} }\leq \vert x \vert\leq  \lambda_{j}c ]}
(\frac{1}{\frac{ \lambda_{j} }{ \lambda_{i} }+\frac{\lambda_{i} }{ \lambda_{j} }r^{2}})^{\alpha'} 
(\frac{1}{1+r^{2}})^{\beta'} \\
& +C
(\frac{\lambda_{i} }{ \lambda_{j} })^{\alpha'}
(\frac{ \lambda_{j} }{ \lambda_{i} })^{n-2\beta} 
+
o(\eps_{i,j}^{\beta}),
\end{split}\end{equation}
since for $\vert x \vert \geq \gamma\frac{ \lambda_{j} }{ \lambda_{i} }$ we may assume using $d( a _{i}, a _{j})\leq \frac{1}{\lambda_{i}} $
\begin{equation}\begin{split}
d( a _{i}, \exp_{ a _{j}}( \frac{x}{\lambda_{j}}))\geq  \frac{r}{\lambda_{j}}.
\end{split}\end{equation}
Therefore
\begin{equation}\begin{split}
\int \varphi_{i}^{\alpha}\varphi_{j}^{\beta}
\leq &
C
(\frac{ \lambda_{j} }{ \lambda_{i} })^{\alpha'}\int_{ [\vert x \vert \geq \gamma\frac{ \lambda_{j} }{ \lambda_{i} }]}
r^{-2n}
+
O(\eps_{i,j}^{\beta}) 
= 
O(\eps_{i,j}^{\beta}).
\end{split}\end{equation}

\end{enumerate}
\item 
By symmetry we may assume $\frac{1}{\lambda_{i}}\leq \frac{1}{\lambda_{j}}$ and thus
\begin{equation}\begin{split}
\eps_{i,j}^{\frac{2}{2-n}}\sim \frac{\lambda_{i} }{ \lambda_{j} }
\; \vee \;
\eps_{i,j}^{\frac{2}{2-n}} \sim \lambda_{i}\lambda_{j}\gamma_{n}G_{g_{0}}^{\frac{2}{2-n}}( a _{i}, a _{j})
\end{split}\end{equation}
We estimate for $c>0$ small
\begin{equation}\begin{split}
\int \varphi_{i}^{\frac{n}{n-2}}&\varphi_{j}^{\frac{n}{n-2}} \\
\leq &
C \int_{B_{c}(0)}(\frac{\lambda_{i} }{1+\lambda_{i}^2r^{2}})^{\frac{n}{2}}
(\frac{ \lambda_{j} }{1 +\lambda_{j}^{2}\gamma_{n}G_{ a _{j}}^{\frac{2}{2-n}}(\exp_{ a _{i}}x)})^{\frac{n}{2}}\\
& +
C  \frac{1}{\lambda_{i}^{\frac{n}{2}}}
\int_{B_{c}(0)}(\frac{ \lambda_{j} }{1+\lambda_{j}^2r^{2}})^{\frac{n}{2}}
+
O(  \frac{1}{\lambda_{i}^{\frac{n}{2}}} \frac{1}{\lambda_{j}^{\frac{n}{2}}} ) \\
= &
C \int_{B_{ c\lambda_{i} }(0)}(\frac{1}{1+r^{2}})^{\frac{n}{2}}
(\frac{1}{\frac{\lambda_{i} }{ \lambda_{j} }
+
\lambda_{i}\lambda_{j}\gamma_{n}G_{ a _{j}}^{\frac{2}{2-n}}(\exp_{ a _{i}}( \frac{x}{\lambda_{i}}))})^{\frac{n}{2}}\\
& +
O(\ln \lambda_{j}\eps_{i,j}^{\frac{n}{n-2}}).
\end{split}\end{equation}
Thus in cases
\begin{equation}\begin{split}
\frac{ \lambda_{j} }{ \lambda_{i} }+\frac{\lambda_{i} }{ \lambda_{j} }
+
\lambda_{i}\lambda_{j}\gamma_{n}G_{g_{0}}^{\frac{2}{2-n}}( a _{i}, a _{j})
\sim \frac{\lambda_{i}}{\lambda_{j}}
\; \text{ or }\;
d( a _{i}, a _{j})>3c
\end{split}\end{equation}
we obtain
\begin{equation}\begin{split}
\int \varphi_{i}^{\frac{n}{n-2}}\varphi_{j}^{\frac{n}{n-2}}
\leq &
C \ln \lambda_{i}\eps_{i,j}^{\frac{n}{n-2}}
+
C\ln \lambda_{j}\eps_{i,j}^{\frac{n}{n-2}}
\leq 
C\ln (\lambda_{i}\lambda_{j})\eps_{i,j}^{\frac{n}{n-2}},
\end{split}\end{equation}
thus $\int \varphi_{i}\varphi_{j}=O(\eps_{i,j}^{\frac{n}{n-2}}\ln \eps_{i,j})$.
Else we may assume  $d( a _{i}, a _{j})<3c$ and 
\begin{equation}\begin{split}
\frac{ \lambda_{j} }{ \lambda_{i} }+\frac{\lambda_{i} }{ \lambda_{j} }+\lambda_{i}\lambda_{j}\gamma_{n}G^{\frac{2}{2-n}}( a _{i}, a _{j}) 
& \sim 
\lambda_{i}\lambda_{j}d^{2}( a _{i}, a _{j}).
\end{split}\end{equation}
We then get with $\mathcal{B}=[\frac{1}{2}d( a _{i}, a _{j})\leq \vert  \frac{x}{\lambda_{i}}\vert \leq 2 d( a _{i}, a _{j})]$
\begin{equation}\begin{split}
\int \varphi_{i}^{\frac{n}{n-2}}&\varphi_{j}^{\frac{n}{n-2}}\\
\leq &
C \underset{\mathcal{B}}{\int}
(\frac{1}{1+r^{2}})^{\frac{n}{2}}(\frac{1}{\frac{\lambda_{i} }{ \lambda_{j} }+\lambda_{i}\lambda_{j}d^{2}( a _{j}, \exp_{ a _{i}}( \frac{x}{\lambda_{i}}))})^{\frac{n}{2}}
\\ & +
O(\eps_{i,j}^{\frac{n}{n-2}}\ln \eps_{i,j})\\
\leq &
C
(\frac{1}{1+\vert\lambda_{i}d( a _{i}, a _{j})\vert^{2}})^{\frac{n}{2}}
\underset{ [\vert  \frac{x}{\lambda_{i}}\vert \leq 4 d( a _{i}, a _{j})]}{\int}
(\frac{1}{\frac{\lambda_{i} }{ \lambda_{j} }+\frac{ \lambda_{j} }{ \lambda_{i} }r^{2}})^{\frac{n}{2}} \\
& +
O(\eps_{i,j}^{\frac{n}{n-2}}\ln \eps_{i,j})\\
\leq &
C
(\frac{1}{1+\vert\lambda_{i}d( a _{i}, a _{j})\vert^{2}})^{\frac{n}{2}}(\frac{ \lambda_{i} }{ \lambda_{j} })^{\frac{n}{2}}
 \underset{ [r \leq 4 \lambda_{j}d( a _{i}, a _{j})]}{\int}
(\frac{1}{1+r ^{2}})^{\frac{n}{2}}\\
& +
O(\eps_{i,j}^{\frac{n}{n-2}}\ln \eps_{i,j})\\
\leq &
C
\eps_{i,j}^{\frac{n}{n-2}}
\ln(\lambda_{j}d(a_{i},a_{j}))
+
O(\eps_{i,j}^{\frac{n}{n-2}}\ln \eps_{i,j}).
\end{split}\end{equation}
The claim follows, as $\lambda_{j}\leq \lambda_{i}$ by assumption.

 \item
$\eps_{i,j}=O(\eps_{i,j})$ is trivial and $\lambda_{i}\partial_{\lambda_{i}}\eps_{i,j}=O(\eps_{i,j})$ follows readily due to
\begin{equation}\begin{split}
\lambda_{i}\partial_{\lambda_{i}}\eps_{i,j}
= &
\frac{2-n}{2}\eps_{i,j}
\frac
{
\frac{\lambda_{i}}{\lambda_{j}}-\frac{\lambda_{j}}{\lambda_{i}}+\lambda_{i}\lambda_{j}\gamma_{n}G_{g_{0}}^{\frac{2}{2-n}}(a_{i},a_{j})
}
{
\frac{\lambda_{i}}{\lambda_{j}}+\frac{\lambda_{j}}{\lambda_{i}}+\lambda_{i}\lambda_{j}\gamma_{n}G_{g_{0}}^{\frac{2}{2-n}}(a_{i},a_{j})
}
.
\end{split}\end{equation}
Last 
$
\frac{1}{\lambda_{i}}\nabla_{a_{i}}\eps_{i,j}
=
O(\eps_{i,j})
$
follows from 
\begin{equation}\begin{split}\label{nablaai/lambdai_epsij}
\frac{1}{\lambda_{i}}\nabla_{a_{i}}\eps_{i,j}
= &
\frac{2-n}{2}\eps_{i,j}
\frac
{
\lambda_{j}\gamma_{n}\nabla_{a_{i}}G_{g_{0}}^{\frac{2}{2-n}}(a_{i},a_{j})
}
{
\frac{\lambda_{i}}{\lambda_{j}}+\frac{\lambda_{j}}{\lambda_{i}}+\lambda_{i}\lambda_{j}\gamma_{n}G_{g_{0}}^{\frac{2}{2-n}}(a_{i},a_{j})
}.
\end{split}\end{equation}
immediately in case $d(a_{i},a_{j})>c>0$. In the contrary case we estimate
\begin{equation}\begin{split}
\frac
{
\lambda_{j}\gamma_{n}\vert \nabla_{a_{i}}G_{g_{0}}^{\frac{2}{2-n}}(a_{i},a_{j})\vert
}
{
\frac{\lambda_{i}}{\lambda_{j}}+\frac{\lambda_{j}}{\lambda_{i}}+\lambda_{i}\lambda_{j}\gamma_{n}G_{g_{0}}^{\frac{2}{2-n}}(a_{i},a_{j})
}
\leq &
\frac
{
\frac{\lambda_{j}}{\lambda_{i}}+\lambda_{i}\lambda_{j}\gamma_{n} \vert \nabla_{a_{i}}G_{g_{0}}^{\frac{2}{2-n}}(a_{i},a_{j})\vert^{2}
}
{
\frac{\lambda_{i}}{\lambda_{j}}+\frac{\lambda_{j}}{\lambda_{i}}+\lambda_{i}\lambda_{j}\gamma_{n}G_{g_{0}}^{\frac{2}{2-n}}(a_{i},a_{j})
}
\end{split}\end{equation}
with the right hand side being bounded for $d(a_{i},a_{j})$ small.
\end{enumerate}
\end{proof}

\begin{proof}[\textbf{Proof of proposition \ref{prop_optimal_choice}}(Cf. \cite{BahriCoronCriticalExponent}, Appendix A)]\label{p_optimal_choice}$_{}$\\
Let us denote by $w(\eps)$ any quantity, for which $\vert w(\eps)\vert\overset{\eps\to 0}{\-}0$ and consider for
\begin{equation}
\begin{split}
u\in V(\omega,p, \eps)\; \text{ with }\;\eps\- 0 
\end{split}
\end{equation} 
a representation
\begin{equation}
\begin{split}
u=u_{\hat\alpha, \hat\beta}+\hat\alpha^{i}\var_{\hat a_{i}, \hat \lambda_{i}}+\hat v, \,
(\hat\alpha, \hat\beta_{k}, \hat\alpha_{i}, \hat a_{i}, \hat\lambda_{i})\in A_{u}(\omega,p, \eps), \Vert \hat v \Vert\leq \eps.
\end{split}
\end{equation} 
Since $A_{u}(\omega,p, \eps)\subset A_{u}(\omega,p,2\eps_{0})$ we have
\begin{equation}
\begin{split}
\inf
_
{
(\tilde \alpha, \tilde\beta_{k}, \tilde\alpha_{i}, \tilde a_{i}, \tilde\lambda_{i})\in A_{u}(\omega,p,2\eps_{0}) 
}
\int 
Ku^{\frac{4}{n-2}}
\vert 
u
-
u_{\tilde \alpha, \tilde \beta}
-
\tilde\alpha^{i}\varphi_{\tilde a_{i}, \tilde \lambda_{i}}
\vert^{2}
=
w(\eps),
\end{split}
\end{equation} 
whence we may consider $(\tilde \alpha, \tilde \beta_{k}, \tilde \alpha_{i}, \tilde a_{i}, \tilde \lambda_{i})\in A_{u}(\omega,p,2\eps_{0})$ such, that
\begin{equation}
\begin{split}
\int Ku^{\frac{4}{n-2}}\vert u-u_{\tilde \alpha, \tilde \beta}-\tilde \alpha^{i}\var_{\tilde a_{i}, \tilde \lambda_{i}}\vert^{2}
=
w(\eps).
\end{split}
\end{equation} 
Expanding this gives in a first step
\begin{equation}
\begin{split}
\int K(\hat\alpha\omega+\hat\alpha^{i}\var_{\hat a_{i}, \hat\lambda_{i}})^{\frac{4}{n-2}}
\vert
u_{\hat\alpha, \hat\beta}-u_{\tilde \alpha, \tilde \beta}
+
\hat\alpha^{i}\var_{\hat a_{i}, \hat \lambda_{i}}-\tilde \alpha^{i}\var_{\tilde a_{i}, \lambda_{i}}
\vert^{2}
=
w(\eps)
\end{split}
\end{equation} 
and using lemma \ref{lem_interactions} and proposition \ref{prop_smoothness_of_u_a_b} we derive
\begin{equation}
\begin{split}
w(\eps)
= &
\int K\omega^{\frac{4}{n-2}}
\vert u_{\hat \alpha, \hat \beta}-u_{\tilde \alpha, \tilde \beta}-\tilde \alpha^{i}\var_{\tilde a_{i}, \tilde \lambda_{i}}\vert^{2} \\
& +
\sum_{i} \int K\var_{\hat a_{i}, \hat \lambda_{i}}^{\frac{4}{n-2}}
\vert \hat \alpha_{i}\var_{\hat a_{i}, \hat \lambda_{i}}-\tilde \alpha^{j}\var_{\tilde a_{j}, \tilde \lambda_{j}}\vert^{2}
\end{split}
\end{equation} 
Consequently for at least one $j=j_{i}$ the quantity
\begin{equation}\begin{split}
\frac{\tilde \lambda_{j_{i}}}{\hat \lambda_{i}}
+
\frac{\hat \lambda_{i}}{\tilde \lambda_{j_{i}}}
+
\tilde \lambda_{j_{i}}\hat \lambda_{i}\gamma_{n}G_{g_{0}}^{\frac{2}{2-n}}(\tilde a_{j_{i}}, \hat  a_{i})
\end{split}\end{equation}
has to stay bounded, whereas on the other hand 
\begin{equation}\begin{split}
\hat \eps_{i,j}
=
(\frac{\hat \lambda_{i}}{\hat \lambda_{j}}+\frac{\hat \lambda_{j}}{\hat \lambda_{i}}
+
\hat \lambda_{i}\hat \lambda_{j}\gamma_{n}G_{g_{0}}^{\frac{2}{2-n}}(\hat a_{i}, \hat a_{j}))^{\frac{2-n}{2}}
<\eps.
\end{split}\end{equation}
Thus for any $j= i, \ldots,p$ there exists exactly one $j_{i}\in \{1, \ldots, p\}$ such,  that
\begin{equation}\begin{split}
\frac{\tilde \lambda_{j_{i}}}{\hat \lambda_{i}}
+
\frac{\hat \lambda_{i}}{\tilde \lambda_{j_{i}}}
+
\tilde \lambda_{j_{i}}\hat \lambda_{i}\gamma_{n}G_{g_{0}}^{\frac{2}{2-n}}(\tilde a_{j_{i}}, \hat a_{i})
\end{split}\end{equation}
remains bounded and we may assume $j_{i}=i$. From this we deduce
\begin{equation}
\begin{split}
w(\eps)
= &
\int K\omega^{\frac{4}{n-2}}
\vert u_{\hat \alpha, \hat \beta}-u_{\tilde \alpha, \tilde \beta}\vert^{2} 
+
\sum_{i} \int K\var_{\hat a_{i}, \hat \lambda_{i}}^{\frac{4}{n-2}}
\vert \hat \alpha_{i}\var_{\hat a_{i}, \hat \lambda_{i}}
-
\tilde \alpha_{i}\var_{\tilde a_{i}, \tilde \lambda_{i}}\vert^{2}.
\end{split}
\end{equation} 
Note, that
\begin{equation}
\begin{split}
 \int K\omega^{\frac{4}{n-2}}
\vert u_{\hat \alpha, \hat \beta}-u_{\tilde \alpha, \tilde \beta}\vert^{2} 
= &
\vert \hat \alpha-\tilde \alpha \vert^{2} \int K\omega^{\frac{2n}{n-2}}
+
\sum_{i}\vert \hat \alpha\hat \beta{i}-\tilde \alpha\tilde \beta_{i}\vert^{2}\int K\omega^{\frac{4}{n-2}}\mathrm{e}_{i}^{2} \\
& +
\int K \omega^{\frac{4}{n-2}}\vert \hat \alpha h(\hat \beta)-\tilde \alpha h(\tilde \beta)\vert^{2},
\end{split}
\end{equation} 
whence due to $\Vert h\hat \beta\Vert=O(\Vert\hat  \beta \Vert^{2})$ we obtain
\begin{equation}
\begin{split}
 \int K\omega^{\frac{4}{n-2}}
\vert u_{\hat \alpha, \hat \beta}-u_{\tilde \alpha, \tilde \beta}\vert^{2} 
\geq &
C(\vert \hat \alpha-\tilde \alpha \vert^{2} +\Vert \hat \beta -\tilde \beta \Vert^{2}).
\end{split}
\end{equation}
Moreover in $g_{\hat a_{i}}$ normal coordinates with 
\begin{equation}
\begin{split}
\gamma(\tau)
=
\tau(\tilde\alpha_{i}, \tilde a_{i}, \tilde \lambda_{i})
+
(1-\tau)(\hat \alpha_{i}, \hat a_{i}, \hat \lambda_{i})
\end{split}
\end{equation} 
we have for some $\tau\in (0,1)$
\begin{equation}
\begin{split}
\tilde \alpha_{i}\var_{\tilde a_{i}, \tilde \lambda_{i}}
& -
\hat \alpha_{i}\var_{\hat a_{i}, \hat \lambda_{i}} \\
= &
\begin{pmatrix}
\partial_{\alpha}\\
\nabla_{a}\\
\partial_{\lambda}
\end{pmatrix}
(\alpha\var_{a, \lambda})
\lfloor_{(\alpha,a, \lambda=\gamma(0))} 
\begin{pmatrix}
\tilde \alpha-\hat \alpha \\
\tilde a-\hat a \\
\tilde \lambda-\hat \lambda
\end{pmatrix}\\
& +
\begin{pmatrix}
\partial^{2}_{\alpha} & \partial_{\alpha}\nabla_{a}&\partial_{\alpha}\partial_{\lambda} \\
\nabla_{a}\partial_{\lambda}& \nabla^{2}_{a} &\nabla_{a}\partial_{\lambda} \\
\partial_{\lambda}&\partial_{\lambda}\nabla_{a}&\partial^{2}_{\lambda}
\end{pmatrix}
(\alpha\var_{a, \lambda})
\lfloor_{(\alpha,a, \lambda)=\gamma(\tau)}
\begin{pmatrix}
\tilde \alpha-\hat \alpha \\
\tilde a-\hat a \\
\tilde \lambda-\hat \lambda
\end{pmatrix}^{2},
\end{split}
\end{equation} 
whence due to lemma \ref{lem_interactions} (i) and $c<\frac{\tilde \lambda_{i}}{\lambda_{i}}<C$ we obtain
\begin{equation}\label{phi-phitilde}
\begin{split}
\Vert \tilde \alpha_{i}\var_{\tilde a_{i}, \tilde \lambda_{i}}
-
\hat \alpha_{i}\var_{\hat a_{i}, \hat \lambda_{i}} 
& -
\begin{pmatrix}
\var_{\hat a_{i}, \hat \lambda_{i}}\\
\frac{1}{\hat \lambda_{i}}\nabla_{\hat a_{i}}\var_{\hat a_{i}, \hat \lambda_{i}}\\
\hat \lambda_{i}\partial_{\hat \lambda_{i}}\var_{\hat a_{i}, \hat \lambda_{i}}
\end{pmatrix}
\begin{pmatrix}
\tilde \alpha_{i}-\hat \alpha_{i} \\
\hat \lambda_{i}(\tilde a_{i}-\hat a_{i}) \\
\frac{\tilde \lambda_{i}-\hat \lambda_{i}}{\hat \lambda_{i}}
\end{pmatrix}
\Vert
\\
= &
O
(
\vert \tilde \alpha_{i}-\hat \alpha_{i}\vert^{2} 
+
\hat \lambda_{i}^{2}\vert \tilde a_{i}-\hat a_{i}\vert^{2}
+
\vert\frac{\tilde \lambda_{i}-\hat \lambda_{i}}{\hat \lambda_{i}}\vert^{2}).
\end{split}
\end{equation} 
So lemma \ref{lem_interactions} (ii) and (iv) yield
\begin{equation}
\begin{split}
\int K\var_{\hat a_{i}, \hat \lambda_{i}}^{\frac{4}{n-2}} 
\vert \hat \alpha_{i}\var_{\hat a_{i}, \hat \lambda_{i}}
& -
\tilde \alpha_{i}\var_{\tilde a_{i}, \tilde \lambda_{i}}\vert^{2} \\
\geq & 
C
(
\vert \hat \alpha_{i}-\tilde \alpha_{i}\vert^{2}
+
\hat \lambda_{i}^{2}\vert \hat a_{i}-\tilde a_{i}\vert^{2}
+
\vert \frac{\tilde \lambda_{i}}{\hat \lambda_{i}}-1\vert^{2}
)
+
w(\eps).
\end{split}
\end{equation}
Collecting terms we arrive at
\begin{equation}\label{convexity}
\begin{split}
\vert \hat \alpha-\tilde \alpha \vert^{2} 
& +
\Vert \hat \beta -\tilde \beta \Vert^{2} \\
& +
\sum_{i}
(
\vert \hat \alpha_{i}-\tilde \alpha_{i}\vert^{2}
+
\hat \lambda_{i}^{2}\vert \hat a_{i}-\tilde a_{i}\vert^{2}
+
\vert \frac{\tilde \lambda_{i}}{\hat \lambda_{i}}-1\vert^{2}
)
=
w(\eps).
\end{split}
\end{equation}
Consequently, if we consider a minimizing sequence 
\begin{equation}
\begin{split}
(\tilde \alpha_{l}, \tilde \beta_{k,l}, \tilde \alpha_{i,l}, \tilde a_{i,l}, \tilde \lambda_{i,l})_{l}
\subseteq 
A_{u}(\omega,p,2\eps_{0})
\end{split}
\end{equation} 
for the functional
\begin{equation}
\begin{split}
\int Ku^{\frac{4}{n-2}}\vert u-u_{\tilde \alpha, \tilde \beta}-\tilde \alpha^{i}\var_{\tilde a_{i}, \tilde \lambda_{i}}\vert^{2} 
\end{split}
\end{equation} 
with $u\in V(\omega,p, \eps)$ fixed, e.g.
\begin{equation}
\begin{split}
u=u_{\hat \alpha, \hat \beta}+\hat \alpha^{i}\var_{\hat a_{i}, \hat \lambda_{i}}+\hat v, \;
(\hat \alpha, \hat \beta_{k}, \hat \alpha_{i}, \hat a_{i}, \hat \lambda_{i})\in A_{u}(\omega,p, \eps)
,
\Vert\hat  v \Vert\leq \eps,
\end{split}
\end{equation} 
then there necessarily holds 
\begin{equation}
\begin{split}
(\tilde \alpha_{l}, \tilde \beta_{k,l}, \tilde \alpha_{i,l}, \tilde a_{i,l}, \tilde \lambda_{i,l})_{l}
\subseteq 
A_{u}(\omega,p, \eps+w(\eps)).
\end{split}
\end{equation}
for all $l$ sufficiently large. Moreover, since $\tilde \lambda_{i,l}\overset{l\to \infty}{\-} \infty$ is not possible due to
\begin{equation}
\begin{split}
\vert \frac{\tilde \lambda_{i,l}}{\hat \lambda_{i}}-1\vert^{2}=w(\eps)
\end{split}
\end{equation} 
the infimum of the functional is attained for some 
\begin{equation}
\begin{split}
(\alpha,  \beta_{k}, \alpha_{i}, a_{i}, \lambda_{i})\in \overline A_{u}(\omega,p, \eps+w(\eps))\subset A_{u}(\omega,p, \eps_{0}),
\end{split}
\end{equation}
provided $\eps\ll \eps_{0}$ is sufficiently small.
\smallskip
\\
To show uniqueness we argue by contradiction and assume, that for some
\begin{equation}
\begin{split}
u=u_{\hat \alpha, \hat \beta}+\hat \alpha^{i}\var_{\hat a_{i}, \hat \lambda_{i}}+\hat v, \,
(\hat \alpha, \hat \beta_{k}, \hat \alpha_{i}, \hat a_{i}, \hat \lambda_{i} )
\in A_{u}(\omega,p, \eps), \Vert v \Vert<\eps,
\end{split}
\end{equation} 
in other words for some $u\in V(\omega,p, \eps)$ with suitable representation there exist
\begin{equation}\begin{split}
(\alpha, \beta_{k}, \alpha_{i},a_{i}, \lambda_{i})
,
(\tilde\alpha, \tilde\beta_{k}, \tilde\alpha_{i}, \tilde a_{i}, \tilde\lambda_{i})
\in A_{u}(\omega,p, \eps_{0})
\end{split}\end{equation}
such, that
\begin{equation}\begin{split}\label{*=tilde*=inf}  
\inf_{\bar\alpha, \bar \beta, \bar\alpha_{i}, \bar a_{i}, \bar\lambda_{i}} &
\int Ku^{\frac{4}{n-2}} 
\vert u-u_{\bar \alpha, \bar\beta}-\bar\alpha^{i}\varphi_{\bar a_{i}, \bar\lambda_{i}}\vert^{2}  \\
= &
\int Ku^{\frac{4}{n-2}}
\vert u-u_{\alpha, \beta}-\alpha^{i}\varphi_{a_{i}, \lambda_{i}}\vert^{2} Ku^{\frac{4}{n-2}} \\
= &
\int Ku^{\frac{4}{n-2}}
\vert u-u_{\tilde \alpha, \tilde \beta}-\tilde\alpha^{i}\varphi_{\tilde a_{i}, \tilde\lambda_{i}}\vert^{2} Ku^{\frac{4}{n-2}}
.
\end{split}\end{equation}
By what was shown before the quantities
\begin{equation}\begin{split}
&
A = \vert \tilde \alpha-\alpha \vert, \; B_{k}=\vert \tilde \beta_{k}-\beta_{k}\vert
, \\ 
&
A_{i}
=\vert\tilde \alpha_{i}-\alpha_{i}\vert
, \;
L_{i}=\vert\frac{\tilde \lambda_{i}}{\lambda_{i}}-1\vert
, \;
D_{i}^{2}=\tilde \lambda_{i}\lambda_{i}d^{2}(\tilde a_{i},a_{i}) 
\end{split}\end{equation}
are well defined and we will prove the proposition by showing 
\begin{equation}
\begin{split}
A,B_{k},A_{i},D_{i},L_{i}=w(\eps) 
\end{split}
\end{equation} 
and 
\begin{equation}\begin{split}\label{uniqueness_equation}
A +\sum_{k=1}^{m}B_{k} +\sum_{i=1}^{p}A_{i}+D_{i}+L_{i}  = 
o(A+\sum_{k=1}^{m}B_{k}+\sum_{i=1}^{p}A_{i}+D_{i}+L_{i}).
\end{split}\end{equation}
The first statement if rather obvious. Indeed \eqref{convexity} shows
\begin{equation}
\begin{split}
\vert \alpha-\hat \alpha\vert, \vert \tilde \alpha-\hat \alpha\vert=w(\eps),
\end{split}
\end{equation} 
so $A=w(\eps)$ and the same argument applies to $B_{k},A_{i},D_{i},L_{i}$ as well.
\smallskip \\
We are left with proving \eqref{uniqueness_equation}. Note, that
\begin{equation}\begin{split}\label{deltaajlj-dtildeajtildelj}
\varphi_{ a_{j}, \lambda_{j}}-\varphi_{\tilde a_{j}, \tilde \lambda_{j}}
= &
u_{ a_{j}}(\frac{\lambda_{j}}{1+\lambda_{j}^{2}\gamma_{n}G_{ a_{j}}^{\frac{2}{2-n}}})^{\frac{n-2}{2}}
-
u_{\tilde a_{j}}(\frac{\tilde \lambda_{j}}{1+\tilde \lambda_{j}^{2}\gamma_{n}G_{\tilde a_{j}}^{\frac{2}{2-n}}})^{\frac{n-2}{2}}\\
= &
\varphi_{ a_{j}, \lambda_{j}}
(
1
-
\frac{u_{\tilde a_{j}}}{u_{ a_{j}}}
(\frac{\tilde \lambda_{j}}{ \lambda_{j}}
\frac
{
1+  \lambda_{j}^{2}\gamma_{n}G_{a_{j}}^{\frac{2}{2-n}}
}
{
1+\tilde \lambda_{j}^{2}\gamma_{n}G_{\tilde a_{j}}^{\frac{2}{2-n}}
}
)^{\frac{n-2}{2}}
)
\end{split}\end{equation}
and therefore
\begin{equation}\begin{split}
\label{rough_estimate_delta_tilde_delta}
\vert \varphi_{ a_{j}, \lambda_{j}}-\varphi_{\tilde a_{j}, \tilde \lambda_{j}}\vert 
\leq c( D_{j} +  L_{j})\varphi_{a_{j}, \lambda_{j}}.
\end{split}\end{equation}

First we make use of
\begin{equation}\begin{split}
\partial_{\alpha}
\int 
Ku^{\frac{4}{n-2}}
\vert 
u
-
u_{\alpha, \beta}
-
\alpha^{i}\varphi_{a_{i}, \lambda_{i}}
\vert^{2}
=0.
\end{split}\end{equation}
Differentiating we obtain
\begin{equation}\begin{split}
0
= &
\int Ku^{\frac{4}{n-2}}
(u-u_{\alpha, \beta}-\alpha^{i}\varphi_{a_{i}, \lambda_{i}})
\partial_{\alpha}u_{\alpha, \beta}\\
= &
\int Ku^{\frac{4}{n-2}}(u_{\tilde \alpha, \tilde \beta}-u_{\alpha, \beta})\partial_{\alpha}u_{\alpha, \beta}
+
(\tilde \alpha^{i}-\alpha^{i})
\int Ku^{\frac{4}{n-2}}\varphi_{\tilde a_{i}, \tilde \lambda_{i}}\partial_{\alpha}u_{\alpha, \beta} \\
& +
\alpha^{i}\int Ku^{\frac{4}{n-2}}(\varphi_{\tilde a_{i}, \tilde \lambda_{i}}-\varphi_{a_{i}, \lambda_{i}})\partial_{\alpha}u_{\alpha, \beta} \\
& +
\int Ku^{\frac{4}{n-2}}\tilde v(\partial_{\alpha}u_{\alpha, \beta}-\partial_{\tilde \alpha}u_{\tilde \alpha, \tilde \beta}),
\end{split}\end{equation}
whence  
$
A=o(A+\sum^{m}_{k=1}B_{k}+\sum_{i=1}^{p}A_{i}+D_{i}+L_{i}).
$

Similarly we make use of 
\begin{equation}\begin{split}
\partial_{\beta_{k}}
\int 
Ku^{\frac{4}{n-2}}
\vert 
u
-
u_{\alpha, \beta}
-
\alpha^{i}\varphi_{a_{i}, \lambda_{i}}
\vert^{2}=0
\end{split}\end{equation}
yielding 
$
B_{k}=o(A+\sum^{m}_{k=1}B_{k}+\sum_{i=1}^{p}A_{i}+D_{i}+L_{i}).
$

We proceed using
\begin{equation}\begin{split}
\partial_{\alpha_{j}}\int Ku^{\frac{4}{n-2}} \vert u-u_{\alpha, \beta}-\alpha^{i}\varphi_{a_{i}, \lambda_{i}}\vert^{2}=0.
\end{split}\end{equation}
This gives
\begin{equation}\begin{split}
0 
= &
\int Ku^{\frac{4}{n-2}} ( u-u_{\alpha, \beta}-\alpha^{i}\varphi_{a_{i}, \lambda_{i}})\varphi_{a_{j}, \lambda_{j}}\\
= &
\int Ku^{\frac{4}{n-2}}\varphi_{a_{j}, \lambda_{j}}(u_{\tilde \alpha, \tilde \beta}-u_{\alpha, \beta})
+
(\tilde \alpha^{i}-\alpha^{i})\int Ku^{\frac{4}{n-2}}\varphi_{a_{j}, \lambda_{j}} \varphi_{a_{i}, \lambda_{i}} \\
& +
\tilde \alpha^{i}\int Ku^{\frac{4}{n-2}}\varphi_{a_{j}, \lambda_{j}}(\varphi_{\tilde a_{i}, \tilde \lambda_{i}}-\varphi_{a_{i}, \lambda_{i}}) \\
& +
\int Ku^{\frac{4}{n-2}}v(\varphi_{a_{j}, \lambda_{j}}-\varphi_{\tilde a_{j}, \tilde \lambda_{j}}),
\end{split}\end{equation}
whence due to 
\eqref{rough_estimate_delta_tilde_delta} and lemma \ref{lem_interactions} 
\begin{equation}\begin{split}
0 
= &
(\tilde \alpha_{j}-\alpha_{j})\int Ku^{\frac{4}{n-2}} \varphi_{a_{j}, \lambda_{j}}^{2}
+
\tilde \alpha_{j}K_{j}\alpha_{j}^{\frac{4}{n-2}}
\int \varphi^{\frac{n+2}{n-2}}_{a_{j}, \lambda_{j}}
(\varphi_{\tilde a_{j}, \tilde \lambda_{j}}-\varphi_{a_{j}, \lambda_{j}}) \\
& +
o(A+\sum^{m}_{k=1}B_{k}+\sum^{p}_{i=1}A_{i}+D_{i}+L_{i}).
\end{split}\end{equation}
Arguing as for \eqref{phi-phitilde} we obtain passing to $g_{a_{i}}$ normal coordinates
\begin{equation}
\begin{split}
\int \varphi^{\frac{n+2}{n-2}}_{a_{j}, \lambda_{j}}
(\varphi_{\tilde a_{j}, \tilde \lambda_{j}}-\varphi_{a_{j}\lambda_{j}})
= &
\int \varphi^{\frac{n+2}{n-2}}_{a_{j}, \lambda_{j}}
\frac{1}{\lambda_{j}}\nabla_{a_{j}}\varphi_{a_{j}, \lambda_{j}}
\, \lambda_{j}(\tilde a_{j}-a_{j})\\
& +
\int \varphi^{\frac{n+2}{n-2}}_{a_{j}, \lambda_{j}}
\lambda_{j}\partial_{\lambda_{j}}\var_{a_{j}, \lambda_{j}}
\,
(\frac{\tilde \lambda_{j}}{\lambda_{j}}-1) \\
& +
o(D_{j}+L_{j}),
\end{split} 
\end{equation} 
whence according to lemma \ref{lem_interactions} (iv) we obtain
\begin{equation}
\begin{split}
\int \varphi^{\frac{n+2}{n-2}}_{a_{j}, \lambda_{j}}
(\varphi_{\tilde a_{j}, \tilde \lambda_{j}}-\varphi_{a_{j}\lambda_{j}})
= &
o(D_{j}+L_{j}).
\end{split} 
\end{equation}
We conclude
\begin{equation}\begin{split}
A_{j}
= &
o(A+\sum^{m}_{k=1}B_{k}+\sum^{p}_{i=1}A_{i}+D_{i}+L_{i}).
\end{split}\end{equation}
Analogously one obtains 
\begin{equation}\begin{split}
L_{j},D_{j}=o(A+\sum^{m}_{k=1}B_{k}+\sum^{p}_{i=1}A_{i}+D_{i}+L_{i})
\end{split}\end{equation}
by exploiting 
\begin{equation}\begin{split}
\partial_{\lambda_{j}}\int Ku^{\frac{4}{n-2}} \vert u-u_{\alpha, \beta} -\alpha^{i}\varphi_{i}\vert^{2}=0
\end{split}\end{equation}
and
\begin{equation}\begin{split}
\nabla_{a_{j}}\int Ku^{\frac{4}{n-2}} \vert u-u_{\alpha, \beta} -\alpha^{i}\varphi_{i}\vert^{2}=0
\end{split}\end{equation}
using
\begin{equation}\begin{split}
\int \varphi_{a_{j}, \lambda_{j}}^{\frac{4}{n-2}} & \lambda_{j}\partial_{ \lambda_{j}}\varphi_{a_{j}, \lambda_{j}}
(
\varphi_{a_{j}, \lambda_{j}}-\varphi_{\tilde a_{j}, \tilde\lambda_{j}}
) \\
= &
-\int \varphi_{a_{j}, \lambda_{j}}^{\frac{4}{n-2}} 
\vert\lambda_{j}\partial_{ \lambda_{j}}\varphi_{a_{j}, \lambda_{j}}\vert^{2}
(\frac{\tilde \lambda_{j}}{\lambda_{j}}-1)
+
o(D_{j}+L_{j})
\end{split}\end{equation}
and 
\begin{equation}\begin{split}
\int \varphi_{a_{j}, \lambda_{j}}^{\frac{4}{n-2}} & \frac{1}{\lambda_{j}}\nabla_{ a_{j}}\varphi_{a_{j}, \lambda_{j}}
(
\varphi_{a_{j}, \lambda_{j}}-\varphi_{\tilde a_{j}, \tilde\lambda_{j}}
) \\
= &
-\int \varphi_{a_{j}, \lambda_{j}}^{\frac{4}{n-2}} 
(\frac{1}{\lambda_{j}}\nabla_{ a_{j}}\varphi_{a_{j}, \lambda_{j}})^{2}
\, \lambda_{j}(\tilde a_{j}-a_{j})
+
o(D_{j}+L_{j})
\end{split}\end{equation}
Finally we show smooth dependence. To that end consider
\begin{equation}
\begin{split}
F(u,(\bar \alpha, \bar \beta_{k}, \bar \alpha_{i}, \bar a_{i}, \bar \lambda_{i}))
=
\int Ku^{\frac{4}{n-2}}\vert u-u_{\bar \alpha, \bar \beta}-\bar \alpha^{i}\var_{\bar a_{i}, \bar \lambda_{i}}\vert^{2}.
\end{split}
\end{equation} 
If $(\alpha, \beta_{k}, \alpha_{i},a_{i}, \lambda_{i})$ denotes the minimizer constructed
for $u\in V(\omega,p, \eps)$, then
\begin{equation}
\begin{split}
D_{(\alpha, \beta_{k}, \alpha_{i},a_{i}, \lambda_{i})}F(u,(\alpha, \beta_{k}, \alpha_{i},a_{i}, \lambda_{i}))
=
0.
\end{split}
\end{equation} 
Moreover in view of lemma \ref{lem_interactions} we easily find, that
\begin{equation}
\begin{split}
D^{2}_{( \alpha, \beta_{k}, \alpha_{i}, a_{i}, \lambda_{i})}F(u,( \alpha, \beta_{k}, \alpha_{i}, a_{i}, \lambda_{i}))
>0 
\end{split}
\end{equation} 
is positive, provided $\eps>0$ is sufficiently small. Thus the implicit function theorem provides
a smooth parametrization of
\begin{equation}
\begin{split}
[D_{(\alpha, \beta_{k}, \alpha_{i},a_{i}, \lambda_{i})}F(u,(\alpha, \beta_{k}, \alpha_{i},a_{i}, \lambda_{i}))
=
0].
\end{split}
\end{equation} 
This proves the statement.  
\end{proof}


\end{document}